\documentclass{amsart}
\usepackage[margin = 1in]{geometry}
\usepackage{graphicx} 

\usepackage[utf8]{inputenc}

\usepackage{amsmath, amssymb, amsthm, mathtools}
\usepackage{mathrsfs, amsfonts}
\usepackage{color}
\usepackage{hyperref}
\hypersetup{hidelinks}
\usepackage{caption}

\usepackage{enumitem}
\setlist[enumerate,1]{label={(\roman*)}, itemsep=5pt, topsep=5pt}

\usepackage{stackrel}

\usepackage{tikz}
\usepackage{tikz-cd}
\usetikzlibrary{arrows}
\usetikzlibrary{decorations.pathmorphing}

\usepackage{adjustbox}

\usepackage{mathtools} 
\newcommand{\fixedxrightarrow}[2][2cm]{%
  \xrightarrow{\mathmakebox[#1]{#2}}
}
\newcommand{\fixedxleftarrow}[2][2cm]{%
  \xleftarrow{\mathmakebox[#1]{#2}}
}

\numberwithin{equation}{section} 
\makeatletter
\@addtoreset{equation}{section}
\makeatother

\newcommand{\mc}[1]{\mathcal{#1}}
\newcommand{\mb}[1]{\mathbb{#1}}

\newcommand{\A}{\mathbb{A}}

\DeclareMathOperator{\id}{id}
\DeclareMathOperator{\Hom}{Hom}
\DeclareMathOperator{\pr}{pr}
\DeclareMathOperator{\im}{im}
\DeclareMathOperator{\colim}{colim}
\DeclareMathOperator{\Fun}{Fun}
\DeclareMathOperator{\Cat}{Cat}

\DeclareMathOperator{\Span}{Span}
\DeclareMathOperator{\PSh}{PSh}
\DeclareMathOperator{\map}{map}
\DeclareMathOperator{\Fact}{Fact}
\DeclareMathOperator{\Sq}{Sq}
\DeclareMathOperator{\Gr}{Gr}
\DeclareMathOperator{\Pull}{Pull}

\DeclareMathOperator{\diag}{diag}

\newcommand{\Cpt}{\mathrm{Cpt}}
\newcommand{\CCpt}{\mathbb{C}\mathrm{pt}}
\newcommand{\op}{\mathrm{op}}
\newcommand{\Ani}{\mathrm{Spc}}
\DeclareMathOperator{\cnr}{cnr}
\DeclareMathOperator{\sd}{sd}
\DeclareMathOperator{\ascat}{ac}

\DeclareMathOperator{\Hor}{Hor}

\DeclareMathOperator{\Tw}{Tw}

\DeclareMathOperator{\Mor}{Mor}
\DeclareMathOperator{\End}{End}
\DeclareMathOperator{\CAlg}{CAlg}
\DeclareMathOperator{\Th}{TH}
\DeclareMathOperator{\Shv}{Shv}
\DeclareMathOperator{\ind}{ind}
\DeclareMathOperator{\coind}{coind}
\DeclareMathOperator{\Nm}{Nm}
\DeclareMathOperator{\Sp}{Sp}

\DeclareMathOperator{\Vect}{Vect}
\DeclareMathOperator{\Exch}{Exch}
\DeclareMathOperator{\Lan}{Lan}
\DeclareMathOperator{\Ver}{Ver}
\newcommand{\coWald}{\mathrm{coWald}}
\DeclareMathOperator{\Alg}{Alg}
\DeclareMathOperator{\fsd}{fsd}
\newcommand{\TwC}{\mc C^\mathrm{VB}}
\newcommand{\InvTwC}{\mc C^\mathrm{vVB}}

\usepackage[mathscr,mathcal]{euscript}

\usepackage[nameinlink]{cleveref}

\theoremstyle{plain}
\newtheorem{theorem}{Theorem}[section]
\newtheorem{proposition}[theorem]{Proposition}
\AddToHook{env/proposition/begin}{\crefalias{theorem}{proposition}}
\newtheorem{lemma}[theorem]{Lemma}
\AddToHook{env/lemma/begin}{\crefalias{theorem}{lemma}}
\newtheorem{corollary}[theorem]{Corollary}
\AddToHook{env/corollary/begin}{\crefalias{theorem}{corollary}}

\AddToHook{env/conjecture/begin}{\crefalias{theorem}{conjecture}}

\theoremstyle{definition}
\newtheorem{definition}[theorem]{Definition}
\AddToHook{env/definition/begin}{\crefalias{theorem}{definition}}
\newtheorem{remark}[theorem]{Remark}
\AddToHook{env/remark/begin}{\crefalias{theorem}{remark}}

\AddToHook{env/notation/begin}{\crefalias{theorem}{notation}}
\newtheorem{example}[theorem]{Example}
\AddToHook{env/example/begin}{\crefalias{theorem}{example}}

\AddToHook{env/construction/begin}{\crefalias{theorem}{construction}}

\crefname{subsubsection}{\S}{\S\S}
\Crefname{subsubsection}{\S}{\S\S}

\newcounter{mainthms}

\theoremstyle{theorem}
\newtheorem{maintheorem}[mainthms]{Theorem}
\AddToHook{env/maintheorem/begin}{\crefalias{mainthms}{theorem}}
\newtheorem{maincorollary}[mainthms]{Corollary}
\AddToHook{env/maincorollary/begin}{\crefalias{mainthms}{corollary}}

\usepackage{xpatch}

\makeatletter 

\renewcommand{\tocsection}[3]{%
  \indentlabel{\@ifnotempty{#2}{\bfseries\ignorespaces#1 #2\quad}}\bfseries#3}
\renewcommand{\tocsubsection}[3]{%
  \indentlabel{\@ifnotempty{#2}{\ignorespaces#1 #2\quad}}#3}
\newcommand\@dotsep{4.5}
\def\@tocline#1#2#3#4#5#6#7{\relax
  \ifnum #1>\c@tocdepth 
  \else
    \par \addpenalty\@secpenalty\addvspace{#2}%
    \begingroup \hyphenpenalty\@M
    \@ifempty{#4}{%
      \@tempdima\csname r@tocindent\number#1\endcsname\relax
    }{%
      \@tempdima#4\relax
    }%
    \parindent\z@ \leftskip#3\relax \advance\leftskip\@tempdima\relax
    \rightskip\@pnumwidth plus1em \parfillskip-\@pnumwidth
    #5\leavevmode\hskip-\@tempdima{#6}\nobreak
    \ifnum #1< 2
      \hfill
    \else
      \leaders\hbox{$\m@th\mkern \@dotsep mu\hbox{.}\mkern \@dotsep mu$}\hfill
    \fi
    \nobreak
    \hbox to\@pnumwidth{\@tocpagenum{\ifnum#1=1\bfseries\fi#7}}\par
    \nobreak
    \endgroup
  \fi}
\AtBeginDocument{%
\expandafter\renewcommand\csname r@tocindent0\endcsname{0pt}
}
\def\l@subsection{\@tocline{2}{0pt}{2.5pc}{5pc}{}}
\makeatother 

\title{Exchange theorems and coherent duality in six functors}
\author{William Fisher}
\address{University of California, Berkeley, USA}
\email{will\_fisher@berkeley.edu}
\date{\today}

\makeatletter
\renewcommand{\l@section}{\@tocline{1}{0pt}{0em}{}{}}
\renewcommand{\l@subsection}{\@tocline{2}{0pt}{1.2em}{}{}}
\makeatother

\begin{document}

\begin{abstract}
    We define the notion of an exchange theorem and show that any two functors satisfying an exchange theorem are canonically related via twisted norm maps. This is done by identifying the universal category receiving a pair of functors satisfying an exchange theorem. Additionally, we show that the twists occurring are $K$-theoretic in nature, parametrized by a categorified analogue of virtual vector bundles. As an application, we show that every $3$-functor formalism has a canonical extension which encodes Poincar\'e duality and Thom twists internal to the formalism. This gives a $1$-categorical realization of the ``coherent six operations'' outlined by Hoyois.\par
    In the process of proving universality, techniques for computing categories associated to bi- and $n$-simplicial spaces are developed. Many of the results in this direction may be viewed as model-independent rederivations of work of Liu--Zheng.
\end{abstract}

\maketitle

\begin{figure}[h!]
    \vspace{4mm}
    \centering
    \includegraphics[width=90mm]{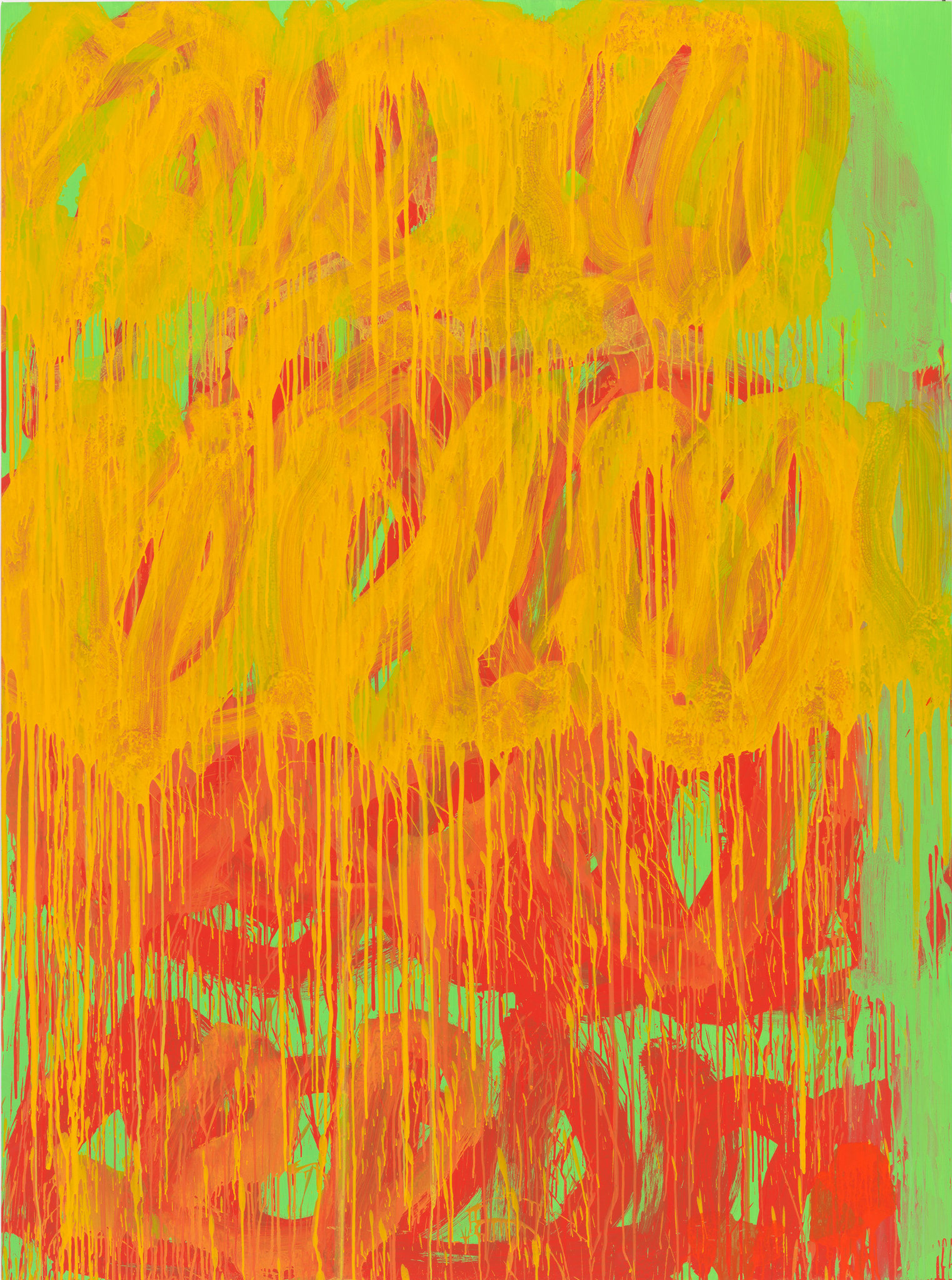}
    \caption*{``Untitled (Camino Real)'' by Cy Twombly.}
\end{figure}

\newpage

\setcounter{tocdepth}{2}
\tableofcontents

\section{Introduction}

Let $\mc C$ be a category and $E\subseteq\mc C$ a wide subcategory which is closed under pullbacks. Given a contravariant functor $(-)^\ast : \mc C^\mathrm{op} \longrightarrow \mc D$ and covariant functor $(-)_! : E\longrightarrow \mc D$ which take the same value on objects, a familiar occurrence in geometry is that of a \emph{base change theorem} between $(-)^\ast$ and $(-)_!$. Informally, a base change theorem posits a collection of coherent equivalences
\begin{equation*}
    \bar f_!\bar g^\ast \simeq g^\ast f_!
\end{equation*}
for every pullback square
\begin{equation*}
    \begin{tikzcd}
    	x & {y'} \\
    	y & z
    	\arrow["{\bar g}", from=1-1, to=1-2]
    	\arrow["{\bar f}"', from=1-1, to=2-1]
    	\arrow["\lrcorner"{anchor=center, pos=0.125}, draw=none, from=1-1, to=2-2]
    	\arrow["f", from=1-2, to=2-2]
    	\arrow["g"', from=2-1, to=2-2]
    \end{tikzcd}
\end{equation*}
in $\mc C$ with $\bar f, f\in E$. In an effort to efficiently package the higher coherences involved in such a setup, Gaitsgory \cite{gaitsgoryindcoh} introduced the span (or correspondence) category $\Span(\mc C, E)$. This category shares the same objects as $\mc C$ but with morphisms given by spans $x\leftarrow y\rightarrow z$ having $y\rightarrow z\in E$ and composition defined via pullback. Shortly after, Liu--Zheng \cite{liu2012enhanced} justified this packaging by proving that the span category $\Span(\mc C, E)$ with its right
\begin{equation*}
    (-)_! : E\longrightarrow \Span(\mc C, E)
\end{equation*}
and left
\begin{equation*}
    (-)^\ast : \mc C^\mathrm{op}\longrightarrow \Span(\mc C, E)
\end{equation*}
leg inclusions is the \emph{universal} base change theorem. That is to say, every other base change theorem between two functors valued in a category $\mc D$ is induced uniquely by a functor $\Span(\mc C, E)\to \mc D$ for which $(-)_!$ and $(-)^\ast$ may be recovered by restriction to the right and left leg spans, respectively. As a result of this, correspondence categories have become central to the modern formalism of Grothendieck's six functors (see \cite{scholze2022six,mann2022p,gaitsgory2019study,cnossen2025universality}). Importantly, Liu--Zheng's theorem allows manipulations of general functors satisfying base change to be done at the universal level, utilizing the internal category theory of $\Span(\mc C, E)$.\par


In this paper, we introduce and study the notion of an \emph{exchange theorem}. Similar to a base change theorem, given two covariant functors $(-)_\natural : E \longrightarrow \mc D$ and $(-)_\ast : \mc C\longrightarrow \mc D$ which take the same value on objects, an exchange theorem between $(-)_\ast$ and $(-)_\natural$ is the data of coherent equivalences
\begin{equation*}
    f_\natural \bar g_\ast \simeq g_\ast \bar f_\natural
\end{equation*}
for every pullback square
\begin{equation*}
    \begin{tikzcd}
    	x & {y'} \\
    	y & z
    	\arrow["{\bar g}", from=1-1, to=1-2]
    	\arrow["{\bar f}"', from=1-1, to=2-1]
    	\arrow["\lrcorner"{anchor=center, pos=0.125}, draw=none, from=1-1, to=2-2]
    	\arrow["f", from=1-2, to=2-2]
    	\arrow["g"', from=2-1, to=2-2]
    \end{tikzcd}
\end{equation*}
in $\mc C$ with $\bar f, f\in E$ (we refer the reader to \Cref{subsec:exchthms} for a rigorous treatment of the coherence data involved). As it turns out, any two functors satisfying exchange are canonically related up to a \emph{twist}---a phenomenon known to geometers and homotopy theorists as ambidexterity. This is a consequence of Verdier's diagonal trick: By contemplating the diagram
\begin{equation*}
    \begin{tikzcd}
    	x && \\
    	& {x\times_y x} & x \\
    	& x & y,
    	\arrow["{\Delta_f}", from=1-1, to=2-2]
    	\arrow["{\pr_2}", from=2-2, to=2-3]
    	\arrow["{\pr_1}"', from=2-2, to=3-2]
    	\arrow["\lrcorner"{anchor=center, pos=0.125}, draw=none, from=2-2, to=3-3]
    	\arrow["f", from=2-3, to=3-3]
    	\arrow["f"', from=3-2, to=3-3]
    \end{tikzcd}
\end{equation*}
exchange yields an equivalence
\begin{equation*}
    f_\natural (\pr_2)_\ast \simeq f_\ast (\pr_1)_\natural.
\end{equation*}
Precomposing with $(\Delta_f)_\ast$, one finds that
\begin{equation}\label{eq:twistnormmapraw}
    f_\natural \simeq f_\ast (\pr_1)_\natural (\Delta_f)_\ast,
\end{equation}
i.e.\ $f_\natural$ canonically agrees with $f_\ast$ up to precomposition with the twisting endomorphism $(\pr_1)_\natural (\Delta_f)_\ast$.\par
These twisting endomorphisms have been studied by Magen \cite{magen2025geometric} and Zavyalov \cite{zavyalov2023poincar} and in fact are instances of a much broader family of endomorphisms. To form the twist $(\pr_1)_\natural (\Delta_f)_\ast$, we only used that $\pr_1$ had a section $\Delta_f$ and belonged to $E$ to make sense of $(\pr_1)_\natural$. We abstract this setup into what we refer to as a vector bundle (see \Cref{rem:vectorbundles} for comparison to geometric notions of vector bundles).

\begin{definition}
    A vector bundle over $x$ is a pair of maps
    \begin{equation*}
        \begin{tikzcd}
            x\arrow{r}{s} & e\arrow{r}{p} & x
        \end{tikzcd}
    \end{equation*}
    with $ps\simeq\id$ and $p\in E$. We denote the category of such objects by $\Vect_{(\mc C, E)}(x)$ or $\Vect(x)$ when $\mc C$ and $E$ are clear from context.
\end{definition}

Every vector bundle $e$ over $x$ induces an endomorphism $\Sigma^e = p_\natural s_\ast$. A consequence of exchange is that the endomorphisms $\Sigma^{(-)}$ are multiplicative in fiber sequences of vector bundles. In particular, they extend to virtual vector bundles, i.e.\ to a map
\begin{equation*}
    \Sigma^{(-)} : \mc K(\Vect_{(\mc C, E)}(x))\longrightarrow \End(\Shv(x)),
\end{equation*}
where $\Shv(x)$ denotes the common value of $(-)_\natural$ and $(-)_\ast$ on objects. In the example of exchange between pullback and exceptional pullback for smooth morphisms arising from Poincar\'e duality (see \Cref{subsec:exchthmexamples} below), these twisting endomorphisms are given by smashing with the Thom spectrum of the associated virtual vector bundle. As such, following \cite{magen2025geometric}, we opt to call these endomorphisms \emph{Thom twists}. Additionally, we refer to the vector bundle
\begin{equation}\label{eq:reltangentbundle}
    T_f\coloneqq\begin{tikzcd}
        x\arrow{r}{\Delta_f} & x\times_y x \arrow{r}{\pr_1} & x
    \end{tikzcd}
\end{equation}
as the \emph{relative tangent bundle of $f$}, denoted $T_f$. Under this formalism, \eqref{eq:twistnormmapraw} may be written as
\begin{equation*}
    f_\natural \simeq f_\ast \Sigma^{T_f}.
\end{equation*}
Ambidextrous relationships of this form appear in motivic homotopy theory \cite{ayoub2007six,hoyois2017six,annala2024atiyah}, representation theory \cite{wirthmuller1974equivariant}, homotopy theory \cite{hopkins2013ambidexterity}, and beyond, and we recall some of these examples in \Cref{subsec:exchthmexamples}.

\subsection{Main results}

In the above, we argued that any two functors satisfying exchange are canonically related via ambidexterity and induce a theory of Thom twists. Conversely, any two functors which satisfy a sufficiently robust form of ambidexterity have an associated exchange theorem. Our main theorem makes this precise by identifying the universal exchange theorem.

\begin{maintheorem}[\Cref{thm:mainthmintext}]\label{thm:mainthm}
    Let $(\mc C, E)$ be a geometric setup (see \Cref{def:geosetup}) and let $\TwC$ be the category whose
    \begin{itemize}
        \item objects are the same as $\mc C$
        \item morphisms $x\to y$ are pairs $(f,\xi)$ with $f : x\to y\in\mc C$ and $\xi\in \mc K(\Vect(x))$
        \item composition is given by $(g, \eta)\circ (f,\xi)\simeq (gf, f^\ast\eta + \xi)$.
    \end{itemize}
    Then $\TwC$ hosts the universal exchange theorem, which is between the two functors
    \begin{equation*}
        \begin{tikzcd}[row sep=0pt]
            (-)_\ast : \mc C\arrow{r} & \TwC \\
            \qquad\; f \arrow[mapsto]{r} & (f,\; 0)
        \end{tikzcd}
    \end{equation*}
    and
    \begin{equation*}
        \begin{tikzcd}[row sep=0pt]
            (-)_\natural : E\arrow{r} & \TwC \\
            \qquad\; f\arrow[mapsto]{r} & (f,\; T_f).
        \end{tikzcd}
    \end{equation*}
    In particular, every other exchange theorem between two functors $(-)_\ast : \mc C\to\mc D$ and $(-)_\natural : E\to\mc D$ valued in a category $\mc D$ is induced uniquely by a functor $\TwC\to\mc D$ with the properties that
    \begin{enumerate}
        \item $(f, 0)$ is sent to $f_\ast$
        \item $(f, T_f)$ is sent to $f_\natural$
        \item for a vector bundle $\xi = \begin{tikzcd}
            x\arrow{r}{s} & e\arrow{r}{p} & x
        \end{tikzcd}$, $(\id_x, \xi)$ is sent to the Thom twist $\Sigma^\xi = p_\natural s_\ast$.
    \end{enumerate}
\end{maintheorem}

Explicitly, $\TwC$ is the Cartesian unstraightening of
\begin{equation*}
    \begin{tikzcd}[row sep=0pt]
        \mc C^\mathrm{op}\arrow{r} & \CAlg(\Ani)\arrow{r}{B} & \Cat \\
        x\arrow[mapsto]{r} & \mc K(\Vect_{(\mc C, E)}(x))
    \end{tikzcd}
\end{equation*}
where $\mc K(-)$ denotes the partial $K$-theory of Yuan \cite{yuan2023integral} and $B$ sends an $\mb E_\infty$-space $A$ to the one-object category with endomorphism space $A$.

\begin{remark}
    When $E$ is left-cancellable, $\mc K(\Vect(x))\simeq K(\Vect(x))$ and the functor $BK(\Vect(-))$ first appeared in work of Dauser--Kuijper \cite{dauser2024uniqueness}. There, it is shown to be intimately related to the uniqueness of six functor formalisms.
\end{remark}

\begin{remark}
    The partial $K$-theory of $\Vect_{(\mc C, E)}(x)$ is an $\mb E_\infty$-space such that
    \begin{enumerate}
        \item its group completion is equivalent to standard algebraic $K$-theory, i.e.\ $\mc K(\Vect(x))^\mathrm{gp}\simeq K(\Vect(x))$
        \item $\pi_0\mc K(\Vect(x))$ is the abelian group freely generated by the objects of $\Vect(x)$ modulo the relation $[e_2] = [e_1] + [e_3]$ for every fiber sequence
        \begin{equation*}
            e_1\longrightarrow e_2\longrightarrow e_3
        \end{equation*}
        of vector bundles with $e_2\to e_3\in E$.
    \end{enumerate}
    That is to say, the partial $K$-theory imposes additivity in fiber sequences without adjoining formal inverses. If one interprets $(\id, \xi)$ in $\TwC$ as the universal Thom twist associated to $\xi$, partial $K$-theory appears because the Thom twist $\Sigma^\xi$ need not be invertible for an arbitrary exchange theorem.
\end{remark}

In the context of ambidexterity, this theorem relates any two functors satisfying exchange via canonical norm maps. The canonicity of these norm maps is a consequence of uniqueness of composition in the category $\TwC$. Indeed, writing $\Sigma^\xi$ for the endomorphism $(\id, \xi)$ in $\TwC$, one has canonical equivalences
\begin{equation*}
    \begin{aligned}
        f_\natural &\coloneqq (f, T_f) \\
        &\simeq (f, 0)\circ (\id, T_f) \\
        &\eqqcolon f_\ast \Sigma^{T_f}
    \end{aligned}
\end{equation*}
in the universal exchange theorem which is then inherited by all other functors satisfying exchange. In existing literature on ambidexterity (see e.g.\ \cite{hopkins2013ambidexterity,cnossen2024parametrized,cnossen2025universality}), norm maps for ambidextrous functors are constructed inductively under the assumption that the maps in question are truncated. This result makes no truncation assumption, and all the standard properties of the norm maps may be deduced from the category theory of $\TwC$.\par
In addition to identifying the universal exchange theorem, we also identify the universal exchange theorem equipped with a compatible base change theorem. By this, we mean three functors $(-)^\ast : \mc C^\mathrm{op}\to \mc D$, $(-)^! : S^\mathrm{op}\to \mc D$ and $(-)_! : E\to \mc D$ such that
\begin{enumerate}
    \item $(-)^\ast$ and $(-)^!$ are equipped with an exchange theorem
    \item $(-)^\ast$ and $(-)_!$ are equipped with a base change theorem
    \item $(-)^!$ and $(-)_!$ are equipped with a base change theorem
\end{enumerate}
in a suitably compatible manner. We refer the reader to \Cref{sec:poincareduality} for more details. Here, $S\subseteq E\subseteq \mc C$ are nested wide subcategories with $(\mc C, E)$ a geometric setup.

\begin{maintheorem}[\Cref{thm:trisimpgluing}]\label{thm:univexchplusbasechange}
    The recipient of the universal exchange theorem with compatible base change theorem is given by the labeled span category $\Span^{\mc K}(\mc C, E)$ whose
    \begin{itemize}
        \item objects are the same as those in $\mc C$
        \item morphisms are labeled spans
        \begin{equation*}
            \begin{tikzcd}
            	& {(Z, \xi)} & \\
            	X && Y
            	\arrow["f"', from=1-2, to=2-1]
            	\arrow["g", from=1-2, to=2-3]
            \end{tikzcd}
        \end{equation*}
        with $g\in E$ and $\xi\in \mc K(\Vect_{(\mc C, S)}(Z))$
        \item composition is given by composing underlying spans and taking external sums of virtual vector bundles.
    \end{itemize}
    Moreover, the three corresponding functors participating in the universal exchange theorem with compatible base change are given by
    \begin{equation*}
        \begin{tikzcd}[row sep=0pt]
            (-)^\ast : \mc C^\mathrm{op} \arrow{r} & \Span^{\mc K}(\mc C, E) \\
            \qquad f : X\to Y\arrow[mapsto]{r} & Y\xleftarrow{f} (X, 0) \;{=\joinrel=}\; X,
        \end{tikzcd}
    \end{equation*}
    \begin{equation*}
        \begin{tikzcd}[row sep=0pt]
            (-)^! : S^\mathrm{op}\arrow{r} & \Span^{\mc K}(\mc C, E) \\
            \qquad f : X\to Y\arrow[mapsto]{r} & Y\xleftarrow{f} (X, T_f) \;{=\joinrel=}\; X
        \end{tikzcd}
    \end{equation*}
    and
    \begin{equation*}
        \begin{tikzcd}[row sep=0pt]
            (-)_! : E \arrow{r} & \Span^{\mc K}(\mc C, E) \\
            \qquad f : X\to Y\arrow[mapsto]{r} &  X \;{=\joinrel=}\; (X, 0) \xrightarrow{f}  Y.
        \end{tikzcd}
    \end{equation*}
    In particular, any other three functors $(-)^\ast : \mc C^\mathrm{op}\to \mc D$, $(-)^! : S^\mathrm{op}\to \mc D$ and $(-)_! : E\to \mc D$ satisfying exchange with compatible base change are induced uniquely by a functor $\Span^{\mc K}(\mc C, E)\to \mc D$ with the properties that
    \begin{enumerate}
        \item $Y\xleftarrow{f} (X, 0) \;{=\joinrel=}\; X$ is sent to $f^\ast$
        \item $Y\xleftarrow{f} (X, T_f) \;{=\joinrel=}\; X$ is sent to $f^!$
        \item $X \;{=\joinrel=}\; (X, 0) \xrightarrow{f}  Y$ is sent to $f_!$
        \item for a vector bundle $\xi = \begin{tikzcd}
            x\arrow{r}{s} & e\arrow{r}{p} & x
        \end{tikzcd}$, $X \;{=\joinrel=}\; (X, \xi) \;{=\joinrel=}\; X$ is sent to the Thom twist $\Sigma^\xi =s^\ast p^!$.
    \end{enumerate}
\end{maintheorem}

As an application of this, by taking $S$ to be cohomologically smooth morphisms and applying Poincar\'e duality, we show that every six functor formalism has a canonical enhancement to one which encodes Poincar\'e duality and a theory of Thom twists. This may be viewed as a $1$-categorical realization of Hoyois' proposed ``coherent six operations'' \cite{Hoyois2020}.

\begin{maintheorem}[\Cref{thm:extforweakcohom}, \Cref{cor:invThomtwistext}]\label{thm:mainpoincthm}
    For every $3$-functor formalism $\mc D : \Span(\mc C, E)^\otimes \to \Cat$, there exists a class $S\subseteq E$ of cohomologically smooth morphisms along with a canonical extension of $\mc D$ to a functor
    \begin{equation*}
        \mc D^\mathrm{Th} : \Span^K(\mc C, E) \longrightarrow \Cat
    \end{equation*}
    where $\Span^K(\mc C, E)$ is the category whose
    \begin{itemize}
        \item objects are the objects of $\mc C$
        \item morphisms are spans
        \begin{equation*}
            \begin{tikzcd}
            	& {(Z, \xi)} & \\
            	X && Y
            	\arrow["f"', from=1-2, to=2-1]
            	\arrow["g", from=1-2, to=2-3]
            \end{tikzcd}
        \end{equation*}
        with $g\in E$ and $\xi \in K(\Vect_{(\mc C, S)}(Z))$
        \item composition is given by composing underlying spans and taking external sums of virtual bundles
    \end{itemize}
    such that
    \begin{enumerate}
        \item \emph{(Extension)} $\mc D$ is canonically given by the composite
        \begin{equation*}
            \begin{tikzcd}
                \Span(\mc C, E) \arrow[hook]{r}{\xi = 0} & \Span^K(\mc C, E)\arrow{r}{\mc D^\mathrm{Th}} & \Cat
            \end{tikzcd}
        \end{equation*}
        \item \emph{(Poincar\'e duality)} the composite
        \begin{equation*}
            \begin{tikzcd}[row sep=0pt]
                S^\mathrm{op}\arrow{r} & \Span^K(\mc C, E) \arrow{r}{\mc D^\mathrm{Th}} & \Cat \\
                (f : X\to Y)\arrow[mapsto]{r} & Y\xleftarrow{f} (X,\; [T_f]) = X
            \end{tikzcd}
        \end{equation*}
        has a canonical identification with the exceptional pullback $(-)^! : S^\mathrm{op}\longrightarrow \Cat$
        \item \emph{(Thom twists)} For a vector bundle $\xi = \begin{tikzcd}
            x\arrow{r}{s} & e \arrow{r}{p} & x
        \end{tikzcd}\in \Vect(x)$,  $X \;{=\joinrel=}\; (X, \xi) \;{=\joinrel=}\; X$ is sent to $\Sigma^\xi \simeq s^\ast p^!(1_x)\otimes (-)$.
    \end{enumerate}
\end{maintheorem}

\begin{remark}
    Here, $\Span^K(\mc C, E)$ in \Cref{thm:mainpoincthm} differs from $\Span^{\mc K}(\mc C, E)$ in \Cref{thm:univexchplusbasechange} in that the labels are parametrized by the usual algebraic $K$-theory of $\Vect_{(\mc C, S)}(x)$ rather than the partial $K$-theory. This is a consequence of (and equivalent to by \Cref{prop:localizationoflabeledspan}) the Thom twists being invertible.
\end{remark}

\begin{remark}
    $\Span^K(\mc C, E)$ may be given a symmetric monoidal structure by tensoring underlying spans and summing virtual bundles. In this case, $\mc D^\mathrm{Th}$ may be upgraded to a lax symmetric monoidal functor $\mc D^\mathrm{Th} : \Span^K(\mc C, E)^\otimes \to \Cat$ which recovers the lax monoidal structure of $\mc D$ by restriction. Additionally, the maps
    \begin{equation*}
        \Sigma^{(-)} : K(\Vect_{(\mc C, S)}(x))\longrightarrow \End(\mc D(x))
    \end{equation*}
    factor as
    \begin{equation*}
        \begin{tikzcd}[row sep=0pt]
            K(\Vect_{(\mc C, S)}(x))\arrow{r}{\Sigma^{(-)}1_x} &[20pt] \mathrm{Pic}(\mc D(x))\arrow{r} & \End(\mc D(x)) \\
            & \omega \arrow[mapsto]{r} & \omega \otimes (-)
        \end{tikzcd}
    \end{equation*}
    where the first map is $\mb E_\infty$. We do not prove this in this paper.
\end{remark}

\begin{remark}\label{rem:vectorbundles}
    The categorified notion of vector bundle appearing in \Cref{thm:mainpoincthm} only captures the feature of having a zero section and projection which is cohomologically smooth. While this definition subsumes any reasonable geometric definition of vector bundles via the total space construction, it is often far more general. To bridge this gap, when the input six functor formalism is sufficiently geometric as to possess a notion of
    \begin{itemize}
        \item closed immersions
        \item open complements
        \item gluing fiber sequences for open-closed decompositions
    \end{itemize}
    one may show that for vector bundles $X\xrightarrow{s} E \xrightarrow{p} X$ where the zero section is a closed immersion, $\Sigma^E$ depends only on the infinitesimal neighborhood of the image of the zero section. By the tubular neighborhood theorem or via deformation to the normal cone, one may deduce that $\Sigma^E\simeq \Sigma^{\mc N_s}$ where $\mc N_s$ is the normal bundle of the zero section of $E$. We refer the reader to \cite[Corollary 4.2.12]{zavyalov2023poincar} and \cite{magen2025geometric} for more detailed study of this phenomenon.\par
    In the case of the relative tangent bundle \eqref{eq:reltangentbundle}, this comports with the fact that the normal bundle of $\Delta_f$ is the relative tangent bundle of $f$.
\end{remark}

\subsubsection{Other results of independent interest}

In an orthogonal direction, in the degenerate case where the partial $K$-theory groups $\mc K(\Vect(x))$ all vanish, \Cref{thm:mainthm} says that any two functors satisfying an exchange theorem may be uniquely identified in a manner compatible with the exchange theorem. As a corollary we obtain a result originally proved via model-dependent methods by Liu--Zheng \cite{liu2012enhanced} which has become essential input to the modern machine for constructing six functor formalisms.

\begin{maincorollary}[Liu--Zheng]
    Suppose that $E$ is left-cancellable and every $f : x\to y\in E$ is $n$-truncated for some $n$ depending on $f$, then the exchange equivalences $f_\natural \bar g_\ast \simeq g_\ast \bar f_\natural$ for every pullback square
    \begin{equation*}
        \begin{tikzcd}
        	x & {y'} \\
        	y & z
        	\arrow["{\bar g}", from=1-1, to=1-2]
        	\arrow["{\bar f}"', from=1-1, to=2-1]
        	\arrow["\lrcorner"{anchor=center, pos=0.125}, draw=none, from=1-1, to=2-2]
        	\arrow["f", from=1-2, to=2-2]
        	\arrow["g"', from=2-1, to=2-2]
        \end{tikzcd}
    \end{equation*}
    extend coherently to equivalences $f_\natural g_\ast' \simeq g_\ast f_\natural'$ for every commutative square
    \begin{equation*}
        \begin{tikzcd}
        	x & {y'} \\
        	y & z
        	\arrow["{g'}", from=1-1, to=1-2]
        	\arrow["{f'}"', from=1-1, to=2-1]
        	\arrow["f", from=1-2, to=2-2]
        	\arrow["g"', from=2-1, to=2-2]
        \end{tikzcd}
    \end{equation*}
    in $\mc C$ with $f',f\in E$.
\end{maincorollary}
\begin{proof}
    By \Cref{thm:mainthm}, it suffices to show that $\TwC\simeq \mc C$, or equivalently that every $\mc K(\Vect(x))$ vanishes. By \Cref{cor:ascatkthy}, we have that $\mc K(\Vect(x))\simeq K(\Vect(x))$.
    Consider a vector bundle
    \begin{equation*}
        \begin{tikzcd}
            x\arrow{r}{s} & e\arrow[two heads]{r}{p} & x
        \end{tikzcd}
    \end{equation*}
    in $\Vect(x)$. If $s$ is $n$-truncated, then $\Omega^{n + 1} e\simeq 0\in \Vect(x)$. Since looping induces an equivalence on $K$-theory, we have that $[e] = 0\in \pi_0K(\Vect(x))$. Similarly, one may show that every element of each $\pi_i K(\Vect(x))$ is zero, and thus $K(\Vect(x))\simeq 0$ for all $x$.
\end{proof}

In the process of proving universality, we also obtain a result involving the gluing of categories equipped with a factorization system.

\begin{definition}
    \begin{enumerate}
        \item Let $\mc C$ be a category with factorization system $(E, M)$. Define $\Fact \mc C$ to be the bisimplicial space whose $(n,m)$-simplex space is the groupoid of $n\times m$-commutative squares in $\mc C$ whose horizontal arrows lie in $E$ and whose vertical arrows lie in $M$.
        \item For a category $\mc D$, let $\Sq(\mc D)$ denote the bisimplicial space whose $(n,m)$-simplex space is the groupoid of $n\times m$-commutative squares in $\mc D$.
    \end{enumerate}
\end{definition}

We then have the following gluing result for categories with factorization systems.

\begin{maintheorem}[\Cref{thm:gluingoffactsystems}]\label{mainthm:gluingoffactsystems}
    Let $\mc C$ be a category with factorization system $(E, M)$. Then the data of a functor $\mc C\to\mc D$ is equivalent to providing a map $\Fact \mc C\to \Sq(\mc D)$. More precisely, we have an equivalence
    \begin{equation*}
        \begin{tikzcd}[row sep=0pt]
            \map_{\Cat}(\mc C, \mc D)\arrow{r}{\simeq} & \map_{\PSh(\Delta^{\times 2})}(\Fact\mc C,\Sq(\mc D)) \\
            F \arrow[mapsto]{r} & (\Fact\mc C\hookrightarrow \Sq(\mc C)\xrightarrow{\Sq(F)}\Sq(\mc D))
        \end{tikzcd}
    \end{equation*}
    for all $\mc D$.
\end{maintheorem}

To unpack this result, we note that the data of a map
\begin{equation*}
    H : \Fact\mc C\longrightarrow\Sq(\mc D)
\end{equation*}
consists of
\begin{enumerate}
    \item a functor $F_E : E\to\mc D$ (whose nerve is) given by
    \begin{equation*}
        \begin{tikzcd}
            N(E)\simeq (\Fact\mc C)_{0,\bullet}\arrow{r}{H_{0,\bullet}} & \Sq(\mc D)_{0,\bullet}\simeq N(\mc D).
        \end{tikzcd}
    \end{equation*}
    \item a functor $F_M : M\to\mc D$ (whose nerve is) given by $H_{\bullet,0}$
    \item $F_E$ and $F_M$ must canonically agree on objects (given by $H_{0,0}$) and for every commutative square
    \begin{equation*}
        \begin{tikzcd}
        	a & b \\
        	c & d
        	\arrow["{e'}", from=1-1, to=1-2]
        	\arrow["m"', from=1-1, to=2-1]
        	\arrow["{m'}", from=1-2, to=2-2]
        	\arrow["e", from=2-1, to=2-2]
        \end{tikzcd}\in (\Fact\mc C)_{1,1}
    \end{equation*}
    an equivalence $F_M(m')F_E(e')\simeq F_E(e)F_M(m)$ which is compatible with pasting of squares.
\end{enumerate}
\Cref{mainthm:gluingoffactsystems} then says that in such a scenario, $F_E$ and $F_M$ uniquely glue to a functor $F : \mc C\to\mc D$ such that $F|_E\simeq F_E$ and $F|_M\simeq F_M$.

\subsection{Exchange theorems in the wild}\label{subsec:exchthmexamples}

We conclude the introduction by recounting some examples of exchange theorems and ambidexterity in geometry. We refer to the example of six functor formalisms for an illustration of how to deduce exchange from a pair of ambidextrous functors.

\subsubsection{Equivariant stable homotopy theory} If $G$ is a compact Lie group with closed subgroup $H\le G$, Wirthm\"uller \cite{wirthmuller1974equivariant} showed that one has an equivalence
\begin{equation*}
    \begin{tikzcd}
        \coind^G_H(-)\arrow{r}{\simeq} & \ind^G_H(-\wedge S^{-L})
    \end{tikzcd}
\end{equation*}
where $L = T_{eH}(G/H)$ is the tangent space of $G/H$ at the identity coset with its induced $H$-representation, and $S^{-L}$ is the stable inverse of the representation sphere of $L$.

\subsubsection{$K(n)$-local stable homotopy theory}
For any $\pi$-finite space $X$ and functor $\rho : X\to \Sp_{K(n)}$ representing an $X$-parametrized family of $K(n)$-local spectra, Hopkins and Lurie \cite{hopkins2013ambidexterity} showed that there is a canonical identification
\begin{equation*}
    \begin{tikzcd}
        \Nm_X(\rho) : \colim(\rho)\arrow{r}{\simeq} & \lim(\rho).
    \end{tikzcd}
\end{equation*}
More generally, for a map $f : X\to Y$ between $\pi$-finite spaces, there is an identification of left and right Kan extension of diagrams $X\to\Sp_{K(n)}$ along $f$, i.e.\ the left and right adjoints of restriction along $f$, respectively.

\subsubsection{$\A^1$-homotopy theory} For a smooth proper map $f : X\to Y$ of schemes, the pullback $f^\ast : \mathrm{SH}(Y)\to \mathrm{SH}(X)$ has both left and right adjoints $f_\natural$ and $f_\ast$, respectively. It was shown by Ayoub \cite{ayoub2007six} that there exist canonical equivalences
\begin{equation*}
    f_\ast\simeq f_\natural \Sigma^{-\Omega_f}
\end{equation*}
where $\Sigma^{-\Omega_f}$ is an automorphism of $\mathrm{SH}(X)$ induced by the sheaf $\Omega_f$ of relative differentials of $f$. This was subsequently extended to equivariant $\A^1$-homotopy theory by Hoyois \cite{hoyois2017six} and even more recently to non-$\A^1$-invariant motivic spectra by Annala--Hoyois--Iwasa \cite{annala2024atiyah}.

\subsubsection{Six functor formalisms}
In the context of a geometry equipped with six functors, the prototypical exchange theorem is that of (shriek) \emph{smooth base change} which provides an exchange theorem between pullback $(-)^\ast$ and exceptional pullback $(-)^!$ for smooth morphisms. To illustrate this, consider the six functor formalism of (spectral) sheaves on locally compact Hausdorff spaces \cite{volpe2025six}. For $f : X\to Y$ a topological submersion, $f$ satisfies a form of Poincar\'e duality which yields a canonical equivalence
\begin{equation*}
    f^!(-)\simeq f^\ast(-)\otimes \omega_f.
\end{equation*}
To produce an exchange theorem, consider a pullback square
\begin{equation*}
    \begin{tikzcd}
    	X & {Y'} \\
    	Y & Z
    	\arrow["{\bar g}", from=1-1, to=1-2]
    	\arrow["{\bar f}"', from=1-1, to=2-1]
    	\arrow["\lrcorner"{anchor=center, pos=0.125}, draw=none, from=1-1, to=2-2]
    	\arrow["f", from=1-2, to=2-2]
    	\arrow["g"', from=2-1, to=2-2]
    \end{tikzcd}
\end{equation*}
of locally compact Hausdorff spaces where $f$ and $\bar f$ are both topological submersions. We then have a chain of equivalences
\begin{equation*}
    \begin{aligned}
        \bar g^\ast f^!(-) &\simeq \bar g^\ast (f^\ast(-)\otimes \omega_f) && \text{(Poincar\'e duality)} \\
        &\simeq \bar g^\ast f^\ast(-)\otimes \bar g^\ast\omega_f && \text{(Symm.\ monoidality of $(-)^\ast$)} \\
        &\simeq \bar g^\ast f^\ast(-)\otimes \omega_{\bar f} && \text{(Base change for dualizing complexes)} \\
        &\simeq \bar f^\ast g^\ast(-)\otimes \omega_{\bar f} && \text{(Functoriality of $(-)^\ast$)} \\
        &\simeq \bar f^!g^\ast(-) && \text{(Poincar\'e duality)}.
    \end{aligned}
\end{equation*}
Fundamentally, the smooth exchange theorem is a consequence of Poincar\'e duality paired with a sufficiently functorial theory of dualizing complexes. To explain the additional functoriality of dualizing complexes, recall relative Atiyah duality \cite[Theorem 7.14]{volpe2025six}. For a submersion $f : X\to Y$ between smooth manifolds, relative Atiyah duality provides an equivalence
\begin{equation*}
    \omega_f \simeq \mathrm{Th}(T_f)
\end{equation*}
between the dualizing complex of $f$ and the Thom spectrum of the relative tangent bundle of $f$. Since the Thom construction extends to all (virtual) vector bundles, the twisting endomorphism $(-)\otimes \omega_f$ occurring in Poincar\'e duality in fact sits within a larger family of endomorphisms
\begin{equation*}
    \begin{tikzcd}[row sep=0pt]
        \Sigma^{(-)} : K(\Vect(X))\arrow{r} & \End(\Shv(X)) \\
        \qquad \xi\arrow[mapsto]{r} & \mathrm{Th}(\xi)\otimes (-)
    \end{tikzcd}
\end{equation*}
referred to as Thom twists. The additional functorialities of $\omega_f$ are then downstream of functorialities within the Thom construction along with relationships between relative tangent bundles.\par
More generally, every six functor formalism has a collection of morphisms referred to as cohomologically smooth morphisms for which Poincar\'e duality holds. This induces a smooth exchange theorem and associated theory of Thom twists internal to that formalism.

\subsubsection{Parametrized category theory}

Given a parametrized category $F : \mc C^\mathrm{op}\longrightarrow \Cat$,
\begin{itemize}
    \item $F$ has $\mc C$-indexed colimits if the image of every pullback square in $\mc C$ is horizontally left adjointable
    \item $F$ has $E$-indexed limits if the image of every pullback square in $\mc C$ with vertical morphisms belonging to $E$ is vertically right adjointable
    \item $\mc C$-indexed colimits commute with $E$-indexed limits in $F$ if, for every pullback square
    \begin{equation*}
        \begin{tikzcd}
        	x & {y'} \\
        	y & z
        	\arrow["{\bar g}", from=1-1, to=1-2]
        	\arrow["{\bar f}"', from=1-1, to=2-1]
        	\arrow["\lrcorner"{anchor=center, pos=0.125}, draw=none, from=1-1, to=2-2]
        	\arrow["f", from=1-2, to=2-2]
        	\arrow["g"', from=2-1, to=2-2]
        \end{tikzcd}
    \end{equation*}
    in $\mc C$ with $\bar f, f\in E$, the image under $F$ is a vertically right adjointable square whose mate is further horizontally left adjointable.
\end{itemize}
In this setting, one obtains an exchange theorem between the pointwise left adjoint $F(-)^L : \mc C\to \Cat$ of $F$ and the pointwise right adjoint $F(-)^R : E \to\Cat$ of $F$. This exchange theorem is given by sending a pullback square
\begin{equation*}
    \begin{tikzcd}
        x & {y'} \\
        y & z
        \arrow["{\bar g}", from=1-1, to=1-2]
        \arrow["{\bar f}"', from=1-1, to=2-1]
        \arrow["\lrcorner"{anchor=center, pos=0.125}, draw=none, from=1-1, to=2-2]
        \arrow["f", from=1-2, to=2-2]
        \arrow["g"', from=2-1, to=2-2]
    \end{tikzcd}
\end{equation*}
in $\mc C$ with $\bar f,f\in E$ to the commutative square
\begin{equation*}
    \begin{tikzcd}
    	{F(x)} & {y'} \\
    	{F(y)} & z
    	\arrow["{F(\bar g)^L}", from=1-1, to=1-2]
    	\arrow["{F(\bar f)^R}"', from=1-1, to=2-1]
    	\arrow["{F(f)^R}", from=1-2, to=2-2]
    	\arrow["{F(g)^L}"', from=2-1, to=2-2]
    \end{tikzcd}
\end{equation*}
with filler given by the associated Beck--Chevalley transformation. For more details on this we refer the reader to \cite{martini2024,shahparam,cnossen2025universality} and \Cref{subsec:exchthms}.

\vspace{\baselineskip}

\noindent\textbf{Outline.} In \Cref{sec:presofcats}, we develop tools for computing the left adjoints of various embeddings of categories into $n$-simplicial spaces. Many of these results may be viewed as rederivations of results of Liu--Zheng \cite{liu2012enhanced} via model-independent techniques. In \Cref{sec:univexchthm}, we define and study various components occurring in the definition of the category which receives the universal exchange theorem and in \Cref{sec:proofofuniv} we prove universality. Finally, \Cref{sec:poincareduality} applies the universal exchange theorem to enhance $3$-functor formalisms into ones which encode Poincar\'e duality. At its core, this section may be understood as identifying the category which hosts the universal exchange theorem equipped with a compatible base change theorem.

\vspace{\baselineskip}

\noindent\textbf{Notation and conventions.}

\begin{itemize}
    \item We make use of the theory of $\infty$-categories, as developed by Lurie in \cite{lurie2009higher,lurie2017higher}. We make no concrete choice of model, however, as all arguments are model independent. Throughout, the word category will always refer to an $\infty$-category. If the hom-spaces happen to be discrete, we may emphasize this by saying a $(1,1)$-category or an ordinary category.
    \item We write $\Ani$ for the ($\infty$-)category of spaces and $\Cat$ for the ($\infty$-)category of ($\infty$-)categories.
    \item For a category $\mc C$ we write $\PSh(\mc C)$ for $\Fun(\mc C^\op, \Ani)$.
    \item We write $\Delta[n_1,n_2,\dots,n_m]$ for the presheaf on $\Delta^{\times m}$ represented by $([n_1], \dots, [n_m])$.
    \item We write $\Lambda_2^2 = 0\rightarrow 2\leftarrow 1$ for the cospan category.
    \item When working with 2D grids or bisimplicial spaces, we will always use (row, column) notation with the upper left corner being $(0,0)$. For example, a point of the $(1,1)$-simplex space of a bisimplicial space $X$ will be denoted by
    \begin{equation*}
        \begin{tikzcd}[row sep=10pt, column sep=10pt]
        	{x_{00}} && {x_{01}} \\
        	& {X_{1,1}} \\
        	{x_{10}} && {x_{11}}.
        	\arrow["{\in X_{0,1}}", from=1-1, to=1-3]
        	\arrow["{\in X_{1,0}}"', from=1-1, to=3-1]
        	\arrow["{\in X_{1,0}}", from=1-3, to=3-3]
        	\arrow["{\in X_{0,1}}"', from=3-1, to=3-3]
        \end{tikzcd}
    \end{equation*}
    \item We denote by $\langle i \rangle : [0]\to [n]$ the inclusion $\{i\}\hookrightarrow [n]$.
\end{itemize}

\vspace{\baselineskip}

\noindent\textbf{Acknowledgements.} I would like to thank my advisor David Nadler for many helpful conversations and guidance during the completion of this work. Additionally, I thank Adam Dauser for enlightening conversations as well as for his work with Josefien Kuijper which inspired many of these ideas. Finally, I also thank Josephine Hlavinka and David Nadler for comments on earlier drafts of this work.

\section{Categorical preliminaries}

In this section we take account of various categorical machinery that will be used throughout. This section may safely be skipped and referred back to as needed.

\subsubsection{Powering and tensoring} Let $\mc V$ be a symmetric monoidal category. In the sequel, we will only need $\mc V = \Ani$.

\begin{definition}
    Let $\mc C$ be a $\mc V$-enriched category.
    \begin{enumerate}
        \item A powering of $\mc C$ is a functor
        \begin{equation*}
            [-, -] : \mc V^\mathrm{op}\times \mc C\longrightarrow\mc C
        \end{equation*}
        equipped with natural equivalences
        \begin{equation*}
            \mc V(v, \mc C(c_1, c_2))\simeq \mc C(c_1, [v, c_2]).
        \end{equation*}
        \item A tensoring of $\mc C$ is a functor
        \begin{equation*}
            (-) \otimes (-) :\mc V\times \mc C\longrightarrow \mc C
        \end{equation*}
        equipped with natural equivalences
        \begin{equation*}
            \mc C(v\otimes c_1, c_2)\simeq \mc V(v, \mc C(c_1, c_2)).
        \end{equation*}
    \end{enumerate}
\end{definition}

We note by the characterizing property, tensoring preserves colimits in each factor whereas powering preserves limits in each variable.

\begin{example}
    Any closed symmetric monoidal category is canonically enriched over itself, as well as tensored and powered over itself via the symmetric monoidal product and internal hom, respectively.\par
    In particular, taking $\Cat$ with the Cartesian monoidal structure, it is tensored and powered over $\Cat$ and by restriction to groupoids $\Ani\subseteq \Cat$ it is tensored and powered over $\Ani$. 
\end{example}

\subsubsection{Some coend calculus} We refer the reader to the introduction of \cite{gepner2017lax} and \cite{loregian2018fubini} for a concise exposition of much of the following content.

\begin{definition}\label{def:twistedarrcat}
    For a category $\mc C$, the \emph{twisted arrow category} $\Tw(\mc C)$ is the total space of the coCartesian unstraightening of the mapping space functor $\map_{\mc C}(-,-) : \mc C^\mathrm{op}\times\mc C\longrightarrow\Ani$.
\end{definition}

One may also explicitly describe the twisted arrow category of $\mc C$ via its nerve. Namely
\begin{equation*}
    N(\Tw(\mc C))_n = \map_{\Cat}([n]\star [n]^\mathrm{op}, \mc C)
\end{equation*}
where $[n]\star [n]^\mathrm{op}\simeq [2n + 1]$ is the join of $[n]$ and $[n]^\mathrm{op}$. The objects of $\Tw(\mc C)$ are morphisms in $\mc C$ and a morphism between $f : x\to y\in\mc C$ and $g : x'\to y'\in \mc C$ in $\Tw(\mc C)$ is the data of a commutative square
\begin{equation*}
    \begin{tikzcd}
    	x & {x'} \\
    	y & {y'}
    	\arrow["f"', from=1-1, to=2-1]
    	\arrow["s"', from=1-2, to=1-1]
    	\arrow["g", from=1-2, to=2-2]
    	\arrow["t"', from=2-1, to=2-2]
    \end{tikzcd}
\end{equation*}
in $\mc C$. By construction of $\Tw(\mc C)$ it comes equipped with a projection $\Tw(\mc C)\longrightarrow \mc C^\mathrm{op}\times\mc C$. In the above description, this corresponds to sending a morphism $f : x\to y$ to its source and target pair $(x,y)$.

\begin{definition}
    Let $F : \mc C^\mathrm{op}\times \mc C\to \mc D$ be a functor. When it exists, the \emph{coend} of $F$ is defined as
    \begin{equation*}
        \int^{c\in C} F(c,c)\coloneqq \colim (\hspace{-4pt}\begin{tikzcd}
            \Tw(\mc C)\arrow{r} & \mc C^\mathrm{op}\times\mc C\arrow{r}{F} & \mc D
        \end{tikzcd}\hspace{-4pt}).
    \end{equation*}
\end{definition}

We recall three important results involving coends that we will make use of.

\begin{proposition}
    Let $\mc D$ be a cocomplete category which is tensored over $\Ani$. Then the left Kan extension of any $F : \mc C\to\mc D$ along any $p : \mc C\to \mc E$ exists and is given by
    \begin{equation*}
        \Lan_p(F)(-) \simeq \int^{c\in C}\hom_{\mc E}(p(c), -)\otimes F(c).
    \end{equation*}
\end{proposition}
\begin{proof}
    See, e.g., \cite[Proposition 2.3.6]{loregian2021co}.
\end{proof}

The next result describes a coend formula for the Cartesian unstraightening of a functor.

\begin{theorem}[{\cite[Theorem 1.1]{gepner2017lax}}]\label{thm:gepcoendunstraightening}
    The Cartesian unstraightening of a functor $F : \mc C^\mathrm{op}\to \Cat$ is given by
    \begin{equation*}
        \begin{tikzcd}
            \displaystyle\int^{x\in C} F(x)\times \mc C_{/x}\arrow{r} & \displaystyle\int^{x\in\mc C}\ast\times \mc C_{/x}\simeq \mc C.
        \end{tikzcd}
    \end{equation*}
\end{theorem}

Finally, we will need a Fubini theorem for coends which is a consequence of colimits commuting with colimits.

\begin{proposition}\label{prop:fubiniforcoends}
    Let $F : \mc C^\mathrm{op}\times \mc D^\mathrm{op}\to\mc E$ where $\mc E$ is tensored over $\Ani$. For every $X : \mc C\to\Ani$ and $Y : \mc D\to\Ani$ there are canonical equivalences
    \begin{equation*}
        \begin{aligned}
            & \int^{(c,d)\in\mc C\times\mc D} F(c,d)\otimes X(c)\otimes Y(d) \\
            \simeq & \int^{c\in\mc C}\Big(\int^{d\in \mc D} F(c,d)\otimes Y(d)\Big)\otimes X(c) \\
            \simeq & \int^{d\in\mc D}\Big(\int^{c\in \mc C} F(c,d)\otimes X(c)\Big)\otimes Y(d).
        \end{aligned}
    \end{equation*}
\end{proposition}
\begin{proof}
    This is a consequence of a more general Fubini theorem for coends given by \cite[Theorem 2.2]{loregian2018fubini} along with tensoring commuting with colimits.
\end{proof}

\subsubsection{Nerve and realization}\label{subsubsec:nerverealization}

Let $\mc C$ be a category and
\begin{equation*}
    S_{\mc C} : S\longrightarrow \mc C
\end{equation*}
a functor. We may form the \emph{nerve} associated to $S_{\mc C}$ which is the functor
\begin{equation*}
    \begin{tikzcd}[row sep=0pt]
        N : \mc C\arrow{r} & \PSh(S) \\
        \qquad x\arrow[mapsto]{r} & \map_{\mc C}(S_{\mc C}(-), x).
    \end{tikzcd}
\end{equation*}
Assuming that $\mc C$ is cocomplete, we may also form the \emph{realization} functor which we take to be the left Kan extension of $S_{\mc C} : S\to \mc C$ along the Yoneda embedding $S\to \PSh(S)$:
\begin{equation*}
    \begin{tikzcd}
    	S &&&& {\mc C} \\
    	\\
    	&& {\PSh(S)}
    	\arrow[""{name=0, anchor=center, inner sep=0}, "{S_{\mc C}}", from=1-1, to=1-5]
    	\arrow["{\mc Y}"', from=1-1, to=3-3]
    	\arrow["{|-|\;\coloneqq\; \Lan_{\mc Y}(S_{\mc C})}"', from=3-3, to=1-5]
    	\arrow[shorten >= 5pt, shorten <= 5pt, Rightarrow, from=0, to=3-3]
    \end{tikzcd}
\end{equation*}

We then have the following fact:

\begin{proposition}\label{prop:coendformulaforrealization}
    Realization $|-|$ is left adjoint to the nerve $N$ and if $\mc C$ is tensored over $\Ani$ the realization may be computed as
    \begin{equation*}
        |X|\simeq \int^{s\in S} X(s)\otimes S_{\mc C}(s).
    \end{equation*}
\end{proposition}

As a consequence of realization being a left adjoint, hence preserving colimits, and every presheaf canonically being the colimit of its simplices, we also get another formula for computing the realization. For this, we recall the arrow category construction.

\begin{definition}
    Let $F : \mc C\to \mc D$ be a functor and $d\in \mc D$. Then the \emph{arrow category} $F\downarrow d$ is the pullback
    \begin{equation*}
        \begin{tikzcd}
        	{F\downarrow d} & {\mc C} \\
        	{\mc D_{/d}} & {\mc D}
        	\arrow[from=1-1, to=1-2]
        	\arrow[from=1-1, to=2-1]
        	\arrow["\lrcorner"{anchor=center, pos=0.125}, draw=none, from=1-1, to=2-2]
        	\arrow["F", from=1-2, to=2-2]
        	\arrow[from=2-1, to=2-2]
        \end{tikzcd}
    \end{equation*}
    in $\Cat$.
\end{definition}

For a presheaf $X\in \PSh(S)$, we have that $X$ is canonically the colimit
\begin{equation*}
    X \simeq \colim (\hspace{-5pt}\begin{tikzcd}
        \mc Y\downarrow X\arrow{r} & S\arrow{r}{\mc Y} & \PSh(S)
    \end{tikzcd}\hspace{-5pt})
\end{equation*}
for which it follows that
\begin{equation*}
    |X|\simeq \colim(\hspace{-5pt}\begin{tikzcd}
        \mc Y\downarrow X\arrow{r} & S\arrow{r}{S_{\mc C}} & \mc C
    \end{tikzcd}\hspace{-5pt}).
\end{equation*}
In the sequel, we will make use of both this colimit formula and the coend formula for computing the realization.

\subsubsection{The (Rezk) nerve and squares functor}

We now recall two instances of nerve and realization that we will make heavy use of in the sequel.

\begin{definition}
    The \emph{(Rezk) nerve} $N : \Cat\to \PSh(\Delta)$ is the nerve construction of \Cref{subsubsec:nerverealization} associated to the inclusion $\Delta\hookrightarrow \Cat$. Its corresponding left adjoint/realization functor $\ascat : \PSh(\Delta)\to \Cat$ is referred to as the \emph{associated category} functor.
\end{definition}

A fundamental result in category theory is that the nerve is fully faithful.

\begin{theorem}[Joyal--Tierney \cite{joyal2007quasi}]\label{thm:nervefullyfaithful}
    The nerve functor $N : \Cat \to \PSh(\Delta)$ is fully faithful.
\end{theorem}

We also have a higher analogue which remembers all commutative squares in $\mc C$ rather than just strings of composable arrows.

\begin{definition}\label{def:squaresfunc}
    The \emph{squares functor} $\Sq : \Cat \to \PSh(\Delta^{\times 2})$ is the nerve construction of \Cref{subsubsec:nerverealization} associated to the functor
    \begin{equation*}
        \begin{tikzcd}[row sep=0pt]
            \Delta^{\times 2} \arrow{r} & \Cat \\
            ([n],[m])\arrow[mapsto]{r} & {[n]\times [m]}.
        \end{tikzcd}
    \end{equation*}
    We denote its left adjoint/realization functor by $\Gr : \PSh(\Delta^{\times 2})\to \Cat$.
\end{definition}

\begin{definition}
    Define the \emph{horizontal embedding} to be the functor
    \begin{equation*}
        \begin{tikzcd}[row sep=0pt]
            \Hor : \Cat \arrow{r} & \PSh(\Delta^{\times 2}) \\
            \qquad \mc C\arrow[mapsto]{r} & ([n], [m])\mapsto \map_{\Cat}([m], \mc C)
        \end{tikzcd}
    \end{equation*}
    and the \emph{vertical embedding} to be the functor
    \begin{equation*}
        \begin{tikzcd}[row sep=0pt]
            \Ver : \Cat \arrow{r} & \PSh(\Delta^{\times 2}) \\
            \qquad \mc C\arrow[mapsto]{r} & ([n], [m])\mapsto \map_{\Cat}([n], \mc C).
        \end{tikzcd}
    \end{equation*}
\end{definition}

\begin{remark}
    To square the terminology, recall that we have adopted (row, column) notation for all 2D diagrams.
\end{remark}

These three functors allow one to produce bisimplicial spaces from categories and we have the following analogue of \Cref{thm:nervefullyfaithful}.

\begin{theorem}
    The functors $\Hor$, $\Ver$ and $\Sq$ are fully faithful. Moreover, the natural maps $\Hor(\mc C)\to \Sq(\mc C)$ and $\Ver(\mc C)\to\Sq(\mc C)$ transpose to equivalences $\Gr\Hor(\mc C)\simeq \mc C$ and $\Gr\Ver(\mc C)\simeq \mc C$ for all $\mc C$.
\end{theorem}
\begin{proof}
    This follows from the much more difficult results \cite[Theorem 5]{abellan2023comparing} and \cite[Theorem A]{loubaton2025squares} which prove similar statements but for extensions of these three functors to the category of $2$-categories.
\end{proof}

We also remark that the three functors $\Sq$, $\Hor$ and $\Ver$ extend along the Rezk nerve to arbitrary simplicial spaces. Indeed, we may extend $\Sq$ to the functor
\begin{equation*}
    \begin{tikzcd}[row sep=0pt]
        \Sq : \PSh(\Delta)\arrow{r} & \PSh(\Delta^{\times 2}) \\
        \qquad X_\bullet \arrow[mapsto]{r} & ([n], [m])\mapsto \map_{\PSh(\Delta)}(\Delta[n]\times \Delta[m], X_\bullet).
    \end{tikzcd}
\end{equation*}
In this case, $\Sq(N(\mc C))\simeq \Sq(\mc C)$ so that the extension is compatible with \Cref{def:squaresfunc}. One may similarly extend $\Hor$ and $\Ver$. We will occasionally use these extensions.

\begin{remark}
    There are natural maps $\Hor(X_{0,\bullet})\to X$ and $\Ver(X_{\bullet,0})\to X$ for every bisimplicial space $X$ induced on $(n,m)$-simplices by precomposition with the collapse maps $\Delta[n,m]\to \Delta[n,0]$ and $\Delta[n,m]\to \Delta[0,m]$.
\end{remark}

\subsubsection{Factorization systems}

\begin{definition}[Joyal]\label{def:factsystem}
    Let $\mc C$ be a category. A \emph{factorization system} on $\mc C$ is a pair of wide subcategories $E,\, M\subseteq\mc C$ such that
    \begin{enumerate}
        \item every morphism $f \in \mc C$ has a factorization $f \simeq m\circ e$ for some $m\in M$ and $e\in E$
        \item for every commutative square of the form
        \begin{equation*}
            \begin{tikzcd}
                \bullet\arrow{r}\arrow{d}{e} & \bullet\arrow{d}{m} \\
                \bullet\arrow{r}\arrow[dashed]{ur} & \bullet
            \end{tikzcd}
        \end{equation*}
        with $e\in E$ and $m\in M$, there exists a contractible space of dashed fillers.
    \end{enumerate}
\end{definition}

We define the category of categories with factorization systems as the full subcategory of $\Fun(\Lambda^2_2, \Cat)$ spanned by cospans $E\hookrightarrow \mc C\hookleftarrow M$ of wide inclusions where $(E,M)$ forms a factorization system on $\mc C$. We denote this category by $\mathrm{OFS}$.

\begin{remark}
    The conditions in \Cref{def:factsystem} may be shown to be equivalent to every morphism in $\mc C$ having a contractible space of factorizations into a morphism of $E$ followed by a morphism of $M$ \cite[Proposition 5.2.8.17]{lurie2009higher}. Note that \Cref{def:factsystem} drops the stability under retracts condition of \cite[Definition 5.2.8.8]{lurie2009higher}, but this condition is redundant by \cite[Sec 1.1]{galvez2018decomposition}.
\end{remark}

\begin{example}
    For any two categories $\mc C$, $\mc D$, $\mc C\times \mc D$ has a factorization system given by the pair $(\mc C^\simeq\times \mc D, \mc C\times \mc D^\simeq)$. We denote $\mc C\times \mc D$ with this factorization system by $\mc C\bar{\times}\mc D$.
\end{example}

\begin{definition}
    We let $\Fact : \mathrm{OFS}\to \PSh(\Delta^{\times 2})$ denote the nerve construction of \Cref{subsubsec:nerverealization} applied to the functor
    \begin{equation*}
        \begin{tikzcd}[row sep=0pt]
            \Delta\times \Delta \arrow{r} & \mathrm{OFS} \\
            ([n], [m]) \arrow[mapsto]{r} & {[n]\bar\times [m]}
        \end{tikzcd}
    \end{equation*}
\end{definition}

We remark that there is a natural inclusion $\Fact \mc C\to \Sq(\mc C)$ for all $\mc C$ induced by the forgetful functor $\mathrm{OFS}\to\Cat$.



\subsubsection{Pushouts of categories}

We will need to compute the pushout of posets in the category $\Cat$. For this we will use the following general criterion.

\begin{proposition}\label{prop:pushoutsofposets}
    Let $P$ be a poset with sub-posets $A$, $B\subseteq P$. Suppose that for all $a\in A\setminus B$ and $b\in B\setminus A$ we have that
    \begin{enumerate}
        \item $A\cap B\cap [b, a]$ is weakly contractible whenever $b\le a$
        \item $A\cap B\cap [a, b]$ is weakly contractible whenever $a\le b$
    \end{enumerate}
    then
    \begin{equation*}
        \begin{tikzcd}
        	{A\cap B} & A \\
        	B & {A\cup B}
        	\arrow[from=1-1, to=1-2]
        	\arrow[from=1-1, to=2-1]
        	\arrow[from=1-2, to=2-2]
        	\arrow[from=2-1, to=2-2]
        \end{tikzcd}
    \end{equation*}
    is a pushout in $\Cat$ where all intersections and unions are taken inside $P$.
\end{proposition}
\begin{proof}
    This is a direct translation of the conditions of \cite[Corollary 3.2]{haine2025fully} to the case of posets.
\end{proof}

\section{Presentations of categories}\label{sec:presofcats}

In this section, we develop tools for computing the category $\Gr X$ for $X$ a bisimplicial space. Informally, this left adjoint sends a bisimplicial space $X$ to the following category:
\begin{enumerate}
    \item The \emph{horizontal fragment} $X_{0,\bullet}$ of $X$ has an associated category $\mc H = \ascat X_{0,\bullet}$
    \item The \emph{vertical fragment} $X_{\bullet,0}$ of $X$ has an associated category $\mc V = \ascat X_{\bullet,0}$
    \item $\Gr X$ is then the category freely generated by $\mc H$ and $\mc V$ such that every $(1,1)$-simplex
    \begin{equation}\label{eq:11simplexsquares}
        \begin{tikzcd}
        	\bullet & \bullet \\
        	\bullet & \bullet
        	\arrow["{h'}", from=1-1, to=1-2]
        	\arrow["{v'}"', from=1-1, to=2-1]
        	\arrow["v", from=1-2, to=2-2]
        	\arrow["h", from=2-1, to=2-2]
        \end{tikzcd},\quad v',v\in X_{1,0},\;\; h',h\in X_{0,1},
    \end{equation}
    imposes a relation $v\circ h'\simeq h\circ v'$ in $\Gr X$
    \item Higher simplex spaces impose higher coherences between the equivalences of (iii).
\end{enumerate}
As such, this section may be understood as providing computational tools for computing categories presented in the form $\langle \mc H, \mc V\; |\; R\rangle$ for two generating categories $\mc H$, $\mc V$ and a collection of relations $R$ of the form \eqref{eq:11simplexsquares}.

\subsection{Diagonal approximation}

\begin{definition}
    For a bisimplicial space $X$, define $\diag X$ to be the simplicial space obtained by restriction to the diagonal, i.e.\
    \begin{equation*}
        \begin{tikzcd}
            \Delta^\mathrm{op}\arrow{r}{\diag} & \Delta^\mathrm{op}\times \Delta^\mathrm{op} \arrow{r}{X} & \Ani.
        \end{tikzcd}
    \end{equation*}
\end{definition}

For every category $\mc C$, there is a natural map
\begin{equation*}
    \begin{tikzcd}[row sep=0pt]
        \diag \Sq(\mc C)\arrow{r} & N(\mc C) \\
        (F : [n]\times [n]\to \mc C)\arrow[mapsto]{r} & ([n]\xrightarrow{\diag} [n]\times [n]\xrightarrow{F}\mc C)
    \end{tikzcd}
\end{equation*}
given by restriction to the diagonal. As such, for every bisimplicial space $X$, there is a natural composition along the diagonal map given by the composite
\begin{equation*}
    \begin{tikzcd}
        \diag X\arrow{r}{\diag\circ\mathrm{unit}} &[20pt] \diag \Sq(\Gr X)\arrow{r} & N(\Gr X).
    \end{tikzcd}
\end{equation*}

\begin{proposition}[Diagonal Approximation]\label{prop:diaglocalequalsgr}
    Composition along the diagonal $\diag X\to N(\Gr X)$ transposes to an equivalence $\ascat\diag X\simeq \Gr X$.
\end{proposition}

\begin{remark}
    Since the nerve functor is fully faithful, this is equivalent to saying that $\diag X \to N(\Gr X)$ becomes an equivalence after taking associated categories.
\end{remark}

\begin{proof}
    Let $\mc Y : \Delta^{\times 2}\to\PSh(\Delta^{\times 2})$ denote the Yoneda embedding and $\delta : \Delta\to \Delta^{\times 2}$ the diagonal embedding. Since $\Gr \Delta[i,j] \simeq [i]\times [j]$, $\Gr X$ is canonically the colimit of the diagram
    \begin{equation*}
        \begin{tikzcd}[row sep=0pt]
            \mc Y\downarrow X\arrow{r} & \Delta\times\Delta\arrow{r} & \Cat \\
            & ([i],[j])\arrow[mapsto]{r} & {[i]\times [j]}.
        \end{tikzcd}
    \end{equation*}
    On the other hand, $\ascat\diag X$ is canonically the colimit of the diagram
    \begin{equation*}
        \begin{tikzcd}[row sep=0pt]
            \mc Y\circ\delta\downarrow X\arrow{r} & \Delta\arrow{r} & \Cat \\
            & {[i]} \arrow[mapsto]{r} & {[i]}
        \end{tikzcd}
    \end{equation*}
    and the induced map $\ascat\diag X\to\Gr X$ comes from the natural transformation
    \begin{equation}\label{eq:kanextdiag}
        \begin{tikzcd}[row sep=10pt]
            {\mc Y\circ\delta\downarrow X} & \Delta & \Cat \\
            & {\Delta\times\Delta} \\
            {\mc Y\downarrow X}
            \arrow[from=1-1, to=1-2]
            \arrow[Rightarrow, from=1-1, to=2-2]
            \arrow["\delta_\ast"', from=1-1, to=3-1]
            \arrow[from=1-2, to=1-3]
            \arrow[from=2-2, to=1-3]
            \arrow[from=3-1, to=2-2]
        \end{tikzcd}
    \end{equation}
    given pointwise by the diagonal $[i]\to [i]\times [i]$. Thus it suffices to show that \eqref{eq:kanextdiag} is a (pointwise) left Kan extension.\par
    The formula for pointwise left Kan extension requires us to check that, for every $\Delta[p,q]\to X\in \mc Y\downarrow X$, the canonical map
    \begin{equation}\label{eq:formforleftkanext}
        \colim (\delta_\ast \downarrow (\Delta[p,q]\to X)\longrightarrow \mc Y\circ\delta\downarrow X\longrightarrow \Cat)\longrightarrow [p]\times [q]
    \end{equation}
    is an equivalence. For this, we have that
    \begin{equation*}
        \delta_\ast \downarrow (\Delta[p,q]\to X)\simeq \delta\downarrow ([p],[q]).
    \end{equation*}
    Next, notice that every morphism $([n],[n])\to ([p],[q])$ corresponds precisely to a map $[n]\to [p]\times [q]$ of categories, i.e.\ if $\iota : \Delta\to\Cat$ denotes this inclusion then we have an equivalence
    \begin{equation}\label{eq:inclintoallcatmaps}
        \begin{tikzcd}
            \delta\downarrow ([p],[q]) \arrow{r}{\simeq} & \iota\downarrow [p]\times [q].
        \end{tikzcd}
    \end{equation}
    Using these rewritings of the indexing category, \eqref{eq:formforleftkanext} is equivalent to the composite
    \begin{equation*}
        \mathop{\colim}\limits_{\substack{[n]\in\Delta \\ ([n],[n])\to ([p],[q]) \in \delta\downarrow ([p],[q])}} [n]\xrightarrow{\simeq} \mathop{\colim}\limits_{\substack{[n]\in\Delta \\ [n]\to [p]\times [q] \in \iota\downarrow [p]\times [q]}} [n]\to [p]\times [q]
    \end{equation*}
    where the second map is an equivalence as it is the counit of the adjunction $\ascat \dashv N$ evaluated at $[p]\times [q]$, and $N$ is fully faithful. This shows that \eqref{eq:kanextdiag} is a pointwise left Kan extension.
\end{proof}

\begin{remark}
    This should be viewed as a strengthening of the fact that the geometric realization of the diagonal of a bisimplicial space recovers the total realization of the bisimplicial space. Indeed, applying the realization functor $| - | : \Cat\to\Ani$ to the conclusion of \Cref{prop:diaglocalequalsgr} recovers this fact.
\end{remark}

\subsection{Corner approximation}

\begin{definition}
    \begin{enumerate}
        \item We denote by $\Cpt^n$ the free category with factorization system generated by $[n]$. Explicitly, $\Cpt^n$ is the poset $\{(i,j) : 0\le i\le j\le n\}$ with $(i,j)\le (\ell, k)$ whenever $i\le \ell$ and $j\le k$ and factorization system given by horizontal and vertical arrows.
        \item We denote $\CCpt^n = \Fact\Cpt^n$.
    \end{enumerate}
\end{definition}

Note that there is a canonical map $\CCpt^n\to \Sq(\Cpt^n)$ as well as a map $\CCpt^n\to \Delta[n,n] = \Fact [n]\bar\times [n]$ induced by the inclusion $\Cpt^n\subseteq [n]\bar\times [n]$ of categories with factorization systems.

\begin{example}
    The (rotated) Hasse diagram (with the initial object in the upper left corner) of $\Cpt^3$ is given by
    \begin{equation}\label{eq:cpthassediagram}
        \begin{tikzcd}
        	{(0,0)} & {(0,1)} & {(0,2)} & {(0,3)} \\
        	& {(1,1)} & {(1,2)} & {(1,3)} \\
        	&& {(2,2)} & {(2,3)} \\
        	&&& {(3,3)}
        	\arrow[from=1-1, to=1-2]
        	\arrow[from=1-2, to=1-3]
        	\arrow[from=1-2, to=2-2]
        	\arrow[from=1-3, to=1-4]
        	\arrow[from=1-3, to=2-3]
        	\arrow[from=1-4, to=2-4]
        	\arrow[from=2-2, to=2-3]
        	\arrow[from=2-3, to=2-4]
        	\arrow[from=2-3, to=3-3]
        	\arrow[from=2-4, to=3-4]
        	\arrow[from=3-3, to=3-4]
        	\arrow[from=3-4, to=4-4]
        \end{tikzcd}
    \end{equation}
    $\CCpt^3$ is the bisimplicial space represented visually by the above diagram.
\end{example}

Our first goal is to prove the following crucial proposition.

\begin{proposition}\label{prop:GrCCptequivtoCpt}
    The natural map $\CCpt^n \to \Sq \Cpt^n$ transposes to an equivalence $\Gr\CCpt^n \xrightarrow{\simeq} \Cpt^n$.
\end{proposition}

We will build up to this result in steps by proving the result inductively for certain nice sub-posets of $\Cpt^n$.

\begin{definition}
    \begin{enumerate}
        \item For a sub-poset $P\subseteq \Cpt^n$, we define condition ($\ast$) as
        \begin{enumerate}
            \item[($\ast$)] for every $x,y\in P$, the interval $[x,y]$ in $\Cpt^n$ is contained in $P$
        \end{enumerate}
        \item For $P$ satisfying ($\ast$), $P$ inherits a factorization system from $\Cpt^n$ and we define $\Box^{\mathrm{int},\le 1}_{/P}$ to be the full subcategory of $\mathrm{OFS}_{/P}$ spanned by objects $[i]\bar\times [j]\hookrightarrow P$ which are inclusions of intervals with $i,j\le 1$.
    \end{enumerate}
\end{definition}

\begin{example}
    Morphisms $[i]\bar\times [j] \hookrightarrow P\in \Box^{\mathrm{int},\le 1}_{/P}$ can only take three forms:
    \begin{enumerate}
        \item Inclusions of a ``length one'' horizontal morphism $(i,j)\to (i, j + 1)$
        \item Inclusions of a ``length one'' vertical morphism $(i,j)\to (i + 1, j)$
        \item Inclusions of an ``area one'' square
        \begin{equation*}
            \begin{tikzcd}
            	{(i,j)} & {(i,j + 1)} \\
            	{(i + 1,j)} & {(i+1,j+1)}.
            	\arrow[from=1-1, to=1-2]
            	\arrow[from=1-1, to=2-1]
            	\arrow[from=1-2, to=2-2]
            	\arrow[from=2-1, to=2-2]
            \end{tikzcd}
        \end{equation*}
    \end{enumerate}
\end{example}

\begin{lemma}\label{lemma:smallboxesbuildbigboxes}
    The canonical map
    \begin{equation*}
        \mathop{\mathrm{colim}}_{[i]\bar\times [j]\hookrightarrow [n]\bar\times [m]\in \Box^{\mathrm{int},\le 1}_{/[m]\bar\times [n]}} [i]\times [j]\to [n]\times [m]
    \end{equation*}
    is an equivalence.
\end{lemma}
\begin{proof}
    We induct on $n$. For $n = 0$, this follows from the fact that
    \begin{equation*}
        [m] = [1]\cup_{[0]} [1]\cup_{[0]}\cdots\cup_{[0]}[1]
    \end{equation*}
    in $\Cat$.\par
    For $n = 1$, one inducts on $m$ in a fashion similar to the induction we perform below, so we omit the proof.\par
    For $n\ge 2$, consider the subcategory $I'$ of $\mathrm{OFS}_{/[n]\bar\times [m]}$ spanned by the objects of $I\coloneqq \Box^{\mathrm{int},\le 1}_{/[n]\bar\times [m]}$ along with the additional three objects
    \begin{equation}\label{eq:extraobjectsboxcolimit}
        \begin{aligned}
            \{0\le 1\le\cdots \le n - 1\}\bar\times [m]\hookrightarrow [n]\bar\times [m] \\
            \{n - 1\le n\}\bar\times [m] \hookrightarrow [n]\bar\times [m] \\
            \{n - 1\}\bar\times [m] \hookrightarrow [n]\bar\times [m].
        \end{aligned}
    \end{equation}
    Let $I''$ denote the full subcategory spanned by just the three objects in \eqref{eq:extraobjectsboxcolimit}. By the induction hypothesis, the forgetful map $I' \to \Cat$ is pointwise left Kan extended from $I\to \Cat$, so it suffices to show that
    \begin{equation*}
        \mathop{\mathrm{colim}}_{P\hookrightarrow [n]\bar\times [m]\in I'} P\to [n]\times [m]
    \end{equation*}
    is an equivalence.\par
    For this, notice that $I''\hookrightarrow I'$ is cofinal. Indeed, by Quillen's theorem, we need to check that $I''\times_{I'} I'_{x/}$ is weakly contractible for every $x\in I'$. When $x\in I''$, this category has an initial object hence is weakly contractible, so assume $x\in I'\setminus I'' = I$. Then one checks that $I''\times_{I'} I'_{x/}$ is a non-empty poset with either $1$ or $3$ objects depending on whether $x$ factors through $\{n-1\}\times [m]\subseteq [n]\times [m]$. In the case of three objects, the poset has the form
    \begin{equation*}
        \begin{tikzcd}
            & \bullet \\
            \bullet\arrow{ur}\arrow{dr} \\
            & \bullet
        \end{tikzcd}
    \end{equation*}
    and hence is weakly contractible.\par
    By cofinality of $I''\hookrightarrow I'$, we've reduced to showing that
    \begin{equation*}
        \begin{tikzcd}
            {\{n - 1\}\times [m]}\arrow{r}\arrow{d} & {\{0\le 1\le\cdots \le n - 1\}\times [m]}\arrow{d} \\
            {\{n - 1\le n\}\times [m]}\arrow{r} & {[n]\times [m]}
        \end{tikzcd}
    \end{equation*}
    is a pushout in $\Cat$. This follows from \Cref{prop:pushoutsofposets}.
\end{proof}

\begin{lemma}\label{lemma:identofGrFactP}
    For any $P$ satisfying ($\ast$), the map $\Gr\Fact P\to P$ is canonically equivalent to
    \begin{equation*}
        \mathop{\mathrm{colim}}_{[i]\bar\times [j]\to P\in \Box^{\mathrm{int},\le 1}_{/P}} [i]\times [j]\longrightarrow P.
    \end{equation*}
\end{lemma}
\begin{proof}
    Writing $P$ as a colimit of its simplices and using that $\Gr \Delta[i,j]\simeq [i]\times [j]$, we have that $\Gr\Fact P\to P$ is canonically equivalent to
    \begin{equation}\label{eq:colimitforGrFactP}
        \mathop{\mathrm{colim}}_{\substack{([i],[j])\in\Delta^{\times 2} \\ [i]\bar\times [j]\to P}} [i]\times [j]\to P.
    \end{equation}
    The result will then follow by showing that the natural map
    \begin{equation*}
        \mathop{\mathrm{colim}}_{[i]\bar\times [j]\in \Box^{\mathrm{int},\le 1}_{/P}} [i]\times [j]\longrightarrow\mathop{\mathrm{colim}}_{\substack{([i],[j])\in\Delta^{\times 2} \\ [i]\bar\times [j]\to P}} [i]\times [j]
    \end{equation*}
    is an equivalence.\par
    Let $\Box$ be the full subcategory of $\mathrm{OFS}$ spanned by the boxes $[n]\bar\times [m]$. The embedding
    \begin{equation*}
        \begin{tikzcd}[row sep=0pt]
            \Delta^{\times 2}\arrow{r} & \mathrm{OFS} \\
            ([n],[m])\arrow[mapsto]{r} & {[n]\bar\times [m]}
        \end{tikzcd}
    \end{equation*}
    is fully faithful, so our indexing category in \eqref{eq:colimitforGrFactP} may be described as $\Box_{/P} \coloneqq \Box\times_{\mathrm{OFS}}\mathrm{OFS}_{/P}$. Let $\Box_{/P}^\mathrm{int}$ denote the full subcategory spanned by maps $[n]\bar\times [m]\to P$ which are inclusions of intervals. Note that if $f : [n]\bar\times [m]\to P$ is a map of categories with factorization system, then
    \begin{gather*}
        \min\im f = f(0,0) \\
        \max\im f = f(n,m).
    \end{gather*}
    Letting $\min\im f = (i,j)$ and $\max\im f = (s,t)$, then the interval $[\min\im f, \max\im f]\subseteq \Cpt^n$ is contained in $P$ by the assumption ($\ast$), and it has an induced factorization system from $P$ for which it is isomorphic to $[s - i]\bar\times [t - j]$. It follows that $\Box_{/P}^\mathrm{int}\hookrightarrow \Box_{/P}$ is cofinal. Indeed, for every $f : [n]\bar\times [m]\to P$, we have that $\Box_{/P}^\mathrm{int}\times_{\Box_{/P}} (\Box_{/P})_{[n]\bar\times [m] /}$ has an initial object given by
    \begin{equation*}
        [n]\bar\times [m] \to [\min\im f, \max\im f]\hookrightarrow P
    \end{equation*}
    and thus is weakly contractible.\par
    We therefore reduce to showing that
    \begin{equation*}
        \mathop{\mathrm{colim}}_{[i]\bar\times [j]\in \Box^{\mathrm{int},\le 1}_{/P}} [i]\times [j]\longrightarrow \mathop{\mathrm{colim}}_{[i]\bar\times [j]\in \Box^{\mathrm{int}}_{/P}} [i]\times [j]
    \end{equation*}
    is an equivalence. However, by \Cref{lemma:smallboxesbuildbigboxes}, the forgetful map $\Box_{/P}^{\mathrm{int}}\to \Cat$ is left Kan extended from $\Box_{/P}^{\mathrm{int},\le 1}\to\Cat$ so the colimits agree.
\end{proof}

\begin{proof}[Proof of \Cref{prop:GrCCptequivtoCpt}]
    We show inductively on the size of $P$ that for every $P\subseteq \Cpt^\ell$ satisfying ($\ast$) the natural map $\Gr\Fact P\to P$ is an equivalence. We note that for $P$ of the form $[n]\bar\times [m]$, the result already follows from \Cref{lemma:identofGrFactP,lemma:smallboxesbuildbigboxes}.\par
    For our induction, we prove the result for all $P\subseteq [(0,0),(i,j)]$ inducting on both $i$ and $j$. When $i = 0$, $P$ is necessarily of the form $[0]\bar\times [m]$ and the result holds. When $j = 0$, $P$ is just a point and the result also holds.\par
    Now suppose that $j\ge 1$ and $P\subseteq [(0,0), (i,j)]$. We assume that $P\cap [0,i]\times \{j\}\neq \emptyset$ otherwise $P\subseteq [(0,0), (i, j - 1)]$ and the result holds by the induction hypothesis. Set
    \begin{equation*}
        \begin{aligned}
            P' &= P\cap [(0,0), (i, j - 1)] \\
            P'' &= P\cap [(0, j - 1), (i, j)].
        \end{aligned}
    \end{equation*}
    Since condition ($\ast$) is closed under intersections, it follows that $P'$ and $P''$ satisfy ($\ast$). We also have that
    \begin{equation*}
        \begin{aligned}
            P'\cap P'' &= P\cap ([0,i]\times \{j - 1\}) \\
            P'\cup P'' &= P
        \end{aligned}
    \end{equation*}
    and since $P$ satisfies ($\ast$), it follows that $P'\cap P''$ is an interval. Thus, by \Cref{prop:pushoutsofposets}, we have that
    \begin{equation}\label{eq:pushoutforPfromP'P''}
        \begin{tikzcd}
        	{P'\cap P''} & {P'} \\
        	{P''} & P
        	\arrow[from=1-1, to=1-2]
        	\arrow[from=1-1, to=2-1]
        	\arrow[from=1-2, to=2-2]
        	\arrow[from=2-1, to=2-2]
        \end{tikzcd}
    \end{equation}
    is a pushout in $\Cat$. By induction, the statement holds for $P'\cap P''$ and $P'$. Moreover, since $P''$ satisfies ($\ast$) and is contained in $[0,i]\times \{j - 1 \le j\}$, it must be of the form $[n]\bar\times [m]$ for which we also know the result to hold.\par
    Now, consider the category $I$ which we take to be the full subcategory of $\mathrm{OFS}_{/P}$ spanned by the objects of $\Box^{\mathrm{int},\le 1}_{/P}$ along with the three additional objects $P'\cap P''\hookrightarrow P$, $P'\hookrightarrow P$ and $P''\hookrightarrow P$, and let $I'$ be the subcategory spanned only by these three objects. By the assumption that the result holds for $P'$, $P''$ and $P'\cap P''$, \Cref{lemma:identofGrFactP} tells us that $I\to\Cat$ is left Kan extended from $\Box^{\mathrm{int},\le 1}_{/P}\to \Cat$.\par
    Finally, by similar reason as in the proof of \Cref{lemma:smallboxesbuildbigboxes}, we have that the inclusion $I'\hookrightarrow I$ is cofinal. Thus by \Cref{lemma:identofGrFactP}, the claim for $P$ reduces to checking that \eqref{eq:pushoutforPfromP'P''} is a pushout, which we have already confirmed.\par
    This completes the induction for $j$. The induction on $i$ follows a similar procedure, so we omit it.
\end{proof}

\begin{corollary}\label{cor:formalcomptorealcompequiv}
    The composite
    \begin{equation*}
        \begin{tikzcd}
            \diag \CCpt^n\arrow{r} & \diag \Sq(\Cpt^n)\arrow{r} & N(\Cpt^n),
        \end{tikzcd}
    \end{equation*}
    where the second map is given by composition along the diagonal, becomes an equivalence after passing to associated categories.
\end{corollary}
\begin{proof}
    We have a commutative square
    \begin{equation*}
        \begin{tikzcd}
            \diag\CCpt^n\arrow{r}\arrow{d} & N(\Gr\CCpt^n)\arrow{d}{\simeq} \\
            N(\Cpt^n) & N(\Gr\Sq \Cpt^n)\arrow{l}[below]{\simeq}
        \end{tikzcd}
    \end{equation*}
    where the top arrow is the composition along the diagonal map of \Cref{prop:diaglocalequalsgr}. Now, the right vertical arrow is an equivalence by \Cref{prop:GrCCptequivtoCpt} and the bottom horizontal arrow is an equivalence since $\Sq$ is fully faithful. The result then follows as the top horizontal arrow becomes an equivalence after passing to associated categories by \Cref{prop:diaglocalequalsgr}.
\end{proof}

\begin{remark}
    $\diag\CCpt^n$ and $N(\Cpt^n)$ are both discrete simplicial spaces. When viewed as simplicial sets, \cite[Lemma 4.2]{liu2012gluing} is precisely the statement that $\diag\CCpt^n\to N(\Cpt^n)$ is inner anodyne.
\end{remark}

\begin{definition}
    Denote by
    \begin{equation*}
        \begin{tikzcd}
            \cnr : \PSh(\Delta^{\times 2})\arrow{r} & \PSh(\Delta)
        \end{tikzcd}
    \end{equation*}
    the nerve construction of \Cref{subsubsec:nerverealization} induced by the functor
    \begin{equation*}
        \begin{tikzcd}[row sep=0pt]
            \Delta\arrow{r} & \PSh(\Delta^{\times 2}) \\
            {[n]}\arrow[mapsto]{r} & \CCpt^n.
        \end{tikzcd}
    \end{equation*}
\end{definition}

We note that the natural map $\CCpt^n = \Fact \Cpt^n\to \Fact [n]\bar\times[n] = \Delta[n,n]$ induces a natural transformation $\diag\Rightarrow\cnr$.

\begin{definition}
    \begin{enumerate}
        \item Define the \emph{subdivision} functor $\sd : \PSh(\Delta)\to\PSh(\Delta)$ to be the nerve construction of \Cref{subsubsec:nerverealization} applied to the functor
        \begin{equation*}
            \begin{tikzcd}[row sep=0pt]
                \Delta\arrow{r} & \PSh(\Delta) \\
                {[n]}\arrow[mapsto]{r} & N(\Cpt^n)
            \end{tikzcd}
        \end{equation*}
        Restriction along the diagonal $\Delta[n]\hookrightarrow N(\Cpt^n)$ induces a natural map $\sd X\to X$. We call sections of this forgetful map \emph{factorings of $X$}. For a category $\mc C$, we write $\sd\mc C$ for $\sd N(\mc C)$.
        \item For a simplicial space $X$, we write $\fsd X$ to denote $\cnr\Sq(X)$ and refer to $\fsd X$ as the \emph{formal subdivision} of $X$. For a category $\mc C$, write $\fsd \mc C$ for $\fsd N(\mc C)$.
    \end{enumerate}
\end{definition}

\begin{example}
    There are always two degenerate factorings given by factoring every morphism as itself followed by the identity on either the left or right. Explicitly, these are given by pullback along (the nerve of) the two projections $\Cpt^n\hookrightarrow [n]\times [n]\xrightarrow{\pr_i} [n]$ for $i = 1,2$.
\end{example}

\begin{example}\label{ex:factsystemshavefact}
    If $\mc C$ is a category with a factorization system $(E,M)$, then $N(\mc C)$ has a canonical factoring. Indeed, since $\Cpt^n$ with its horizontal--vertical factorization system is the free category with factorization system on $[n]$, we have a natural in $n$ map
    \begin{equation*}
        N(\mc C)_n = \map_{\Cat}([n],\mc C)\simeq \map_{\mathrm{OFS}}(\Cpt^n, \mc C)\to \map_{\Cat}(\Cpt^n, \mc C) = (\sd \mc C)_n
    \end{equation*}
    which gives a section of the forgetful map $\sd \mc C \to N(\mc C)$.
\end{example}

\begin{example}\label{ex:factofdiag}
    There is a canonical factoring of $\diag X$. At level $n$ we take this to be
    \begin{equation*}
        \begin{aligned}
            (\diag X)_n = \map_{\PSh(\Delta^{\times 2})}(\Delta[n,n], X) &\fixedxrightarrow[1cm]{\diag_\ast} \map_{\PSh(\Delta)}(\diag\Delta[n,n],\diag X) \\
            &\simeq \map_{\PSh(\Delta)}(\Delta[n]\times\Delta[n], \diag X) \\
            &\fixedxleftarrow[1cm]{} \map_{\PSh(\Delta)}(N(\Cpt^n), \diag X) \\
            &= (\sd\diag X)_n.
        \end{aligned}
    \end{equation*}
    Informally, this factoring decomposes a point
    \begin{equation*}
        \begin{tikzcd}
            \bullet\arrow{r}\arrow{d} & \bullet\arrow{d} \\
            \bullet\arrow{r} & \bullet
        \end{tikzcd}\in X_{1,1} = (\diag X)_1
    \end{equation*}
    as the composite
    \begin{equation*}
        \begin{tikzcd}
        	\bullet & \bullet & \bullet \\
        	\bullet & \bullet & \bullet \\
        	\bullet & \bullet & \bullet
        	\arrow[from=1-1, to=1-2]
        	\arrow[equals, from=1-1, to=2-1]
        	\arrow[equals, from=1-2, to=1-3]
        	\arrow[equals, from=1-2, to=2-2]
        	\arrow[equals, from=1-3, to=2-3]
        	\arrow[from=2-1, to=2-2]
        	\arrow[from=2-1, to=3-1]
        	\arrow[equals, from=2-2, to=2-3]
        	\arrow[from=2-2, to=3-2]
        	\arrow[from=2-3, to=3-3]
        	\arrow[from=3-1, to=3-2]
        	\arrow[equals, from=3-2, to=3-3]
        \end{tikzcd}\in X_{2,2} = (\diag X)_2 = (\sd \diag X)_1.
    \end{equation*}
\end{example}

We provide some motivation for $\cnr$ and $\fsd$. For a general bisimplicial space $X$, horizontal and vertical morphisms, i.e.\ points of $X_{0,1}$ and $X_{1,0}$ respectively, cannot be composed. However, in the gluing $\Gr X$ these morphisms \emph{can} be composed. Thus an intermediate approximation to $\Gr X$ is a simplicial space $Y$ whose morphism space is ``formal composites'' of horizontal and vertical morphisms, i.e.\
\begin{equation*}
    Y_1 = X_{0,1}\times_{X_{0,0}} X_{1, 0}.
\end{equation*}
Since
\begin{equation*}
    \CCpt^1 = \Delta[0,1]\coprod_{\Delta[0,0]} \Delta[1,0],
\end{equation*}
this can be written as $Y_1 = \map_{\PSh(\Delta^{\times 2})}(\CCpt^1, X)$. We may now try to imagine what $Y_2$ should look like. A naive guess would be
\begin{equation*}
    Y_2 = X_{0,1}\times_{X_{0,0}} X_{1, 0}\times_{X_{0,0}} X_{0,1}\times_{X_{0,0}} X_{1, 0},
\end{equation*}
i.e.\ two formal composites which share a target and source. However, from this description it is unclear how to compose these two formal composites into a formal composite, i.e.\ what the map $d_1 : Y_2\to Y_1$ should be. Additionally, whatever choice we make for this composite should be compatible with taking the bona fide composite after pushing forward to $\Gr X$.\par
To correct this, we first need a way to transform the ``wrong way'' formal composite $X_{1, 0}\times_{X_{0,0}} X_{0,1}$ of a vertical followed by a horizontal morphism back into a horizontal followed by a vertical morphism. This may be done by supplying a $(1,1)$-simplex of $X$. After correcting the wrong way composite, we then need to compose the horizontal and vertical arrows to reduce back down to a single horizontal morphism followed by a vertical morphism, i.e.\ a point of $Y_1$. This is witnessed by a $(2,0)$- and $(0,2)$-simplex of $X$. Taken together, the correct amount of data is corepresented by the bisimplicial space visually depicted as the gluing
\begin{equation}\label{eq:diagforCCpt2}
    \begin{tikzcd}
        \bullet && \bullet &[-20pt]& \bullet \\
        &&& {\Delta[1,1]} \\
        && \bullet && \bullet \\
        \\
        &&&& \bullet
        \arrow[from=1-1, to=1-3]
        \arrow[""{name=0, anchor=center, inner sep=0}, bend left=30, from=1-1, to=1-5]
        \arrow[from=1-3, to=1-5]
        \arrow[from=1-3, to=3-3]
        \arrow[from=1-5, to=3-5]
        \arrow[""{name=1, anchor=center, inner sep=0}, bend left=50, from=1-5, to=5-5]
        \arrow[from=3-3, to=3-5]
        \arrow[from=3-5, to=5-5]
        \arrow["{\Delta[0,2]}"{description, pos=0.7}, draw=none, from=0, to=1-3]
        \arrow["{\Delta[2,0]}"{description, pos=0.6}, draw=none, from=1, to=3-5]
    \end{tikzcd}
\end{equation}
which is precisely $\CCpt^2$. Thus $Y_2 = \map_{\PSh(\Delta^{\times 2})}(\CCpt^2, X)$. Continuing this process for higher simplices one arrives at the definition of $\cnr X$.\par
In the special case of $X = \Sq(\mc C)$ for a category $\mc C$, there is no distinction between horizontal and vertical morphisms---both are morphisms in $\mc C$. The resulting simplicial space $\fsd \mc C = \cnr\Sq(\mc C)$ is then a simplicial space whose $1$-simplex space is the groupoid of two composable morphisms in $\mc C$, i.e.\
\begin{equation*}
    (\fsd \mc C)_1 = N(\mc C)_1\times_{N(\mc C)_0}N(\mc C)_1.
\end{equation*}
This is in contrast to $\sd \mc C$ whose $1$-simplex space is given by
\begin{equation*}
    (\sd \mc C)_1 = \map_{\Cat}(\Cpt^1, \mc C) = N(\mc C)_2,
\end{equation*}
i.e.\ the groupoid of composable arrows in $\mc C$ with a specified choice of composite. Since $\mc C$ is a category, these two spaces are canonically equivalent. The next proposition formalizes this observation.

\begin{proposition}\label{prop:sdequivfsdforcats}
    There is a forgetful natural transformation $\sd(-) \Rightarrow \fsd(-)$ of functors $\PSh(\Delta)\to \PSh(\Delta)$ which is an equivalence when evaluated on (the nerve of) a category.
\end{proposition}
\begin{proof}
    We take the transformation to be given by
    \begin{equation*}
        \begin{aligned}
            \sd X = \map_{\PSh(\Delta)}(N(\Cpt^\bullet), X) &\fixedxrightarrow[1cm]{\Sq_\ast} \map_{\PSh(\Delta^{\times 2})}(\Sq(\Cpt^\bullet), \Sq(X)) \\
            &\fixedxrightarrow[1cm]{} \map_{\PSh(\Delta^{\times 2})}(\CCpt^\bullet, \Sq(X)) = \fsd X
        \end{aligned}
    \end{equation*}
    where the last map is pre-composition with $\CCpt^n\hookrightarrow \Sq(\Cpt^n)$. When $X = N(\mc C)$, this map is an equivalence by \Cref{prop:GrCCptequivtoCpt} and the fact that $\Sq : \Cat\to \PSh(\Delta^{\times 2})$ is fully faithful.
\end{proof}

It turns out that $\cnr X$ is a good approximation to $\Gr X$. The unit $X\to \Sq(\Gr X)$ induces a natural map
\begin{equation*}
    \cnr X\longrightarrow \fsd \Gr X
\end{equation*}
and by \Cref{prop:sdequivfsdforcats}, the target is naturally equivalent to $\sd \Gr X$. We therefore obtain a map
\begin{equation*}
    \begin{tikzcd}
        \cnr X\arrow{r} & \fsd\Gr X\simeq \sd \Gr X\arrow{r}{\diag^\ast} & N(\Gr X)
    \end{tikzcd}
\end{equation*}
which takes a formal composite of horizontal and vertical morphisms in $X$ and performs the composition in $\Gr X$. Our next goal is to show that this map transposes to an equivalence $\ascat\cnr X\simeq \Gr X$.

We show this using the comparison map $\diag X\to \cnr X$ and invoking diagonal approximation. The following is originally due to Liu--Zheng via model dependent methods.

\begin{theorem}[Corner Approximation {\cite[Theorem 4.27]{liu2012gluing}}]\label{thm:cornerapprox}
    The natural map $\diag X\to \cnr X$ induces an equivalence on associated categories.
\end{theorem}

To show this, we need a lemma.

\begin{lemma}\label{lemma:cnrofCCpt}
    There is a natural in $n$ equivalence $\cnr \CCpt^n\simeq N(\Cpt^n)$.
\end{lemma}
\begin{proof}
    By \cite[Theorem A]{juran2026orthogonal}, $\Fact : \mathrm{OFS}\longrightarrow\PSh(\Delta^{\times 2})$ is fully faithful. Thus we have natural in $n$ and $\ell$ equivalences
    \begin{equation*}
        \begin{aligned}
            (\cnr \CCpt^n)_\ell &= \map_{\PSh(\Delta^{\times 2})}(\CCpt^\ell, \CCpt^n) \\
            &= \map_{\PSh(\Delta^{\times 2})}(\Fact \Cpt^\ell, \Fact \Cpt^n) \\
            &\simeq \map_{\mathrm{OFS}}(\Cpt^\ell, \Cpt^n) \\
            &\simeq \map_{\Cat}([\ell], \Cpt^n) \\
            &\simeq N(\Cpt^n)_\ell
        \end{aligned}
    \end{equation*}
    as required. Here we have used that $\Cpt^\ell$ is the free category with factorization system generated by $[\ell]$. In fact, in \cite{juran2026orthogonal}, fully faithfulness of $\Fact$ is shown by constructing a natural equivalence $\cnr\Fact (-)\simeq N(-)$ of functors $\mathrm{OFS}\to \PSh(\Delta)$.
\end{proof}

\begin{proof}[Proof of \Cref{thm:cornerapprox}]
    Denote by $\alpha : \diag X\to \cnr X$ the natural map. There is a morphism of factorings
    \begin{equation*}
        \begin{tikzcd}
        	{\diag X} & {\cnr X} \\
        	{\sd \diag X} & {\sd \cnr X} \\
        	{\diag X} & {\cnr X}
        	\arrow["\alpha", from=1-1, to=1-2]
        	\arrow["f_1"', from=1-1, to=2-1]
        	\arrow["f_2", from=1-2, to=2-2]
        	\arrow["{\sd\alpha}", from=2-1, to=2-2]
        	\arrow[from=2-1, to=3-1]
        	\arrow[from=2-2, to=3-2]
        	\arrow["\alpha", from=3-1, to=3-2]
        \end{tikzcd}
    \end{equation*}
    where the left factoring is that of \Cref{ex:factofdiag} and the right factoring is given by
    \begin{equation*}
        \begin{tikzcd}
            (\cnr X)_n = \map(\CCpt^n, X)\arrow{r}{\cnr_\ast} & \map(\cnr \CCpt^n, \cnr X)\simeq \map(N(\Cpt^n), \cnr X) = (\sd \cnr X)_n
        \end{tikzcd}
    \end{equation*}
    where we have applied \Cref{lemma:cnrofCCpt}.\par
    Moreover, there is a natural map $\beta : X\to \Sq(\diag X)$ given on level $(n,m)$ by
    \begin{equation*}
        \begin{aligned}
            X_{n,m} &\simeq \map(\Delta[n,m], X) \\
            &\fixedxrightarrow[.75cm]{\diag_\ast}\map(\diag \Delta[n,m], \diag X) \\
            &\simeq \map(\Delta[n]\times \Delta[m], \diag X) \\
            &= \Sq(\diag X)_{n,m}
        \end{aligned}
    \end{equation*}
    which fits into a commutative diagram
    \begin{equation}\label{eq:diagtocnrfct}
        \begin{tikzcd}
        	{\diag X} & {\cnr X} & {\sd \cnr X} \\
        	{\sd \diag X} & {\fsd \diag X} & {\fsd \cnr X}
        	\arrow["\alpha", from=1-1, to=1-2]
        	\arrow["{f_1}"', from=1-1, to=2-1]
        	\arrow["{f_2}", from=1-2, to=1-3]
        	\arrow["{\cnr\beta}"', from=1-2, to=2-2]
        	\arrow[from=1-3, to=2-3]
        	\arrow[from=2-1, to=2-2]
        	\arrow["{{\cnr \alpha}}", from=2-2, to=2-3]
        \end{tikzcd}
    \end{equation}
    where the bottom map is the forgetful map of \Cref{prop:sdequivfsdforcats}.\par
    By \Cref{prop:sdequivfsdforcats}, the natural transformation $\sd (-)\Rightarrow \fsd (-)$ of functors $\Cat \to \PSh(\Delta^{\times 2})$ is an equivalence, so we fix a choice of natural inverse. The left square of \eqref{eq:diagtocnrfct} then yields a diagram
    \begin{equation*}
        \begin{tikzcd}
        	{\diag X}\arrow{r}{\alpha} & {\cnr X} \\
        	{\sd \diag X} & {\fsd \diag X} \\
        	{\sd \ascat\diag X} & {\fsd \ascat\diag X} \\
        	{N(\ascat\diag X)}
        	\arrow[from=1-1, to=1-2]
        	\arrow["f_1", from=1-1, to=2-1]
        	\arrow["{\mathrm{unit}}"', bend right=50, out=-70, in=-110, from=1-1, to=4-1]
        	\arrow["\cnr\beta", from=1-2, to=2-2]
        	\arrow[from=2-1, to=2-2]
        	\arrow[from=2-1, to=3-1]
        	\arrow[from=2-2, to=3-2]
        	\arrow["\simeq"', from=3-1, to=3-2]
        	\arrow[from=3-1, to=4-1]
        \end{tikzcd}
    \end{equation*}
    where the left most composite is the unit since $\diag X\to \sd\diag X$ is a factoring. It follows that the composite
    \begin{equation}\label{eq:leftinvcnrapprox}
        \begin{tikzcd}
            \cnr X\arrow{r} & \fsd \diag X\arrow{r} & \fsd \ascat\diag X\simeq \sd \ascat \diag X\arrow{r} & N(\ascat\diag X)
        \end{tikzcd}
    \end{equation}
    when precomposed with $\alpha$ is the unit. In particular, after passing to associated categories, \eqref{eq:leftinvcnrapprox} is a left inverse to $\ascat \alpha$.\par
    It thus remains to show that, after passing to associated categories, \eqref{eq:leftinvcnrapprox} is also a right inverse to $\ascat\alpha$. For this, the right square of \eqref{eq:diagtocnrfct} gives rise to a diagram
    \begin{equation*}
        \begin{tikzcd}[column sep=45pt]
        	{\cnr X} & {\sd \cnr X} \\
        	{\fsd \diag X} & {\fsd \cnr X} \\
        	{\fsd \ascat\diag X} & {\fsd \ascat\cnr X} \\
        	{\sd \ascat\diag X} & {\sd \ascat \cnr X} \\
        	{N(\ascat\diag X)} & {N(\ascat\cnr X)}
        	\arrow["f_2", from=1-1, to=1-2]
        	\arrow["\cnr\beta"', from=1-1, to=2-1]
        	\arrow[from=1-2, to=2-2]
        	\arrow["{{\fsd \alpha}}", from=2-1, to=2-2]
        	\arrow[from=2-1, to=3-1]
        	\arrow[from=2-2, to=3-2]
        	\arrow["{{\fsd \ascat\alpha}}", from=3-1, to=3-2]
        	\arrow["\simeq", from=3-1, to=4-1]
        	\arrow["\simeq", from=3-2, to=4-2]
        	\arrow["{{\sd \ascat\alpha}}", from=4-1, to=4-2]
        	\arrow[from=4-1, to=5-1]
        	\arrow[from=4-2, to=5-2]
        	\arrow["{{N(\ascat\alpha)}}", from=5-1, to=5-2]
        \end{tikzcd}
    \end{equation*}
    Traversing the left-most arrows followed by the bottom arrow gives $N(\ascat \alpha)\circ$\eqref{eq:leftinvcnrapprox}. On the other hand, since $\cnr X\to \sd \cnr X$ is a factoring, traversing the upper and right-most arrows gives the unit $\cnr X\to N(\ascat\cnr X)$ for which the result follows.
\end{proof}

The essential content of this theorem is as follows. A priori, $\cnr X$ has strictly more points in its $1$-simplex space than $\diag X$. Indeed, the points of $(\cnr X)_1$ are diagrams of the form
\begin{equation}\label{eq:morincnrD}
    \begin{tikzcd}
    	\bullet & \bullet \\
    	& \bullet
    	\arrow["{\in X_{0,1}}", from=1-1, to=1-2]
    	\arrow["{\in X_{1,0}}", from=1-2, to=2-2]
    \end{tikzcd}
\end{equation}
which need not be in the image of $(\diag X)_1\to (\cnr X)_1$ which takes a $(1,1)$-simplex of $X$ and forgets the bottom left half. However, in $\cnr X$, \eqref{eq:morincnrD} can be decomposed as a composite
\begin{equation*}
    \begin{tikzcd}
        \bullet & \bullet & \bullet \\
        & \bullet & \bullet \\
        && \bullet.
        \arrow[from=1-1, to=1-2]
        \arrow[equals, from=1-2, to=1-3]
        \arrow[equals, from=1-2, to=2-2]
        \arrow[equals, from=1-3, to=2-3]
        \arrow[equals, from=2-2, to=2-3]
        \arrow[from=2-3, to=3-3]
    \end{tikzcd}
\end{equation*}
where both arrows in the composite \emph{do} extend to $(1,1)$-simplices in $X$. This is precisely the factoring $\cnr X\to \sd \cnr X$ occurring in the proof of \Cref{thm:cornerapprox}. However, the formal composite
\begin{equation*}
    \begin{tikzcd}
    	\bullet & \bullet & \\
    	\bullet & \bullet & \bullet \\
    	& \bullet & \bullet
    	\arrow[from=1-1, to=1-2]
    	\arrow[equals, from=1-1, to=2-1]
    	\arrow[equals, from=1-2, to=2-2]
    	\arrow[from=2-1, to=2-2]
    	\arrow[equals, from=2-2, to=2-3]
    	\arrow[from=2-2, to=3-2]
    	\arrow[from=2-3, to=3-3]
    	\arrow[equals, from=3-2, to=3-3]
    \end{tikzcd}\in (\diag X)_1\times_{(\diag X)_0}(\diag X)_1
\end{equation*}
which lifts this composite in $\cnr X$ need not possess a composite in $\diag X$, i.e.\ extend to a $2$-simplex of $\diag X$. It is only after passing to associated categories that we can compose this ``formal composite.'' Thus the procedure becomes:
\begin{enumerate}
    \item Functorially factor simplices in $\cnr X$ into factorings that lift to formal composites in $\diag X$. This is modeled by a map $\cnr X\to \fsd \diag X$.
    \item Pushforward formal composites along $\diag X\to \ascat\diag X$, i.e.\ post-compose with $\fsd \diag X \to \fsd \ascat \diag X$.
    \item Complete the formal composites into bona fide factorings using $\fsd \ascat\diag X\simeq \sd \ascat\diag X$ since $\ascat \diag X$ is a category.
    \item Forget the factorings and restrict to the composite arrows, i.e.\ postcompose with $\sd \ascat\diag X\to N(\ascat\diag X)$.
\end{enumerate}
This allows us to build the map $\cnr X\to N(\ascat\diag X)$ which is inverse (after passing to associated categories) to $\diag X\to\cnr X$. This philosophy is discussed further in \Cref{subsec:weakfactsystems} as it will be used again.\par

We may summarize the results of this section as follows.

\begin{theorem}[Approximation]\label{thm:approximationtheorem}
    There exists a natural in $X$ commutative diagram
    \begin{equation*}
        \begin{tikzcd}
            \diag X\arrow{rr}\arrow{dr} && \cnr X\arrow{dl} \\
            & N(\Gr X)
        \end{tikzcd}
    \end{equation*}
    whose downwards arrows are composition along the diagonal. Moreover, all maps become equivalences after passing to associated categories.
\end{theorem}
\begin{proof}
    We construct the map $\cnr X\to N(\Gr X)$ and argue naturality. We have the following commutative diagram of functors $\Cat\to \PSh(\Delta)$:
    \begin{equation*}
        \begin{tikzcd}
        	{\diag\Sq(-)} & {\fsd (-)} \\
        	{N(-)} & {\sd (-)}
        	\arrow[from=1-1, to=1-2]
        	\arrow[from=1-1, to=2-1]
        	\arrow[from=1-1, to=2-2]
        	\arrow[from=2-2, to=2-1]
        	\arrow["\simeq"', from=2-2, to=1-2]
        \end{tikzcd}
    \end{equation*}
    where $\diag \Sq(-)\to N(-)$ is restriction along the diagonal of \Cref{prop:diaglocalequalsgr}, $\sd (-)\to \fsd (-)$ is the forgetful map of \Cref{prop:sdequivfsdforcats}, and $\diag \Sq(-)\to \sd (-)$ is given by
    \begin{equation*}
        \begin{tikzcd}[row sep=0pt]
            \diag\Sq(\mc C)\arrow{r} & \sd \mc C \\
            F : [n]\times [n]\to \mc C\arrow[mapsto]{r} & \Cpt^n\hookrightarrow [n]\times [n]\xrightarrow{F}\mc C.
        \end{tikzcd}
    \end{equation*}
    It follows that $\sd (-)\Rightarrow \fsd (-)$ is invertible in $\Fun(\Cat, \PSh(\Delta))_{\diag\Sq(-)/}$ so we may invert it in such a way as to obtain a commutative square
    \begin{equation*}
        \begin{tikzcd}
        	{\diag\Sq(-)} & {\fsd (-)} \\
        	{N(-)} & {\sd (-)}.
        	\arrow[from=1-1, to=1-2]
        	\arrow[from=1-1, to=2-1]
        	\arrow[from=1-1, to=2-2]
        	\arrow[from=2-2, to=2-1]
        	\arrow["\simeq"', from=1-2, to=2-2]
        \end{tikzcd}
    \end{equation*}
    Our diagram is then the outer triangle of
    \begin{equation*}
        \begin{tikzcd}
        	{\diag X} & {\cnr X} \\
        	{\diag \Sq(\Gr X)} & {\fsd \Gr X} \\
        	{N(\Gr X)} & {\sd \Gr X}
        	\arrow[from=1-1, to=1-2]
        	\arrow["{\diag\circ\mathrm{unit}}"', from=1-1, to=2-1]
        	\arrow["{\cnr\circ\mathrm{unit}}", from=1-2, to=2-2]
        	\arrow[from=2-1, to=2-2]
        	\arrow[from=2-1, to=3-1]
        	\arrow["\simeq", from=2-2, to=3-2]
        	\arrow[from=3-2, to=3-1]
        \end{tikzcd}
    \end{equation*}
    which is evidently natural in $X$. The fact that all maps become equivalences after passing to associated categories follows from \Cref{prop:diaglocalequalsgr} and \Cref{thm:cornerapprox}.
\end{proof}

\subsection{Gluing of categories with factorization system}

\begin{definition}
    Define the category of \emph{marked categories} to be the full subcategory of $\Fun(\Lambda_2^2, \Cat)$ spanned by cospans $E_1\rightarrow \mc C\leftarrow E_2$ where $E_1,E_2\to\mc C$ are inclusions of wide subcategories. We denote the category by $\Cat^+$ and will write objects as triples $(\mc C, E_1, E_2)$.
\end{definition}

There is a fully faithful embedding $\mathrm{OFS}\subseteq \Cat^+$ given by sending $\mc C$ with factorization system $(E,M)$ to $E\rightarrow\mc C\leftarrow M$. As such, we get a simultaneous extension of $\Fact$ and $\Sq$ to all marked categories by taking the nerve functor associated to the cobisimplicial object
\begin{equation*}
    \begin{tikzcd}[row sep=0pt]
        \Delta^{\times 2}\arrow{r} & \Cat^+ \\
        {([n],[m])}\arrow[mapsto]{r} & {[n]\bar\times [m]}.
    \end{tikzcd}
\end{equation*}
By overload of notation, we refer to the induced functor $\Cat^+\to \PSh(\Delta^{\times 2})$ as the squares functor and denote it by $\Sq$. The $(n,m)$-simplex space of $\Sq(\mc C, E_1, E_2)$ is given by the classifying space of $n\times m$-commutative grids in $\mc C$ whose horizontal arrows lie in $E_1$ and vertical arrows lie in $E_2$.\par
Note that every $(\mc C, E_1, E_2)\in\Cat^+$ has a canonical map $(\mc C, E_1, E_2)\to (\mc C,\mc C,\mc C)$ given by the identity $\id :\mc C\to \mc C$. This induces a natural map $\Sq(\mc C, E_1, E_2)\to \Sq(\mc C)$. Our goal in this section is to study when this natural map transposes to an equivalence $\Gr\Sq(\mc C, E_1, E_2)\to \mc C$. One may informally view $\Gr\Sq(\mc C, E_1, E_2)$ as the category freely generated by morphisms in both $E_1$ and $E_2$, modulo the relations $e_2 e_1 \simeq e_1' e_2'$ for every commutative square
\begin{equation*}
    \begin{tikzcd}
        \bullet\arrow{r}{e_1}\arrow{d}[left]{e_2'} & \bullet\arrow{d}{e_2} \\
        \bullet\arrow{r}{e_1'} & \bullet
    \end{tikzcd}
\end{equation*}
in $\mc C$ with $e_1,e_1'\in E_1$, $e_2,e_2'\in E_2$.\par
Our approach to this question will follow \cite{liu2012gluing} and make use of \Cref{thm:approximationtheorem}. Indeed, by \Cref{thm:approximationtheorem}, we have a commutative diagram
\begin{equation*}
    \begin{tikzcd}[column sep=0pt]
        {\diag\Sq(\mc C, E_1, E_2)} && {\cnr\Sq(\mc C, E_1, E_2)} \\
        & {N(\Gr\Sq(\mc C, E_1, E_2))} && {\diag\Sq(\mc C)} && {\fsd\mc C} \\
        &&&& {N(\Gr\Sq(\mc C))\simeq N(\mc C)}
        \arrow[from=1-1, to=1-3]
        \arrow[from=1-1, to=2-2]
        \arrow[from=1-1, to=2-4]
        \arrow[from=1-3, to=2-2]
        \arrow[from=1-3, to=2-6]
        \arrow[from=2-2, to=3-5]
        \arrow[from=2-4, to=2-6]
        \arrow[from=2-4, to=3-5]
        \arrow[from=2-6, to=3-5]
    \end{tikzcd}
\end{equation*}
where the front and back triangles consist of equivalences after passing to associated categories. Hence our question is equivalent to asking when the composite
\begin{equation*}
    \begin{tikzcd}
        \cnr\Sq(\mc C, E_1, E_2)\arrow{r} & \fsd\mc C\arrow{r} & N(\mc C)
    \end{tikzcd}
\end{equation*}
is an equivalence after passing to associated categories.

\begin{proposition}\label{prop:cnrapproxforsq}
    Under the equivalence $\fsd\mc C \simeq \sd \mc C$ of \Cref{prop:sdequivfsdforcats}, the map
    \begin{equation*}
        \cnr\Sq(\mc C, E_1, E_2)\to \fsd \mc C
    \end{equation*}
    is an equivalence onto the sub-simplicial space
    \begin{equation*}
        \map_{\Cat^+}(\Cpt^\bullet, (\mc C, E_1, E_2))\hookrightarrow\map_{\Cat}(\Cpt^\bullet, \mc C) = \sd\mc C
    \end{equation*}
    where $\Cpt^n$ is viewed as a marked category via its usual factorization system.
\end{proposition}
\begin{proof}
    Since $E_1$, $E_2$ are wide subcategories, the map
    \begin{equation*}
        \map_{\Cat^+}([n]\bar\times [m], (\mc C, E_1, E_2)) \simeq \Sq(\mc C, E_1, E_2)_{n,m}\to \Sq(\mc C)_{n,m}\simeq \map_{\Cat}([n]\times [m], \mc C)
    \end{equation*}
    is a monomorphism which corresponds to the inclusion of connected components consisting of diagrams whose horizontal arrows lie in $E_1$ and whose vertical arrows lie in $E_2$.\par
    We now make use of the following observation: Suppose that $\mc C$ has finite limits and $K$ is another category. Then the monomorphisms in $\Fun(K, \mc C)$ are precisely the natural transformations which are componentwise monomorphisms. Indeed, $\Fun(K,\mc C)$ has finite limits, so $f : X\to Y\in\Fun(K,\mc C)$ is a monomorphism if and only if $\Delta_f : X\to X\times_Y X$ is an equivalence. But limits in $\Fun(K,\mc C)$ are computed pointwise and equivalences are precisely the pointwise equivalences. Thus $\Delta_f$ is an equivalence if and only if each $X(k) \to X(k)\times_{Y(k)}X(k)$ is an equivalence, which is to say each $f_k : X(k)\to Y(k)$ is a monomorphism.\par
    It follows from this observation and the above that $\Sq(\mc C, E_1, E_2)\to \Sq(\mc C)$ is a monomorphism. Again checking pointwise,
    \begin{equation*}
        \cnr \Sq(\mc C, E_1, E_2)\to \fsd \mc C
    \end{equation*}
    is therefore also a monomorphism. At level $n$, it corresponds to the inclusion of connected components of $\map_{\PSh(\Delta^{\times 2})}(\CCpt^n, \Sq(\mc C))$ such that the induced map $(\CCpt^n)_{0,\bullet}\to \Sq(C)_{0,\bullet}\simeq N(\mc C)$ lands inside $N(E_1)$ and $(\CCpt^n)_{\bullet,0}\to \Sq(\mc C)_{\bullet, 0}\simeq N(\mc C)$ lands inside $N(E_2)$. But under the equivalence
    \begin{equation*}
        \map_{\PSh(\Delta^{\times 2})}(\CCpt^n, \Sq(\mc C))\simeq \map_{\Cat}(\Cpt^n, \mc C),
    \end{equation*}
    the maps on horizontal and vertical fragments may be recovered by restricting to the horizontal and vertical arrows in $\Cpt^n$, from which the result follows.
\end{proof}

We thus get the following equivalent formulation of the problem.

\begin{corollary}\label{cor:condforgluing}
    The natural map $\Gr\Sq(\mc C, E_1, E_2)\to\mc C$ is equivalent, after passing to associated categories, to the map
    \begin{equation*}
        \map_{\Cat^+}(\Cpt^\bullet, (\mc C, E_1, E_2))\to N(\mc C)
    \end{equation*}
    given by restriction along the diagonal.
\end{corollary}

From this, we may immediately derive a vast generalization of \Cref{prop:GrCCptequivtoCpt}.

\begin{theorem}\label{thm:gluingoffactsystems}
    If $(\mc C,E,M)$ is a category with factorization system, then the natural map $\Fact\mc C\to \Sq(\mc C)$ transposes to an equivalence.
\end{theorem}
\begin{proof}
    First note that $\Fact \mc C = \Sq(\mc C, E, M)$ by definition. Since $\Cpt^n$ is the free category with factorization system generated by $[n]$ and $\mathrm{OFS}\subseteq \Cat^+$ is a full subcategory, we have that
    \begin{equation*}
        \begin{tikzcd}
            \map_{\Cat^+}(\Cpt^n, (\mc C, E, M))\simeq \map_{\mathrm{OFS}}(\Cpt^n, (\mc C, E, M))\simeq \map_{\Cat}([n], \mc C)
        \end{tikzcd}
    \end{equation*}
    where the composite is restriction along the diagonal. Thus the map of \Cref{cor:condforgluing} is an equivalence of simplicial spaces, before passing to associated categories.
\end{proof}

\subsection{Factorization systems for simplicial spaces}\label{subsec:weakfactsystems}

In this section we introduce a technique for factoring simplicial spaces which will be used in the sequel.\par
Suppose that $X$ is a simplicial space which has two distinguished sub-simplicial spaces $E,M\subseteq X$ such that every $n$-simplex of $X$ coherently extends to a factorization $N(\Cpt^n)\to X$ where the horizontal arrows lie in $E$ and the vertical arrows lie in $M$. We can then imagine that we have access to some bisimplicial space $Y_{\bullet,\bullet}$ which takes the role of $\Fact X$ for this weak version of a factorization system. In particular, we should imagine that
\begin{enumerate}
    \item $Y_{0,\bullet}\simeq E$
    \item $Y_{\bullet,0}\simeq M$
    \item $Y_{n,m}$ represents an $n\times m$ square in $X$ where the horizontal arrows lie in $E$ and the vertical arrows lie in $M$.
\end{enumerate}
In particular, we should have a map $\alpha : Y\to \Sq(X)$ which captures (iii). On horizontal fragments the map $\alpha_{0,\bullet} : Y_{0,\bullet}\to \Sq(X)_{0,\bullet}\simeq X$ should be a monomorphism onto $E$, recovering (i) and on vertical fragments $\alpha_{\bullet,0} : Y_{\bullet,0}\to \Sq(X)_{\bullet,0}\simeq X$ should be a monomorphism onto $M$, recovering (ii).\par
Additionally, one would expect that the factorization $X\longrightarrow \sd X$ of simplices in $X$ into those belonging to $E$ and $M$ lifts to a diagram
\begin{equation*}
    \begin{tikzcd}
    	{\cnr Y} & {\fsd X} \\
    	X & {\sd X}
    	\arrow[from=1-1, to=1-2, "\cnr\alpha"]
    	\arrow[from=2-1, to=1-1]
    	\arrow[from=2-1, to=2-2]
    	\arrow[from=2-2, to=1-2]
    \end{tikzcd}
\end{equation*}
where the right vertical map is the forgetful map of \Cref{prop:sdequivfsdforcats}.\par
The map $\alpha : Y\to \Sq(X)$ induces a map $\Gr Y\to \ascat X$ by transposing the composite
\begin{equation*}
    \begin{tikzcd}
        Y\arrow{r}{\alpha} & \Sq(X)\arrow{r}{\Sq\circ\mathrm{unit}} &[20pt] \Sq(\ascat X).
    \end{tikzcd}
\end{equation*}
A hopeful thought may be that this induced map $\Gr Y\to\ascat X$ is an equivalence in analogy with \Cref{thm:gluingoffactsystems}. However, the current hypotheses only guarantee the existence of factorizations, not the uniqueness that is also inherent to bona fide factorization systems. In applications, uniqueness is often too strong to ask. This is because, if $X$ is a mere simplicial space, morphisms which may later become equivalent in $\ascat X$ may not easily be witnessed as equivalent in $X$. However, what we may show is that the map $\Gr Y\to \ascat X$ has a right inverse defined by factorizing morphisms. We encode this in \Cref{prop:weakfactsystemrightinv} below.\par
Before, however, we take stock of additional desirable properties in this setup.

\begin{definition}\label{def:degenerateonhoriz}
    In the setup above, we say that the factorization is \emph{degenerate on horizontal fragments} if the composite
    \begin{equation*}
        \begin{tikzcd}
            Y_{0,\bullet}\arrow{r}{\alpha_{0,\bullet}} & \Sq(X)_{0,\bullet}\simeq X\arrow{r} & \cnr Y
        \end{tikzcd}
    \end{equation*}
    is equivalent to
    \begin{equation*}
        \begin{tikzcd}
            Y_{0,\bullet} \simeq \map_{\PSh(\Delta^{\times 2})}(\Delta[0,\bullet], Y)\arrow{r} & \map_{\PSh(\Delta^{\times 2})}(\CCpt^\bullet, Y)\simeq \cnr Y
        \end{tikzcd}
    \end{equation*}
    where the middle map is precomposition with the projection $\CCpt^n = \Fact\Cpt^n\to \Fact [0]\bar\times [n] = \Delta[0,n]$ onto the second coordinate.\par
    Similarly, we say that the factorization is \emph{degenerate on vertical fragments} if the composite
    \begin{equation*}
        \begin{tikzcd}
            Y_{\bullet,0}\arrow{r}{\alpha_{\bullet,0}} & \Sq(X)_{\bullet,0}\simeq X\arrow{r} & \cnr Y
        \end{tikzcd}
    \end{equation*}
    is equivalent to
    \begin{equation*}
        \begin{tikzcd}
            Y_{\bullet,0} \simeq \map_{\PSh(\Delta^{\times 2})}(\Delta[\bullet,0], Y)\arrow{r} & \map_{\PSh(\Delta^{\times 2})}(\CCpt^\bullet, Y)\simeq \cnr Y
        \end{tikzcd}
    \end{equation*}
    where the middle map is precomposition with the projection $\CCpt^n = \Fact\Cpt^n\to \Fact [n]\bar\times [0] = \Delta[n,0]$ onto the first coordinate.
\end{definition}

\begin{proposition}\label{prop:weakfactsystemrightinv}
    Let $Y$ be a bisimplicial space and $X$ a simplicial space. Suppose we have a map $\alpha : Y\to \Sq(X)$ and a factoring $X\to\sd X$ of $X$ along with a commutative diagram
    \begin{equation}\label{eq:factorizationdiagYX}
        \begin{tikzcd}
        	{\cnr Y} & {\fsd X} & \\
        	X & {\sd X}.
        	\arrow[from=1-1, to=1-2]
        	\arrow[from=2-1, to=1-1]
        	\arrow[from=2-1, to=2-2]
        	\arrow[from=2-2, to=1-2]
        \end{tikzcd}
    \end{equation}
    Then the induced map $\gamma : \Gr Y\to \ascat X$ has a right inverse $\delta$ given by (taking associated categories of)
    \begin{equation*}
        \begin{tikzcd}
            X\arrow{r} & \cnr Y\arrow{r} & N(\Gr Y)
        \end{tikzcd}
    \end{equation*}
    where $\cnr Y\to N(\Gr Y)$ is the map of \Cref{thm:approximationtheorem}.\par
    Additionally, viewing $\Gr Y$ and $\ascat X$ as categories under $\ascat Y_{0,\bullet}$ via
    \begin{equation*}
        \ascat Y_{0,\bullet}\simeq \Gr \Hor(Y_{0,\bullet})\longrightarrow\Gr Y
    \end{equation*}
    and $\ascat \alpha_{0,\bullet} : \ascat Y_{0,\bullet}\to \ascat X$, and similarly as categories under $\ascat Y_{\bullet,0}$, then
    \begin{enumerate}
        \item if the factorization is degenerate on horizontal fragments then both $\gamma$ and $\delta$ lift to morphisms in $\Cat_{\ascat Y_{0,\bullet}/}$
        \item if the factorization is degenerate on vertical fragments then both $\gamma$ and $\delta$ lift to morphisms in $\Cat_{\ascat Y_{\bullet,0}/}$.
    \end{enumerate}
\end{proposition}

\begin{remark}
    In the statement of \Cref{prop:weakfactsystemrightinv} we have excluded the hypothesis that $\alpha_{0,\bullet}$ and $\alpha_{\bullet,0}$ induce monomorphisms onto subsimplicial spaces of $X$ as it is not necessary for the proof. However, in the setting discussed earlier, (i) and (ii) of \Cref{prop:weakfactsystemrightinv} say that both $\gamma$ and $\delta$ lift to morphisms in $\Cat_{\ascat E/}$ and $\Cat_{\ascat M/}$, respectively.
\end{remark}

\begin{proof}
    Building off \eqref{eq:factorizationdiagYX}, we have a commutative diagram
    \begin{equation*}
        \begin{tikzcd}
        	X & {\sd X} \\
        	{\cnr Y} & {\fsd X} \\
        	{\fsd \Gr Y} & {\fsd\ascat X} \\
        	{\sd\Gr Y} & {\sd\ascat X} \\
        	{N(\Gr Y)} & {N(\ascat X).}
        	\arrow[from=1-1, to=1-2]
        	\arrow[from=1-1, to=2-1]
        	\arrow[from=1-2, to=2-2]
        	\arrow["{\cnr\alpha}", from=2-1, to=2-2]
        	\arrow[from=2-1, to=3-1]
        	\arrow[from=2-2, to=3-2]
        	\arrow["{\fsd\gamma}", from=3-1, to=3-2]
        	\arrow["\simeq", from=3-1, to=4-1]
        	\arrow["\simeq", from=3-2, to=4-2]
        	\arrow["{\sd\gamma}", from=4-1, to=4-2]
        	\arrow[from=4-1, to=5-1]
        	\arrow[from=4-2, to=5-2]
        	\arrow["{N(\gamma)}", from=5-1, to=5-2]
        \end{tikzcd}
    \end{equation*}
    By definition, traversing the left--bottom edges gives (after passing to associated categories) $\gamma\delta$. However, since $X\to\sd X$ is a factoring traversing the top--right edges gives the unit $X\to N(\ascat X)$ and thus $\gamma\delta\simeq \id$.\par
    For (i), notice that for any simplicial space $Z$, the composite
    \begin{equation*}
        \begin{tikzcd}
            Z\simeq \map_{\PSh(\Delta^{\times 2})}(\Delta[0,\bullet], \Hor(Z))\arrow{r} & \map_{\PSh(\Delta^{\times 2})}(\CCpt^\bullet, \Hor(Z))\simeq \cnr \Hor(Z),
        \end{tikzcd}
    \end{equation*}
    where the middle map is precomposition with $\CCpt^n\to \Delta[0,n]$, is an equivalence. Thus the hypothesis yields a commutative diagram
    \begin{equation*}
        \begin{tikzcd}
        	{Y_{0,\bullet}} & X \\
        	{\cnr \Hor(Y_{0,\bullet})} & {\cnr Y} \\
        	{N(\Gr \Hor(Y_{0,\bullet}))} & {N(\Gr Y)}
        	\arrow["{\alpha_{0,\bullet}}", from=1-1, to=1-2]
        	\arrow["\simeq", from=1-1, to=2-1]
        	\arrow[from=1-2, to=2-2]
        	\arrow[from=2-1, to=2-2]
        	\arrow[from=2-1, to=3-1]
        	\arrow[from=2-2, to=3-2]
        	\arrow[from=3-1, to=3-2]
        \end{tikzcd}
    \end{equation*}
    which, upon taking associated categories, lifts $\delta$ to a morphism in $\Cat_{\ascat Y_{0,\bullet}/}$. On the other hand, $\gamma$ always lifts to a morphism in $\Cat_{\ascat Y_{0,\bullet}/}$ as a consequence of the diagram
    \begin{equation*}
        \begin{tikzcd}
        	{\Hor(Y_{0,\bullet})} &[15pt] {\Hor(X)} & {\Hor(\ascat X)} \\
        	Y & {\Sq(X)} & {\Sq(\ascat X)}
        	\arrow["{\Hor(\alpha_{0,\bullet})}", from=1-1, to=1-2]
        	\arrow[from=1-1, to=2-1]
        	\arrow[from=1-2, to=1-3]
        	\arrow[from=1-2, to=2-2]
        	\arrow[from=1-3, to=2-3]
        	\arrow["\alpha", from=2-1, to=2-2]
        	\arrow[from=2-2, to=2-3]
        \end{tikzcd}
    \end{equation*}
    and the fact that $\Hor(\ascat X)\to \Sq(\ascat X)$ transposes to an equivalence.\par
    (ii) follows from a symmetrical argument to (i).
\end{proof}

\section{The universal exchange theorem}\label{sec:univexchthm}

In this section, we discuss the machinery involved in constructing the recipient of the universal exchange theorem. For this, we restrict our attention to geometric setups.

\begin{definition}\label{def:geosetup}
    A \emph{geometric setup} is a pair $(\mc C, E)$ where $\mc C$ is a category and $E\subseteq \mc C$ is a wide subcategory such that
    \begin{enumerate}
        \item $\mc C$ has pullbacks along morphisms in $E$
        \item $E$ is closed under pullbacks.
    \end{enumerate}
\end{definition}

For the remainder of this section, fix a geometric setup $(\mc C, E)$.

\subsection{Exchange theorems}\label{subsec:exchthms}

Informally, an exchange theorem for the pair $(\mc C, E)$ valued in a category $\mc D$ is a pair of functors $(-)_\ast : \mc C \to \mc D$ and $(-)_\natural : E\to \mc D$ which take the same value on objects along with coherent equivalences
\begin{equation*}
    f_\natural \bar g_\ast \simeq g_\ast \bar f_\natural
\end{equation*}
for every pullback square
\begin{equation*}
    \begin{tikzcd}
    	x & {y'} \\
    	y & z
    	\arrow["{\bar g}", from=1-1, to=1-2]
    	\arrow["{\bar f}"', from=1-1, to=2-1]
    	\arrow["\lrcorner"{anchor=center, pos=0.125}, draw=none, from=1-1, to=2-2]
    	\arrow["f", from=1-2, to=2-2]
    	\arrow["g"', from=2-1, to=2-2]
    \end{tikzcd}
\end{equation*}
in $\mc C$ with $\bar f, f\in E$. To handle the higher coherences involved, we make use of bisimplicial spaces.

\begin{definition}
    Let $\Pull(\mc C, E)$ denote the bisimplicial space whose $(n,m)$-simplex space is the groupoid of diagrams
    \begin{equation*}
        \begin{tikzcd}
            {x_{00}} & {x_{01}} & \cdots & {x_{0m}} \\
            {x_{10}} & {x_{11}} & \cdots & {x_{1m}} \\
            {x_{20}} & {x_{21}} & \cdots & {x_{2m}} \\
            \vdots & \vdots & \ddots & \vdots \\
            {x_{n0}} & {x_{n1}} & \cdots & {x_{nm}}
            \arrow[from=1-1, to=1-2]
            \arrow[from=1-1, to=2-1]
            \arrow["\lrcorner"{anchor=center, pos=0.125}, draw=none, from=1-1, to=2-2]
            \arrow[from=1-2, to=1-3]
            \arrow[from=1-2, to=2-2]
            \arrow["\lrcorner"{anchor=center, pos=0.125}, draw=none, from=1-2, to=2-3]
            \arrow[from=1-3, to=1-4]
            \arrow[from=1-4, to=2-4]
            \arrow[from=2-1, to=2-2]
            \arrow[from=2-1, to=3-1]
            \arrow["\lrcorner"{anchor=center, pos=0.125}, draw=none, from=2-1, to=3-2]
            \arrow[from=2-2, to=2-3]
            \arrow[from=2-2, to=3-2]
            \arrow["\lrcorner"{anchor=center, pos=0.125}, draw=none, from=2-2, to=3-3]
            \arrow[from=2-3, to=2-4]
            \arrow[from=2-4, to=3-4]
            \arrow[from=3-1, to=3-2]
            \arrow[from=3-1, to=4-1]
            \arrow["\lrcorner"{anchor=center, pos=0.125}, draw=none, from=3-1, to=4-2]
            \arrow[from=3-2, to=3-3]
            \arrow[from=3-2, to=4-2]
            \arrow["\lrcorner"{anchor=center, pos=0.125}, draw=none, from=3-2, to=4-3]
            \arrow[from=3-3, to=3-4]
            \arrow[from=3-4, to=4-4]
            \arrow[from=4-1, to=5-1]
            \arrow[from=4-2, to=5-2]
            \arrow[from=4-4, to=5-4]
            \arrow[from=5-1, to=5-2]
            \arrow[from=5-2, to=5-3]
            \arrow[from=5-3, to=5-4]
        \end{tikzcd}
    \end{equation*}
    in $\mc C$ where each square is Cartesian and the vertical arrows belong to $E$.
\end{definition}

\begin{definition}\label{def:exchangethm}
    An \emph{exchange theorem} for the pair $(\mc C, E)$ valued in $\mc D$ is a map $\Pull(\mc C, E)\to \Sq(\mc D)$. We write
    \begin{equation*}
        \Exch_{(\mc C, E)}(-)\coloneqq \map_{\PSh(\Delta^{\times 2})}(\Pull(\mc C, E), \Sq(-)) : \Cat \longrightarrow \Ani
    \end{equation*}
    for the functor classifying exchange theorems.
\end{definition}

\begin{remark}
    Given a map $F : \Pull(\mc C, E)\longrightarrow \Sq(\mc D)$, one may recover the informal description by taking the map
    \begin{equation*}
        \begin{tikzcd}
            N(\mc C)\simeq \Pull(\mc C, E)_{0,\bullet}\arrow{r}{F_{0,\bullet}} & \Sq(\mc D)_{0,\bullet}\simeq N(\mc D)
        \end{tikzcd}
    \end{equation*}
    to be (the nerve of) $(-)_\ast$ and taking
    \begin{equation*}
        \begin{tikzcd}
            N(E)\simeq \Pull(\mc C, E)_{\bullet,0}\arrow{r}{F_{\bullet,0}} & \Sq(\mc D)_{\bullet,0}\simeq N(\mc D)
        \end{tikzcd}
    \end{equation*}
    to be (the nerve of) $(-)_\natural$. These two functors canonically agree on objects with common value $F_{0,0}$. Looking at $(1,1)$-simplices, $F$ provides a map $F_{1,1} : \Pull(\mc C, E)_{1,1}\to \map_{\Cat}([1]\times [1], \mc D)$ which sends a point
    \begin{equation*}
        \begin{tikzcd}
        	x & {y'} \\
        	y & z
        	\arrow["{\bar g}", from=1-1, to=1-2]
        	\arrow["{\bar f}"', from=1-1, to=2-1]
        	\arrow["\lrcorner"{anchor=center, pos=0.125}, draw=none, from=1-1, to=2-2]
        	\arrow["f"', from=1-2, to=2-2]
        	\arrow["g", from=2-1, to=2-2]
        \end{tikzcd}\in \Pull(\mc C, E)_{1,1}
    \end{equation*}
    to a commutative square
    \begin{equation*}
        \begin{tikzcd}
        	{F_{0,0}(x)} & {F_{0,0}(y')} \\
        	{F_{0,0}(y)} & {F_{0,0}(z)}
        	\arrow["{\bar g_\ast}", from=1-1, to=1-2]
        	\arrow["{\bar f_\natural}"', from=1-1, to=2-1]
        	\arrow["{f_\natural}"', from=1-2, to=2-2]
        	\arrow["{g_\ast}", from=2-1, to=2-2]
        \end{tikzcd}\in \Sq(\mc D)_{1,1}.
    \end{equation*}
    The maps $F_{n,m}$ on higher simplex spaces ensure that these equivalences are compatible with pasting of pullback squares.
\end{remark}

\subsubsection{Exchange theorems from biadjointability}

In this section we discuss an important source of exchange theorems coming from \emph{biadjointable} functors which is relevant to the study of six functor formalisms, as well as parametrized category theory. This will not be utilized in the sequel, so the uninterested reader may skip this section.\par
We recall the definition of biadjointability from \cite{cnossen2025universality}.

\begin{definition}
    Let $\mb D$ be a $2$-category.
    \begin{enumerate}
        \item A square
        \begin{equation*}
            \begin{tikzcd}
            	a & b \\
            	c & d
            	\arrow["{\ell'}", from=1-1, to=1-2]
            	\arrow["{r'}"', from=1-1, to=2-1]
            	\arrow["r"', from=1-2, to=2-2]
            	\arrow["\ell", from=2-1, to=2-2]
            \end{tikzcd}
        \end{equation*}
        in $\mb D$ is said to be \emph{biadjointable} if it is vertically left adjointable and the mated square
        \begin{equation*}
            \begin{tikzcd}
            	a & b \\
            	c & d
            	\arrow["{\ell'}", from=1-1, to=1-2]
            	\arrow["{(r')^L}", from=2-1, to=1-1]
            	\arrow["\ell", from=2-1, to=2-2]
            	\arrow["{r^L}", from=2-2, to=1-2]
            \end{tikzcd}
        \end{equation*}
        is further horizontally right adjointable.
        \item A functor $F : \mc C^\mathrm{op}\to \mb D$ is said to be $(\mc C, E)$-biadjointable if, for every pullback square
        \begin{equation*}
            \begin{tikzcd}
            	x & {y'} \\
            	y & z
            	\arrow["{\bar g}", from=1-1, to=1-2]
            	\arrow["{\bar f}"', from=1-1, to=2-1]
            	\arrow["\lrcorner"{anchor=center, pos=0.125}, draw=none, from=1-1, to=2-2]
            	\arrow["f", from=1-2, to=2-2]
            	\arrow["g"', from=2-1, to=2-2]
            \end{tikzcd}
        \end{equation*}
        in $\mc C$ with $f,\bar f\in E$, its image under $F$ is biadjointable. We write $\map_{(\mc C, E)\mathrm{-badj}}(\mc C^\mathrm{op}, \mb D)$ for the subspace of components of $\map_{\Cat}(\mc C^\mathrm{op}, \mb D)$ consisting of $(\mc C, E)$-biadjointable functors.
    \end{enumerate}
\end{definition}

The exchange theorem associated to a biadjointable functor $F : \mc C^\mathrm{op}\to \mb D$ will send a pullback square
\begin{equation*}
    \begin{tikzcd}
        x & {y'} \\
        y & z
        \arrow["{\bar g}", from=1-1, to=1-2]
        \arrow["{\bar f}"', from=1-1, to=2-1]
        \arrow["\lrcorner"{anchor=center, pos=0.125}, draw=none, from=1-1, to=2-2]
        \arrow["f", from=1-2, to=2-2]
        \arrow["g"', from=2-1, to=2-2]
    \end{tikzcd}\in \Pull(\mc C, E)_{1,1}
\end{equation*}
to the biadjointed square
\begin{equation*}
    \begin{tikzcd}
    	{F(x)} & {F(y')} \\
    	{F(y)} & {F(z)}.
    	\arrow["{F(\bar g)^R}", from=1-1, to=1-2]
    	\arrow["{F(\bar f)^L}"', from=1-1, to=2-1]
    	\arrow["{F(f)^L}", from=1-2, to=2-2]
    	\arrow["{F(g)^R}"', from=2-1, to=2-2]
    \end{tikzcd}
\end{equation*}
To do this, we will make use of the \emph{walking adjunction} $\textsc{Adj}$ from \cite{riehl2016homotopy} which is the 2-category freely generated by two adjoint morphisms
\begin{equation*}
    \begin{tikzcd}
        \ell : \ominus
        \arrow[r, ""{name=F}, bend left=10, yshift=3pt] &
        \oplus : r.
        \arrow[l, ""{name=G}, bend left=10, yshift=-3pt]
        \arrow[phantom, from=F, to=G, "\dashv" rotate=-90]
    \end{tikzcd}
\end{equation*}
From its universal property, one may show (see, e.g., \cite[Proposition 2.12 (3)]{cnossen2025universality}) that the restriction
\begin{equation*}
    (\ell\times r)^\ast : \map_{\Cat_2}(\textsc{Adj}\times \textsc{Adj}, \mb D)\longrightarrow \map_{\Cat}([1]\times [1], \mb D)
\end{equation*}
is an equivalence onto the subspace of components consisting of biadjointable squares.\par
More generally, let $\textsc{Adj}_n$ denote the $2$-category consisting of $n$ composable adjunctions, i.e.\ the iterated coproduct
\begin{equation*}
    \textsc{Adj}_n\coloneqq \textsc{Adj}\coprod_{\oplus\simeq \ominus} \textsc{Adj} \coprod_{\oplus\simeq \ominus}\cdots \coprod_{\oplus\simeq \ominus}\textsc{Adj}
\end{equation*}
taken in the category $\Cat_2$ of $2$-categories. Since the inclusion $\Cat \hookrightarrow\Cat_2$ has both adjoints, hence preserves colimits, we get an induced map
\begin{equation*}
    \begin{tikzcd}[column sep=50pt]
        \ell_n : [n]\simeq [1]\,\displaystyle\coprod_{[0]}\,[1]\,\displaystyle\coprod_{[0]}\,\cdots\, \displaystyle\coprod_{[0]} \,[1]\arrow{r}{\ell\,\amalg\,\ell\,\amalg\,\cdots\,\amalg\,\ell} & \textsc{Adj}_n
    \end{tikzcd}
\end{equation*}
and similarly a map $r_n : [n]\longrightarrow \textsc{Adj}_n$. Mimicking the proof of \cite[Proposition 2.12 (3)]{cnossen2025universality}, one has that the restriction
\begin{equation*}
    (\ell_n\times r_m)^\ast : \map_{\Cat_2}(\textsc{Adj}_n\times \textsc{Adj}_m, \mb D)\longrightarrow \map_{\Cat}([n]\times [m], \mb D)
\end{equation*}
is an equivalence onto the subspace of components consisting of $n\times m$-grids such that every square consisting of vertical and horizontal arrows is biadjointable.

\begin{definition}
    Let $\mb D$ be a $2$-category. Define the bisimplicial of space of biadjointable squares to be
    \begin{equation*}
        \Sq_\mathrm{badj}(\mb D)\coloneqq \map_{\Cat_2}(\textsc{Adj}_\bullet\times\textsc{Adj}_\bullet, \mb D).
    \end{equation*}
\end{definition}

By the above discussion, there is a natural monomorphism $\Sq_\mathrm{badj}(\mb D)\to \Sq(\mb D)$ onto the sub-bisimplicial space of commutative grids consisting of biadjointable squares.

\begin{proposition}\label{prop:biadjoinsqsymmetry}
    There is a natural in $\mb D$ equivalence $\Sq_\mathrm{badj}(\mb D)\simeq \Sq_\mathrm{badj}(\mb D^{\mathrm{1-op},\;\mathrm{2-op}})$ which sends a biadjointable square to its biadjoined square.
\end{proposition}

\begin{remark}
    The appearance of $2$-op is because we have fixed the convention that in a biadjointable square, the vertical morphisms are left adjointable and the horizontal morphisms are right adjointable. By passing to the biadjoined square, we get a biadjointable square where the convention has been swapped.
\end{remark}

\begin{proof}
    Since $(-)^{\mathrm{1-op},\;\mathrm{2-op}}$ is an automorphism of $\Cat_2$, we have a natural in $\mb D$ equivalence
    \begin{equation*}
        \Sq_\mathrm{badj}(\mb D)\simeq \map_{\Cat_2}(\textsc{Adj}_\bullet^{\mathrm{1-op},\;\mathrm{2-op}}\times \textsc{Adj}_\bullet^{\mathrm{1-op},\;\mathrm{2-op}}, \mb D^{\mathrm{1-op},\;\mathrm{2-op}}).
    \end{equation*}
    However, by the universal property of $\textsc{Adj}$ we have that
    \begin{equation*}
        \textsc{Adj}^{\mathrm{1-op},\;\mathrm{2-op}}\simeq \textsc{Adj}
    \end{equation*}
    via the functor which fixes the two objects $\oplus$, $\ominus$ and swaps the left and right adjoints $\ell$ and $r$. Since $(-)^{\mathrm{1-op},\;\mathrm{2-op}}$ is an automorphism, and hence preserves colimits, this equivalence extends to a functorial in $n$ equivalence
    \begin{equation*}
        \textsc{Adj}_n^{\mathrm{1-op},\;\mathrm{2-op}}\simeq \textsc{Adj}_n.
    \end{equation*}
    Thus
    \begin{equation*}
        \begin{aligned}
            \Sq_\mathrm{badj}(\mb D) &\simeq \map_{\Cat_2}(\textsc{Adj}_\bullet^{\mathrm{1-op},\;\mathrm{2-op}}\times \textsc{Adj}_\bullet^{\mathrm{1-op},\;\mathrm{2-op}}, \mb D^{\mathrm{1-op},\;\mathrm{2-op}}) \\
            &\simeq \map_{\Cat_2}(\textsc{Adj}_\bullet\times \textsc{Adj}_\bullet, \mb D^{\mathrm{1-op},\;\mathrm{2-op}}) \\
            &= \Sq_\mathrm{badj}(\mb D^{\mathrm{1-op},\;\mathrm{2-op}})
        \end{aligned}
    \end{equation*}
    as required.
\end{proof}

We may now construct the desired exchange theorems.

\begin{proposition}
    There is a natural in $\mb D$ map
    \begin{equation*}
        \map_{(\mc C, E)\mathrm{-badj}}(\mc C^\mathrm{op}, \mb D)\longrightarrow \Exch_{(\mc C, E)}(\mb D).
    \end{equation*}
\end{proposition}
\begin{proof}
    Given a biadjointable $F : \mc C^\mathrm{op}\to\mb D$ we may consider the composite
    \begin{equation*}
        \begin{tikzcd}
            \Pull(\mc C, E)\arrow[hook]{r} & \Sq(\mc C)\arrow{r}{\Sq(F^\mathrm{op})} & \Sq(\mb D^{\mathrm{1-op},\;\mathrm{2-op}})
        \end{tikzcd}
    \end{equation*}
    where we are free to add the $2$-op since $\Sq$ only depends on the underlying $1$-category of $\mb D$. By the definition of biadjointability, this lands inside the sub-bisimplicial space of biadjointable squares and hence has a unique lift
    \begin{equation*}
        \Pull(\mc C, E)\longrightarrow \Sq_\mathrm{badj}(\mb D^{\mathrm{1-op},\;\mathrm{2-op}})
    \end{equation*}
    which by \Cref{prop:biadjoinsqsymmetry} produces via biadjoining squares an exchange theorem
    \begin{equation*}
        \Pull(\mc C, E)\longrightarrow \Sq_\mathrm{badj}(\mb D^{\mathrm{1-op},\;\mathrm{2-op}})\simeq \Sq_\mathrm{badj}(\mb D)\longrightarrow \Sq(\mb D).\qedhere
    \end{equation*}
\end{proof}


\subsection{Vector bundles}\label{subsec:vectorbundles}

In this section, we define and study the analog of vector bundles which appear in the universal exchange theorem.

\begin{definition}
    \begin{enumerate}
        \item For each $x\in \mc C$, we let $\Vect_{(\mc C, E)}(x)$ denote the full subcategory of the double slice category
        \begin{equation*}
            \mc C_{x/,/x}\coloneqq (\mc C_{x/})_{/\id_x}
        \end{equation*}
        spanned by objects
        \begin{equation*}
            \begin{tikzcd}
                x\arrow{r}{s} & e\arrow{r}{p} & x
            \end{tikzcd}
        \end{equation*}
        with $p\in E$. We refer to objects of $\Vect_{(\mc C, E)}(x)$ as vector bundles over $x$.
        \item Denote by $\Vect_{(\mc C, E)}(x)_\dagger$ the wide subcategory of $\Vect_{(\mc C, E)}(x)$ fitting into the pullback
        \begin{equation*}
            \begin{tikzcd}
            	{\Vect_{(\mc C, E)}(x)_\dagger} & {\Vect_{(\mc C, E)}(x)} \\
            	E & {\mc C}
            	\arrow[hook, from=1-1, to=1-2]
            	\arrow[from=1-1, to=2-1]
            	\arrow["\lrcorner"{anchor=center, pos=0.125}, draw=none, from=1-1, to=2-2]
            	\arrow[from=1-2, to=2-2]
            	\arrow[hook, from=2-1, to=2-2]
            \end{tikzcd}
        \end{equation*}
        where $\Vect_{(\mc C, E)}(x)\to\mc C$ is the forgetful map. That is, $\Vect_{(\mc C, E)}(x)_\dagger$ is the wide subcategory consisting of morphisms
        \begin{equation*}
            \begin{tikzcd}
            	& x & \\
            	e && {e'} \\
            	& x
            	\arrow[from=1-2, to=2-1]
            	\arrow[from=1-2, to=2-3]
            	\arrow["f", two heads, from=2-1, to=2-3]
            	\arrow[two heads, from=2-1, to=3-2]
            	\arrow[two heads, from=2-3, to=3-2]
            \end{tikzcd}
        \end{equation*}
        of vector bundles with $f\in E$.
    \end{enumerate}
\end{definition}

Since $\mc C$ and $E$ have been fixed, for convenience of notation we will simply write $\Vect(x)$ and $\Vect(x)_\dagger$ for $\Vect_{(\mc C, E)}(x)$ and $\Vect_{(\mc C, E)}(x)_\dagger$, respectively.\par
Given a morphism $f : x\to y\in \mc C$, since $\mc C$ has pullbacks along morphisms in $E$ and $E$ is closed under pullback, we obtain a functor
\begin{equation*}
    \begin{tikzcd}[row sep=0pt]
        f^\ast : \Vect(y)\arrow{r} & \Vect(x) \\
        y\to e\twoheadrightarrow y\arrow[mapsto]{r} & x\to e\times_y x\twoheadrightarrow x
    \end{tikzcd}
\end{equation*}
given by pulling back vector bundles along $f$. These assemble into a functor
\begin{equation*}
    \begin{tikzcd}[row sep=0pt]
        \Vect(-) : \mc C^\mathrm{op}\arrow{r} & \Cat.
    \end{tikzcd}
\end{equation*}\par
We now show that the pair $(\Vect(x), \Vect(x)_\dagger)$ is a co-Waldhausen category in the sense of \cite{barwick2016algebraic}, the definition of which we recall now.

\begin{definition}[{\cite[Definition 2.7, Definition 2.16]{barwick2016algebraic}}]\label{def:cowaldhausencat}
    A co-Waldhausen category is a pair $(\mc D, \mc D_\dagger)$ of a category $\mc D$ and a wide subcategory $\mc D_\dagger\subseteq \mc D$ such that
    \begin{enumerate}
        \item $\mc D$ has a zero object
        \item every map $x\to 0$ to a zero object in $\mc D$ belongs to $\mc D_\dagger$
        \item pullbacks along morphisms in $\mc D_\dagger$ in $\mc D$ exist
        \item $\mc D_\dagger$ is closed under pullback.
    \end{enumerate}
    If $(\mc E,\mc E_\dagger)$ is another co-Waldhausen category, a functor $\psi : \mc D\to \mc E$ is said to be \emph{exact} if
    \begin{enumerate}
        \item $\psi(\mc D_\dagger)\subseteq \mc E_\dagger$
        \item $\psi$ preserves zero objects
        \item $\psi$ preserves pullbacks along morphisms in $\mc D_\dagger$.
    \end{enumerate}
    We denote by $\coWald$ the category of co-Waldhausen categories and exact functors.
\end{definition}

\begin{proposition}\label{prop:liftofVecttocoWald}
    The pair $(\Vect(x), \Vect(x)_\dagger)$ is a co-Waldhausen category and the functor $\Vect : \mc C^\mathrm{op}\to \Cat$ lifts to a functor $\Vect : \mc C^\mathrm{op}\to\coWald$.
\end{proposition}
\begin{proof}
    We check each condition of \Cref{def:cowaldhausencat}.\par
    Firstly, the forgetful map $\mc C_{x/,/x}\to\mc C$ is conservative, so the fact that $\Vect(x)_\dagger$ is a wide subcategory follows from the fact that $E$ is a wide subcategory.\par
    For (i), we have that $\mc C_{x/,/x}$ has a zero object given by $x\xrightarrow{\id} x\xrightarrow{\id} x$. Since $E$ is a wide subcategory of $\mc C$, this object belongs to the full subcategory $\Vect(x)$ and is thus still a zero object in $\Vect(x)$. (ii) then follows from the definition of $\Vect(x)_\dagger$ and $\Vect(x)$ as every map $e\to 0$ is equivalent to the structure map $e\to x$ which is assumed to lie in $E$.\par
    For (iii), $\mc C_{x/,/x}$ has all weakly contractible limits that $\mc C$ does, and the forgetful functor $\mc C_{x/,/x}\to\mc C$ preserves and reflects these limits. Thus, suppose we have a cospan $e_1\xrightarrow{f} e_2\xleftarrow{g} e_3$ in $\Vect(x)$ with $f\in \Vect(x)_\dagger$. Then we have a pullback square
    \begin{equation*}
        \begin{tikzcd}
        	{e_1\times_{e_2}e_3} & {e_1} \\
        	{e_3} & {e_2}
        	\arrow["{\bar g}", from=1-1, to=1-2]
        	\arrow["{\bar f}"', from=1-1, to=2-1]
        	\arrow["\lrcorner"{anchor=center, pos=0.125}, draw=none, from=1-1, to=2-2]
        	\arrow["f", from=1-2, to=2-2]
        	\arrow["g"', from=2-1, to=2-2]
        \end{tikzcd}
    \end{equation*}
    in $\mc C_{x/,/x}$ where the underlying diagram in $\mc C$ is a pullback and the incoming structure map on $e_1\times_{e_2}e_3$ is the product of the incoming structure maps on $e_1$ and $e_3$, and the outgoing structure map is
    \begin{equation*}
        e_1\times_{e_2}e_3\longrightarrow e_2\longrightarrow x.
    \end{equation*}
    This is therefore a pullback in $\Vect(x)$ so long as $e_1\times_{e_2}e_3\in \Vect(x)$. However, the outgoing structure map of $e_1\times_{e_2}e_3$ is equivalent to
    \begin{equation*}
        \begin{tikzcd}
            e_1\times_{e_2}e_3\arrow{r}{\bar f} & e_3\arrow{r} & x
        \end{tikzcd}
    \end{equation*}
    which is a composite of two morphisms belonging to $E$ since $E$ is stable under pullbacks in $\mc C$.\par
    Finally, for (iv), the above shows that $\Vect(x)\hookrightarrow\mc C_{x/,/x}$ preserves pullbacks along morphisms in $\Vect(x)_\dagger$ and hence so does the composite $\Vect(x)\to \mc C$. It then follows from the definition of $\Vect(x)_\dagger$ that $\Vect(x)_\dagger$ is closed under pullback by virtue of $E$ being closed under pullback.\par
    Lastly, one checks that for every $f : x\to y$, the pullback functor $f^\ast : \Vect(y)\to\Vect(x)$ is exact. This uses similar techniques, so we omit the proof.
\end{proof}



\subsection{Partial $K$-theory}

For the remainder of this subsection, let $(\mc D, \mc D_\dagger)$ be a co-Waldhausen category in the sense of \Cref{def:cowaldhausencat}. The goal of this section is to introduce a non-group completed variant of $K$-theory. Upon writing this, the author was made aware of \cite{yuan2023integral} which introduces the same construction. For completeness, we continue to include this section but we refer the reader to work of Yuan \cite{yuan2023integral} for a more comprehensive study of partial $K$-theory.\par

Informally, to a co-Waldhausen category $(\mc D, \mc D_\dagger)$, we wish to associate an $\mb E_\infty$-space $\mc K(\mc D)$ which at the level of $\pi_0$ is freely generated by the objects of $\mc D$ modulo the relation $[b] = [a] + [c]$ for every fiber sequence
\begin{equation*}
    a\to b\twoheadrightarrow c
\end{equation*}
with $b\twoheadrightarrow c\in \mc D_\dagger$.\par
The partial $K$-theory should recover the ordinary algebraic $K$-theory of $(\mc D,\mc D_\dagger)$ upon group completion, i.e.\ $\mc K(\mc D)^\mathrm{gp}\simeq K(\mc D)$ where $(-)^\mathrm{gp}$ is the left adjoint to the inclusion of grouplike $\mb E_\infty$-spaces into $\mb E_\infty$-spaces. We remark that this distinction only appears in the $K$-theory of (co-)Waldhausen categories. For a stable category $\mc D$, which may be viewed as a (co-)Waldhausen category with $\mc D_\dagger = \mc D$, the existence of inverses is already guaranteed by the relations imposed via fiber sequences. Indeed, for every $x$ we may consider the fiber sequence
\begin{equation*}
    y\to 0\twoheadrightarrow x
\end{equation*}
which in $\pi_0\mc K(\mc D)$ gives the relation
\begin{equation*}
    0 = [x] + [y],
\end{equation*}
witnessing $[y]$ as an inverse to $[x]$. However, for a general co-Waldhausen category, it will often not be the case that $0\to x$ belongs to $\mc D_\dagger$.\par
To do this, we first recall the $S_\bullet$-construction.

\begin{definition}
    The $S_\bullet$-construction associated to $\mc D$, denoted $S_\bullet(\mc D)$, is the simplicial space whose $n$-simplex space is the groupoid of diagrams $\Cpt^n\to\mc D$
    \begin{equation*}
        \begin{tikzcd}
        	0 & {x_{11}} & {x_{12}} & \cdots & {x_{1n}} \\
        	& 0 & {x_{22}} & \cdots & {x_{2n}} \\
        	&& 0 & \ddots & \vdots \\
        	&&& \ddots & {x_{nn}} \\
        	&&&& 0
        	\arrow[from=1-1, to=1-2]
        	\arrow[from=1-2, to=1-3]
        	\arrow[two heads, from=1-2, to=2-2]
        	\arrow["\lrcorner"{anchor=center, pos=0.125}, draw=none, from=1-2, to=2-3]
        	\arrow[from=1-3, to=1-4]
        	\arrow[two heads, from=1-3, to=2-3]
        	\arrow["\lrcorner"{anchor=center, pos=0.125}, draw=none, from=1-3, to=2-4]
        	\arrow[from=1-4, to=1-5]
        	\arrow[two heads, from=1-5, to=2-5]
        	\arrow[from=2-2, to=2-3]
        	\arrow[from=2-3, to=2-4]
        	\arrow[two heads, from=2-3, to=3-3]
        	\arrow["\lrcorner"{anchor=center, pos=0.125}, draw=none, from=2-3, to=3-4]
        	\arrow[from=2-4, to=2-5]
        	\arrow[two heads, from=2-5, to=3-5]
        	\arrow[two heads, from=3-5, to=4-5]
        	\arrow[two heads, from=4-5, to=5-5]
        \end{tikzcd}
    \end{equation*}
    where the diagonal consists of zero objects, all squares consisting of vertical and horizontal arrows are Cartesian, and all vertical arrows belong to $\mc D_\dagger$.
\end{definition}

The ordinary algebraic $K$-theory of $\mc D$ is then defined as endomorphisms of $0$ in the geometric realization of $S_\bullet(\mc D)$ (see \cite{barwick2016algebraic}). To define partial $K$-theory, we take the associated category rather than the associated groupoid.

\begin{definition}
    The \emph{partial $K$-theory} of $\mc D$, denoted $\mc K(\mc D)$, is defined as $\End_0(\ascat S_\bullet(\mc D))$.
\end{definition}

We first want to show that $\mc K(\mc D)$ has an $\mb E_\infty$-structure. This will be a consequence of the following lemma.

\begin{lemma}\label{lemma:Sbulletandascatpresprod}
    Both $S_\bullet : \coWald\to \PSh(\Delta)$ and $\ascat : \PSh(\Delta)\to\Cat$ preserve products.
\end{lemma}
\begin{proof}
    By a dual version of \cite[Proposition 4.4]{barwick2016algebraic}, $\coWald$ has small limits and the forgetful map
    \begin{equation*}
        \begin{tikzcd}[row sep=0pt]
            \coWald\arrow{r} & \Fun([1], \Cat) \\
            (\mc D, \mc D_\dagger)\arrow[mapsto]{r} & \mc D_\dagger \hookrightarrow\mc D
        \end{tikzcd}
    \end{equation*}
    preserves them. In particular, we have that $(\mc D,\mc D_\dagger)\times (\mc E, \mc E_\dagger)\simeq (\mc D\times \mc E, \mc D_\dagger\times \mc E_\dagger)$. Since limits in the product category $\mc D\times\mc E$ are computed component-wise, we see that the induced projections
    \begin{equation*}
        \begin{tikzcd}
            S_\bullet(\mc D)\arrow[<-]{r} & S_\bullet(\mc D\times \mc E)\arrow{r} & S_\bullet(\mc E)
        \end{tikzcd}
    \end{equation*}
    realize $S_\bullet(\mc D\times \mc E)$ as the product $S_\bullet(\mc D)\times S_\bullet(\mc E)$.\par
    For $\ascat$, we have a natural transformation $\ascat((-)\times (-))\Rightarrow \ascat(-)\times \ascat(-)$ and both sides preserve colimits in each variable. Thus it suffices to check for $\Delta[m]$ and $\Delta[n]$ in which case the natural transformation is an equivalence, both yielding $[n]\times [m]$.
\end{proof}

\begin{proposition}\label{prop:partialKthygroupcmpl}
    The following hold.
    \begin{enumerate}
        \item $\mc K : \coWald\to \Ani$ naturally lifts to $\mc K : \coWald\to \CAlg(\Ani)$
        \item There is a natural transformation $\mc K\Rightarrow K$ of functors $\coWald\to \CAlg(\Ani)$ which realizes $K(\mc D)$ as the group completion of $\mc K(\mc D)$.
    \end{enumerate}
\end{proposition}
\begin{proof}
    We may equip each of $\coWald$, $\PSh(\Delta)$, $\Cat$ and $\Ani$ with the Cartesian monoidal structure. Then by \Cref{lemma:Sbulletandascatpresprod}, along with the standard fact that $|-| : \Cat\to \Ani$ preserves products, we may lift all these functors to functors
    \begin{equation*}
        \begin{aligned}
            S_\bullet : & \CAlg(\coWald)\longrightarrow \CAlg(\PSh(\Delta)) \\
            \ascat : & \CAlg(\PSh(\Delta))\longrightarrow \CAlg(\Cat) \\
            | - | : & \CAlg(\Cat)\longrightarrow \CAlg(\Ani).
        \end{aligned}
    \end{equation*}
    Moreover, since products in $\coWald$ are computed as products of the underlying pair, and exact functors preserve products, there is a section $\coWald\to \CAlg(\coWald)$ of the forgetful map $\CAlg(\coWald)\to \coWald$ which equips each coWaldhausen category with its Cartesian monoidal structure.\par
    Thus we may take the lift of $\mc K$ to be the composite
    \begin{equation*}
        \begin{tikzcd}
            \coWald\arrow{r} & \CAlg(\coWald)\arrow{r}{S_\bullet} & \CAlg(\PSh(\Delta))\arrow{r}{\ascat} & \CAlg(\Cat)\arrow{r}{\End_{1}(-)} & \CAlg(\Ani),
        \end{tikzcd}
    \end{equation*}
    which agrees with the definition of $\mc K$ since $0$ is the monoidal unit for $\ascat S_\bullet(\mc D)$.\par
    The natural transformation $\mc K\Rightarrow K$ is given by taking endomorphisms at the unit of the symmetric monoidal transformation
    \begin{equation*}
        \ascat S_\bullet(-)\Longrightarrow |S_\bullet(-)|.
    \end{equation*}
    Since $\ascat S_\bullet(-)$ is (equivalent to) a one object category, $|S_\bullet(-)|$ is formed by group completing the endomorphism space of this object and the claim follows.
\end{proof}

We also prove the claim that, when $\mc D_\dagger$ is sufficiently large, the partial $K$-theory of $\mc D$ agrees with the usual algebraic $K$-theory of $\mc D$.

\begin{lemma}\label{lemma:simpspacerightinverses}
    Let $X$ be a simplicial space in which every $f\in X_1$ has a right (resp.\ left) inverse, i.e.\ there exists some $h\in X_2$ such that $[d_0 h] = [f]\in \pi_0(X_1)$ and $[d_1 h]\in \pi_0(X_1)$ is in the image of $\pi_0 s_0 : \pi_0(X_0)\to \pi_0(X_1)$, then $\ascat X$ is a groupoid.
\end{lemma}
\begin{proof}
    We will show that in the homotopy category of $\ascat X$, every morphism has a right inverse at which point it follows that every morphism is in fact invertible. Indeed, if $\mc C$ is an ordinary category in which every morphism has a right inverse, then given an arbitrary morphism $f$ let $g$ be its right inverse and $h$ be the right inverse of $g$. Then $fg = \id$ by construction and $f = fgh = h$ so $g$ is also the left inverse of $f$.\par
    First, we reduce to the case where each $X_n$ is discrete, following \cite[Example 1.22]{barkan2025segalification}. Let $Z_0\to X_0$ be a $\pi_0$-surjective map and define $Z_n = Z_0^{n + 1}\times_{X_0^{n + 1}} X_n$ to get a new simplicial space $Z_\bullet$. Using the path space model for the homotopy fiber product, one readily checks that every morphism of $Z_\bullet$ also has a right inverse. Since the projection $Z\to X$ induces an equivalence on associated categories, by taking $Z_0$ to be discrete we may reduce to the case that $X_0$ is discrete.\par
    If $X_0$ is discrete, then the simplicial space $\pi_0 X\coloneqq (\pi_0 X_n)_n$ has the property that $X\to \pi_0 X$ induces an equivalence on homotopy categories of associated categories. Since $\pi_0 X$ also clearly inherits the property of having right inverses from $X$, we reduce to the case where $X$ is discrete.\par
    If $X$ is discrete, then we have from \cite[Example 1.22]{barkan2025segalification} that
    \begin{equation*}
        \Mor h\ascat X = \left(\coprod_{n\ge 0} X_1\times_{X_0}X_1\times_{X_0}\cdots \times_{X_0}X_1\right)/\sim
    \end{equation*}
    where $(f_1, f_2,\dots, f_n)\sim (f_1, \dots, f_{i-1}, g, f_{i + 1},\dots, f_n)$ whenever there is a $2$-simplex witnessing $f_{i + 1}\circ f_i = g$ and $(f_1,\dots, f_n)\sim (f_1,\dots, f_{i - 1}, f_{i + 1}, \dots, f_n)$ whenever $f_i$ is a degenerate simplex, and composition is given by concatenation. From this formula, it is clear that if $(f_1, \dots, f_n)$ represents a morphism and $g_1,\dots, g_n$ are right inverses of $f_1,\dots,f_n$ respectively, then $(g_n,\dots, g_1)$ is a right inverse to $(f_1,\dots, f_n)$ as required.
\end{proof}

\begin{corollary}\label{cor:ascatkthy}
    Suppose that $(\mc D,\mc D_\dagger)$ is such that every morphism $0\to x$ belongs to $\mc D_\dagger$. Then $\mc K(\mc D)$ is grouplike, i.e.\ $\mc K(\mc D)\simeq K(\mc D)$.
\end{corollary}
\begin{proof}
    Let $0$ be a zero object of $\mc D$ and let (up to equivalence)
    $0\to x\twoheadrightarrow 0$ be an arbitrary morphism in $S_1(\mc D)$. Then we have a $2$-simplex
    \begin{equation*}
        \begin{tikzcd}
            0 & {\Omega x} & 0 \\
            & 0 & x \\
            && 0
            \arrow[from=1-1, to=1-2]
            \arrow[from=1-2, to=1-3]
            \arrow[two heads, from=1-2, to=2-2]
            \arrow["\lrcorner"{anchor=center, pos=0.125}, draw=none, from=1-2, to=2-3]
            \arrow[two heads, from=1-3, to=2-3]
            \arrow[from=2-2, to=2-3]
            \arrow[two heads, from=2-3, to=3-3]
        \end{tikzcd}
    \end{equation*}
    which witnesses a right inverse to $0\to x\twoheadrightarrow 0$. The result then follows from \Cref{lemma:simpspacerightinverses}.
\end{proof}

For the case of vector bundles as discussed in \Cref{subsec:vectorbundles}, this has the following consequence.

\begin{corollary}
    If $E$ is left-cancellable, then $\Vect(x)_\dagger = \Vect(x)$ for all $x$. In particular, $\mc K(\Vect(x))\simeq K(\Vect(x))$.
\end{corollary}
\begin{proof}
    Given a morphism of vector bundles
    \begin{equation*}
        \begin{tikzcd}
        	& x & \\
        	e && {e'}, \\
        	& x
        	\arrow[from=1-2, to=2-1]
        	\arrow[from=1-2, to=2-3]
        	\arrow["g", from=2-1, to=2-3]
        	\arrow["p"', two heads, from=2-1, to=3-2]
        	\arrow["{p'}", two heads, from=2-3, to=3-2]
        \end{tikzcd}
    \end{equation*}
    where $p,p'\in E$ by definition of $\Vect(x)$, left-cancellability of $E$ implies that $g\in E$ and the result follows.
\end{proof}

\subsection{Adjoining endomorphisms and the universal recipient}\label{subsec:adjendos}

We are now equipped to construct the category which is the recipient of the universal exchange theorem.

\begin{definition}\label{def:TwC}
    Define $\TwC$ to be the total space of the Cartesian unstraightening of the functor
    \begin{equation*}
        \begin{tikzcd}
            \mc C^\mathrm{op}\arrow{r}{\Vect} & \coWald\arrow{r}{\mc K} & \CAlg(\Ani)\arrow{r}{B} & \Cat.
        \end{tikzcd}
    \end{equation*}
    where $\Vect : \mc C^\mathrm{op}\to\coWald$ is the lift of \Cref{prop:liftofVecttocoWald}.
\end{definition}

Since each $B\mc K(\Vect(x))$ is a one object category, $\TwC$ may informally be described as the category whose:
\begin{itemize}
    \item objects are the same as those in $\mc C$
    \item a morphism $x\to y$ in $\TwC$ is a pair $(f, \xi)$ where $f : x\to y\in\mc C$ and $\xi\in \mc K(\Vect(x))$
    \item composition is given by
    \begin{equation*}
        \begin{tikzcd}[row sep=35pt]
            x\arrow{r}{(f,\;\xi)} & y\arrow{r}{(g,\;\eta)} & z
        \end{tikzcd} = \begin{tikzcd}[column sep=55pt]
            x\arrow{r}{(gf,\; f^\ast \eta + \xi)} & z.
        \end{tikzcd}
    \end{equation*}
\end{itemize}
We spend the remainder of this section making precise this claim.\par
More generally, given a functor
\begin{equation}\label{eq:Gfuncofendos}
    \begin{tikzcd}
        G : \mc D^\mathrm{op}\arrow{r} & \Alg_{\mb E_1}(\Ani)
    \end{tikzcd}
\end{equation}
where $\Alg_{\mb E_1}(\Ani)$ is the category of $\mb E_1$-spaces, we give a description of the Cartesian unstraightening of
\begin{equation}\label{eq:paramofendos}
    \begin{tikzcd}
        \mc D^\mathrm{op}\arrow{r}{G} & \Alg_{\mb E_1}(\Ani)\arrow{r}{B} & \Cat
    \end{tikzcd}
\end{equation}
which has the effect of adding $G(x)$ worth of endomorphisms to each object $x$.

\begin{definition}
    Given a functor \eqref{eq:Gfuncofendos}, we set $\mc D^G$ to be the Cartesian unstraightening of \eqref{eq:paramofendos}.
\end{definition}

\begin{remark}
    Because unstraightening is functorial, the construction $(\mc D, G)\mapsto \mc D^G$ is functorial in $(\mc D, G)\in \mathrm{LabCat}$ where $\mathrm{LabCat}$ is the category whose
    \begin{itemize}
        \item objects are pairs $(\mc D, G)$ where $\mc D$ is a category and $G : \mc D^\mathrm{op}\longrightarrow \Alg_{\mb E_1}(\Ani)$ is a functor
        \item morphisms $(\mc D, G)\to (\mc D', G')$ are pairs $(F, \eta)$ where $F : \mc D\to\mc D'$ is a functor and $\eta : G\to G'\circ F^\mathrm{op}$ is a natural transformation.
    \end{itemize}
    Explicitly, $\mathrm{LabCat}$ is the (total space of the) Cartesian unstraightening of
    \begin{equation*}
        \begin{tikzcd}
            \Fun((-)^\mathrm{op}, \Alg_{\mb E_1}(\Ani)) : \Cat^\mathrm{op}\arrow{r} & \Cat.
        \end{tikzcd}
    \end{equation*}
\end{remark}

Importantly, the functor \eqref{eq:paramofendos} factors as
\begin{equation*}
    \begin{tikzcd}
        \mc D^\mathrm{op}\arrow{r}{G} & \Alg_{\mb E_1}(\Ani)\arrow{r}{B} & \Cat_{\ast/}\arrow{r} & \Cat
    \end{tikzcd}
\end{equation*}
by remembering the pointing coming from delooping. This pointing under the unstraightening correspondence gives rise to a Cartesian section
\begin{equation*}
    s : \mc D\longrightarrow\mc D^G
\end{equation*}
of the fibration $\mc D^G\to \mc D$. Alternatively, by functoriality of the construction $(\mc D, G)\mapsto \mc D^G$, this section arises from the morphism $(\mc D, \ast)\to (\mc D, G)$ of pairs given by $(\id_{\mc D}, 1 : \ast\Rightarrow G)$ where $1 : \ast\Rightarrow G$ is the transformation picking out the unit at each level.

\begin{proposition}\label{prop:descrofadjendos}
    $s : \mc D\longrightarrow\mc D^G$ is essentially surjective and there exist equivalences
    \begin{equation*}
        \map_{\mc D^G}(s(x), s(y))\simeq G(x)\times \map_{\mc D}(x,y)
    \end{equation*}
    under which composition corresponds to
    \begin{equation*}
        \begin{tikzcd}[row sep=0pt]
            G(x)\times \map_{\mc D}(x,y)\times G(y)\times \map_{\mc D}(y,z)\arrow{r} & G(x)\times \map_{\mc D}(x,z) \\
            ((\xi,\; f),\; (\eta,\; g))\arrow[mapsto]{r} & (G(f)(\eta) \cdot \xi,\; gf).
        \end{tikzcd}
    \end{equation*}
\end{proposition}
\begin{proof}
    Let $\pi : \mc D^G\to\mc D$ be the projection map. For essential surjectivity of $s$, let $\alpha\in \mc D^G$. Then we have that $\alpha$ and $s(\pi(\alpha))$ both belong to the fiber $\pi^{-1}(\pi(\alpha))\simeq BG(\pi(\alpha))$ which is a one object category. Hence we have an equivalence $s(\pi(\alpha))\simeq \alpha$ in $\pi^{-1}(\pi(\alpha))$ and therefore in $\mc D^G$ as well.\par
    Next, because $s$ preserves Cartesian edges, for every $f : x\to y\in\mc D$ we have that $s(f)$ is a $\pi$-Cartesian lift of $f$. It follows that we have a pullback square
    \begin{equation}\label{eq:bigcartsqcomp}
        \begin{tikzcd}
        	{(\alpha, f)} & {s(f)\circ \alpha} \\[-15pt]
        	{\map_{\mc D^G}(s(x),s(x))\times\map_{\mc D}(x,y)} & {\map_{\mc D^G}(s(x), s(y))} \\
        	{\map_{\mc D}(x,x)\times \map_{\mc D}(x,y)} & {\map_{\mc D}(x,y)}. \\[-15pt]
        	{(\beta, f)} & {f\circ \beta}
        	\arrow[maps to, from=1-1, to=1-2]
        	\arrow[from=2-1, to=2-2]
        	\arrow[from=2-1, to=3-1]
        	\arrow["\lrcorner"{anchor=center, pos=0.125}, draw=none, from=2-1, to=3-2]
        	\arrow[from=2-2, to=3-2]
        	\arrow[from=3-1, to=3-2]
        	\arrow[maps to, from=4-1, to=4-2]
        \end{tikzcd}
    \end{equation}
    This may be verified as a pullback by checking it on fibers and applying \cite[Proposition 2.4.4.3]{lurie2009higher}. By \cite[Proposition 2.4.4.2]{lurie2009higher}, we also have a fiber sequence
    \begin{equation*}
        \begin{tikzcd}
            G(x)\simeq \map_{\pi^{-1}(x)}(s(x),s(x))\arrow{r} & \map_{\mc D^G}(s(x), s(x))\arrow{r} & \map_{\mc D}(x,x)
        \end{tikzcd}
    \end{equation*}
    where the fiber is taken over $\id_x$. Thus pulling back \eqref{eq:bigcartsqcomp} along $\{\id_x\}\times \map_{\mc D}(x,y)\to \map_{\mc D}(x,x)\times \map_{\mc D}(x,y)$ we have a Cartesian square
    \begin{equation*}
        \begin{tikzcd}
        	{G(x)\times\map_{\mc D}(x,y)} & {\map_{\mc D^G}(s(x), s(y))} \\
        	{\map_{\mc D}(x,y)} & {\map_{\mc D}(x,y)}
        	\arrow[from=1-1, to=1-2]
        	\arrow[from=1-1, to=2-1]
        	\arrow["\lrcorner"{anchor=center, pos=0.125}, draw=none, from=1-1, to=2-2]
        	\arrow[from=1-2, to=2-2]
        	\arrow[equals, from=2-1, to=2-2]
        \end{tikzcd}
    \end{equation*}
    from which we deduce the top map is an equivalence. Moreover, the top morphism is given by the composite
    \begin{equation*}
        \begin{tikzcd}[column sep=35pt]
            G(x)\times \map_{\mc D}(x,y)\arrow{r} & \map_{\pi^{-1}(x)}(s(x), s(x))\times \map_{\mc D^G}(s(x), s(y))\arrow{r}{(-)\circ (-)} & \map_{\mc D^G}(s(x), s(y)).
        \end{tikzcd}
    \end{equation*}
    from which one readily checks that composition is as claimed.
\end{proof}

This gives an easy criterion for checking when endomorphisms of $\mc D^G$ are equivalences, which we will use later.

\begin{corollary}\label{cor:checkonfibers}
    Suppose one has a commutative diagram
    \begin{equation}\label{eq:triangleunders}
        \begin{tikzcd}
        	& {\mc D} & \\
        	{\mc D^G} && {\mc D^G} \\
        	& {\mc D}
        	\arrow["s"', from=1-2, to=2-1]
        	\arrow["s", from=1-2, to=2-3]
        	\arrow["F"', from=2-1, to=2-3]
        	\arrow[from=2-1, to=3-2]
        	\arrow[from=2-3, to=3-2]
        \end{tikzcd}
    \end{equation}
    Then $F$ is an equivalence if and only if it induces equivalences on fibers.
\end{corollary}
\begin{proof}
    Only if is clear, so we show that this is sufficient.\par
    By the essential surjectivity of $s$, we have that $F$ is essentially surjective. Thus we only need to check fully faithfulness.\par
    By \eqref{eq:triangleunders}, we have a commutative diagram
    \begin{equation*}
        \begin{tikzcd}
        	& {\map_{\mc D}(x,y)} & \\
        	{\map_{\mc D^G}(s(x), s(y))} && {\map_{\mc D^G}(s(x),s(y))}.
        	\arrow["{s_\ast}"', from=1-2, to=2-1]
        	\arrow["{s_\ast}", from=1-2, to=2-3]
        	\arrow["{F_\ast}"', from=2-1, to=2-3]
        \end{tikzcd}
    \end{equation*}
    Thus by \Cref{prop:descrofadjendos}, $F_\ast$ is an equivalence if the composite
    \begin{equation*}
        \begin{tikzcd}
            G(x)\simeq \map_{(\mc D^G)_x}(s(x), s(x))\arrow{r}{(F_x)_\ast} & \map_{(\mc D^G)_x}(s(x), s(x))\simeq G(x)
        \end{tikzcd}
    \end{equation*}
    is an equivalence, i.e.\ if $F$ induces an equivalence on fibers.
\end{proof}

\begin{remark}
    Alternatively, the fact that $F\circ s\simeq s$ shows that for every $f : x\to y\in \mc D$, $F$ preserves the preferred Cartesian lift $s(f)$ of $f$ terminating at the essentially unique lift $s(y)$ of $y$. By uniqueness of Cartesian lifts, $F$ sends Cartesian morphisms to Cartesian morphisms and thus by unstraightening $F$ is an equivalence if and only if it induces an equivalence on fibers.
\end{remark}

There is also a global version of \Cref{prop:descrofadjendos} that we will make use of.

\begin{proposition}\label{prop:globaldescrofadjendos}
    Let $\mc G$ be the Cartesian unstraightening of
    \begin{equation*}
        \begin{tikzcd}
            (\mc D^\simeq)^\mathrm{op}\subseteq \mc D^\mathrm{op}\arrow{r}{G} & \Ani.
        \end{tikzcd}
    \end{equation*}
    Then there is a map $\alpha : \mc G\to \map([1],\mc D^G)\times_{\map(\partial[1], \mc D^G)}\mc D^\simeq$ over $\mc D^\simeq$ which on the fiber over $x\in\mc D^\simeq$ corresponds to
    \begin{equation*}
        \mc G_x \simeq G(x)\simeq \map_{(\mc D^G)_x}(s(x), s(x))\longrightarrow \map_{\mc D^G}(s(x), s(x)).
    \end{equation*}
    Moreover, the map
    \begin{equation*}
        \begin{tikzcd}
            \mc G\times_{\mc D^\simeq} N(\mc D)_1\arrow{r} & N(\mc D^G)_1\times_{\map(\partial[1],\mc D^G)}\map(\partial[1], \mc D)
        \end{tikzcd}
    \end{equation*}
    given by sending $(\xi\in G(x), f : x\to y)$ to $s(f)\circ \alpha(\xi)$ is an equivalence.
\end{proposition}
\begin{proof}
    Since delooping $B : \Alg_{\mb E_1}(\Ani)\longrightarrow\Cat_{\ast/}$ is fully faithful, the unit of the $B\dashv \Omega$ adjunction gives an equivalence
    \begin{equation*}
        G(-)\simeq \map([1], BG(-))\times_{\map(\partial [1], BG(-))} \ast
    \end{equation*}
    of functors $\mc D^\mathrm{op}\to \Ani$. Since unstraightening is a monoidal equivalence, it preserves internal homs and limits. Thus we see that
    \begin{equation*}
        \mathrm{Un}(G)\simeq \Fun_{/\mc D}(\mc D\times [1], \mc D^G)\times_{\Fun_{/\mc C}(\mc D\times\partial [1], \mc D^G)} \mc D.
    \end{equation*}
    Now, we have a pullback
    \begin{equation*}
        \begin{tikzcd}
            {\Fun_{/\mc D}(\mc D\times K,\mc D^G)} & {\Fun(K,\mc D^G)} \\
            {\mc D} & {\Fun(K,\mc D)}
            \arrow[from=1-1, to=1-2]
            \arrow[from=1-1, to=2-1]
            \arrow["\lrcorner"{anchor=center, pos=0.125, rotate=0}, draw=none, from=1-1, to=2-2]
            \arrow["{\pi_\ast}", from=1-2, to=2-2]
            \arrow["{\mathrm{const}}"', from=2-1, to=2-2]
        \end{tikzcd}
    \end{equation*}
    for all categories $K$. Thus we see that
    \begin{equation*}
        \mathrm{Un}(G)\simeq \mc D\times_{\Fun([1],\mc D)\times_{\Fun(\partial[1],\mc D)}\mc D} (\Fun([1],\mc D^G)\times_{\Fun(\partial [1], \mc D^G)}\mc D).
    \end{equation*}
    In particular, we have a projection map
    \begin{equation}\label{eq:projfromunG}
        \mathrm{Un}(G)\longrightarrow \Fun([1],\mc D^G)\times_{\Fun(\partial [1], \mc D^G)}\mc D.
    \end{equation}
    Now,
    \begin{equation*}
        \mc G = \mathrm{Un}(G|_{\mc D^{\simeq,\mathrm{op}}}) \simeq \mc D^\simeq\times_{\mc D}\mathrm{Un}(G)
    \end{equation*}
    and we also know that $\mc G$ is a groupoid as it is a right fibration over a groupoid. Thus pulling back \eqref{eq:projfromunG} along $\mc D^\simeq\to\mc D$ and taking $(-)^\simeq$ which preserves limits we get a map
    \begin{equation*}
        \mc G\longrightarrow \map([1], \mc D^G)\times_{\map(\partial [1], \mc D^G)} \mc D^\simeq
    \end{equation*}
    which by construction has the claimed properties.\par
    To verify the final claimed map is an equivalence, it suffices to check it on fibers over $(x,y)\in \map(\partial [1],\mc D)$. But on fibers, this map is an equivalence by \Cref{prop:descrofadjendos}.
\end{proof}

We conclude this section with one final result that will be used later.

\begin{proposition}\label{prop:locofCG}
    Let $(\mc C, G)\in \mathrm{LabCat}$ and let $G^\mathrm{gp} : \mc C^\mathrm{op}\to \Alg_{\mb E_1}(\Ani)$ be the functor given by pointwise groupifying $G$. Then the natural map $\mc C^G\to \mc C^{G^\mathrm{gp}}$ is a localization of $\mc C^G$ at morphisms of the form $(\id_x, \xi)$, $x\in \mc C$, $\xi\in G(x)$.
\end{proposition}
\begin{proof}
    We have that $BG^{\mathrm{gp}} \simeq |BG|$ and the natural map $BG\to BG^\mathrm{gp}$ corresponds to localization $BG\to |BG|$. It follows by \cite[\href{https://kerodon.net/tag/038U}{Tag 038U}]{kerodon} that the induced map on total spaces of the Cartesian unstraightening is a localization at the claimed arrows, which are precisely the morphisms (up to equivalence) lying within a single fiber.
\end{proof}

\subsection{Tangent bundles and the universal exchange theorem}

In order to construct the exchange theorem on $\TwC$ rigorously, it is easier to prove an equivalence between $\TwC$ and $\Gr\Pull(\mc C, E)$ and then pushforward the tautological exchange theorem on $\Gr\Pull(\mc C, E)$. The universal exchange theorem is easy to describe informally, however, so we outline it in this section.

\begin{definition}
    For $f : x\to y \in E$, we define the \emph{relative tangent bundle} of $f$ to be
    \begin{equation*}
        \begin{tikzcd}
            T_f \coloneqq x\arrow{r}{\Delta_f} & x\times_y x\arrow[two heads]{r}{\pr_1} & x\in \Vect(x).
        \end{tikzcd}
    \end{equation*}
\end{definition}

\begin{remark}
    For matters of convention, we will always take the projection in $T_f$ to be projection onto the left factor. However, the two choices are canonically equivalent in $\Vect(x)$.
\end{remark}

We remark that this belongs to $\Vect(x)$ since $E$ is closed under base change, so $\pr_1 : x\times_y x\to x\in E$. The relative tangent bundle satisfies the following properties.

\begin{proposition}\label{prop:propsofreltan}
    The following hold:
    \begin{enumerate}
        \item If $f : x\to y$, $g : y\to z\in E$, then there exists a canonical fiber sequence
        \begin{equation*}
            \begin{tikzcd}
                T_f\arrow{r} & T_{gf}\arrow[two heads]{r} & f^\ast T_g
            \end{tikzcd}
        \end{equation*}
        in $\Vect(x)$. In particular, $[T_{gf}] = f^\ast[T_g] + [T_f]$ in $\pi_0\mc K(\Vect(x))$.
        \item Given a pullback square
        \begin{equation*}
            \begin{tikzcd}
            	x & {y'} \\
            	y & z
            	\arrow["{\bar g}", from=1-1, to=1-2]
            	\arrow["{\bar f}"', from=1-1, to=2-1]
            	\arrow["\lrcorner"{anchor=center, pos=0.125}, draw=none, from=1-1, to=2-2]
            	\arrow["f", from=1-2, to=2-2]
            	\arrow["g"', from=2-1, to=2-2]
            \end{tikzcd}
        \end{equation*}
        in $\mc C$ with $f,\bar f\in E$, there is a canonical equivalence
        \begin{equation*}
            T_{\bar f}\simeq \bar g^\ast T_f
        \end{equation*}
        in $\Vect(x)$.
    \end{enumerate}
\end{proposition}
\begin{proof}
    (i) follows from contemplating the diagram
    \begin{equation*}
        \begin{tikzcd}
        	{x\times_y x} & {x\times_z x} & x \\
        	x & {x\times_z y} & y \\
        	& x & z
        	\arrow[from=1-1, to=1-2]
        	\arrow[two heads, from=1-1, to=2-1]
        	\arrow["\lrcorner"{anchor=center, pos=0.125}, draw=none, from=1-1, to=2-2]
        	\arrow[two heads, from=1-2, to=1-3]
        	\arrow[two heads, from=1-2, to=2-2]
        	\arrow["\lrcorner"{anchor=center, pos=0.125}, draw=none, from=1-2, to=2-3]
        	\arrow["f", two heads, from=1-3, to=2-3]
        	\arrow[from=2-1, to=2-2]
        	\arrow[two heads, from=2-2, to=2-3]
        	\arrow[two heads, from=2-2, to=3-2]
        	\arrow["\lrcorner"{anchor=center, pos=0.125}, draw=none, from=2-2, to=3-3]
        	\arrow["g", two heads, from=2-3, to=3-3]
        	\arrow["gf"', two heads, from=3-2, to=3-3]
        \end{tikzcd}
    \end{equation*}
    for which the top left square gives the claimed fiber sequence in $\Vect(x)$.\par
    For (ii), on one hand we have a diagram
    \begin{equation*}
        \begin{tikzcd}
        	{\bar g^\ast T_f} & {T_f} & {y'} \\
        	x & {y'} & z.
        	\arrow[from=1-1, to=1-2]
        	\arrow[from=1-1, to=2-1]
        	\arrow["\lrcorner"{anchor=center, pos=0.125}, draw=none, from=1-1, to=2-2]
        	\arrow[from=1-2, to=1-3]
        	\arrow[from=1-2, to=2-2]
        	\arrow["\lrcorner"{anchor=center, pos=0.125}, draw=none, from=1-2, to=2-3]
        	\arrow["f", from=1-3, to=2-3]
        	\arrow["{\bar g}"', from=2-1, to=2-2]
        	\arrow["f"', from=2-2, to=2-3]
        \end{tikzcd}
    \end{equation*}
    However, the outer cospan agrees with the outer cospan of
    \begin{equation*}
        \begin{tikzcd}
        	{T_{\bar f}} & x & {y'} \\
        	x & y & z,
        	\arrow[from=1-1, to=1-2]
        	\arrow[from=1-1, to=2-1]
        	\arrow["\lrcorner"{anchor=center, pos=0.125}, draw=none, from=1-1, to=2-2]
        	\arrow["{\bar g}", from=1-2, to=1-3]
        	\arrow["{\bar f}", from=1-2, to=2-2]
        	\arrow["\lrcorner"{anchor=center, pos=0.125}, draw=none, from=1-2, to=2-3]
        	\arrow["f", from=1-3, to=2-3]
        	\arrow["{\bar f}"', from=2-1, to=2-2]
        	\arrow["g"', from=2-2, to=2-3]
        \end{tikzcd}
    \end{equation*}
    hence giving an equivalence $T_{\bar f}\simeq \bar g^\ast T_f$.
\end{proof}

As a consequence of \Cref{prop:propsofreltan}(i) and \Cref{prop:descrofadjendos}, we may informally define two sections of the projection $\pi : \TwC\to\mc C$. Namely,
\begin{equation}\label{eq:horsectionofTw2}
    \begin{tikzcd}[row sep=0pt]
        (-)_\ast : \mc C\arrow{r} & \TwC \\
        \qquad\;\;\; x\arrow[mapsto]{r} & x \\
        (f : x\to y)\arrow[mapsto]{r} & (f,\; 0)
    \end{tikzcd}
\end{equation}
and
\begin{equation}\label{eq:versectionofTw2}
    \begin{tikzcd}[row sep=0pt]
        (-)_\natural : E\arrow{r} & \TwC \\
        \qquad\;\;\; x\arrow[mapsto]{r} & x \\
        (f : x\to y)\arrow[mapsto]{r} & (f,\; [T_f]).
    \end{tikzcd}
\end{equation}
Moreover, by \Cref{prop:propsofreltan}(ii), these two sections satisfy an exchange theorem valued in $\TwC$ which sends a square
\begin{equation*}
    \begin{tikzcd}
        x\arrow{r}{\bar g}\arrow{d}[left]{\bar f}\arrow[draw=none, "\lrcorner"{anchor=center, pos=0.125}]{dr} & 
        y'\arrow{d}{f} & 
        {}\arrow[d, draw=none, ""{name=phantomLeft}] & 
        {}\arrow[d, draw=none, ""{name=phantomRight}] && 
        x\arrow{r}{(\bar g,\; 0)}\arrow[d, "{(\bar f,\; {[T_{\bar f}]})}"'] & 
        y'\arrow{d}[right]{(f,\; [T_f])} \\
        y\arrow{r}[below]{g} & z & {} & {} && y\arrow{r}[below]{(g,\; 0)} & z.
        \arrow[squiggly, from=phantomLeft, to=phantomRight]
    \end{tikzcd}
\end{equation*}
making use of the canonical identification $T_{\bar f}\simeq \bar g^\ast T_f$.\par
The goal of the next section will be to rigorously construct this exchange theorem and prove its universality.

\section{Proof of universality}\label{sec:proofofuniv}

Since
\begin{equation*}
    \begin{aligned}
        \Exch_{(\mc C, E)}(-) &= \map_{\PSh(\Delta)}(\Pull(\mc C, E), \Sq(-)) \\
        &\simeq \map_{\Cat}(\Gr\Pull(\mc C, E), -),
    \end{aligned}
\end{equation*}
$\Gr \Pull(\mc C, E)$ is the tautological recipient of the universal exchange theorem. In this section we explicitly provide an equivalence $\Gr\Pull(\mc C, E)\simeq \TwC$.

\subsection{Reformulation and outline of strategy}

First, we give an alternative description of $\Gr\Pull(\mc C, E)$ using corner approximation.

\begin{definition}\label{def:SCEx}
    Let $S_\bullet(\mc C, E)$ denote the simplicial space which at level $n$ is given by the subspace of connected components of $\map_{\Cat^+}(\Cpt^n, \mc C\;{=\joinrel=}\;\mc C\hookleftarrow E)$ consisting of diagrams where every square consisting of vertical and horizontal arrows is a pullback, i.e.\ the $(n,m)$-simplex space is the groupoid of diagrams $\Cpt^n\to\mc C$ of the form
    \begin{equation*}
        \begin{tikzcd}
        	{x_{00}} & {x_{01}} & {x_{02}} & \cdots & {x_{0m}} \\
        	& {x_{11}} & {x_{12}} & \cdots & {x_{1m}} \\
        	&& {x_{22}} & \ddots & {x_{2m}} \\
        	&&&& \vdots \\
        	&&&& {x_{nm}}
        	\arrow[from=1-1, to=1-2]
        	\arrow[from=1-2, to=1-3]
        	\arrow[from=1-2, to=2-2]
        	\arrow["\lrcorner"{anchor=center, pos=0.125}, draw=none, from=1-2, to=2-3]
        	\arrow[from=1-3, to=1-4]
        	\arrow[from=1-3, to=2-3]
        	\arrow["\lrcorner"{anchor=center, pos=0.125}, draw=none, from=1-3, to=2-4]
        	\arrow[from=1-4, to=1-5]
        	\arrow[from=1-5, to=2-5]
        	\arrow[from=2-2, to=2-3]
        	\arrow[from=2-3, to=2-4]
        	\arrow[from=2-3, to=3-3]
        	\arrow["\lrcorner"{anchor=center, pos=0.125}, draw=none, from=2-3, to=3-4]
        	\arrow[from=2-4, to=2-5]
        	\arrow[from=2-5, to=3-5]
        	\arrow[from=3-5, to=4-5]
        	\arrow[from=4-5, to=5-5]
        \end{tikzcd}
    \end{equation*}
    where all squares are Cartesian and the vertical arrows belong to $E$.
\end{definition}

We remark that there is a monomorphism $S_\bullet(\mc C, E)\to \sd \mc C$ by only remembering the underlying diagram $\Cpt^n\to\mc C$. The following lemma identifies this simplicial space as the corner approximation of $\Pull(\mc C, E)$.

\begin{lemma}\label{lemma:sbulletisgrpull}
    Under the equivalence $\fsd \mc C\simeq \sd \mc C$ of \Cref{prop:sdequivfsdforcats}, the map
    \begin{equation*}
        \cnr\Pull(\mc C, E)\to \fsd \mc C
    \end{equation*}
    induced by the inclusion $\Pull(\mc C, E)\to \Sq(\mc C)$ is an equivalence onto the sub-simplicial space $S_\bullet(\mc C, E)$. In particular, there is a natural equivalence $\ascat S_\bullet(\mc C, E)\xrightarrow{\simeq}\Gr\Pull(\mc C, E)$.\par
    Moreover, under this equivalence,
    \begin{enumerate}
        \item the map $\Gr \Pull(\mc C, E)\to \mc C$ given by transposing $\Pull(\mc C, E)\to \Sq(\mc C)$ corresponds to restriction along the diagonal $\diag^\ast : S_\bullet(\mc C, E)\to N(\mc C)$
        \item the map
        \begin{equation*}
            \mc C \simeq \Gr \Pull(\mc C, E)_{0,\bullet}\to \Gr \Pull(\mc C, E)
        \end{equation*}
        corresponds to (associated categories of) the map
        \begin{equation*}
            N(\mc C)\longrightarrow S_\bullet(\mc C, E)
        \end{equation*}
        given by diagrams which are degenerate in the vertical component
        \item the map
        \begin{equation*}
            E \simeq \Gr \Pull(\mc C, E)_{\bullet,0}\to \Gr \Pull(\mc C, E)
        \end{equation*}
        corresponds to (associated categories of) the map
        \begin{equation*}
            N(E)\longrightarrow S_\bullet(\mc C, E)
        \end{equation*}
        given by diagrams which are degenerate in the horizontal component.
    \end{enumerate}
\end{lemma}
\begin{proof}
    This follows the same proof as \Cref{prop:cnrapproxforsq}.
\end{proof}

Under this description, the tautological universal exchange theorem between
\begin{equation*}
    (-)_\ast : \mc C \simeq \Gr \Pull(\mc C, E)_{0,\bullet}\to \Gr \Pull(\mc C, E)
\end{equation*}
and
\begin{equation*}
    (-)_\natural : E \simeq \Gr \Pull(\mc C, E)_{\bullet,0}\to \Gr \Pull(\mc C, E)
\end{equation*}
corresponds to an exchange theorem between morphisms of the form
\begin{equation*}
    \begin{tikzcd}
    	\bullet & \bullet \\
    	& \bullet
    	\arrow[equals, from=1-1, to=1-2]
    	\arrow[from=1-2, to=2-2]
    \end{tikzcd}
\end{equation*}
and
\begin{equation*}
    \begin{tikzcd}
    	\bullet & \bullet \\
    	& \bullet
    	\arrow[from=1-1, to=1-2]
    	\arrow[equals, from=1-2, to=2-2]
    \end{tikzcd}
\end{equation*}
in $S_\bullet(\mc C, E)$, which we refer to as vertical and horizontal morphisms, respectively. By ambidexterity, vertical morphisms should be related to horizontal morphisms up to a twist.\par
On the level of homotopy categories, this follows from \emph{Verdier's diagonal trick}. Given a vertical morphism
\begin{equation*}
    \begin{tikzcd}
    	x & x \\
    	& y
    	\arrow[equals, from=1-1, to=1-2]
    	\arrow["f", from=1-2, to=2-2]
    \end{tikzcd}
\end{equation*}
in $S_1(\mc C, E)$ representing $f_\natural \in \ascat S_\bullet(\mc C, E)\simeq \Gr\Pull(\mc C, E)$, we may express it as the composite
\begin{equation*}
    \begin{tikzcd}
    	x & {x\times_y x} & x \\
    	& x & y \\
    	&& y.
    	\arrow["{\Delta_f}", from=1-1, to=1-2]
    	\arrow["{\pr_2}", from=1-2, to=1-3]
    	\arrow["{\pr_1}"', from=1-2, to=2-2]
    	\arrow["\lrcorner"{anchor=center, pos=0.125}, draw=none, from=1-2, to=2-3]
    	\arrow["f", from=1-3, to=2-3]
    	\arrow["f", from=2-2, to=2-3]
    	\arrow[equals, from=2-3, to=3-3]
    \end{tikzcd}
\end{equation*}
This $2$-simplex witnesses the relationship
\begin{equation*}
    f_\natural \simeq f_\ast (\Delta_f)_\natural (\pr_1)_\ast
\end{equation*}
in $\Gr\Pull(\mc C, E)\simeq \ascat S_\bullet(\mc C, E)$. In general, we may distinguish three subclasses of morphisms in $S_\bullet(\mc C, E)$.

\begin{definition}
    \begin{enumerate}
        \item The sub-simplicial space of \emph{twists} is the sub-simplicial space of $S_\bullet(\mc C, E)$ given at level $n$ by the union of components of diagrams $F : \Cpt^n\to\mc C\in S_n(\mc C, E)$ which send all morphisms $(i,i)\to (j,j)$ on the diagonal to equivalences.
        \item The sub-simplicial space of \emph{horizontal morphisms} is the sub-simplicial space of $S_\bullet(\mc C, E)$ given at level $n$ by the union of components of diagrams $F : \Cpt^n\to\mc C\in S_n(\mc C, E)$ which send all vertical morphisms $(i,j)\to (i',j)$ to equivalences.
        \item The sub-simplicial space of \emph{vertical morphisms} is the sub-simplicial space of $S_\bullet(\mc C, E)$ given at level $n$ by the union of components of diagrams $F : \Cpt^n\to\mc C\in S_n(\mc C, E)$ which send all horizontal morphisms $(i,j)\to (i,j')$ to equivalences.
    \end{enumerate}
\end{definition}

A general morphism $x\xrightarrow{f}y\xrightarrow{g} z\in S_1(\mc C, E)$ has a canonical factorization
\begin{equation*}
    \begin{tikzcd}
        x & y & y \\
        & y & y \\
        && z
        \arrow["f", from=1-1, to=1-2]
        \arrow[equals, from=1-2, to=1-3]
        \arrow[equals, from=1-2, to=2-2]
        \arrow[equals, from=1-3, to=2-3]
        \arrow[equals, from=2-2, to=2-3]
        \arrow["g", from=2-3, to=3-3]
    \end{tikzcd}
\end{equation*}
into a horizontal followed by a vertical morphism. Verdier's diagonal trick on the other hand shows that any morphism may be factored as a twist followed by a horizontal morphism. The identification $\ascat S_\bullet(\mc C, E)\simeq \TwC$ will be proven by generalizing Verdier's diagonal trick to higher dimensional simplices, constructing a twist--horizontal factorization system in the spirit of \Cref{subsec:weakfactsystems}.

\subsection{A gluing description}\label{subsec:gluingdescr}

In this section, we reformulate $\TwC$ as the gluing $\Gr$ of a bisimplicial space. Recall from \Cref{def:TwC} that $\TwC$ is the Cartesian unstraightening of the functor
\begin{equation*}
    \begin{tikzcd}[row sep=0pt]
        \mc C^\mathrm{op} \arrow{r} & \Cat \\
        x\arrow[mapsto]{r} & B\mc K(\Vect(x)).
    \end{tikzcd}
\end{equation*}

\begin{definition}
    \begin{enumerate}[wide, labelwidth=!, labelindent=0pt]
        \item For a category $\mc D$, let $\Fun^\mathrm{cart}(\mc D, \mc C)$ denote the non-full subcategory of $\Fun(\mc D, \mc C)$ consisting of natural transformations $\eta : F_1\Rightarrow F_2$ such that every square
        \begin{equation*}
            \begin{tikzcd}
            	{F_1(x)} & {F_2(x)} \\
            	{F_1(y)} & {F_2(y)}
            	\arrow["{\eta_x}", from=1-1, to=1-2]
            	\arrow[from=1-1, to=2-1]
            	\arrow[from=1-2, to=2-2]
            	\arrow["{\eta_y}"', from=2-1, to=2-2]
            \end{tikzcd}
        \end{equation*}
        is Cartesian.
        \item Let $\mc T_m(\mc C, E)\subseteq \Fun^\mathrm{cart}(\Cpt^m, \mc C)$ denote the full subcategory on objects $F : \Cpt^m\to \mc C$ such that
        \begin{enumerate}[leftmargin=40pt]
            \item $F\in S_m(\mc C, E)$
            \item $F$ sends all diagonal morphisms $(i,i)\to (j,j)$ to equivalences
        \end{enumerate}
        These categories assemble to form a simplicial category $\mc T_\bullet(\mc C, E)$. We will also write $T_m(\mc C, E)$ for $\mc T_m(\mc C, E)^\simeq$.
        \item Given $x\in\mc C$, write $T_\bullet(\mc C, E, x)$ for the fiber product
        \begin{equation*}
            \begin{tikzcd}
            	{T_\bullet(\mc C, E, x)} & {T_\bullet(\mc C, E)} \\
            	{\{x\}} & {N(\mc C)}.
            	\arrow[from=1-1, to=1-2]
            	\arrow[from=1-1, to=2-1]
            	\arrow["\lrcorner"{anchor=center, pos=0.125}, draw=none, from=1-1, to=2-2]
            	\arrow[from=1-2, to=2-2, "\diag^\ast"]
            	\arrow[from=2-1, to=2-2]
            \end{tikzcd}
        \end{equation*}
    \end{enumerate}
\end{definition}

We remark that $T_\bullet(\mc C, E, x)$ is the simplicial space whose $m$-simplex space is the groupoid of diagrams $\Cpt^m\to \mc C$ which are (equivalent to) one of the form
\begin{equation}\label{eq:diagrinTmCEx}
    \begin{tikzcd}
    	x & \bullet & \bullet & \cdots & \bullet \\
    	& x & \bullet & \cdots & \bullet \\
    	&&& \ddots & \vdots \\
    	&&&& \bullet \\
    	&&&& x
    	\arrow[from=1-1, to=1-2]
    	\arrow[from=1-2, to=1-3]
    	\arrow[two heads, from=1-2, to=2-2]
    	\arrow["\lrcorner"{anchor=center, pos=0.125}, draw=none, from=1-2, to=2-3]
    	\arrow[from=1-3, to=1-4]
    	\arrow[two heads, from=1-3, to=2-3]
    	\arrow["\lrcorner"{anchor=center, pos=0.125}, draw=none, from=1-3, to=2-4]
    	\arrow[from=1-4, to=1-5]
    	\arrow[two heads, from=1-5, to=2-5]
    	\arrow[from=2-2, to=2-3]
    	\arrow[from=2-3, to=2-4]
    	\arrow[from=2-4, to=2-5]
    	\arrow[two heads, from=2-5, to=3-5]
    	\arrow[two heads, from=3-5, to=4-5]
    	\arrow[two heads, from=4-5, to=5-5]
    \end{tikzcd}
\end{equation}
where
\begin{itemize}
    \item the diagonal morphisms are all equivalences
    \item each square consisting of vertical and horizontal morphisms is a pullback square in $\mc C$
    \item all vertical morphisms belong to $E$.
\end{itemize}
Since the forgetful map $\Vect(x)\to \mc C$ preserves pullbacks along morphisms in $\Vect(x)_\dagger$ (see the proof of \Cref{prop:liftofVecttocoWald}), pushforward along $\Vect(x)\to \mc C$ induces a map $S_\bullet(\Vect(x))\to T_\bullet(\mc C, E, x)$.

\begin{proposition}\label{prop:twistsareKthy}
    The forgetful map $S_\bullet(\Vect(x))\longrightarrow T_\bullet(\mc C, E, x)$ is an equivalence.
\end{proposition}

This should be believable as each diagram of the form \eqref{eq:diagrinTmCEx} uniquely lifts to a diagram in $\Vect(x)$ where each object gets its incoming structure map from the initial object which is (equivalent to) $x$, and its outgoing structure map from the terminal object which is (equivalent to) $x$. Since the forgetful map $\Vect(x)\to\mc C$ creates pullbacks along $\Vect(x)_\dagger$, all diagrams in $S_m(\Vect(x))$ should come from this procedure.\par
To do this rigorously, we need an alternative description of the functors $D : \Cpt^n\to\mc C$ which belong to $S_n(\mc C, E)$ (and hence to each $T_m(\mc C, E, x)$).

\begin{definition}
    Denote by $\mathrm{Sp}^n$ the subposet of $\Cpt^n = \{(i,j) : 0\le i\le j\le n\}$ spanned by the objects $\{(i,i) : 0\le i\le n\}\cup \{(i, n) : 0\le i \le n\}$.
\end{definition}


\begin{example}
    The (rotated) Hasse diagram of $\Cpt^3$, so that the initial object is in the top-left corner, is given by
    \begin{equation*}
        \begin{tikzpicture}
            \node at (0,0) {$\bullet$};
            \draw (0,0) -- (1,0);
            \node at (1,0) {$\bullet$};
            \draw (1,0) -- (2,0);
            \node at (2,0) {$\bullet$};
            \node at (3,0) {$\bullet$};
            \node at (1,-1) {$\bullet$};
            \node at (2, -1) {$\bullet$};
            \node at (3, -1) {$\bullet$};
            \node at (2, -2) {$\bullet$};
            \node at (3, -2) {$\bullet$};
            \node at (3,-3) {$\bullet$};
            \draw (2,0) -- (3,0);
            \draw (1,-1) -- (2,-1);
            \draw (2,-1) -- (3, -1);
            \draw (2, -2) -- (3, -2);
            \draw (1,0) -- (1, -1);
            \draw (2,0) -- (2, -1);
            \draw (3,0) -- (3, -1);
            \draw (2,-1) -- (2,-2);
            \draw (3,-1) -- (3,-2);
            \draw (3, -2) -- (3, -3);
        \end{tikzpicture}
    \end{equation*}
    and the rotated Hasse diagram of $\mathrm{Sp}^3$ is given by
    \begin{equation*}
        \begin{tikzpicture}
            \node at (0,0) {$\bullet$};
            \node at (3,0) {$\bullet$};
            \node at (1,-1) {$\bullet$};
            \node at (3, -1) {$\bullet$};
            \node at (2, -2) {$\bullet$};
            \node at (3, -2) {$\bullet$};
            \node at (3,-3) {$\bullet$};

            \draw (0,0) -- (1, -1);
            \draw (1, -1) -- (2, -2);
            \draw (2,-2) -- (3,-3);
            \draw (3,0) -- (3, -1);
            \draw (3,-1) -- (3,-2);
            \draw (3, -2) -- (3, -3);
            \draw (0,0) -- (3, 0);
            \draw (1, -1) -- (3, -1);
            \draw (2, -2) -- (3, -2);
        \end{tikzpicture}
    \end{equation*}
\end{example}

\begin{proposition}\label{prop:goodisdetectonspine}
    A functor $D : \Cpt^n\to \mc C$ belongs to $S_n(\mc C, E)$ if and only if
    \begin{enumerate}
        \item $D$ is right Kan extended from its restriction $\mathrm{Sp}^n\to \Cpt^n\xrightarrow{D}\mc C$
        \item the right-most vertical arrows $\{(i,n) : 0\le i\le n\}\hookrightarrow \mathrm{Sp}^n\to \Cpt^n\xrightarrow{D} \mc C$ lie inside $E\subseteq \mc C$.
    \end{enumerate}
    In particular, if $\map_E(\mathrm{Sp}^n, \mc C)$ denotes the subspace of components of $\map_{\Cat}(\mathrm{Sp}^n, \mc C)$ satisfying (ii), then the restriction
    \begin{equation*}
        S_n(\mc C, E)\to \map_E(\mathrm{Sp}^n, \mc C)
    \end{equation*}
    is an equivalence.
\end{proposition}
\begin{proof}
    By the pointwise formula for right Kan extensions and standard finality arguments, we see that $D$ is pointwise right Kan extended from $\mathrm{Sp}^n\to \Cpt^n\xrightarrow{D} \mc C$ if and only if the natural maps
    \begin{equation}\label{eq:condforrightkan}
        \begin{tikzcd}
            D(i,j)\arrow{r}{\simeq} & D(j,j)\times_{D(j,n)} D(i,n)
        \end{tikzcd}
    \end{equation}
    are equivalences. This shows that if $D$ belongs to $S_n(\mc C, E)$, then (i) and (ii) are satisfied.\par
    For the other direction, suppose that $D$ satisfies (i) and (ii). Since $\mc C$ has pullbacks along morphisms in $E$, given a functor $\mathrm{Sp}^n\to \mc C$ satisfying (ii), \eqref{eq:condforrightkan} tells us that the right Kan extension along $\mathrm{Sp}^n\hookrightarrow \Cpt^n$ always exists and is computed pointwise. Therefore we may assume in (i) that the extension is pointwise. Then, by \eqref{eq:condforrightkan} and the pasting laws for pullback diagrams, we learn that $D : \Cpt^n\to \mc C$ sends every square consisting of horizontal and vertical arrows to a pullback. Using this, since $E$ is closed under pullback, (ii) then also tells us that every vertical arrow gets sent to a morphism in $E$ by $D$. These two conditions then precisely say that $D$ belongs to $S_n(\mc C, E)$.\par
    The final claim then follows by general results on the uniqueness of Kan extensions as well as the aforementioned fact that every $\mathrm{Sp}^n\to\mc C$ satisfying (ii) has a right Kan extension along $\mathrm{Sp}^n\hookrightarrow \Cpt^n$.
\end{proof}

\begin{corollary}\label{cor:descrofTnCE}
    Let $L_n = [n + 1]((0\to n + 1)^{-1})$. Then restriction to
    \begin{equation*}
        [n + 1]\simeq \{(0,0)\}\cup \{(i,n) : 0\le i\le n\}\subseteq \Cpt^n
    \end{equation*}
    induces an equivalence
    \begin{equation*}
        \begin{tikzcd}
            T_n(\mc C, E)\arrow{r}{\simeq} & \map_E(L_n, \mc C).
        \end{tikzcd}
    \end{equation*}
    where $\map_E(L_n, \mc C)\subseteq \map(L_n,\mc C)$ is the subspace of components which send $\{1\le 2\le \cdots \le n + 1\}\subseteq [n + 1]\subseteq L_n$ to morphisms in $E$.
\end{corollary}
\begin{proof}
    The restriction is well defined as, given a diagram $D : \Cpt^n\to\mc C\in T_n(\mc C, E)$, $T(0,0)\to T(n,n)$ is an equivalence.\par
    By \Cref{prop:goodisdetectonspine}, we have that restriction to $\mathrm{Sp}^n$ induces an equivalence
    \begin{equation*}
        \begin{tikzcd}
            T_n(\mc C, E)\arrow{r}{\simeq} & \map_E(\mathrm{Sp}^n[\diag^{-1}], \mc C)
        \end{tikzcd}
    \end{equation*}
    where $\map_E(\mathrm{Sp}^n[\diag^{-1}], \mc C)$ is the subspace of components of $\map(\mathrm{Sp}^n[\diag^{-1}], \mc C)$ which send the rightmost vertical arrows into $E$. The result then follows by remarking that the composite
    \begin{equation*}
        [n + 1]\simeq \{(0,0)\}\cup \{(i,n) : 0\le i\le n\}\hookrightarrow \mathrm{Sp}^n\longrightarrow \mathrm{Sp}^n[\diag^{-1}]
    \end{equation*}
    induces an equivalence $L_n \simeq \mathrm{Sp}^n[\diag^{-1}]$ of categories.
\end{proof}

We may now prove \Cref{prop:twistsareKthy}.

\begin{proof}[Proof of \Cref{prop:twistsareKthy}]
    We need to prove that $S_m(\Vect(x))\to T_m(\mc C, E, x)$ is an equivalence for each $m$.\par
    By definition, we have a pullback
    \begin{equation*}
        \begin{tikzcd}
        	{T_m(\mc C, E, x)} & {T_m(\mc C, E)} \\
        	{\{x\}} & {\map([m],\mc C)}.
        	\arrow[from=1-1, to=1-2]
        	\arrow[from=1-1, to=2-1]
        	\arrow["\lrcorner"{anchor=center, pos=0.125}, draw=none, from=1-1, to=2-2]
        	\arrow[from=1-2, to=2-2, "\diag^\ast"]
        	\arrow[from=2-1, to=2-2]
        \end{tikzcd}
    \end{equation*}
    Since the morphisms on the diagonal for diagrams in $T_m(\mc C, E)$ are equivalences, the rightmost morphism factors as
    \begin{equation*}
        \begin{tikzcd}
            T_m(\mc C, E)\arrow{r} & \map(|[m]|, \mc C)\arrow{r} & \map([m], \mc C).
        \end{tikzcd}
    \end{equation*}
    Thus we have a diagram
    \begin{equation*}
        \begin{tikzcd}
        	{T_m(\mc C, E, x)} & {T_m(\mc C, E)} & {T_m(\mc C, E)} \\
        	{\{x\}} & {\map(|[m]|,\mc C)} & {\mc C^\simeq} \\
        	{\{x\}} & {\map([m],\mc C)}
        	\arrow[from=1-1, to=1-2]
        	\arrow[from=1-1, to=2-1]
        	\arrow["\lrcorner"{anchor=center, pos=0.125}, draw=none, from=1-1, to=2-2]
        	\arrow[equals, from=1-2, to=1-3]
        	\arrow[from=1-2, to=2-2]
        	\arrow["\lrcorner"{anchor=center, pos=0.125}, draw=none, from=1-2, to=2-3]
        	\arrow["{\mathrm{ev}_{(m,m)}}", from=1-3, to=2-3]
        	\arrow[from=2-1, to=2-2]
        	\arrow[equals, from=2-1, to=3-1]
        	\arrow["\lrcorner"{anchor=center, pos=0.125}, draw=none, from=2-1, to=3-2]
        	\arrow["{\mathrm{ev}_m}"', "\simeq", from=2-2, to=2-3]
        	\arrow[from=2-2, to=3-2]
        	\arrow[from=3-1, to=3-2]
        \end{tikzcd}
    \end{equation*}
    of pullbacks. Here, the bottom left square is a pullback because we are taking the fiber of a monomorphism, and the rightmost square is a pullback because both horizontal morphisms are equivalences. We conclude that $T_m(\mc C, E, x)$ also naturally fits into a pullback square
    \begin{equation*}
        \begin{tikzcd}
        	{T_m(\mc C, E, x)} & {T_m(\mc C, E)} \\
        	{\{x\}} & {\mc C^\simeq}.
        	\arrow[from=1-1, to=1-2]
        	\arrow[from=1-1, to=2-1]
        	\arrow["\lrcorner"{anchor=center, pos=0.125}, draw=none, from=1-1, to=2-2]
        	\arrow["{\mathrm{ev}_{(m,m)}}", from=1-2, to=2-2]
        	\arrow[from=2-1, to=2-2]
        \end{tikzcd}
    \end{equation*}
    Applying \Cref{cor:descrofTnCE}, we learn that we have a pullback
    \begin{equation*}
        \begin{tikzcd}
        	{T_m(\mc C, E, x)} & {\map_E(L_m, \mc C)} \\
        	{\{x\}} & {\mc C^\simeq}
        	\arrow[from=1-1, to=1-2]
        	\arrow[from=1-1, to=2-1]
        	\arrow["\lrcorner"{anchor=center, pos=0.125}, draw=none, from=1-1, to=2-2]
        	\arrow["{\mathrm{ev}_{m + 1}}", from=1-2, to=2-2]
        	\arrow[from=2-1, to=2-2]
        \end{tikzcd}
    \end{equation*}
    so it suffices to show that the induced square
    \begin{equation}\label{eq:SmVectpullbacksq}
        \begin{tikzcd}
        	{S_m(\Vect(x))} & {\map_E(L_m, \mc C)} \\
        	{\{x\}} & {\mc C^\simeq}
        	\arrow[from=1-1, to=1-2]
        	\arrow[from=1-1, to=2-1]
        	\arrow["{\mathrm{ev}_{m + 1}}", from=1-2, to=2-2]
        	\arrow[from=2-1, to=2-2]
        \end{tikzcd}
    \end{equation}
    is a pullback. By standard results involving the $S_\bullet$-construction, the map
    \begin{equation*}
        \begin{tikzcd}
            S_m(\Vect(x))\arrow{r}{\simeq} & \map([m - 1], \Vect(x)_\dagger)
        \end{tikzcd}
    \end{equation*}
    given by restriction along $[m - 1]\simeq \{(i, m) : 0\le i \le m - 1\}\subseteq \Cpt^m$ is an equivalence. We proceed by computing $\map([m - 1], \Vect(x)_\dagger)$.\par
    By the definition of slice categories, for any category $\mc D$, there is a Cartesian square
    \begin{equation*}
        \begin{tikzcd}
            N(\mc D_{/x})_n\arrow{r}\arrow{d} & N(\mc D)_{n + 1}\arrow{d}{\mathrm{ev}_{n + 1}} \\
            \{x\}\arrow{r} & \mc D^\simeq.
        \end{tikzcd}
    \end{equation*}
    Similarly for under slice categories. It follows that we have a diagram
    \begin{equation}\label{eq:doubleslicepullback}
        \begin{tikzcd}
        	{\map([\ell],\mc C_{x/,/x})} & {\map([\ell + 1], \mc C_{/x})} & {\map([\ell + 2], \mc C)} \\
        	{\{\id_x\}} & {\mc C_{/x}^\simeq} & {\map([1], \mc C)} \\
        	& {\{x\}} & {\mc C^\simeq}
        	\arrow[from=1-1, to=1-2]
        	\arrow[from=1-1, to=2-1]
        	\arrow["\lrcorner"{anchor=center, pos=0.125}, draw=none, from=1-1, to=2-2]
        	\arrow[from=1-2, to=1-3]
        	\arrow["{\mathrm{ev}_0}", from=1-2, to=2-2]
        	\arrow["\lrcorner"{anchor=center, pos=0.125}, draw=none, from=1-2, to=2-3]
        	\arrow["{(0\to \ell + 2)^\ast}", from=1-3, to=2-3]
        	\arrow[from=2-1, to=2-2]
        	\arrow[from=2-2, to=2-3]
        	\arrow[from=2-2, to=3-2]
        	\arrow["\lrcorner"{anchor=center, pos=0.125}, draw=none, from=2-2, to=3-3]
        	\arrow[from=2-3, to=3-3]
        	\arrow[from=3-2, to=3-3]
        \end{tikzcd}
    \end{equation}
    of pullback squares. Indeed, in \eqref{eq:doubleslicepullback}, the upper right square is Cartesian by the pasting lemma looking at the bottom right square and the composite of the two right-most squares. Because $\{\id_x\}\to \map([1],\mc C)$ factors through $\map(|[1]|,\mc C)\to \map([1],\mc C)$, the composite of the two upper squares in \eqref{eq:doubleslicepullback} may then alternatively be written as the composite
    \begin{equation}\label{eq:compofsquaresSmlemma}
        \begin{tikzcd}
        	{\map([\ell], \mc C_{x/,/x})} & {\map(L_{\ell + 1},\mc C)} & {\map([\ell + 2], \mc C)} \\
        	{\{\id_x\}} & {\map(|[1]|,\mc C)} & {\map([1],\mc C)}
        	\arrow[from=1-1, to=1-2]
        	\arrow[from=1-1, to=2-1]
        	\arrow["\lrcorner"{anchor=center, pos=0.125}, draw=none, from=1-1, to=2-2]
        	\arrow[from=1-2, to=1-3]
        	\arrow["{(0\to \ell + 2)^\ast}"', from=1-2, to=2-2]
        	\arrow["\lrcorner"{anchor=center, pos=0.125}, draw=none, from=1-2, to=2-3]
        	\arrow["{(0\to \ell + 2)^\ast}", from=1-3, to=2-3]
        	\arrow[from=2-1, to=2-2]
        	\arrow[from=2-2, to=2-3]
        \end{tikzcd}
    \end{equation}
    where the right-most square is a pullback since we have a pushout
    \begin{equation*}
        \begin{tikzcd}[column sep=35pt]
        	{[1]} & {[\ell + 2]} \\
        	{|[1]|} & {L_{\ell + 1}}
        	\arrow["{0\to\ell + 2}", from=1-1, to=1-2]
        	\arrow[from=1-1, to=2-1]
        	\arrow[from=1-2, to=2-2]
        	\arrow[from=2-1, to=2-2]
        	\arrow["\ulcorner"{anchor=center, pos=0.125, rotate=180}, draw=none, from=2-2, to=1-1]
        \end{tikzcd}
    \end{equation*}
    in $\Cat$ by definition of $L_{\ell + 1}$. Looking at the left-most square in \eqref{eq:compofsquaresSmlemma} and using that $|[1]|\simeq\ast$ yields a pullback square
    \begin{equation*}
        \begin{tikzcd}
        	{\map([\ell], \mc C_{x/,/x})} & {\map(L_{\ell + 1},\mc C)} \\
        	{\{x\}} & {\mc C^\simeq}.
        	\arrow[from=1-1, to=1-2]
        	\arrow[from=1-1, to=2-1]
        	\arrow["\lrcorner"{anchor=center, pos=0.125}, draw=none, from=1-1, to=2-2]
        	\arrow["{\mathrm{ev}_{\ell + 2}}", from=1-2, to=2-2]
        	\arrow[from=2-1, to=2-2]
        \end{tikzcd}
    \end{equation*}
    It then follows from the definition of $\Vect(x)_\dagger$ that there is a pullback
    \begin{equation*}
        \begin{tikzcd}
        	{\map([m - 1], \Vect(x)_\dagger)} & {\map_E(L_{m},\mc C)} \\
        	{\{x\}} & {\mc C^\simeq}.
        	\arrow[from=1-1, to=1-2]
        	\arrow[from=1-1, to=2-1]
        	\arrow["\lrcorner"{anchor=center, pos=0.125}, draw=none, from=1-1, to=2-2]
        	\arrow["{\mathrm{ev}_{m + 1}}", from=1-2, to=2-2]
        	\arrow[from=2-1, to=2-2]
        \end{tikzcd}
    \end{equation*}
    One then checks that precomposing with $S_m(\Vect(x))\xrightarrow{\simeq} \map([m - 1],\Vect(x)_\dagger)$ recovers the square \eqref{eq:SmVectpullbacksq} and thus \eqref{eq:SmVectpullbacksq} is a pullback as required.
\end{proof}

Having identified $T_\bullet(\mc C, E, x)$ with $S_\bullet(\Vect(x))$, we now proceed to study the categories $\mc T_m(\mc C, E)$. In particular, we will show that for each $m$, $\mc T_m(\mc C, E)\to\mc C$ given by evaluating diagrams at $(m,m)$ is a Cartesian fibration. We first prove a lemma that is of a similar flavor to \Cref{prop:goodisdetectonspine}.

\begin{lemma}\label{lemma:AnmdescrbyKan}
    The natural map $N(\mc T_m(\mc C, E))_n\subseteq \map_{\Cat}(\Cpt^m\times [n],\mc C)$ is an inclusion of connected components. Moreover, a functor $F : \Cpt^m\times [n]\to\mc C$ belongs to $N(\mc T_m(\mc C, E))_n$ if and only if
    \begin{enumerate}
        \item $F|_{\Cpt^m\times \{n\}} : \Cpt^m\to\mc C\in S_m(\mc C, E)$ and sends diagonals to equivalences
        \item $F$ is right Kan extended from its restriction to $\Cpt^m\times \{n\}\cup \{(m,m)\}\times [n] \subseteq \Cpt^m\times [n]$.
    \end{enumerate}
\end{lemma}
\begin{proof}
    The fact that $N(\mc T_m(\mc C, E))_n\subseteq \map_{\Cat}(\Cpt^m\times [n],\mc C)$ is an inclusion of connected components is immediate from the definition of $\mc T_m(\mc C, E)$. From the definition, one finds that $N(\mc T_m(\mc C, E))_n$ consists of precisely those $F : \Cpt^m\times [n]\to\mc C$ such that
    \begin{enumerate}[label=(\alph*)]
        \item for all $0\le i\le n$, $F|_{\Cpt^m\times\{i\}} : \Cpt^m\to\mc C$ belongs to $S_m(\mc C, E)$ and sends diagonal morphisms to equivalences
        \item every square of the form
        \begin{equation*}
            \begin{tikzcd}
                F((i,j),k)\arrow{r}\arrow{d} & F((i,j), k + 1)\arrow{d} \\
                F((p,q), k)\arrow{r} & F((p,q), k + 1)
            \end{tikzcd}
        \end{equation*}
        is Cartesian.
    \end{enumerate}
    Thus, it suffices to show that conditions (i) and (ii) of the lemma are equivalent to (a) and (b).\par
    We make use of the following temporary notation:
    \begin{itemize}
        \item let $P = \Cpt^m\times \{n\}\cup \{(m,m)\}\times [n] \subseteq \Cpt^m\times [n]$
        \item let $\iota$ be the inclusion $\iota : P \hookrightarrow \Cpt^m\times [n]$
    \end{itemize}
    Consider $((i,j), k)\in \Cpt^m\times [n]$. Since $\iota$ is a map of posets, we have that
    \begin{equation*}
        ((i,j), k)\downarrow\iota = \{x\in P : x \le ((i,j), k)\}.
    \end{equation*}
    An easy application of Quillen's theorem shows that the inclusion of
    \begin{equation*}
        \begin{tikzcd}
        	{((i,j), n)} & \\
        	{((m,m), n)} & {((m,m),k)}
        	\arrow[from=1-1, to=2-1]
        	\arrow[from=2-2, to=2-1]
        \end{tikzcd}\subseteq ((i,j), k)\downarrow P
    \end{equation*}
    is final. Under the assumption that $F : P\to\mc C$ satisfies $F|_{\Cpt^m\times\{n\}}\in S_m(\mc C, E)$ it follows that, since $\mc C$ has pullbacks along $E$, the right Kan extension of any such $F : P\to \mc C$ along $\iota$ exists and is computed pointwise. This also shows that under the assumption of (i), condition (ii) is equivalent to
    \begin{enumerate}
        \item[(ii)'] Every square of the form
        \begin{equation*}
            \begin{tikzcd}
                 F((i,j),k) \arrow{d}\arrow{r} & F((i,j),n)\arrow{d} \\
                F((m,m),k)\arrow{r} & F((m,m),n)
            \end{tikzcd}
        \end{equation*}
        is Cartesian.
    \end{enumerate}
    By pasting pullbacks appearing in (ii)', one also readily checks that condition (ii)' is equivalent to the a priori stronger condition (b).\par
    Since (i) is a strictly weaker condition than (a), it suffices to show that (i) and (b) imply (a). However, this follows from pasting laws of pullbacks for cubes and the stability of both $E$ and equivalences under pullback.
\end{proof}

As a consequence of the proof of the above lemma which shows that the relevant right Kan extensions exist, we conclude the following corollary from the uniqueness of right Kan extensions.

\begin{corollary}\label{cor:identofAnmSbullet}
    Restriction along $\Cpt^m\times \{n\}\cup \{(m,m)\}\times [n] \subseteq \Cpt^m\times [n]$ induces an equivalence
    \begin{equation}\label{eq:descrofAnm}
        \begin{tikzcd}
            N(\mc T_m(\mc C, E))_n\arrow{r}{\simeq} & T_m(\mc C, E)\times_{\mc C^\simeq}\map_{\Cat}([n], \mc C)
        \end{tikzcd}
    \end{equation}
    where the maps $T_m(\mc C, E)\to \mc C^\simeq$ and $\map([n],\mc C)\to\mc C^\simeq$ are given by evaluation at $(m,m)$ and $n$ respectively.
\end{corollary}

\begin{corollary}\label{cor:projisconservative}
    The map $\pi : \mc T_m(\mc C, E)\to\mc C$ given by evaluation at $(m,m)$ is conservative.
\end{corollary}
\begin{proof}
    On the level of nerves, the induced map $N(\mc T_m(\mc C, E))_n\to N(\mc C)_n$ is the equivalence of \Cref{cor:identofAnmSbullet} followed by the projection. Thus, take $f : F_1\to F_2\in N(\mc T_m(\mc C, E))_1$ such that $\pi(f)$ is an equivalence. Then we have an equivalence $\pi(f)\simeq \id_{\pi(F_2)}$ in $\mc C_{/\pi(F_2)}$. From this we obtain an equivalence between the image of $f$ and the image of $\id_{F_2}$ under the equivalence of \Cref{cor:identofAnmSbullet}. Thus $f\simeq \id_{F_2}$ in $N(\mc T_m(\mc C, E))_1$ and so $f$ is an equivalence.
\end{proof}

\begin{corollary}\label{cor:slicesofTmissliceofCm}
    Given $F\in\mc T_m(\mc C, E)$, evaluation at $(m,m)$ induces an equivalence $\mc T_m(\mc C, E)_{/F}\to \mc C_{/F(m,m)}$.
\end{corollary}
\begin{proof}
    We check this on nerves. Indeed, we have a commutative cube
    \begin{equation*}
        \begin{tikzcd}
        	& {N(\mc C_{/F(m,m)})_n} && {N(\mc C)_{n + 1}} \\
        	{N(\mc T_m(\mc C, E)_{/F})_n} && {N(\mc T_m(\mc C, E))_{n + 1}} \\
        	& {\{F(m,m)\}} && {\mc C^\simeq} \\
        	{\{F\}} && {T_m(\mc C, E)}
        	\arrow[from=1-2, to=1-4]
        	\arrow[from=1-2, to=3-2]
        	\arrow["\lrcorner"{anchor=center, pos=0.125, rotate=0}, draw=none, from=1-2, to=3-4]
        	\arrow["\mathrm{ev}_{n + 1}", from=1-4, to=3-4]
        	\arrow[from=2-1, to=1-2]
        	\arrow[from=2-1, to=2-3]
        	\arrow[from=2-1, to=4-1]
        	\arrow["\lrcorner"{anchor=center, pos=0.125, rotate=0}, draw=none, from=2-1, to=4-3]
        	\arrow[from=2-3, to=1-4]
        	\arrow[from=2-3, to=4-3]
        	\arrow[from=3-2, to=3-4]
        	\arrow[from=4-1, to=3-2]
        	\arrow[from=4-1, to=4-3]
        	\arrow[from=4-3, to=3-4]
        \end{tikzcd}
    \end{equation*}
    where the front and back faces are pullbacks by definition of slice categories. However, by \Cref{cor:identofAnmSbullet}, we also have that the right face is a pullback. It follows by the pasting lemma that the left face is a pullback. Since the bottom arrow on the left face is an equivalence it follows that $N(\mc T_m(\mc C, E)_{/F})_n\to N(\mc C_{/F(m,m)})_n$ is an equivalence, as required.
\end{proof}

\begin{proposition}\label{prop:TmisCartunstraightofSm}
    The map $\mc T_m(\mc C, E)\to \mc C$ given by evaluation at $(m,m)$ is the Cartesian unstraightening of the functor
    \begin{equation*}
        \begin{tikzcd}[row sep=0pt]
            \mc C^\mathrm{op}\arrow{r} & \Ani \\
            \qquad x\arrow[mapsto]{r} & S_m(\Vect(x)).
        \end{tikzcd}
    \end{equation*}
\end{proposition}
\begin{proof}
    The fact that $\pi : \mc T_m(\mc C, E)\to \mc C$ is a right fibration is immediate from \Cref{cor:slicesofTmissliceofCm}. Indeed, this implies that every morphism in $\mc T_m(\mc C, E)$ is a $\pi$-Cartesian morphism. Moreover, every morphism has a lift by lifting the source and target to constant diagrams. See, e.g., \cite{mazel2015user} for a model independent definition of Cartesian lifts and fibrations.\par
    Since $\mc C_{/x}\to \mc C$ is the unstraightening of $\map_{\mc C}(-,x) : \mc C^\mathrm{op}\to\Ani$, it follows from the Yoneda lemma that $\mc T_m(\mc C, E)\to\mc C$ unstraightens to the functor
    \begin{equation*}
        \begin{tikzcd}[row sep=0pt]
            G : \mc C^\mathrm{op}\arrow{r} & \Ani \\
            \quad x\arrow[mapsto]{r} & \map_{/\mc C}(\mc C_{/x}, \mc T_m(\mc C, E))
        \end{tikzcd}
    \end{equation*}
    Now, we have a natural transformation $S_m(\Vect(-)) \Rightarrow G$ given at level $x$ by
    \begin{equation*}
        \begin{tikzcd}[row sep=0pt]
            S_m(\Vect(x))\arrow{r} & \map_{/\mc C}(\mc C_{/x}, \mc T_m(\mc C, E)) \\
            D\arrow[mapsto]{r} & ((y\xrightarrow{f} x)\mapsto f^\ast D)
        \end{tikzcd}
    \end{equation*}
    which we argue is an equivalence.\par
    Let $F$ be the fiber of $\mc T_m(\mc C, E)\to\mc C$ over $x$. We have a diagram
    \begin{equation*}
        \begin{tikzcd}
        	{N(F)_n} & {\map(|[n]|, \mc T_m(\mc C, E))} & {N(\mc T_m(\mc C, E))_n} \\
        	{\{x\}} & {\map(|[n]|, \mc C)} & {N(\mc C)_n}
        	\arrow[from=1-1, to=1-2]
        	\arrow[from=1-1, to=2-1]
        	\arrow["\lrcorner"{anchor=center, pos=0.125, rotate=0}, draw=none, from=1-1, to=2-2]
        	\arrow[from=1-2, to=1-3]
        	\arrow[from=1-2, to=2-2]
        	\arrow["\lrcorner"{anchor=center, pos=0.125, rotate=0}, draw=none, from=1-2, to=2-3]
        	\arrow[from=1-3, to=2-3]
        	\arrow[from=2-1, to=2-2]
        	\arrow[from=2-2, to=2-3]
        \end{tikzcd}
    \end{equation*}
    of pullbacks. Indeed, the outer square is a pullback by definition of $F$ and $N$ preserving limits and the right inner square is a pullback by \Cref{cor:projisconservative}. However, applying $\map(|[n]|, -)$ to the definition of $T_m(\mc C, E, x)$, we have a diagram of pullback squares
    \begin{equation*}
        \begin{tikzcd}
        	{\map(|[n]|, T_m(\mc C, E, x))} & {\map(|[n]|,\mc T_m(\mc C, E))} & {\map(|[n]|,\mc T_m(\mc C, E)).} \\
        	{\{x\}} & {\map(|[n]|,\map(|[m]|, \mc C))} & {\map(|[n]|,\mc C)}
        	\arrow[from=1-1, to=1-2]
        	\arrow[from=1-1, to=2-1]
        	\arrow["\lrcorner"{anchor=center, pos=0.125, rotate=0}, draw=none, from=1-1, to=2-2]
        	\arrow[equals, from=1-2, to=1-3]
        	\arrow[from=1-2, to=2-2]
        	\arrow["\lrcorner"{anchor=center, pos=0.125, rotate=0}, draw=none, from=1-2, to=2-3]
        	\arrow[from=1-3, to=2-3]
        	\arrow[from=2-1, to=2-2]
        	\arrow["{(\mathrm{ev}_{m})_\ast}", "\simeq"', from=2-2, to=2-3]
        \end{tikzcd}
    \end{equation*}
    Thus we obtain an equivalence $N(F)\simeq N(T_m(\mc C, E,x))$, i.e.\ $F\simeq T_m(\mc C, E,x)$.\par
    Moreover, since $\mc T_m(\mc C, E)\to\mc C$ is already known to be a right fibration, the Yoneda lemma says that evaluation at $\id_x\in\mc C_{/x}$ is an equivalence onto the fiber, i.e.\ we have a composite
    \begin{equation*}
        \begin{tikzcd}
            S_m(\Vect(x))\arrow{r} & \map_{/\mc C}(\mc C_{/x}, \mc T_m(\mc C, E))\arrow{r}[above]{\mathrm{ev}_{\id_x}}[below]{\simeq} & T_m(\mc C, E, x).
        \end{tikzcd}
    \end{equation*}
    But this composite is the forgetful map $S_m(\Vect(x))\to T_m(\mc C, E, x)$, so the result follows from \Cref{prop:twistsareKthy}.
\end{proof}

Since $\ascat S_\bullet(\Vect(x))$ is a model for $B\mc K(\Vect(x))$, we may combine the categories $\mc T_m(\mc C, E)$ to obtain a model for $\TwC$.

\begin{corollary}\label{cor:TwiscolimitofTm}
    The map
    \begin{equation}\label{eq:intofTmtoC}
        \begin{tikzcd}
            \displaystyle\int^{m\in\Delta} \mc T_m(\mc C, E)\times [m]\arrow{r} & \mc C
        \end{tikzcd}
    \end{equation}
    induced by the natural in $m$ maps
    \begin{equation*}
        \begin{tikzcd}
            \mc T_m(\mc C, E)\arrow{r} & \Fun([m], \mc C)
        \end{tikzcd}
    \end{equation*}
    given by restriction along the diagonal is equivalent to $\TwC\to\mc C$ in $\Cat_{/\mc C}$.
\end{corollary}
\begin{proof}
    By the coend formula for the Cartesian unstraightening (\Cref{thm:gepcoendunstraightening}), we know that $\TwC$ is equivalent to
    \begin{equation}\label{eq:intformulaforTw}
        \int^{x\in \mc C}\ascat S_\bullet(\Vect(x)) \times \mc C_{/x}
    \end{equation}
    with projection to $\mc C$ given by
    \begin{equation*}
        \int^{x\in \mc C}\ascat S_\bullet(\Vect(x)) \times \mc C_{/x}\longrightarrow \int^{x\in \mc C} \mc C_{/x}\simeq \mc C.
    \end{equation*}
    By the coend formula for associated categories (\Cref{prop:coendformulaforrealization}), we have that
    \begin{equation*}
        \ascat S_\bullet(\Vect(x)) \simeq \int^{m\in\Delta} S_m(\Vect(x))\times [m].
    \end{equation*}
    It follows that
    \begin{equation*}
        \TwC \simeq \int^{x\in\mc C}\Big(\int^{m\in\Delta}S_m(\Vect(x))\times [m]\Big)\times \mc C_{/x}.
    \end{equation*}
    By the Fubini theorem for coends (\Cref{prop:fubiniforcoends}), this is equivalent to
    \begin{equation*}
        \int^{m\in\Delta}\Big(\int^{x\in\mc C} S_m(\Vect(x))\times \mc C_{/x}\Big)\times [m]\simeq \int^{m\in\Delta}\mc T_m(\mc C, E)\times [m].
    \end{equation*}
    For the last equivalence we again applied the coend formula for unstraightening and \Cref{prop:TmisCartunstraightofSm}.
\end{proof}

We now finally arrive at a gluing description for $\TwC$.

\begin{corollary}\label{cor:GrdescriptionofTw}
    The diagram of bisimplicial spaces
    \begin{equation*}
        N(\mc T_\bullet(\mc C, E))_\ast \longrightarrow \Sq(\mc C)
    \end{equation*}
    given by restriction along the diagonal adjoins to $\TwC\to \mc C$.\par
    Moreover, under this equivalence, the zero section $s : \mc C\to\TwC$ corresponds (after gluing) to
    \begin{equation*}
        \Ver(\mc C)\simeq N(\mc T_0(\mc C, E))_\ast \longrightarrow N(\mc T_\bullet(\mc C, E))_\ast.
    \end{equation*}
\end{corollary}
\begin{proof}
    By the coend formula for the gluing of bisimplicial spaces (\Cref{prop:coendformulaforrealization}), we have that
    \begin{equation*}
        \begin{aligned}
            \Gr N(\mc T_\bullet(\mc C, E))_\ast &\simeq \int^{m,n\in\Delta} N(\mc T_m(\mc C, E))_n \times [n]\times [m] \\
            &\simeq \int^{m\in \Delta}\Big(\int^{n\in\Delta} N(\mc T_m(\mc C, E))_n\times [n]\Big)\times [m] \\
            &\simeq \int^{m\in\Delta}\mc T_m(\mc C, E)\times [m].
        \end{aligned}
    \end{equation*}
    The result then follows from \Cref{cor:TwiscolimitofTm}.
\end{proof}

\begin{definition}
    We let $\Th_{\ast,\bullet}(\mc C, E)$ denote the bisimplicial space $N(\mc T_\bullet(\mc C, E))_\ast$.
\end{definition}

In the philosophy of \Cref{subsec:weakfactsystems}, $\Th_{\bullet,\bullet}(\mc C, E)$ will represent the analog of $\Fact S_\bullet(\mc C, E)$ for the twist--horizontal factorization system. Towards this, we show that there is a natural in $n$ and $m$ map
\begin{equation*}
    \Th_{n,m}(\mc C, E)\longrightarrow \map(\Delta[n]\times\Delta[m], S_\bullet(\mc C, E))
\end{equation*}
which on horizontal fragments is the inclusion of the sub-simplicial space of twists, and on vertical fragments is the inclusion of the sub-simplicial space of horizontal morphisms.

\begin{lemma}\label{lemma:buildsquarefromnerve}
    Suppose that $F : \Cpt^m\times [n]\to\mc C\in \Th_{n,m}(\mc C, E)$. Then for every $(\alpha,\beta) : [\ell]\to [n]\times [m]$, the induced composite
    \begin{equation*}
        \begin{tikzcd}[row sep=0pt]
            \Cpt^\ell\arrow{r} & \Cpt^m\times \Cpt^n\arrow{r}{\id\times\pr_2} & \Cpt^m\times [n]\arrow{r}{F} & \mc C, \\
            (i,j)\arrow[mapsto]{r} & ((\beta(i), \beta(j)), (\alpha(i), \alpha(j))) \\
            & ((i,j), (p,q))\arrow[mapsto]{r} & ((i,j), q)
        \end{tikzcd}
    \end{equation*}
    belongs to $S_\ell(\mc C, E)$.
\end{lemma}
\begin{proof}
    Consider a square
    \begin{equation*}
        \begin{tikzcd}
            (i,j)\arrow{r}\arrow{d} & (i + 1, j)\arrow{d} \\
            (i, j + 1)\arrow{r} & (i + 1, j + 1)
        \end{tikzcd}
    \end{equation*}
    in $\Cpt^\ell$. Under this composite, this gets sent to
    \begin{equation*}
        \begin{tikzcd}
            F((\beta(i), \beta(j)), \alpha(j))\arrow{r}\arrow{d} & F((\beta(i + 1), \beta(j)), \alpha(j + 1))\arrow{d} \\
            F((\beta(i), \beta(j + 1)), \alpha(j))\arrow{r} & F((\beta(i + 1), \beta(j + 1)), \alpha(j + 1)).
        \end{tikzcd}
    \end{equation*}
    This is the composite of the two squares
    \begin{equation*}
        \begin{tikzcd}
            F((\beta(i), \beta(j)), \alpha(j))\arrow{r}\arrow{d} & F((\beta(i + 1), \beta(j)), \alpha(j))\arrow{d} \\
            F((\beta(i), \beta(j + 1)), \alpha(j))\arrow{r} & F((\beta(i + 1), \beta(j + 1)), \alpha(j))
        \end{tikzcd}
    \end{equation*}
    and
    \begin{equation*}
        \begin{tikzcd}
            F((\beta(i + 1), \beta(j)), \alpha(j))\arrow{r}\arrow{d} & F((\beta(i + 1), \beta(j)), \alpha(j + 1))\arrow{d} \\
            F((\beta(i + 1), \beta(j + 1)), \alpha(j))\arrow{r} & F((\beta(i + 1), \beta(j + 1)), \alpha(j + 1))
        \end{tikzcd}
    \end{equation*}
    which are both pullbacks.\par
    We also have that a vertical morphism $(i,j) \to (i', j)$ in $\Cpt^\ell$ gets sent to $F((\beta(i), \beta(j)),\alpha(j))\to F((\beta(i'), \beta(j)), \alpha(j))$ which belongs to $E$.
\end{proof}

\begin{proposition}\label{prop:mapsonTm}
    There exists a map
    \begin{equation*}
        \Th_{\bullet,\bullet}(\mc C, E)\longrightarrow \Sq(S_\bullet(\mc C, E))
    \end{equation*}
    of bisimplicial spaces over $\Sq(\mc C)$ which
    \begin{enumerate}
        \item on vertical fragments is the inclusion of the horizontal sub-simplicial space
        \item on horizontal fragments is the inclusion of the twist sub-simplicial space.
    \end{enumerate}
\end{proposition}
\begin{proof}
    We must give a natural in $n$, $m$ map
    \begin{equation*}
        N(\mc T_m(\mc C, E))_n\longrightarrow \map_{\PSh(\Delta)}(\Delta[n]\times\Delta[m], S_\bullet(\mc C, E))
    \end{equation*}
    of spaces. For a given $F : \Cpt^m\times [n]\to \mc C\in N(\mc T_m(\mc C, E))_n$, we send this to the map $\Delta[n]\times\Delta[m]\to S_\bullet(\mc C, E)$ which sends an $\ell$-simplex $[\ell]\to [n]\times [m]$ of $\Delta[n]\times \Delta[m]$ to the composite
    \begin{equation*}
        \begin{tikzcd}
            \Cpt^\ell\arrow{r} & \Cpt^m\times \Cpt^n\arrow{r} & \Cpt^m\times [n]\arrow{r}{F} & \mc C
        \end{tikzcd}
    \end{equation*}
    of \Cref{lemma:buildsquarefromnerve} which, by \Cref{lemma:buildsquarefromnerve}, belongs to $S_\ell(\mc C, E)$. This is evidently functorial and natural in all necessary variables.\par
    (i) and (ii) are then immediate from the definition.
\end{proof}

As in \Cref{subsec:weakfactsystems}, this induces a map $\Gr\Th_{\bullet,\bullet}(\mc C, E)\to \ascat S_\bullet(\mc C, E)$ by taking the transpose of the composite
\begin{equation*}
    \Th_{\bullet,\bullet}(\mc C, E)\longrightarrow \Sq(S_\bullet(\mc C, E))\longrightarrow \Sq(\ascat S_\bullet(\mc C, E)).
\end{equation*}

\subsection{The twist--horizontal factorization system}

In this section, we show that every simplex in $S_\bullet(\mc C, E)$ may functorially be represented as a composite of a twist and a horizontal morphism. At its heart, this generalizes the observation that every $1$-simplex in $S_\bullet(\mc C, E)$ represented by a diagram
\begin{equation*}
    \begin{tikzcd}
    	x & y \\
    	& z
    	\arrow["f", from=1-1, to=1-2]
    	\arrow["g", from=1-2, to=2-2]
    \end{tikzcd}
\end{equation*}
may be written as a composite of a twist and a horizontal morphism via the $2$-simplex
\begin{equation*}
    \begin{tikzcd}
    	x & {x\times_z y} & y \\
    	& x & z \\
    	&& {z.}
    	\arrow["{(\id, f)}", from=1-1, to=1-2]
    	\arrow[from=1-2, to=1-3]
    	\arrow[from=1-2, to=2-2]
    	\arrow["\lrcorner"{anchor=center, pos=0.125}, draw=none, from=1-2, to=2-3]
    	\arrow["g", from=1-3, to=2-3]
    	\arrow["gf", from=2-2, to=2-3]
    	\arrow[equals, from=2-3, to=3-3]
    \end{tikzcd}
\end{equation*}

\begin{definition}
    We call a diagram $\Cpt^n\to \mc C$ \emph{good} if it belongs to the image of $S_n(\mc C, E)\hookrightarrow \map_{\Cat}(\Cpt^n, \mc C)$.
\end{definition}

In \Cref{subsec:gluingdescr}, we showed (\Cref{prop:goodisdetectonspine}) that a good diagram $D : \Cpt^n\to \mc C$ corresponding to an $n$-simplex of $S_\bullet(\mc C, E)$ is fully determined by its restriction to $\mathrm{Sp}^n$. Moreover, the value of $D(i,j)$ may be recovered by three elements on $\mathrm{Sp}^n$, namely the canonical map
\begin{equation*}
    D(i,j) \xrightarrow{\simeq} D(j,j)\times_{D(j, n)} D(i, n)
\end{equation*}
is an equivalence where $(j,j),(i,n),(j,n)\in\mathrm{Sp}^n$. To construct our twist--horizontal factorizations for simplices of $S_\bullet(\mc C, E)$ we will need to reference more general fiber products of images of elements of $\mathrm{Sp}^n$. Namely, we will need to reference
\begin{equation}\label{eq:genfiberproductforPielement}
    D(i,i)\times_{D(j,n)}D(k,n)
\end{equation}
for all $i,j,k$ with $i,k\le j$. We encode these possible fiber products by a poset $\Pi^n$ which we define below. In what follows, one should think of the element $(i,j,k)\in \Pi^n$ as representing the fiber product \eqref{eq:genfiberproductforPielement} for some fixed good diagram $D : \Cpt^n\to\mc C$.

\begin{definition}
    We denote the poset $\Pi^n$ to be the collection $\{(i,j,k) : i,k\le j\}$ with the relation $(i,j,k)\le (p,q,r)$ if and only if $i\le p$, $j\le q$ and $k\le r$. This upgrades to a functor $\Pi^\bullet : \Delta\to \Cat$ where for $d : [n]\to [m]$ we have
    \begin{equation*}
        \begin{tikzcd}[row sep=0pt]
            \Pi(d) : \Pi^n\arrow{r} & \Pi^m \\
            \qquad\;\;\; (i,j,k)\arrow[mapsto]{r} & (d(i), d(j), d(k)).
        \end{tikzcd}
    \end{equation*}
\end{definition}

\begin{remark}
    There is a natural transformation $\Cpt^\bullet\Rightarrow \Pi^\bullet$ of functors $\Delta\to\Cat$ given at level $n$ by
    \begin{equation*}
        \begin{tikzcd}[row sep=0pt]
            \iota : \Cpt^n\arrow{r} & \Pi^n \\
            \quad (i,j)\arrow[mapsto]{r} & (j,j,i).
        \end{tikzcd}
    \end{equation*}
\end{remark}

Given a good diagram $D : \Cpt^n\to\mc C$, we may produce a functor $D' : \Pi^n\to\mc C$ by sending $(i,j,k)\in \Pi^n$ to $D(i,i)\times_{D(j,n)}D(k,n)$. In light of \Cref{prop:goodisdetectonspine}, the composition $\Cpt^n\xrightarrow{\iota}\Pi^n\xrightarrow{D'}\mc C$ recovers the original functor $D$. We will only ever care about functors $\Pi^n\to\mc C$ that arise in this manner, so we abstract some of their properties into what we will also call a good functor.

\begin{definition}\label{def:Pigooddiagram}
    We call a functor $\Pi^n\to \mc C$ \emph{good} if
    \begin{enumerate}
        \item every $(n,n,i)\to (n,n,j)$ gets sent to an arrow in $E$
        \item for every $i\le p\le j\le q$, the square
        \begin{equation*}
            \begin{tikzcd}
                (j,j,i)\arrow{r}\arrow{d} & (j,j,p)\arrow{d} \\
                (q,q,i)\arrow{r} & (q,q,p)
            \end{tikzcd}
        \end{equation*}
        gets sent to a pullback in $\mc C$
        \item for every $j\le i,k$ the square
        \begin{equation*}
            \begin{tikzcd}
                (i,j,k)\arrow{r}\arrow{d} & (n,n,k)\arrow{d} \\
                (i,j,j)\arrow{r} & (n,n,j)
            \end{tikzcd}
        \end{equation*}
        gets sent to a pullback square in $\mc C$
        \item for every $i\le j\le k$, the map $(i,j,j)\to (i,k,k)$ gets sent to an equivalence in $\mc C$.
    \end{enumerate}
    Since this notion is invariant under equivalence, we will denote by $\map^\mathrm{good}(\Pi^n, \mc C)$ the union of connected components of $\map_{\Cat}(\Pi^n, \mc C)$ consisting of good functors.
\end{definition}

\begin{remark}\label{rem:reformofcondiandii}
    (i) and (ii) of \Cref{def:Pigooddiagram} are equivalent to asking that the restriction $\Cpt^n\xrightarrow{\iota}\Pi^n\to \mc C$ is good.
\end{remark}

\begin{lemma}\label{lemma:extrapullbacksingoodPi}
    \begin{enumerate}
        \item \Cref{def:Pigooddiagram}(iii) is equivalent to every diagram of the form
        \begin{equation}\label{eq:biggerclassofsquares}
            \begin{tikzcd}
                (i,j,k)\arrow{r}\arrow{d} & (i',j',k)\arrow{d} \\
                (i,j,k')\arrow{r} & (i',j',k')
            \end{tikzcd}
        \end{equation}
        being sent to a pullback in $\mc C$.
        \item If $\Pi^m\to\mc C$ is good, then every morphism of the form $(i,j,k)\to (i,j,k')$ gets sent to a morphism in $E$.
    \end{enumerate}
\end{lemma}
\begin{proof}
    First we prove (i). It is clear that this condition implies the condition of \Cref{def:Pigooddiagram}(iii) as a special case, so we prove the other direction.\par
    For the other direction, we will show that each of these squares may be deduced to be Cartesian via pasting laws from squares occurring in \Cref{def:Pigooddiagram}(iii). To start, \eqref{eq:biggerclassofsquares} may be decomposed as the composite
    \begin{equation*}
        \begin{tikzcd}
            (i,j,k)\arrow{r}\arrow{d} & (i,j',k)\arrow{r}\arrow{d} & (i',j',k)\arrow{d} \\
            (i,j,k')\arrow{r} & (i,j',k')\arrow{r} & (i',j',k'),
        \end{tikzcd}
    \end{equation*}
    which by pasting reduces to the two cases of $i = i'$ and $j = j'$.\par
    First we handle the case of $i = i'$. By considering the composite diagram
    \begin{equation*}
        \begin{tikzcd}
            (i,j,k)\arrow{r}\arrow{d} & (i,j',k)\arrow{d}\arrow{r} & (i,n,k)\arrow{d} \\
            (i,j,k')\arrow{r} & (i,j',k')\arrow{r} & (i,n,k')
        \end{tikzcd}
    \end{equation*}
    and looking at the outer and right-most square, we reduce further by pasting to the case of $j' = n$. Next, looking at the composite diagram
    \begin{equation*}
        \begin{tikzcd}
            (i,j,k)\arrow{r}\arrow{d} & (i,n,k)\arrow{d} \\
            (i,j,k')\arrow{r}\arrow{d} & (i,n,k')\arrow{d} \\
            (i,j,j)\arrow{r} & (i,n,j)
        \end{tikzcd}
    \end{equation*}
    we reduce to when $k' = j$. For this, we look at the pasted diagram
    \begin{equation*}
        \begin{tikzcd}
            (i,j,k)\arrow{r}\arrow{d} & (i,n,k)\arrow{d}\arrow{r} & (n,n,k)\arrow{d} \\
            (i,j,j)\arrow{r} & (i,n,j)\arrow{r} & (n,n,j).
        \end{tikzcd}
    \end{equation*}
    The outer square gets sent to a pullback by \Cref{def:Pigooddiagram}(iii), so it suffices to prove that the right-most square gets sent to a pullback. For this, we consider the composite
    \begin{equation*}
        \begin{tikzcd}
            (i,n,k)\arrow{d}\arrow{r} & (n,n,k)\arrow{d} \\
            (i,n,j)\arrow{r}\arrow{d} & (n,n,j)\arrow{d} \\
            (i,n,n)\arrow{r} & (n,n,n).
        \end{tikzcd}
    \end{equation*}
    The outer and bottom squares get sent to pullbacks by \Cref{def:Pigooddiagram}(iii) so we conclude that the top square gets sent to a pullback as required.\par
    Next we handle the case of $j = j'$. By considering the composite diagram
    \begin{equation*}
        \begin{tikzcd}
            (i,j,k)\arrow{r}\arrow{d} & (i',j,k)\arrow{d} \\
            (i,j,k')\arrow{r}\arrow{d} & (i',j, k')\arrow{d} \\
            (i,j,j)\arrow{r} & (i',j,j),
        \end{tikzcd}
    \end{equation*}
    we reduce to the case of $k' = j$. We may then consider the composite diagram
    \begin{equation*}
        \begin{tikzcd}
            (i,j,k)\arrow{r}\arrow{d} & (i',j,k)\arrow{r}\arrow{d} & (n,n,k)\arrow{d} \\
            (i,j,j)\arrow{r} & (i',j,j)\arrow{r} & (n,n,j)
        \end{tikzcd}
    \end{equation*}
    from which we see that the right-most square and outer square get sent to pullbacks by \Cref{def:Pigooddiagram}(iii). The result then follows by pasting.\par
    For (ii), using (i), we have that the diagram
    \begin{equation*}
        \begin{tikzcd}
            (i,j,k)\arrow{r}\arrow{d} & (n,n,k)\arrow{d} \\
            (i,j,k')\arrow{r} & (n,n,k')
        \end{tikzcd}
    \end{equation*}
    gets sent to a pullback in $\mc C$. Since $E$ is stable under base change, the result then follows from \Cref{def:Pigooddiagram}(i).
\end{proof}

\begin{corollary}\label{cor:restrictionpresgoodnessPi}
    For every $d : [n]\to [m]$, the induced restriction $\map_{\Cat}(\Pi^m, \mc C)\to \map_{\Cat}(\Pi^n, \mc C)$ sends good diagrams to good diagrams. In particular, $\map^\mathrm{good}(\Pi^n,\mc C)$ upgrades to a simplicial space $\map^\mathrm{good}(\Pi^\bullet, \mc C) : \Delta^\mathrm{op}\to \Ani$.
\end{corollary}
\begin{proof}
    Fix $F : \Pi^m\to \mc C$ a good diagram. We check that $F\circ \Pi(d)$ is good from the definition.\par
    For (i) and (ii), this follows from \Cref{rem:reformofcondiandii} and the fact that the inclusion $\Cpt^n\to \Pi^n$ is natural in $n$ and that restrictions of good diagrams $\Cpt^m\to \mc C$ along $\Cpt(d) : \Cpt^n\to\Cpt^m$ are good.\par
    For (iii), the reformulation given by \Cref{lemma:extrapullbacksingoodPi} is clearly preserved under restriction.\par
    (iv) is immediate.
\end{proof}

We now come to our quintessential source of good diagrams $\Pi^n\to\mc C$.

\begin{proposition}\label{prop:lifttoPiposet}
    There exists a functor $S_\bullet(\mc C, E)\to \map^\mathrm{good}(\Pi^\bullet, \mc C)$ fitting into a commutative diagram
    \begin{equation*}
        \begin{tikzcd}
            S_\bullet(\mc C, E)\arrow{r}\arrow[bend right=20]{rr}[below]{\id} & \map^\mathrm{good}(\Pi^\bullet, \mc C)\arrow{r}{\iota^\ast} & S_\bullet(\mc C, E)
        \end{tikzcd}
    \end{equation*}
    which sends a simplex $\Delta^n\to S_\bullet(\mc C, E)$ given by a functor $D : \Cpt^n\to \mc C$ to the map
    \begin{equation*}
        \begin{tikzcd}[row sep=0pt]
            \Pi^n \arrow{r} & \mc C \\
            (i,j,k)\arrow[mapsto]{r} & D(i,i)\times_{D(j,n)} D(k,n).
        \end{tikzcd}
    \end{equation*}
\end{proposition}
\begin{proof}
    We have a functor
    \begin{equation*}
        \begin{tikzcd}[row sep=0pt]
            C^n : \Pi^n\arrow{r} & \Fun(\Lambda_2^2, \Cpt^n) \\
            (i,j,k)\arrow[mapsto]{r} & (i,i)\rightarrow (j,n)\leftarrow (k,n)
        \end{tikzcd}
    \end{equation*}
    which is natural in $n$. We may then consider the composite
    \begin{equation*}
        \begin{tikzcd}[column sep=17pt]
            S_\bullet(\mc C, E) \arrow[hook]{r} & \map_{\Cat}(\Cpt^\bullet, \mc C)\arrow{r}{\Fun(\Lambda_2^2, -)_\ast} &[30pt] \map_{\Cat}(\Fun(\Lambda_2^2, \Cpt^\bullet), \Fun(\Lambda_2^2,\mc C))\arrow{r}[above]{(C^\bullet)^\ast} & \map_{\Cat}(\Pi^\bullet, \Fun(\Lambda_2^2, \mc C)).
        \end{tikzcd}
    \end{equation*}
    Since diagrams $F :\Cpt^n\to\mc C\in S_n(\mc C, E)$ send vertical morphisms to morphisms in $E$ this composite factors as
    \begin{equation*}
        S_\bullet(\mc C, E)\longrightarrow  \map_{\Cat}(\Pi^\bullet, \Fun_E(\Lambda_2^2,\mc C))\longrightarrow \map_{\Cat}(\Pi^\bullet, \Fun(\Lambda_2^2,\mc C))
    \end{equation*}
    where $\Fun_E(\Lambda_2^2, \mc C)$ denotes the full subcategory of $\Fun(\Lambda_2^2,\mc C)$ consisting of cospans which send the right leg $2\leftarrow 1$ to a morphism in $E$. Since $\mc C$ has pullbacks along morphisms in $E$, there is a functor $\lim : \Fun_E(\Lambda_2^2, \mc C)\to \mc C$ which forms the limit of these cospans. We may thus take the composite
    \begin{equation}\label{eq:upgradetoPi}
        \begin{tikzcd}
            S_\bullet(\mc C, E)\arrow{r} &  \map_{\Cat}(\Pi^\bullet, \Fun_E(\Lambda_2^2,\mc C))\arrow{r}{\lim_\ast} & \map(\Pi^\bullet, \mc C).
        \end{tikzcd}
    \end{equation}
    This map sends a simplex $\Delta^n\to S_\bullet(\mc C, E)$ corresponding to a diagram $D : \Cpt^n\to\mc C$ to the diagram
    \begin{equation*}
        \begin{tikzcd}[row sep=0pt]
            \Pi^n \arrow{r} & \mc C \\
            (i,j,k)\arrow[mapsto]{r} & D(i,i)\times_{D(j,n)} D(k,n)
        \end{tikzcd}
    \end{equation*}
    which one readily checks from the definition is good. It follows that the map \eqref{eq:upgradetoPi} lands inside the sub-simplicial space $\map^\mathrm{good}(\Pi^\bullet,\mc C)\subseteq \map_{\Cat}(\Pi^\bullet, \mc C)$.\par
    When restricted to $\mathrm{Sp}^n\subseteq \Cpt^n\subseteq \Pi^n$, this formula reduces to the identity up to canonical homotopy so that we have a commutative diagram
    \begin{equation*}
        \begin{tikzcd}
            S_\bullet(\mc C, E)\arrow{r}\arrow{dr}[below left]{\mathrm{res}} & \map^\mathrm{good}(\Pi^\bullet, \mc C)\arrow{r}{\iota^\ast} & S_\bullet(\mc C, E)\arrow{dl}[below right]{\mathrm{res}} \\
            & \map_E(\mathrm{Sp}^\bullet, \mc C)
        \end{tikzcd}
    \end{equation*}
    and thus the result follows from \Cref{prop:goodisdetectonspine}.
\end{proof}

\begin{definition}\label{def:pathinPifromCpt}
    For every $[m]\in \Delta$, define
    \begin{equation*}
        \begin{tikzcd}[row sep=0pt]
            \mathrm{Fct}(m) : N(\Cpt^m)\arrow{r} & \map_{\Cat}(\Cpt^\bullet, \Pi^m)
        \end{tikzcd}
    \end{equation*}
    which on $n$-simplices sends $(r,c) : [n]\to \Cpt^m$ to
    \begin{equation*}
        (i,j)\mapsto (r(j), c(j), c(i)) : \Cpt^n\to \Pi^m.
    \end{equation*}
    This is natural in $m$ and upgrades to a map $\mathrm{Fct} : N(\Cpt^\ast)_\bullet\to \map(\Cpt^\bullet, \Pi^\ast)$ of bisimplicial spaces.
\end{definition}

\begin{remark}
    All the categories occurring in \Cref{def:pathinPifromCpt} are ordinary categories, and therefore all the simplicial and bisimplicial spaces are discrete. Thus one is able to define the maps pointwise and verify naturality in $m$ on the nose, as there are no higher topological considerations.
\end{remark}

\begin{definition}
    We call a functor $\Cpt^n\to \Pi^m$ \emph{universally good} if for every good diagram $\Pi^m\to\mc C$, the composite $\Cpt^n\to \Pi^m\to\mc C$ is good.
\end{definition}

\begin{lemma}
    Let $\map^\mathrm{good}(\Cpt^n, \Pi^m)$ denote the subspace of components of $\map_{\Cat}(\Cpt^n, \Pi^m)$ consisting of universally good functors.
    \begin{enumerate}
        \item $\map^\mathrm{good}(\Cpt^n, \Pi^m)$ upgrades to a sub-bisimplicial space of $\map_{\Cat}(\Cpt^\bullet, \Pi^\ast)$. That is,
        \begin{enumerate}
            \item For every $d : [m]\to [m']$, if $\Cpt^n\to\Pi^m$ is universally good then so is the composite $\Cpt^n\to\Pi^m\to\Pi^{m'}$
            \item For every $d : [n']\to [n]$, if $\Cpt^n\to\Pi^m$ is universally good then so is $\Cpt^{n'}\to\Cpt^n\to \Pi^m$
        \end{enumerate}
        \item $\mathrm{Fct}$ lands inside the sub-bisimplicial space $\map^\mathrm{good}(\Cpt^\bullet, \Pi^\ast)\subseteq \map(\Cpt^\bullet, \Pi^\ast)$.
    \end{enumerate}
\end{lemma}
\begin{proof}
    For (i)(a), this follows from the fact that if $\Pi^{m'}\to \mc C$ is good, then the composite $\Pi^m\to \Pi^{m'}\to \mc C$ is good (see \Cref{cor:restrictionpresgoodnessPi}). For (i)(b), this follows from the fact that if $\Cpt^n\to\mc C$ is good, then so is $\Cpt^{n'}\to\Cpt^n\to\mc C$.\par
    For (ii), fix $(r,c) : [n]\to\Cpt^m$ and suppose that $F : \Pi^m\to\mc C$ is good. We need to check that $F\circ \mathrm{Fct}(m)([n]\to\Cpt^m)$ is good. Firstly, given
    \begin{equation*}
        \begin{tikzcd}
            (i,j)\arrow{r}\arrow{d} & (p,j)\arrow{d} \\
            (i,q)\arrow{r} & (p,q)
        \end{tikzcd}
    \end{equation*}
    in $\Cpt^n$, we need
    \begin{equation*}
        \begin{tikzcd}
            (r(j), c(j), c(i))\arrow{r}\arrow{d} & (r(j), c(j), c(p))\arrow{d} \\
            (r(q), c(q), c(i))\arrow{r} & (r(q), c(q), c(p))
        \end{tikzcd}
    \end{equation*}
    to be sent to a pullback under $F : \Pi^m\to\mc C$. This follows from \Cref{lemma:extrapullbacksingoodPi}. Additionally, we need to check that every morphism $(i,p)\to (j,p)$ gets sent to a morphism in $E$, or equivalently that $F$ sends morphisms of the form $(r(p), c(p), c(i))\to (r(p), c(p), c(j))$ to a morphism in $E$. This again follows from \Cref{lemma:extrapullbacksingoodPi}.
\end{proof}

We are now equipped to construct the twist--horizontal factorization. We refer to \Cref{ex:workedfactex} below for a worked example to illustrate this factorization procedure.

\begin{theorem}\label{thm:twisthorfact}
    There exists a map
    \begin{equation*}
        \begin{tikzcd}
            ff_{\mathrm{TH}} : S_\bullet(\mc C, E)\arrow{r} & \cnr \Th_{\bullet,\bullet}(\mc C, E)
        \end{tikzcd}
    \end{equation*}
    and a factoring
    \begin{equation*}
        \begin{tikzcd}
            f_\mathrm{TH} : S_\bullet(\mc C, E)\arrow{r} & \sd S_\bullet(\mc C, E)
        \end{tikzcd}
    \end{equation*}
    fitting into a commutative diagram
    \begin{equation}\label{eq:twhorfactsquare}
        \begin{tikzcd}
            \cnr \Th_{\bullet,\bullet}(\mc C, E)\arrow{r}{\cnr\alpha} & \fsd S_\bullet(\mc C, E) \\
            S_\bullet(\mc C, E)\arrow{u}[left]{ff_\mathrm{TH}}\arrow{r}[below]{f_\mathrm{TH}} & \sd S_\bullet(\mc C, E)\arrow{u}
        \end{tikzcd}
    \end{equation}
    where $\alpha : \Th_{\bullet,\bullet}(\mc C, E)\to \Sq(S_\bullet(\mc C, E))$ is the map of \Cref{prop:mapsonTm}. Moreover, this factorization is degenerate on horizontal and vertical fragments (see \Cref{def:degenerateonhoriz}).
\end{theorem}
\begin{proof}
    We first construct the factoring on $S_\bullet(\mc C, E)$. Observe that we have a map
    \begin{equation*}
        \begin{tikzcd}[row sep=0pt]
            \mathrm{pc} : \map^{\mathrm{good}}(\Pi^\bullet,\mc C)\arrow{r} & \map_{\PSh(\Delta)}(\map^\mathrm{good}(\Cpt^\ast, \Pi^\bullet), S_\ast(\mc C, E)) \\
            \quad (F : \Pi^m\to\mc C)\arrow[mapsto]{r} & (\Cpt^n\to\Pi^m\mapsto \Cpt^n\to \Pi^m\xrightarrow{F}\mc C).
        \end{tikzcd}
    \end{equation*}
    We may then define $f_\mathrm{TH}$ to be the composite
    \begin{equation*}
        \begin{aligned}
            S_\bullet(\mc C, E)\fixedxrightarrow[2cm]{\text{\normalfont\Cref{prop:lifttoPiposet}}} \map^\mathrm{good}(\Pi^\bullet, \mc C) &\fixedxrightarrow[1cm]{\mathrm{pc}} \map_{\PSh(\Delta)}(\map^\mathrm{good}(\Cpt^\ast, \Pi^\bullet), S_\ast(\mc C, E)) \\
            &\fixedxrightarrow[1cm]{\mathrm{Fct}^\ast} \map_{\PSh(\Delta)}(N(\Cpt^\bullet), S_\ast(\mc C, E)).
        \end{aligned}
    \end{equation*}
    The fact that this is a factoring follows from \Cref{prop:lifttoPiposet} and the fact that $\iota^\ast \simeq \diag^\ast\circ\;\mathrm{Fct}^\ast\circ\mathrm{pc}$.\par
    We now need to construct $ff_\mathrm{TH}$. However, $ff_\mathrm{TH}$ is essentially determined by \eqref{eq:twhorfactsquare}. Indeed, for each $D : \Cpt^\ell\to \mc C\in S_\ell(\mc C, E)$ we must produce a map $\CCpt^\ell\to \Th_{\bullet,\bullet}(\mc C, E)$. This means that for every
    \begin{equation*}
        \beta : [n]\bar\times [m]\longrightarrow \Cpt^\ell
    \end{equation*}
    we must provide a map
    \begin{equation*}
        F_\beta : \Cpt^m\times [n]\longrightarrow \mc C\in N(\mc T_m(\mc C, E))_n.
    \end{equation*}
    \eqref{eq:twhorfactsquare} is then the statement that, for every
    \begin{equation*}
        (r,c) : [s]\longrightarrow [n]\times [m],
    \end{equation*}
    the composite
    \begin{equation}\label{eq:bigcompwithFalpha}
        \begin{tikzcd}[column sep=35pt, row sep=0pt]
            \Cpt^s\arrow{r} & \Cpt^m\times \Cpt^n\arrow{r}{\id\times\pr_2} & \Cpt^m\times [n]\arrow{r}{F_\beta} & \mc C\in S_s(\mc C, E) \\
            (i,j)\arrow[mapsto]{rr} && ((c(i), c(j)), r(j))
        \end{tikzcd}
    \end{equation}
    agrees with the $s$-simplex of $S_\bullet(\mc C, E)$ given by applying $f_\mathrm{TH}(D)$ to the $s$-simplex
    \begin{equation}\label{eq:ssimplofCpt}
        \begin{tikzcd}
            {[s]}\arrow{r}{(r,c)} & {[n]\times [m]}\arrow{r}{\beta} & \Cpt^\ell.
        \end{tikzcd}
    \end{equation}
    of $N(\Cpt^\ell)$. To check this condition, let $\beta = (\varepsilon, \eta)$. Since $\beta$ preserves horizontal and vertical arrows, $\varepsilon$ only depends on the row and $\eta$ only depends on the column, i.e.\
    \begin{equation*}
        \beta(i,j) = (\varepsilon(i), \eta(j)).
    \end{equation*}
    Then, the composite \eqref{eq:bigcompwithFalpha} becomes
    \begin{equation*}
        \begin{tikzcd}[row sep=0pt]
            \Cpt^s\arrow{r} & \Cpt^m\times\Cpt^n\arrow{r} & \Cpt^m\times [n]\arrow{r} & \mc C \\
            (i,j)\arrow[mapsto]{rr} && ((c(i), c(j)), r(j))\arrow[mapsto]{r} & F_\beta((c(i), c(j)), r(j)).
        \end{tikzcd}
    \end{equation*}
    On the other hand, $f_\mathrm{TH}(D)$ sends the $s$-simplex \eqref{eq:ssimplofCpt} to
    \begin{equation*}
        \begin{tikzcd}[row sep=0pt]
            \Cpt^s\arrow{r} & \Pi^\ell\arrow{r} & \mc C. \\
            (i,j)\arrow[mapsto]{r} & (\varepsilon(r(j)),\eta(c(j)),\eta(c(i))) \\
            & (i,j,k)\arrow[mapsto]{r} & D(i,i)\times_{D(j,n)} D(k,n)
        \end{tikzcd}
    \end{equation*}
    From this, it is clear that $F_\beta$ needs to be the composite
    \begin{equation*}
        \begin{tikzcd}[row sep=0pt]
            F_\beta : \Cpt^m\times [n]\arrow{r} & \Pi^\ell\arrow{r} & \mc C \\
            ((i,j), k)\arrow[mapsto]{r} & (\varepsilon(k), \eta(j), \eta(i)) \\
            & (i,j,k)\arrow[mapsto]{r} & D(i,i)\times_{D(j,n)} D(k,n)
        \end{tikzcd}
    \end{equation*}
    We leave the details of this construction to the reader as it is analogous to $f_\mathrm{TH}$. \eqref{eq:twhorfactsquare} is then immediate from the construction as it holds on the nose.\par
    To see that this factorization is degenerate on horizontal fragments, i.e.\ on twists, we need to show that $F_\beta$ depends only on the horizontal projection of $\beta$ when working with $D : \Cpt^\ell\to\mc C\in \Th_{0,\bullet}(\mc C, E)\subseteq S_\ell(\mc C, E)$. Recall that
    \begin{equation*}
        F_\beta((i,j), k) = D(\varepsilon(k), \varepsilon(k)) \times_{D(\eta(j), \ell)} D(\eta(i), \ell).
    \end{equation*}
    Since on $\Th_{0,\bullet}(\mc C, E)$, each $D$ sends diagonal maps to equivalences, we have that $F_\beta$ extends to a map
    \begin{equation*}
        \tilde F_\beta : \Cpt^m\times |[n]|\longrightarrow \mc C.
    \end{equation*}
    In particular, if we consider the bisimplicial space $\widetilde{\Th}_{\bullet,\bullet}(\mc C, E)$ where the $(n,m)$-simplex space are diagrams $\Cpt^m\times |[n]|\to \mc C$ such that each restriction $\Cpt^m\times \{i\}\to\mc C$ belongs to $\Th_{0,m}(\mc C, E)$, then we have a factorization
    \begin{equation*}
        \begin{tikzcd}
        	{\Th_{0,\bullet}(\mc C, E)} & {S_\bullet(\mc C, E)} \\
        	{\cnr \widetilde{\Th}_{\bullet,\bullet}(\mc C, E)} & {\cnr \Th_{\bullet,\bullet}(\mc C, E).}
        	\arrow["\alpha_{0,\bullet}", hook, from=1-1, to=1-2]
        	\arrow[from=1-1, to=2-1]
        	\arrow[from=1-2, to=2-2, "ff_\mathrm{TH}"]
        	\arrow[from=2-1, to=2-2]
        \end{tikzcd}
    \end{equation*}
    But the map
    \begin{equation*}
        \Hor(\Th_{0,\bullet}(\mc C, E))\longrightarrow \widetilde{\Th}_{\bullet,\bullet}(\mc C, E)
    \end{equation*}
    given on level $(n,m)$ by
    \begin{equation*}
        \begin{tikzcd}[row sep=0pt]
            \Th_{0,m}(\mc C, E)\arrow{r} & \widetilde{\Th}_{n,m}(\mc C, E) \\
            F : \Cpt^m\to\mc C\arrow[mapsto]{r} & (\Cpt^m\times [n]\to \Cpt^m\xrightarrow{F} \mc C)
        \end{tikzcd}
    \end{equation*}
    is an equivalence. Thus we have a factoring
    \begin{equation*}
        \begin{tikzcd}
        	{\Th_{0,\bullet}(\mc C, E)} & {S_\bullet(\mc C, E)} \\
        	{\cnr \Hor(\Th_{0,\bullet}(\mc C, E))} & {\cnr \Th_{\bullet,\bullet}(\mc C, E).}
        	\arrow[hook, from=1-1, to=1-2]
        	\arrow[from=1-1, to=2-1]
        	\arrow[from=1-2, to=2-2]
        	\arrow[from=2-1, to=2-2]
        \end{tikzcd}
    \end{equation*}
    But
    \begin{equation*}
        \cnr \Hor(X) \simeq \map_{\PSh(\Delta^{\times 2})}(\Delta[0,\bullet], \Hor(X))\simeq X
    \end{equation*}
    for any simplicial space $X$ and the claim follows.\par
    Degeneracy on vertical fragments follows by similar definition chasing so we omit it.
\end{proof}

\begin{example}\label{ex:workedfactex}
    We illustrate how this factorization works in the $n = 2$ case. Consider the following $2$-simplex of $S_\bullet(\mc C, E)$
    \begin{equation*}
        \begin{tikzcd}
            a & b & f \\
            & c & d \\
            && e.
            \arrow[from=1-1, to=1-2]
            \arrow[from=1-2, to=1-3]
            \arrow[from=1-2, to=2-2]
            \arrow["\lrcorner"{anchor=center, pos=0.125}, draw=none, from=1-2, to=2-3]
            \arrow[from=1-3, to=2-3]
            \arrow[from=2-2, to=2-3]
            \arrow[from=2-3, to=3-3]
        \end{tikzcd}
    \end{equation*}
    which we will refer to by $D : \Cpt^2\to\mc C$. $f_\mathrm{TH}$ produces a map $N(\Cpt^2)\to S_\bullet(\mc C, E)$ which is a factoring of the above $2$-simplex into twists and horizontal morphisms. We denote the diagram by
    \begin{equation*}
        \begin{tikzcd}
            \bullet & \bullet & \bullet \\
            & \bullet & \bullet \\
            && \bullet.
            \arrow["{t_1}", from=1-1, to=1-2]
            \arrow["{t_3}", from=1-2, to=1-3]
            \arrow["{h_1}", from=1-2, to=2-2]
            \arrow["{h_3}", from=1-3, to=2-3]
            \arrow["{t_2}", from=2-2, to=2-3]
            \arrow["{h_2}", from=2-3, to=3-3]
        \end{tikzcd}
    \end{equation*}
    The path
    \begin{equation*}
        \gamma = (0,0)\to (0,1)\to (0,2)\to (1,2)\to (2,2)
    \end{equation*}
    in $\Cpt^2$ should produce a $4$-simplex in $S_\bullet(\mc C, E)$ corresponding to the image of $h_2h_3t_3t_1$. To get this simplex, we first apply $\mathrm{Fct}$ to get the following diagram $\Cpt^4\to \Pi^2$:
    \begin{equation*}
        \begin{tikzcd}
            {(0,0,0)} & {(0,1,0)} & {(0,2,0)} & {(1,2,0)} & {(2,2,0)} \\
            & {(0,1,1)} & {(0,2,1)} & {(1,2,1)} & {(2,2,1)} \\
            && {(0,2,2)} & {(1,2,2)} & {(2,2,2)} \\
            &&& {(1,2,2)} & {(2,2,2)} \\
            &&&& {(2,2,2).}
            \arrow[from=1-1, to=1-2]
            \arrow[from=1-2, to=1-3]
            \arrow[from=1-2, to=2-2]
            \arrow[from=1-3, to=1-4]
            \arrow[from=1-3, to=2-3]
            \arrow[from=1-4, to=1-5]
            \arrow[from=1-4, to=2-4]
            \arrow[from=1-5, to=2-5]
            \arrow[from=2-2, to=2-3]
            \arrow[from=2-3, to=2-4]
            \arrow[from=2-3, to=3-3]
            \arrow[from=2-4, to=2-5]
            \arrow[from=2-4, to=3-4]
            \arrow[from=2-5, to=3-5]
            \arrow[from=3-3, to=3-4]
            \arrow[from=3-4, to=3-5]
            \arrow[from=3-4, to=4-4]
            \arrow[from=3-5, to=4-5]
            \arrow[from=4-4, to=4-5]
            \arrow[from=4-5, to=5-5]
        \end{tikzcd}
    \end{equation*}
    Then one sends $(i,j,k)\in \Pi^2$ to $D(i,i)\times_{D(j,n)}D(k,n)$ to obtain the $4$-simplex
    \begin{equation*}
        f_\mathrm{TH}(D)(\gamma) = \begin{tikzcd}
            a & {a\times_c b\simeq a\times_d f} & {a\times_e f} & {c\times_ef} & f \\
            & a & {a\times_e d} & {c\times_e d} & d \\
            && a & c & e \\
            &&& c & e \\
            &&&& e.
            \arrow[from=1-1, to=1-2]
            \arrow[from=1-2, to=1-3]
            \arrow[from=1-2, to=2-2]
            \arrow["\lrcorner"{anchor=center, pos=0.125}, draw=none, from=1-2, to=2-3]
            \arrow[from=1-3, to=1-4]
            \arrow[from=1-3, to=2-3]
            \arrow["\lrcorner"{anchor=center, pos=0.125}, draw=none, from=1-3, to=2-4]
            \arrow[from=1-4, to=1-5]
            \arrow[from=1-4, to=2-4]
            \arrow["\lrcorner"{anchor=center, pos=0.125}, draw=none, from=1-4, to=2-5]
            \arrow[from=1-5, to=2-5]
            \arrow[from=2-2, to=2-3]
            \arrow[from=2-3, to=2-4]
            \arrow[from=2-3, to=3-3]
            \arrow["\lrcorner"{anchor=center, pos=0.125}, draw=none, from=2-3, to=3-4]
            \arrow[from=2-4, to=2-5]
            \arrow[from=2-4, to=3-4]
            \arrow["\lrcorner"{anchor=center, pos=0.125}, draw=none, from=2-4, to=3-5]
            \arrow[from=2-5, to=3-5]
            \arrow[from=3-3, to=3-4]
            \arrow[from=3-4, to=3-5]
            \arrow[equals, from=3-4, to=4-4]
            \arrow["\lrcorner"{anchor=center, pos=0.125}, draw=none, from=3-4, to=4-5]
            \arrow[equals, from=3-5, to=4-5]
            \arrow[from=4-4, to=4-5]
            \arrow[equals, from=4-5, to=5-5]
        \end{tikzcd}
    \end{equation*}
    of $S_\bullet(\mc C, E)$. Similarly, one finds that the $1$-simplices of the factorization are given by the following points of $S_1(\mc C, E)$:
    \begin{equation*}
        \begin{array}{ccc}
            t_1 = \begin{tikzcd}
                a\arrow{r} & a\times_c b\arrow{d} \\
                & a
            \end{tikzcd} & h_1 = \begin{tikzcd}
                a\arrow{r} & c\arrow[equals]{d} \\
                & c
            \end{tikzcd} & t_2 = \begin{tikzcd}
                c\arrow{r} & c\times_e d\arrow{d} \\
                & c
            \end{tikzcd} \\
            h_2 = \begin{tikzcd}
                c\arrow{r} & e\arrow[equals]{d} \\
                & e
            \end{tikzcd} & t_3 = \begin{tikzcd}
                a\arrow{r} & a\times_e d\arrow{d} \\
                & a
            \end{tikzcd} & h_3 = \begin{tikzcd}
                a\arrow{r} & c\arrow[equals]{d} \\
                & c.
            \end{tikzcd}
        \end{array}
    \end{equation*}
    For the lift $ff_\mathrm{TH}(D)$ to $\Th_{\bullet,\bullet}(\mc C, E)$, the square
    \begin{equation*}
        \begin{tikzcd}
            (0,1)\arrow{r}\arrow{d} & (0,2)\arrow{d} \\
            (1,1)\arrow{r} & (1,2)
        \end{tikzcd}
    \end{equation*}
    in $\Cpt^2$ gives rise to a $(1,1)$-simplex of $\Th_{\bullet,\bullet}(\mc C, E)$ which in this case is the diagram
    \begin{equation*}
        \begin{tikzcd}[row sep=10pt, column sep=15pt]
        	&&& c &&&& {c\times_e d} \\
        	\\
        	a &&&& {a\times_e d} \\
        	&&&&&&& c. \\
        	\\
        	&&&& a
        	\arrow[from=1-4, to=1-8]
        	\arrow[from=1-8, to=4-8]
        	\arrow[from=3-1, to=1-4]
        	\arrow["\lrcorner"{anchor=center, pos=0.15, rotate=90}, draw=none, from=3-1, to=1-8]
        	\arrow[from=3-1, to=3-5]
        	\arrow[from=3-5, to=1-8]
        	\arrow["\lrcorner"{anchor=center, pos=0.1, rotate=0}, draw=none, from=3-5, to=4-8]
        	\arrow[from=3-5, to=6-5]
        	\arrow[from=6-5, to=4-8]
        \end{tikzcd}
    \end{equation*}
\end{example}

We are now equipped to show the desired equivalence.

\begin{proposition}\label{prop:grthsimeqacS}
    The map $\Gr \Th_{\bullet,\bullet}(\mc C, E)\longrightarrow \ascat S_\bullet(\mc C, E)$ induced by the map $\alpha : \Th_{\bullet,\bullet}(\mc C, E)\to \Sq(S_\bullet(\mc C, E))$ of \Cref{prop:mapsonTm} is an equivalence.
\end{proposition}
\begin{proof}
    Since $\alpha : \Th_{\bullet,\bullet}(\mc C, E)\to \Sq(S_\bullet(\mc C, E))$ is a map over $\Sq(\mc C)$, the induced map $\gamma : \Gr \Th_{\bullet,\bullet}(\mc C, E)\to \ascat S_\bullet(\mc C, E)$ lifts to one in $\Cat_{/\mc C}$. Additionally, the twist--horizontal factorization system is compatible with projection to $\mc C$ in the sense that
    \begin{equation*}
        \begin{tikzcd}
        	{S_\bullet(\mc C, E)} & {\cnr \Th_{\bullet,\bullet}(\mc C, E)} & {\fsd \mc C} \\
        	& {N(\mc C)}
        	\arrow["{ff_\mathrm{TH}}", from=1-1, to=1-2]
        	\arrow[from=1-1, to=2-2]
        	\arrow[from=1-2, to=1-3]
        	\arrow[from=1-3, to=2-2]
        \end{tikzcd}
    \end{equation*}
    commutes, where the second horizontal morphism is $\cnr$ of $\Th_{\bullet,\bullet}(\mc C, E)\to\Sq(\mc C)$. It follows that the induced map $\delta : \ascat S_\bullet(\mc C, E)\to \Gr\Th_{\bullet,\bullet}(\mc C, E)$ given by transposing
    \begin{equation*}
        \begin{tikzcd}
            S_\bullet(\mc C, E)\arrow{r}{ff_\mathrm{TH}} & \cnr \Th_{\bullet,\bullet}(\mc C, E)\arrow{r} & N(\Gr \Th_{\bullet,\bullet}(\mc C, E))
        \end{tikzcd}
    \end{equation*}
    lifts to a map in $\Cat_{/\mc C}$ as well.\par
    Combining \Cref{thm:twisthorfact} and \Cref{prop:weakfactsystemrightinv}, we have that $\gamma\delta\simeq \id$. Thus it only remains to show that $\delta\gamma\simeq \id$. By the above, $\delta\gamma : \Gr\Th_{\bullet,\bullet}(\mc C, E)\to \Gr\Th_{\bullet,\bullet}(\mc C, E)$ lifts to a map in $\Cat_{/\mc C}$. Moreover, since the twist--horizontal factorization system is degenerate on horizontal and vertical fragments, it also lifts to a morphism under $\ascat\Th_{\bullet,0}(\mc C, E)$ and $\ascat\Th_{0,\bullet}(\mc C, E)$. We show that any such morphism must be an equivalence.\par
    Let $f$ be any such endomorphism of $\Gr\Th_{\bullet,\bullet}(\mc C, E)$. Under the equivalence $\Gr\Th_{\bullet,\bullet}(\mc C, E)\simeq\TwC$ of \Cref{cor:GrdescriptionofTw}, $\Gr\Th_{\bullet,\bullet}(\mc C, E)\to\mc C$ corresponds to the Cartesian projection $\TwC\to \mc C$ and $\mc C\simeq \ascat \Th_{\bullet,0}(\mc C, E)\to \Gr\Th_{\bullet,\bullet}(\mc C, E)$ corresponds to the zero section. Thus, by \Cref{cor:checkonfibers}, it suffices to check that $f$ is an equivalence on fibers over $x\in\mc C$. For this, it will suffice to show that
    \begin{equation*}
        \begin{tikzcd}
        	{\ascat \Th_{0,\bullet}(\mc C, E)} & {\Gr \Th_{\bullet,\bullet}(\mc C, E)} \\
        	{\mc C^\simeq} & {\mc C}
        	\arrow[from=1-1, to=1-2]
        	\arrow[from=1-1, to=2-1]
        	\arrow[from=1-2, to=2-2]
        	\arrow[from=2-1, to=2-2]
        \end{tikzcd}
    \end{equation*}
    is Cartesian. Equivalently, that $\ascat \Th_{0,\bullet}(\mc C, E)\to\mc C^\simeq$ is the Cartesian unstraightening of the restricted functor
    \begin{equation}\label{eq:restrtwfunc}
        \begin{tikzcd}[row sep=0pt]
            (\mc C^\simeq)^\mathrm{op}\subseteq \mc C^\mathrm{op}\arrow{r} & \Cat \\
            \;\;\qquad x\arrow[mapsto]{r} & B\mc K(\Vect(x)).
        \end{tikzcd}
    \end{equation}
    By \Cref{cor:projisconservative}, we have that the square
    \begin{equation*}
        \begin{tikzcd}
        	{\mc T_m(\mc C, E)^\simeq} & {\mc T_m(\mc C, E)} \\
        	{\mc C^\simeq} & {\mc C}
        	\arrow[from=1-1, to=1-2]
        	\arrow["{\mathrm{ev}_{(m,m)}}"', from=1-1, to=2-1]
        	\arrow["\lrcorner"{anchor=center, pos=0.125, rotate=0}, draw=none, from=1-1, to=2-2]
        	\arrow["{\mathrm{ev}_{(m,m)}}", from=1-2, to=2-2]
        	\arrow[from=2-1, to=2-2]
        \end{tikzcd}
    \end{equation*}
    is Cartesian from which it follows that $\mc T_m(\mc C, E)^\simeq\to\mc C^{\simeq}$ is the Cartesian unstraightening of the restricted functor
    \begin{equation*}
        \begin{tikzcd}[row sep=0pt]
            (\mc C^\simeq)^\mathrm{op}\subseteq \mc C^\mathrm{op}\arrow{r} & \Cat \\
            \;\;\qquad x\arrow[mapsto]{r} & S_m(\Vect(x)).
        \end{tikzcd}
    \end{equation*}
    Arguing the same as \Cref{cor:GrdescriptionofTw}, it follows that $\Gr N_\ast(\mc T_\bullet(\mc C, E)^\simeq)\to \mc C^\simeq$ is the unstraightening of \eqref{eq:restrtwfunc}. But
    \begin{equation*}
        N_\ast(\mc T_\bullet(\mc C, E)^\simeq)\simeq \Hor(\Th_{0,\bullet}(\mc C, E))
    \end{equation*}
    under which the result follows.
\end{proof}

Book-keeping everything up to this point, we have the following theorem identifying the universal exchange theorem.

\begin{theorem}\label{thm:mainthmintext}
    There exists an exchange theorem on the pair $(\mc C, E)$ valued in $\TwC$ with $(-)_\ast : \mc C\to \TwC$ the zero section which witnesses $\TwC$ as corepresenting $\Exch_{(\mc C, E)}(-)$.\par
    Moreover, this universal exchange theorem has the following properties:
    \begin{enumerate}
        \item \emph{(Thom twists)} For every $x\in \mc C$, there exists an $\mb E_1$-map
        \begin{equation*}
            \Sigma^{(-)} : \mc K(\Vect(x))\longrightarrow \End_{\TwC}(x)
        \end{equation*}
        parametrized by the partial $K$-theory of vector bundles over $x$ with the following properties:
        \begin{enumerate}
            \item for a class
            \begin{equation*}
                \xi = [\begin{tikzcd}
                    x\arrow{r}{s} & e\arrow[two heads]{r}{p} & x
                \end{tikzcd}]\in \mc K(\Vect(x))
            \end{equation*}
            the endomorphism is given by $\Sigma^\xi = p_\natural s_\ast$
            \item for every $f : x\to y\in \mc C$ and $\xi\in\mc K(\Vect(y))$, there is a canonical equivalence
            \begin{equation*}
                \Sigma^\xi f_\ast \simeq f_\ast\, \Sigma^{f^\ast\xi}.
            \end{equation*}
        \end{enumerate}
        \item \emph{(Tangent bundles)} For every $f : x\to y \in E$, there exists a relative tangent bundle
        \begin{equation*}
            \begin{tikzcd}
                T_f \coloneqq x\arrow{r}{\Delta_f} & x\times_y x\arrow[two heads]{r}{\pr_1} & x \in \Vect(x)
            \end{tikzcd}
        \end{equation*}
        with the following properties:
        \begin{enumerate}
            \item for composable morphisms $f,g\in E$, there is a canonical fiber sequence
            \begin{equation*}
                \begin{tikzcd}
                    T_f\arrow{r} & T_{gf}\arrow[two heads]{r} & f^\ast T_g
                \end{tikzcd}
            \end{equation*}
            \item for every pullback square
            \begin{equation*}
                \begin{tikzcd}
                	x & {y'} \\
                	y & z
                	\arrow["{\bar g}", from=1-1, to=1-2]
                	\arrow["{\bar f}"', from=1-1, to=2-1]
                	\arrow["\lrcorner"{anchor=center, pos=0.125}, draw=none, from=1-1, to=2-2]
                	\arrow["f", from=1-2, to=2-2]
                	\arrow["g"', from=2-1, to=2-2]
                \end{tikzcd}
            \end{equation*}
            in $\mc C$ with $\bar f, f\in E$ there is a canonical equivalence $\bar g^\ast T_f\simeq T_{\bar f}$
        \end{enumerate}
        \item \emph{(Ambidexterity)} There exists a canonical equivalence
        \begin{equation*}
            f_\natural\simeq f_\ast\Sigma^{T_f}
        \end{equation*}
        for every $f \in E$.
    \end{enumerate}
\end{theorem}
\begin{proof}
    We have a composite of equivalences
    \begin{equation*}
        \begin{tikzcd}[column sep=50pt]
            \TwC\arrow{r}[below]{\simeq}[above]{\text{\normalfont\Cref{cor:GrdescriptionofTw}}} & \Gr \Th_{\bullet,\bullet}(\mc C, E)\arrow{r}[below]{\simeq}[above]{\text{\normalfont\Cref{prop:grthsimeqacS}}} & \ascat S_\bullet(\mc C, E)\arrow{r}[below]{\simeq}[above]{\text{\normalfont\Cref{lemma:sbulletisgrpull}}} & \Gr\Pull(\mc C, E)
        \end{tikzcd}
    \end{equation*}
    which shows that $\TwC$ hosts the universal exchange theorem. Moreover, under these equivalences, the map
    \begin{equation*}
        (-)_\ast : \mc C\simeq \Gr \Hor(\Pull(\mc C, E)_{0,\bullet})\longrightarrow \Gr \Pull(\mc C, E)
    \end{equation*}
    gets sent to the zero section $s : \mc C\to\TwC$.\par
    The Thom twists come from inclusions of the fibers $\mathrm{fib}_x (\TwC\to\mc C)\simeq B\mc K(\Vect(x))$ over each $x\in\mc C$. (i)(b) is then a consequence of \Cref{prop:descrofadjendos} which gives
    \begin{equation*}
        \begin{aligned}
            \Sigma^\xi f_\ast &= (\id,\xi)\circ (f, 0) \\
            &\simeq (f, f^\ast \xi) \\
            &\simeq (f, 0)\circ (\id, f^\ast\xi) \\
            &= f_\ast \Sigma^{f^\ast\xi}.
        \end{aligned}
    \end{equation*}
    For (i)(a), given $\xi = [\begin{tikzcd} x\arrow{r}{s} & e\arrow[two heads]{r}{p} & x \end{tikzcd}]$ we have that $\Sigma^\xi$ is the image of $x\to e\to x$ under $\mc T_1(\mc C, E)\to \Fun([1], \TwC)$. Chasing through the equivalences, this gets sent to
    \begin{equation*}
        \begin{tikzcd}
        	x & e \\
        	& x
        	\arrow["s", from=1-1, to=1-2]
        	\arrow["p", from=1-2, to=2-2]
        \end{tikzcd}\in S_1(\mc C, E).
    \end{equation*}
    Factoring this as
    \begin{equation*}
        \begin{tikzcd}
        	x & e & e \\
        	& e & e \\
        	&& x,
        	\arrow["s", from=1-1, to=1-2]
        	\arrow[equals, from=1-2, to=1-3]
        	\arrow[equals, from=1-2, to=2-2]
        	\arrow[equals, from=1-3, to=2-3]
        	\arrow[equals, from=2-2, to=2-3]
        	\arrow["p", from=2-3, to=3-3]
        \end{tikzcd}
    \end{equation*}
    under the equivalence $\ascat S_\bullet(\mc C, E)\simeq \Gr\Pull(\mc C, E)$ this then gets sent to $p_\natural s_\ast$ as claimed.\par
    (ii) is a restatement of \Cref{prop:propsofreltan}.\par
    For ambidexterity, we have by \Cref{lemma:sbulletisgrpull} that $(-)_\natural : E\to \Gr\Pull(\mc C, E)$ corresponds to (associated categories) of the map $N(E)\to S_\bullet(\mc C, E)$ which sends a string of $n$-composable arrows in $E$ to the degenerate diagram $\Cpt^n\xrightarrow{\pr_1} [n]\to E$. The twist--horizontal factorization system then sends
    \begin{equation*}
        \begin{tikzcd}
        	x & x \\
        	& y
        	\arrow[equals, from=1-1, to=1-2]
        	\arrow["f", from=1-2, to=2-2]
        \end{tikzcd}
    \end{equation*}
    to the composite
    \begin{equation*}
        \begin{tikzcd}
        	x & {x\times_y x} & x \\
        	& x & y \\
        	&& y
        	\arrow["{\Delta_f}", from=1-1, to=1-2]
        	\arrow["{\pr_2}", from=1-2, to=1-3]
        	\arrow["{\pr_1}", from=1-2, to=2-2]
        	\arrow["f", from=1-3, to=2-3]
        	\arrow["f"', from=2-2, to=2-3]
        	\arrow[equals, from=2-3, to=3-3]
        \end{tikzcd}
    \end{equation*}
    which is the image of the composite $f_\ast\Sigma^{T_f}$ under $\TwC\simeq \ascat S_\bullet(\mc C, E)$, as claimed.
\end{proof}

As a corollary, we may also identify exactly which exchange theorems may extend to have Thom twists for every virtual vector bundle.

\begin{corollary}\label{cor:exttovirtvbexch}
    Let $\InvTwC$ be the unstraightening of
    \begin{equation*}
        \begin{tikzcd}[row sep=0pt]
            \mc C^\mathrm{op}\arrow{r} & \Ani \\
            x\arrow[mapsto]{r} & BK(\Vect(x)).
        \end{tikzcd}
    \end{equation*}
    Then the natural in $\mc D$ map
    \begin{equation*}
        \map_{\Cat}(\InvTwC, \mc D)\longrightarrow \map_{\Cat}(\TwC,\mc D)\simeq \Exch_{(\mc C, E)}(\mc D)
    \end{equation*}
    induced by precomposition with $\TwC\to\InvTwC$ is a monomorphism onto the connected components of exchange theorems which satisfy the following equivalent conditions:
    \begin{enumerate}
        \item For every $x\in\mc D$ and $\xi\in \mc K(\Vect(x))$, $\Sigma^\xi$ is invertible
        \item For every $x\in \mc D$ and $\begin{tikzcd}
            x\arrow{r}{s} & e\arrow[two heads]{r}{p} & x
        \end{tikzcd}\in \Vect(x)$, $\Sigma^\xi\simeq p_\natural s_\ast$ is invertible.
    \end{enumerate}
\end{corollary}
\begin{proof}
    Condition (i) follows by combining \Cref{prop:locofCG} and \Cref{prop:partialKthygroupcmpl} which says that $\TwC\to\InvTwC$ is a localization at the morphisms $\Sigma^\xi$ for $x\in\mc C$ and $\xi\in\mc K(\Vect(x))$.\par
    It therefore only remains to show that conditions (i) and (ii) are equivalent. Since (i) $\Rightarrow$ (ii), it suffices to show that (ii) $\Rightarrow$ (i). By (the dual of) \cite[Proposition 4.10]{yuan2023integral} we have that $\pi_0 \mc K(\Vect(x))$ is freely generated by the objects of $\Vect(x)$ modulo the relationship $[e_2] = [e_1] + [e_3]$ for every fiber sequence
    \begin{equation*}
        \begin{tikzcd}
            e_1\arrow{r} & e_2\arrow[two heads]{r} & e_3
        \end{tikzcd}
    \end{equation*}
    in $\Vect(x)$. In particular, $\Vect(x)^\simeq\to \mc K(\Vect(x))$ is surjective on $\pi_0$ so (ii) implies (i).
\end{proof}

\section{Poincar\'e duality}\label{sec:poincareduality}

In this section, we apply the universal exchange theorem to canonically enhance every six functor formalism to one which coherently encodes Thom twists and Poincar\'e duality.\par
First we recall the notion of a three and six functor formalism in the sense of Mann \cite{mann2022p}.

\begin{definition}
    Let $(\mc C, E)$ be a geometric setup. A \emph{3-functor formalism} is a lax symmetric monoidal functor $\mc D : \Span(\mc C, E)\to\Cat$ where $\Cat$ is given the Cartesian monoidal structure.
\end{definition}

Given a 3-functor formalism $\mc D : \Span(\mc C, E)\to \Cat$, one obtains three functors:
\begin{itemize}
    \item Restriction to $\mc C^\mathrm{op}\subseteq \Span(\mc C, E)$ yields for each $f : x\to y\in\mc C$ a pullback functor $f^\ast : \mc D(y)\to\mc D(x)$
    \item Restriction to $E\subseteq \Span(\mc C, E)$ yields for each $f : x\to y\in E$ an exceptional pushforward functor $f_! : \mc D(x)\to\mc D(y)$
    \item Lax monoidality of $\mc D$ equips each $\mc D(x)$ with a symmetric monoidal structure with tensor product given by
    \begin{equation*}
        \begin{tikzcd}
            (-)\otimes (-) : \mc D(x)\times \mc D(x)\arrow{r} & \mc D(x\times x)\arrow{r}{\Delta^\ast} & \mc D(x).
        \end{tikzcd}
    \end{equation*}
\end{itemize}

Positing additional adjoints, one arrives at the definition of a six functor formalism.

\begin{definition}
    A \emph{six functor formalism} is a 3-functor formalism $\mc D : \Span(\mc C, E)\to\Cat$ such that
    \begin{enumerate}
        \item Every $f^\ast$ has a right adjoint $f_\ast$
        \item Every $f_!$ has a right adjoint $f^!$
        \item Each $\mc D(x)$ is closed symmetric monoidal, i.e.\ possesses an internal hom $\underline{\Hom}$.
    \end{enumerate}
\end{definition}

In this abstract setting, Poincar\'e duality attempts to relate $f^!$ and $f^\ast$ up to a twist, i.e.\ to give an equivalence
\begin{equation}\label{eq:poincaredualitystatement}
    f^!(-)\simeq f^\ast(-)\otimes \omega_f
\end{equation}
for some $\otimes$-invertible object $\omega_f$. Evaluating at the monoidal unit, one necessarily has $\omega_f\simeq f^!(1)$. By imposing compatibility with base change as well as canonicity of the equivalence \eqref{eq:poincaredualitystatement}, one arrives at the notion of \emph{cohomologically smooth} morphisms originally defined by Scholze \cite{scholze2017etale}.

\begin{definition}[{\cite[Definition 5.1]{scholze2022six}}]\label{def:weakcohomsmooth}
    Let $\mc D : \Span(\mc C, E)\to\Cat$ be a 3-functor formalism and let $f : x\to y\in E$. We say that $f$ is \emph{weakly} $(\mc D-)$\emph{cohomologically smooth} if
    \begin{enumerate}
        \item the right adjoint $f^!$ to $f_!$ exists and the natural transformation
        \begin{equation*}
            f^!(1_y)\otimes f^\ast(-)\longrightarrow f^!(-)
        \end{equation*}
        coming from transposing
        \begin{equation*}
            f_!(f^!(1_y)\otimes f^\ast(-))\simeq f_!f^!(1_y)\otimes (-)\longrightarrow (-)
        \end{equation*}
        is an equivalence
        \item for any $g : y'\to y$, (i) also holds for the base change $\bar f : x\times_y y'\to y'$ and the natural map
        \begin{equation*}
            \bar g^\ast f^!(1_y)\longrightarrow \bar f^!(1_{y'})
        \end{equation*}
        is an equivalence where $\bar g$ is the base change of $g$.
    \end{enumerate}
    We say that $f$ is \emph{cohomologically smooth} if it is weakly cohomologically smooth and in addition $f^!(1_y)$ is $\otimes$-invertible.
\end{definition}

\begin{remark}
    Weak cohomological smoothness is preserved under base change by definition. Cohomological smoothness is as well by \Cref{def:weakcohomsmooth}(ii) since $\bar g^\ast$ is symmetric monoidal, hence preserves $\otimes$-invertible objects.
\end{remark}

\subsection{Labeled span categories}

Let $(\mc C, E)$ be a geometric setup. Assume we are also given a functor
\begin{equation*}
    \begin{tikzcd}
        G : \mc C^\mathrm{op}\arrow{r} & \Alg_{\mb E_1}(\Ani).
    \end{tikzcd}
\end{equation*}
In this section, we construct a labeled span category $\Span^G(\mc C, E)$ whose
\begin{itemize}
    \item objects are the objects of $\mc C$
    \item morphisms from $X$ to $Y$ are spans
    \begin{equation}\label{eq:labeledspanmorphism}
        \begin{tikzcd}
            & (Z, \xi)\arrow{dl}[above left]{f}\arrow{dr}[above right]{g} \\
            X && Y
        \end{tikzcd}
    \end{equation}
    with $g\in E$ and $\xi\in G(Z)$
    \item composition of morphisms is given by
    \begin{equation*}
        \begin{tikzcd}[column sep=10pt, row sep=5pt]
        	& {(W,\xi)} && {(V,\eta)} &&&& {(W\times_Y V, \pi_V^\ast\eta\cdot\pi_W^\ast\xi)} & \\
        	&&&&& \simeq \\
        	X && Y && Z && X && Z
        	\arrow[from=1-2, to=3-1]
        	\arrow[from=1-2, to=3-3]
        	\arrow[from=1-4, to=3-3]
        	\arrow[from=1-4, to=3-5]
        	\arrow[from=1-8, to=3-7]
        	\arrow[from=1-8, to=3-9]
        \end{tikzcd}
    \end{equation*}
    i.e.\ by composing underlying spans and taking external products of the elements of $G$.
\end{itemize}

\begin{remark}\label{rem:labspansofelmanto}
    A more general procedure for producing such labeled spans has already been developed in \cite{elmanto2021motivic} with further extensions in \cite{elmanto2020modules}. One may realize our labeled span categories as a special instance of the labeled span categories of Elmanto \emph{et al.} However, for our applications, it will be easier to work from scratch with a more concrete model for this special case.
\end{remark}

\subsubsection{Construction}

First we recall the construction of classical, unlabeled span categories. Let $\Sigma_n$ denote the poset
\begin{equation*}
    \Sigma_n = \{(i,j) : 0\le i\le j\le n\}
\end{equation*}
with the relation that $(i,j)\le (k,\ell)$ if and only if $i\le k$ and $j\ge \ell$. For example, $\Sigma_2$ is represented by the diagram
\begin{equation*}
    \begin{tikzcd}[column sep=5pt]
        && {(0\le 2)} && \\
        & {(0\le 1)} && {(1\le 2)} \\
        {(0\le 0)} && {(1\le 1)} && {(2\le 2)}.
        \arrow[from=1-3, to=2-2]
        \arrow[from=1-3, to=2-4]
        \arrow[from=2-2, to=3-1]
        \arrow[from=2-2, to=3-3]
        \arrow[from=2-4, to=3-3]
        \arrow[from=2-4, to=3-5]
    \end{tikzcd}
\end{equation*}
The classical span category $\Span(\mc C, E)$ is built by declaring the space of $n$-simplices to be a certain subspace of diagrams $\Sigma_n\to\mc C$. To organize this, we make use of the category of \emph{triples} $\mathrm{Trip}$ (see \cite[Section 5]{barwick2017spectral}), which is defined to be category whose:
\begin{itemize}
    \item objects are triples $(\mc C, M, N)$ where $\mc C$ is a category and $M$, $N$ are wide subcategories which have and are preserved by pullbacks along each other
    \item morphisms $(\mc C, M, N)\to (\mc C', M', N')$ are functors $f : \mc C\to\mc C'$ satisfying $f(M)\subseteq M'$, $f(N)\subseteq N'$ and which preserve pullbacks of morphisms in $M$ along morphisms in $N$.
\end{itemize}
The posets $\Sigma_n$ upgrade to a cosimplicial triple
\begin{equation*}
    \begin{tikzcd}[row sep=0pt]
        \Delta\arrow{r} & \mathrm{Trip} \\
        {[n]}\arrow[mapsto]{r} & (\Sigma_n, \Sigma_n^L, \Sigma_n^R)
    \end{tikzcd}
\end{equation*}
where $\Sigma_n^L$ and $\Sigma_n^R$ are the wide subcategories consisting of left facing and right facing arrows in $\Sigma_n$, respectively. Applying the nerve construction of \Cref{subsubsec:nerverealization} yields a functor
\begin{equation*}
    \begin{tikzcd}
        N(\Span(-)) : \mathrm{Trip}\arrow{r} & \PSh(\Delta).
    \end{tikzcd}
\end{equation*}

\begin{proposition}[{\cite[Proposition 5.9]{barwick2017spectral}}]
    For every triple $(\mc C, M, N)\in\mathrm{Trip}$, $N(\Span(\mc C, M, N))$ is a complete Segal space.
\end{proposition}

The associated category of $N(\Span(\mc C, M, N))$, denoted $\Span(\mc C, M, N)$, is then the ordinary span category associated to the triple $(\mc C, M, N)$.\par
To motivate our construction of the labeled span category, recall the construction $\mc C^G$ of \Cref{subsec:adjendos} which is the Cartesian unstraightening of
\begin{equation*}
    \begin{tikzcd}
        \mc C^\mathrm{op}\arrow{r}{G} & \Alg_{\mb E_1}(\Ani)\arrow{r}{B} & \Cat.
    \end{tikzcd}
\end{equation*}
\Cref{prop:descrofadjendos} tells us that this is the category whose
\begin{itemize}
    \item objects are the objects of $\mc C$
    \item morphisms $x\to y$ are pairs $(f, \xi)$ with $f : x\to y\in\mc C$ and $\xi\in G(x)$
    \item composition is given by $(f,\xi)\circ (g, \eta)\simeq (fg, g^\ast \xi\cdot \eta)$.
\end{itemize}
Additionally, $\mc C^G$ comes equipped with an identity section map $s : \mc C\to \mc C^G$ which sends a morphism $f$ to $(f, 1)$. In a labeled span \eqref{eq:labeledspanmorphism}, we may pair the element $\xi\in G(Z)$ with the morphism $f$ and represent the span as the diagram
\begin{equation*}
    \begin{tikzcd}
        & s(Z)\arrow{dl}[above left]{(f,\, \xi)}\arrow{dr}[above right]{s(g)} \\
        s(X) && s(Y)
    \end{tikzcd}
\end{equation*}
in $\mc C^G$. Given a cospan
\begin{equation*}
    \begin{tikzcd}
        s(W)\arrow{dr}[below left]{s(g)} && s(V)\arrow{dl}[below right]{(f,\, \eta)} \\
        & s(Y)
    \end{tikzcd}
\end{equation*}
there is an essentially unique extension of this diagram to a square of the form
\begin{equation*}
    \begin{tikzcd}
    	& s(T) & \\
    	s(W) && s(V) \\
    	& s(Y)
    	\arrow["{(\bar f,\, \bar\eta)}"', from=1-2, to=2-1]
    	\arrow["{s(\bar g)}", from=1-2, to=2-3]
    	\arrow["{s(g)}"', from=2-1, to=3-2]
    	\arrow["{(f,\,\eta)}", from=2-3, to=3-2]
    \end{tikzcd}
\end{equation*}
in $\mc C^G$ such that the pushforward along the projection $\pi : \mc C^G\to\mc C$ is Cartesian. Indeed, $\bar f$ and $\bar g$ are determined as the respective pullbacks of $f$ and $g$ in $\mc C$ and the composition rule forces $\bar\eta \simeq \bar g^\ast\eta$. Using this to compose spans we arrive at the composition rule
\begin{equation*}
    \begin{tikzcd}[column sep=12pt]
    	&&&&&&&& {s(W\times_Y V)} && \\
    	& s(W) && s(V) && \simeq && s(W) && s(V) \\
    	s(X) && s(Y) && s(Z) && s(X) && s(Y) && s(Z) \\
    	\\
    	&&&&& \simeq && {s(W\times_Y V)} \\
    	&&&&&& s(X) && s(Z)
    	\arrow["{(\bar{f_2},\, \bar{g_1}^\ast\eta)}"', from=1-9, to=2-8]
    	\arrow["{s(\bar{g_1})}", from=1-9, to=2-10]
    	\arrow["{(f_1,\,\xi)}"', from=2-2, to=3-1]
    	\arrow["{s(g_1)}"', from=2-2, to=3-3]
    	\arrow["{(f_2,\,\eta)}", from=2-4, to=3-3]
    	\arrow["{s(g_2)}", from=2-4, to=3-5]
    	\arrow["{(f_1,\,\xi)}"', from=2-8, to=3-7]
    	\arrow["{s(g_1)}"', from=2-8, to=3-9]
    	\arrow["{(f_2,\,\eta)}", from=2-10, to=3-9]
    	\arrow["{s(g_2)}", from=2-10, to=3-11]
    	\arrow["{(\bar{f_2}f_1,\, \bar{g_1}^\ast\eta \cdot \bar{f_2}^\ast\xi)}"', from=5-8, to=6-7]
    	\arrow["{s(g_2\bar{g_1})}", from=5-8, to=6-9]
    \end{tikzcd}
\end{equation*}
which is compatible with the composition rule outlined at the beginning of the section.\par
With this in mind we make the following definition.

\begin{definition}
    Define $N_\mathrm{unc}(\Span^G(\mc C, E))$ to be the simplicial space which at level $n$ is given by the pullback
    \begin{equation*}
        \begin{tikzcd}
            N_\mathrm{unc}(\Span^G(\mc C, E))_n \arrow{r}\arrow{d} & \map_{\Fun(\Lambda_2^2, \Cat)}(\Sigma_n^L\hookrightarrow\Sigma_n\hookleftarrow\Sigma_n^R, \mc C^G\; {=\joinrel=}\; \mc C^G\xleftarrow{s}\, E)\arrow{d}{\pi_\ast} \\
            N(\Span(\mc C, E))_n\arrow{r} & \map_{\Cat}(\Sigma_n,\mc C)
        \end{tikzcd}
    \end{equation*}
    where the right vertical map sends a diagram $D : \Sigma_n\to \mc C^G$ to $\Sigma_n\xrightarrow{D} \mc C^G\xrightarrow{\pi}\mc C$ and the bottom horizontal map is the forgetful map.
\end{definition}

\begin{remark}
    Because $N(\Span(\mc C, E))_n\to \map_{\Cat}(\Sigma_n,\mc C)$ is a monomorphism, the map
    \begin{equation*}
        N_\mathrm{unc}(\Span^G(\mc C, E))_n \longrightarrow \map_{\Fun(\Lambda_2^2, \Cat)}(\Sigma_n^L\hookrightarrow\Sigma_n\hookleftarrow\Sigma_n^R, \mc C^G\; {=\joinrel=}\; \mc C^G\xleftarrow{s}\, E)
    \end{equation*}
    is also a monomorphism, namely onto the subspace of components consisting of diagrams $D$ such that the composite $\Sigma_n\xrightarrow{D}\mc C^G\xrightarrow{\pi}\mc C$ sends every square in $\Sigma_n$ of the form
    \begin{equation*}
        \begin{tikzcd}[column sep=5pt, row sep=10pt]
        	& {(i\le \ell)} & \\
        	{(i\le j)} && {(k\le\ell)} \\
        	& {(k\le j)}
        	\arrow[from=1-2, to=2-1]
        	\arrow[from=1-2, to=2-3]
        	\arrow[from=2-1, to=3-2]
        	\arrow[from=2-3, to=3-2]
        \end{tikzcd}
    \end{equation*}
    to a Cartesian square.
\end{remark}

We may now define our labeled span categories.

\begin{definition}\label{def:labeledspancat}
    The \emph{labeled span category} $\Span^G(\mc C, E)$ is the associated category of the simplicial space $N_\mathrm{unc}(\Span^G(\mc C, E))$.
\end{definition}

The definition also comes with an inherent functoriality. Let $\mathrm{LabPair}$ denote the category whose
\begin{itemize}
    \item objects are triples $(\mc C, E, G)$ with $(\mc C, E)$ a geometric setup and $G : \mc C^\mathrm{op}\to \Alg_{\mb E_1}(\Ani)$ a functor
    \item morphisms $(\mc C, E, G)\to (\mc C', E', G')$ are pairs $(F, \eta)$ where $F : \mc C\to \mc C'$ is a functor satisfying $F(E)\subseteq E'$ which preserves pullbacks along morphisms in $E$, and $\eta : G\Rightarrow G'\circ F^\mathrm{op}$.
\end{itemize}
Explicitly, let $\mathrm{Pair}$ denote the full subcategory of $\mathrm{Trip}$ spanned by triples of the form $(\mc C, \mc C, E)$. Then $\mathrm{LabPair}$ is the total space of the Cartesian unstraightening of the composite
\begin{equation*}
    \begin{tikzcd}[row sep=0pt]
        \mathrm{Pair}^\mathrm{op}\arrow{r} & \Cat^\mathrm{op}\arrow{r}{\Fun((-)^\mathrm{op},\; \Alg_{\mb E_1}(\Ani))} &[65pt] \Cat. \\
        (\mc C, E)\arrow[mapsto]{r} & \mc C 
    \end{tikzcd}
\end{equation*}
Functoriality of $\Span^G(\mc C, E)$ in $(\mc C, E, G)\in\mathrm{LabPair}$ then follows from the definition of $N_\mathrm{unc}(\Span^G(\mc C, E))$ and the functoriality of $(\mc C, G)\mapsto \mc C^G$.\par

The remainder of this section is devoted to proving the following two propositions which connect $\Span^G(\mc C, E)$ to the informal description at the start of the section.

\begin{proposition}\label{prop:descr1ofSpanG}
    \begin{enumerate}
        \item There is a natural in $(\mc C, E)$ equivalence $\Span^{\ast}(\mc C, E)\simeq \Span(\mc C, E)$ where $\ast : \mc C^\mathrm{op}\to \Alg_{\mb E_1}(\Ani)$ is the constant functor with value $\ast$.
        \item There is a natural in $(\mc C, G)$ equivalence $\Span^G(\mc C, \mc C^\simeq)\simeq (\mc C^G)^\mathrm{op}$.
    \end{enumerate}
\end{proposition}

\begin{proposition}\label{prop:descr2ofSpanG}
    Fix $(\mc C, E, G)\in\mathrm{LabPair}$. Let
    \begin{itemize}
        \item $\iota : \Span(\mc C, E)\to \Span^G(\mc C, E)$ denote the map induced by $(\mc C, E, \ast)\to (\mc C, E, G)$ where $\ast\Rightarrow G$ is the identity section
        \item $\pi : \Span^G(\mc C, E)\to \Span(\mc C, E)$ the map induced by $(\mc C, E, G)\to (\mc C, E, \ast)$.
    \end{itemize}
    Then,
    \begin{enumerate}
        \item $\iota$ is essentially surjective
        \item if $\mc G$ denotes the Cartesian unstraightening of
        \begin{equation*}
            \begin{tikzcd}
                (\mc C^\simeq)^\mathrm{op} \subseteq \mc C^\mathrm{op}\arrow{r}{G} & \Ani
            \end{tikzcd}
        \end{equation*}
        then there exist equivalences
        \begin{equation*}
            \map_{\Span^G(\mc C, E)}(\iota X, \iota Y)\simeq \mc G\times_{\mc C^\simeq} \map_{\Span(\mc C, E)}(X, Y)
        \end{equation*}
        where the projection $\map_{\Span^G(\mc C, E)}(\iota X, \iota Y)\to \map_{\Span(\mc C, E)}(X, Y)$ is given by pushforward along $\pi$ and $\map_{\Span(\mc C, E)}(X, Y)\to \mc C^\simeq$ is given by projection to the apex of the span.
    \end{enumerate}
    Additionally, representing a point $(\xi\in G(Z), X\leftarrow Z\rightarrow Y)\in \mc G\times_{\mc C^\simeq}\map_{\Span(\mc C, E)}(X, Y)$ by
    \begin{equation*}
        \begin{tikzcd}
            & (Z,\xi)\arrow{dl}\arrow{dr} \\
            X && Y,
        \end{tikzcd}
    \end{equation*}
    composition is given by
    \begin{equation*}
        \begin{tikzcd}
            & {(W,\xi)} && {(V,\eta)} & {} && {(W\times_Y V, \eta\boxtimes\xi)} & \\
            X && Y && Z & X && Z.
            \arrow[from=1-2, to=2-1]
            \arrow[from=1-2, to=2-3]
            \arrow[from=1-4, to=2-3]
            \arrow[from=1-4, to=2-5]
            \arrow["\simeq"{description}, draw=none, from=1-5, to=2-6]
            \arrow[from=1-7, to=2-6]
            \arrow[from=1-7, to=2-8]
        \end{tikzcd}
    \end{equation*}
\end{proposition}

\subsubsection{Bisimplicial description}

To connect these labeled span categories to the informal description at the start of the section, as well as for later use, we give a bisimplicial gluing description of $\Span^G(\mc C, E)$.

\begin{definition}
    \begin{enumerate}
        \item Denote by $(-)^\mathrm{op} : \Delta\to \Delta$ the automorphism of $\Delta$ given by
        \begin{equation*}
            \begin{tikzcd}[row sep=0pt]
                (-)^\mathrm{op} : \Delta\arrow{r} & \Delta. \\
                \quad\;\;\qquad {[n]}\arrow[mapsto]{r} & {[n]^\mathrm{op}}
            \end{tikzcd}
        \end{equation*}
        \item For an $n$-simplicial space $X$ and a subset $I\subseteq \{1,2,\dots, n\}$, denote by $X^{I\mathrm{-op}}$ the simplicial space
        \begin{equation*}
            \begin{tikzcd}
                (\Delta^{\times n})^\mathrm{op}\arrow{r}{\prod_{1\le i\le n} \varepsilon_i} &[35pt] (\Delta^{\times n})^\mathrm{op}\arrow{r}{X} & \Ani
            \end{tikzcd}
        \end{equation*}
        where $\varepsilon_i = \id$ if $i\not\in I$ and $\varepsilon_i = (-)^\mathrm{op}$ if $i\in I$. We will also write $X^\mathrm{op}$ for $X^{\{1,2,\dots, n\}\mathrm{-op}}$.
    \end{enumerate}
\end{definition}

We may now define our bisimplicial space of interest.

\begin{definition}
    Let $\Pull^G(\mc C, E)$ be the bisimplicial space given as the pullback
    \begin{equation*}
        \begin{tikzcd}
            \Pull^G(\mc C, E)\arrow{r}\arrow{d} & \Sq(\mc C^G\;{=\joinrel=}\; \mc C^G \,\xleftarrow{s}\,E)\arrow{d} \\
            \Pull(\mc C, E)\arrow{r} & \Sq(\mc C)
        \end{tikzcd}
    \end{equation*}
    where the right vertical arrow is the composite $\Sq(\mc C^G\;{=\joinrel=}\; \mc C^G \,\xleftarrow{s}\,E)\to \Sq(\mc C^G)\xrightarrow{\pi_\ast}\Sq(\mc C)$.
\end{definition}

\begin{remark}
    Here we are overloading notation to have $\Sq$ represent the functor
    \begin{equation*}
        \Sq : \Fun(\Lambda_2^2, \Cat)\longrightarrow \PSh(\Delta^{\times 2})
    \end{equation*}
    given by the nerve construction of \Cref{subsubsec:nerverealization} applied to the cobisimplicial cospan of categories
    \begin{equation*}
        \begin{tikzcd}[row sep=0pt]
            \Delta^{\times 2}\arrow{r} & \mathrm{OFS}\subseteq \Fun(\Lambda_2^2, \Cat). \\
            ([n], [m])\arrow[mapsto]{r} & {[n]\bar\times [m]}
        \end{tikzcd}
    \end{equation*}
\end{remark}

Since $\Pull(\mc C, E)\to \Sq(\mc C)$ is a monomorphism, the map $\Pull^G(\mc C, E)\to \Sq(\mc C^G\;{=\joinrel=}\; \mc C^G \,\xleftarrow{s}\,E)$ is a monomorphism as well, namely onto the subspace of components consisting of diagrams where the underlying commutative squares in $\mc C$ consist of Cartesian squares.\par
At this point, we remark that $\Sigma_n$ with its left and right facing arrows forms a factorization system $(\Sigma_n, \Sigma_n^L, \Sigma_n^R)$. Moreover, under this factorization system we have natural in $n$ equivalences
\begin{equation*}
    (\CCpt^n)^{2\mathrm{-op}}\simeq \Fact \Sigma_n.
\end{equation*}
From this we obtain a diagram
\begin{equation*}
    \begin{tikzcd}[column sep=-12pt]
    	& {\cnr \Pull^G(\mc C, E)^{2\mathrm{-op}}} && {\cnr \Sq(\mc C^G\;{=\joinrel=}\; \mc C^G \,\xleftarrow{s}\,E)^{2\mathrm{-op}}} \\
    	{N_\mathrm{unc}(\Span^G(\mc C, E))} && {\map_{\Fun(\Lambda_2^2, \Cat)}(\Sigma^L_\bullet\hookrightarrow\Sigma_\bullet\hookleftarrow\Sigma^R_\bullet, \mc C^G\;{=\joinrel=}\;\mc C^G\,\xleftarrow{s} E)} \\
    	& {\cnr \Pull(\mc C, E)^{2\mathrm{-op}}} && {\cnr \Sq(\mc C)^{2\mathrm{-op}}} \\
    	{N(\Span(\mc C, E))_n} && {\map_{\Cat}(\Sigma_\bullet,\mc C)}
    	\arrow[from=1-2, to=1-4]
    	\arrow[from=1-2, to=3-2]
    	\arrow["\lrcorner"{anchor=center, pos=0.125, rotate=0}, draw=none, from=1-2, to=3-4]
    	\arrow[from=1-4, to=3-4]
    	\arrow[from=2-1, to=1-2]
    	\arrow[from=2-1, to=2-3]
    	\arrow[from=2-1, to=4-1]
    	\arrow["\lrcorner"{anchor=center, pos=0.125, rotate=0}, draw=none, from=2-1, to=4-3]
    	\arrow["\simeq"{description}, from=2-3, to=1-4]
    	\arrow[from=2-3, to=4-3]
    	\arrow[from=3-2, to=3-4]
    	\arrow["\simeq"{description}, from=4-1, to=3-2]
    	\arrow[from=4-1, to=4-3]
    	\arrow["\simeq"{description}, from=4-3, to=3-4]
    \end{tikzcd}
\end{equation*}
where we have used that
\begin{itemize}
    \item $(-)^{2\mathrm{-op}}$ is a self-inverse automorphism of $\PSh(\Delta^{\times 2})$
    \item the functor
    \begin{equation*}
        \Sq : \Fun(\Lambda_2^2, \Cat)\longrightarrow \PSh(\Delta^{\times 2})
    \end{equation*}
    has left adjoint given by sending $X_{\bullet,\bullet}$ to $\ascat X_{\bullet, 0}\to \Gr X\leftarrow \ascat X_{0,\bullet}$
    \item \Cref{thm:gluingoffactsystems} to deduce $\Gr\Fact \Sigma_n\simeq \Sigma_n$.
\end{itemize}
It follows therefore that
\begin{equation*}
    N_\mathrm{unc}(\Span^G(\mc C, E))\simeq \cnr \Pull^G(\mc C, E)^{2\mathrm{-op}}
\end{equation*}
and hence by corner approximation that
\begin{equation*}
    \Span^G(\mc C, E)\simeq \Gr \Pull^G(\mc C, E)^{2\mathrm{-op}}
\end{equation*}
yielding the desired bisimplicial model for $\Span^G(\mc C, E)$. We remark as well that this equivalence is natural in $(\mc C, E, G)\in \mathrm{LabPair}$ by naturality of the construction $(\mc C, E, G)\mapsto \Pull^G(\mc C, E)$ and corner approximation.\par

The following is a rephrasing of \cite[Lemma 3.12]{juran2026orthogonal}. The cited result is proven under the definition of ``double category'' being a bisimplicial space $X$ such that $X(n, -)$ and $X(-, n)$ are \emph{complete} Segal spaces for every $n$, however the proof only makes use of the assumption that each is a Segal space.

\begin{proposition}\label{prop:condcnrsegal}
    Let $X$ be a bisimplicial space such that for every $n\ge 0$ both $X(n, -)$ and $X(-, n)$ are Segal spaces. Then the following are equivalent:
    \begin{enumerate}
        \item the square
        \begin{equation*}
            \begin{tikzcd}[row sep=30pt, column sep=40pt]
                X(1, 1)\arrow{r}{(\id,\, \langle 0\rangle)^\ast}\arrow{d}[left]{(\langle 1\rangle,\,\id)^\ast} & X(1, 0)\arrow{d}[right]{(\langle 1\rangle,\,\id)^\ast} \\
                X(0,1)\arrow{r}[below]{(\id,\, \langle 0\rangle)^\ast} & X(0,0)
            \end{tikzcd}
        \end{equation*}
        is Cartesian
        \item $\cnr X$ is a Segal space.
    \end{enumerate}
\end{proposition}

We then have the following two results.

\begin{lemma}\label{lemma:equivcondforsegalforD}
    Let $\mb D = \Pull^G(\mc C, E)^{2\mathrm{-op}}$. Then for every $n\ge 0$, $\mb D(n, -)$ and $\mb D(-, n)$ are Segal spaces and the equivalent conditions of \Cref{prop:condcnrsegal} hold.
\end{lemma}
\begin{proof}
    The statement that each $\mb D(n, -)$ and $\mb D(-, n)$ are Segal spaces follows from the analogous statement for $\Pull^G(\mc C, E)$ which follows from the definition of $\Pull^G(\mc C, E)$ as a pullback of three bisimplicial spaces having this property.\par
    We show condition (i) of \Cref{prop:condcnrsegal}. Let $\mc G$ be the unstraightening of the restricted functor
    \begin{equation*}
        \begin{tikzcd}[row sep=0pt]
            (\mc C^\simeq)^\mathrm{op} \subseteq \mc C^\mathrm{op}\arrow{r}{G} & \Ani
        \end{tikzcd}
    \end{equation*}
    which comes with a projection $\pi : \mc G\to \mc C^\simeq$. We have a commutative square
    \begin{equation}\label{eq:formulaforD11}
        \begin{tikzcd}
            \mb D(1,1)\arrow{r}\arrow{d} & \Pull(\mc C, E)_{1,1}\arrow{d} \\
            \mc G \arrow{r}{\pi} & \mc C^\simeq
        \end{tikzcd}
    \end{equation}
    where
    \begin{itemize}
        \item $\mb D(1,1)\to \Pull(\mc C, E)_{1,1}$ is given by pushforward along $\mc C^G\to \mc C$, i.e.\
        \begin{equation*}
            \mb D(1,1)\ni \begin{tikzcd}
                s(x) & {s(y')} && x & {y'} \\
            	s(y) & s(z) && y & z
            	\arrow["{(\bar g, \bar \xi)}", from=1-1, to=1-2]
            	\arrow["{(\bar f, 0)}"', from=1-1, to=2-1]
            	\arrow[""{name=0, anchor=center, inner sep=0}, "{(f, 0)}", from=1-2, to=2-2]
            	\arrow["{\bar g}", from=1-4, to=1-5]
            	\arrow[""{name=1, anchor=center, inner sep=0}, "{\bar f}"', from=1-4, to=2-4]
            	\arrow["\lrcorner"{anchor=center, pos=0.125}, draw=none, from=1-4, to=2-5]
            	\arrow["f"', from=1-5, to=2-5]
            	\arrow["{(g, \xi)}"', from=2-1, to=2-2]
            	\arrow["g", from=2-4, to=2-5]
            	\arrow[shorten >=15pt, shorten <=25pt, maps to, from=0, to=1]
            \end{tikzcd}\in \Pull(\mc C, E)_{1,1}
        \end{equation*}
        \item $\mb D(1,1)\to\mc G$ is given by
        \begin{equation*}
            \mb D(1,1)\ni \begin{tikzcd}
                s(x) & {s(y')} \\
            	s(y) & s(z)
            	\arrow["{(\bar g, \bar \xi)}", from=1-1, to=1-2]
            	\arrow["{(\bar f, 0)}"', from=1-1, to=2-1]
            	\arrow["{(f, 0)}", from=1-2, to=2-2]
            	\arrow["{(g, \xi)}"', from=2-1, to=2-2]
            \end{tikzcd}\mapsto (y, \xi) \in \mc G
        \end{equation*}
        \item $\Pull(\mc C, E)_{1,1}\to\mc C^\simeq$ is evaluation at the bottom left corner.
    \end{itemize}
    By \Cref{prop:descrofadjendos}, given a square
    \begin{equation*}
        \begin{tikzcd}
            s(x) & {s(y')} \\
            s(y) & s(z)
            \arrow["{(\bar g, \bar \xi)}", from=1-1, to=1-2]
            \arrow["{(\bar f, 0)}"', from=1-1, to=2-1]
            \arrow["{(f, 0)}", from=1-2, to=2-2]
            \arrow["{(g, \xi)}"', from=2-1, to=2-2]
        \end{tikzcd}
    \end{equation*}
    belonging to $\mb D(1,1)$, $\bar \xi$ is essentially uniquely determined as $\bar\xi\simeq \bar f^\ast\xi$. From this one deduces that \eqref{eq:formulaforD11} is Cartesian.\par
    Next, we have a commutative cube
    \begin{equation*}
        \begin{tikzcd}
        	& {\Pull(\mc C, E)_{1,1}} && {N(\mc C)_1} \\
        	{\mb D(1,1)} && {\mb D(0,1)} \\
        	& {N(E)_1} && {\mc C^\simeq} \\
        	{\mb D(1,0)} && {\mb D(0,0)}
        	\arrow[from=1-2, to=1-4]
        	\arrow[from=1-2, to=3-2]
        	\arrow["\lrcorner"{anchor=center, pos=0.125}, draw=none, from=1-2, to=3-4]
        	\arrow[from=1-4, to=3-4]
        	\arrow[from=2-1, to=1-2]
        	\arrow[from=2-1, to=2-3]
        	\arrow[from=2-1, to=4-1]
        	\arrow[from=2-3, to=1-4]
        	\arrow[from=2-3, to=4-3]
        	\arrow[from=3-2, to=3-4]
        	\arrow["\simeq", from=4-1, to=3-2]
        	\arrow["\lrcorner"{anchor=center, pos=0.125, rotate=90}, draw=none, from=4-1, to=3-4]
        	\arrow[from=4-1, to=4-3]
        	\arrow["\simeq"', from=4-3, to=3-4]
        \end{tikzcd}
    \end{equation*}
    where the bottom face is a pullback since two of the parallel arrows are equivalences, and the back face is a pullback by uniqueness of pullback squares. Thus by pasting laws for pullback squares, the front face is Cartesian (as required) if and only if the top face is Cartesian.\par
    However, by definition of $\mb D$, we have that
    \begin{equation*}
        \begin{aligned}
            \mb D(0,1) &\simeq N(\mc C^G)_1\times_{(\mc C^G)^\simeq\times (\mc C^G)^\simeq} (\mc C^\simeq \times \mc C^\simeq) \\
            &\simeq \mc G\times_{\mc C^\simeq} N(\mc C)_1
        \end{aligned}
    \end{equation*}
    where the last line uses \Cref{prop:globaldescrofadjendos}. Thus the upper face being a pullback is the statement that the morphism
    \begin{equation*}
        \begin{aligned}
            \mb D(1,1) &\longrightarrow \mb D(0,1)\times_{N(\mc C)_1} \Pull(\mc C, E)_{1,1} \\
            &\simeq (\mc G\times_{\mc C^\simeq} N(\mc C)_1)\times_{N(\mc C)_1} \Pull(\mc C, E)_{1,1} \\
            &\simeq \mc G\times_{\mc C^\simeq} \Pull(\mc C, E)_{1,1}
        \end{aligned}
    \end{equation*}
    is an equivalence, i.e.\ that \eqref{eq:formulaforD11} is Cartesian.
\end{proof}

\begin{corollary}\label{cor:NuncSpanGsegal}
    $N_\mathrm{unc}(\Span^G(\mc C, E))$ is a Segal space.
\end{corollary}

\begin{remark}
    This may be compared to \cite[Theorem 4.1.23]{elmanto2021motivic} which gives the analogous result for the more general labeled spans of Elmanto \emph{et al.} (see \Cref{rem:labspansofelmanto}).
\end{remark}

\begin{lemma}\label{lemma:ascatofD0}
    Let $\mb D = \Pull^G(\mc C, E)^{2\mathrm{-op}}$. Then there are natural in $n$ pullback squares
    \begin{equation*}
        \begin{tikzcd}
        	{\mb D(0,n)} & {(\mc C^{\simeq})^{\times (n + 1)}} \\
        	{N((\mc C^G)^\mathrm{op})_n} & {((\mc C^G)^{\simeq})^{\times (n + 1)}}
        	\arrow[from=1-1, to=1-2]
        	\arrow[from=1-1, to=2-1]
        	\arrow["\lrcorner"{anchor=center, pos=0.125}, draw=none, from=1-1, to=2-2]
        	\arrow["{s^{\times n}}", from=1-2, to=2-2]
        	\arrow["{\prod_i \langle i\rangle^\ast}", from=2-1, to=2-2]
        \end{tikzcd}
    \end{equation*}
    Moreover, the projection $\mb D(0, -)\to N((\mc C^G)^\mathrm{op})$ is the localization onto complete Segal spaces.
\end{lemma}
\begin{proof}
    The pullback square is immediate from the definition of $\mb D$ using the fact that $E^\simeq \simeq \mc C^\simeq$ since $E$ is wide, so we prove the localization claim.\par
    For this, notice that the composite $\mc C^\simeq = \mb D(0,0)\to \ascat \mb D(0, -)\to (\mc C^G)^\mathrm{op}$ is the restriction of (opposite of) the zero section $s : \mc C\to \mc C^G$, hence essentially surjective by \Cref{prop:descrofadjendos}. Moreover, $\mb D(0,-)$ is already a Segal space so completion does not change the mapping spaces for objects in the image of $\mc C^\simeq = \mb D(0,0)\to \ascat \mb D(0, -)$. Thus, for $x,y\in \mb D(0,0)$, we have that
    \begin{equation*}
        \begin{aligned}
            \map_{\ascat \mb D(0, -)}(x, y) &\simeq N((\mc C^G)^\mathrm{op})_1\times_{((\mc C^G)^{\simeq})^{\times 2}} (\mc C^{\simeq})^{\times 2}\times_{(\mc C^{\simeq})^{\times 2}} \{(x,y)\} \\
            &\simeq N((\mc C^G)^\mathrm{op})_1 \times_{((\mc C^G)^{\simeq})^{\times 2}} \{(s(x), s(y))\} \\
            &\simeq \map_{(\mc C^G)^\mathrm{op}}(s(x), s(y))
        \end{aligned}
    \end{equation*}
    which shows that the map $\ascat \mb D(0, -)\to (\mc C^G)^\mathrm{op}$ is also fully faithful, as required.
\end{proof}

We may now prove the two promised propositions.

\begin{proof}[Proof of \Cref{prop:descr1ofSpanG}]
    For (i), under the functorial in $\mc C$ equivalence $\mc C^\ast\simeq \mc C$, the definition of $N_\mathrm{unc}(\Span^G(\mc C, E))$ reduces to the nerve of the classical span category $\Span(\mc C, E)$.\par
    For (ii), we make use of the natural equivalence $\Span^G(\mc C, E)\simeq \Gr \Pull^G(\mc C, E)^{2\mathrm{-op}}$. On the right hand side, one sees from definition that the inclusion
    \begin{equation*}
        \Hor(\Pull^G(\mc C, \mc C^\simeq)^{2\mathrm{-op}}_{0,\bullet})\to \Pull^G(\mc C, \mc C^\simeq)^{2\mathrm{-op}}
    \end{equation*}
    of the horizontal fragment is an equivalence. The result then follows from \Cref{lemma:ascatofD0}.
\end{proof}

\begin{proof}[Proof of \Cref{prop:descr2ofSpanG}]
    The induced map
    \begin{equation*}
        \begin{tikzcd}
            \mc C^\simeq \subseteq \Span(\mc C, E)^\simeq\subseteq \Span(\mc C, E) \arrow{r}{\iota} & \Span^G(\mc C, E)
        \end{tikzcd}
    \end{equation*}
    by definition of $\iota$ is the natural map $X_0\to \ascat X$ where $X = N_\mathrm{unc}(\Span^G(\mc C, E))$. By \Cref{cor:NuncSpanGsegal}, $X$ is already a Segal space and completion does not change the mapping spaces of objects in the image of this map. Thus we have that the mapping space may be computed as the fiber product
    \begin{equation*}
        \map_{\Fun(\Lambda_2^2, \Cat)}(\Sigma_1, \; \TwC \; {=\joinrel=}\; \TwC \xleftarrow{s} E)\times_{E^\simeq \times E^\simeq} \{(X,Y)\}.
    \end{equation*}
    Now, as an element of $\Fun(\Lambda_2^2, \Cat)$, $\Sigma_1$ is given by
    \begin{equation*}
        \begin{tikzcd}[column sep=-40pt]
            \{(0\le 0)\leftarrow (0\le 1)\}\cup \{(1\le 1)\} \arrow{dr} && \{(0\le 1)\rightarrow (1\le 1)\}\cup \{(0\le 0)\}\arrow{dl}\\
            & \{(0\le 0)\leftarrow (0\le 1)\}\;\displaystyle\coprod_{(0\le 1)}\;\{(0\le 1)\rightarrow (1\le 1)\}
        \end{tikzcd}
    \end{equation*}
    so
    \begin{equation}\label{eq:arrowgrpspank}
        \begin{aligned}
            \map_{\Fun(\Lambda_2^2,\Cat)}(\Sigma_1, \; \TwC \; {=\joinrel=}\; \TwC \xleftarrow{s} E) &\simeq (N((\TwC)^\mathrm{op})_1\times (\mc C^G)^{\simeq})\times_{((\mc C^G)^{\simeq})^{\times 3}} (E^\simeq \times N(E)_1) \\
            &\simeq (N((\TwC)^\mathrm{op})_1 \times_{((\mc C^G)^{\simeq})^{\times 2}} (E^{\simeq})^{\times 2})\times_{E^\simeq} N(E)_1 \\
            &\simeq \mc G\times_{\mc C^\simeq} N(\mc C^\mathrm{op})_1 \times_{E^\simeq} N(E)_1
        \end{aligned}
    \end{equation}
    where we have made use of \Cref{prop:globaldescrofadjendos}. Taking the fiber over $\{(X,Y)\}$ we see that
    \begin{equation*}
        \map_{\Span^G(\mc C, E)}(\iota X,\iota Y)\simeq \mc G\times_{\mc C^\simeq}\map_{\Span(\mc C, E)}(X,Y)
    \end{equation*}
    as claimed. Moreover, the decomposition \eqref{eq:arrowgrpspank} is compatible with the $\pi$ so that the projection $\map_{\Span^G(\mc C, E)}(\iota X,\iota Y)\to \map_{\Span(\mc C, E)}(X,Y)$ is induced by pushforward along $\pi$.\par
    Finally, for the claimed composition, the morphism $X\xleftarrow{g} (Z,\xi)\xrightarrow{f} Y$ represents the $\Sigma_1$-indexed diagram
    \begin{equation*}
        \begin{tikzcd}
            & s(Z)\arrow{dl}[above left]{(g,\xi)}\arrow{dr}{(f, 0)} \\
            s(X) && s(Y)
        \end{tikzcd}
    \end{equation*}
    and so composition is witnessed by the $2$-simplex
    \begin{equation*}
        \begin{tikzcd}
        	&& {s(W\times_Y V)} && \\
        	& s(W) && s(V) \\
        	s(X) && s(Y) && s(Z).
        	\arrow["{(\bar g_2,\, \bar f_1^\ast\eta)}"', from=1-3, to=2-2]
        	\arrow["{(\bar f_1,\, 0)}", from=1-3, to=2-4]
        	\arrow["{(g_1,\, \xi)}"', from=2-2, to=3-1]
        	\arrow["{(f_1,\, 0)}", from=2-2, to=3-3]
        	\arrow["{(g_2,\, \eta)}"', from=2-4, to=3-3]
        	\arrow["{(f_2,\, 0)}", from=2-4, to=3-5]
        \end{tikzcd}
    \end{equation*}
    which has composite
    \begin{equation*}
        \begin{tikzcd}
        	& {s(W\times_Y V)} & \\
        	s(X) && s(Z)
        	\arrow["{(g_1\bar g_2, \bar f_1^\ast\eta \cdot \bar g_2^\ast\xi)}"', from=1-2, to=2-1]
        	\arrow["{(f_2\bar f_1, 0)}", from=1-2, to=2-3]
        \end{tikzcd}
    \end{equation*}
    as claimed.
\end{proof}

\subsection{Collapsing dimensions and a trisimplicial representation}

Fix a geometric setup $(\mc C, E)$ along with another wide subcategory $S\subseteq E$ which is closed under pullback. We also fix $\mc K$ to refer to the functor
\begin{equation*}
    \begin{tikzcd}[row sep=0pt]
        \mc C^\mathrm{op}\arrow{r} & \Alg_{\mb E_\infty}(\Ani). \\
        x\arrow[mapsto]{r} & \mc K(\Vect_{(\mc C, S)}(x))
    \end{tikzcd}
\end{equation*}
In this section, we give a representation of $\Span^{\mc K}(\mc C, E)$ as the gluing of a certain trisimplicial space.\par
We have already seen in the case of the universal exchange theorem that the bisimplicial space whose $(n,m)$-simplex space is $n\times m$ commutative grids in $\mc C$ built out of Cartesian squares of the form
\begin{equation*}
    \begin{tikzcd}
    	\bullet & \bullet \\
    	\bullet & \bullet
    	\arrow["{\in\mc C}", from=1-1, to=1-2]
    	\arrow["{\in S}"', from=1-1, to=2-1]
    	\arrow["\lrcorner"{anchor=center, pos=0.125}, draw=none, from=1-1, to=2-2]
    	\arrow["{\in S}", from=1-2, to=2-2]
    	\arrow["{\in\mc C}"', from=2-1, to=2-2]
    \end{tikzcd}
\end{equation*}
glues to $\mc C^{\mc K}$. We also saw in the previous section that the bisimplicial space with $(n,m)$-simplex space consisting of $n\times m$ commutative grids in $\mc C$ built out of Cartesian squares of the form
\begin{equation*}
    \begin{tikzcd}
    	\bullet & \bullet \\
    	\bullet & \bullet
    	\arrow["{\in E}"', from=1-1, to=2-1]
    	\arrow["{\in\mc C}"', from=1-2, to=1-1]
    	\arrow["\lrcorner"{anchor=center, pos=0.125, rotate=-90}, draw=none, from=1-2, to=2-1]
    	\arrow["{\in E}", from=1-2, to=2-2]
    	\arrow["{\in\mc C}", from=2-2, to=2-1]
    \end{tikzcd}
\end{equation*}
glues to $\Span(\mc C, E)$. In this section, we will prove that $\Span^{\mc K}(\mc C, E)$ may be obtained by gluing a trisimplicial space whose $(n,m,\ell)$-simplex space is given by $n\times m\times\ell$ commutative grids built out of cubes of the form
\begin{equation*}
    \begin{tikzcd}
    	& \bullet && \bullet \\
    	\bullet && \bullet \\
    	& \bullet && \bullet \\
    	\bullet && \bullet
    	\arrow["{\in E}"{pos=0.3}, from=1-2, to=3-2]
    	\arrow["{\in \mc C}", from=1-4, to=1-2]
    	\arrow["{\in E}"{pos=0.3}, from=1-4, to=3-4]
    	\arrow["{\in S}"', from=2-1, to=1-2]
    	\arrow["{\in E}"{pos=0.3}, from=2-1, to=4-1]
    	\arrow["{\in S}"'{pos=0.4}, from=2-3, to=1-4]
    	\arrow["{\in \mc C}"{pos=0.1}, from=2-3, to=2-1]
    	\arrow["{\in E}"{pos=0.3}, from=2-3, to=4-3]
    	\arrow["{\in \mc C}"{pos=0.7}, from=3-4, to=3-2]
    	\arrow["{\in S}"', from=4-1, to=3-2]
    	\arrow["{\in S}"', from=4-3, to=3-4]
    	\arrow["{\in \mc C}", from=4-3, to=4-1]
    \end{tikzcd}
\end{equation*}
where all faces are Cartesian. This will be done by keeping the $E$ component fixed and gluing the $\mc C$ and $S$ components, applying the universal exchange theorem.

\subsubsection{Higher squares and diagonal approximation}

First we define what it means to glue a trisimplicial space by extending the squares functor to higher dimensions.

\begin{definition}
    For each $n\ge 1$, define $\Sq^{(n)} : \Cat\to \PSh(\Delta^{\times n})$ to be the nerve construction corresponding to the functor
    \begin{equation*}
        \begin{tikzcd}[row sep=0pt]
            \Delta^{\times n}\arrow{r} & \Cat \\
            ([m_1],\dots, [m_n])\arrow[mapsto]{r} & {[m_1]\times\cdots\times [m_n]}.
        \end{tikzcd}
    \end{equation*}
    This has a left adjoint which we denote by $\Gr^{(n)}$. Since $n$ may implicitly be determined by the input to $\Gr^{(n)}$, we will occasionally abuse notation and simply write $\Gr$ for $\Gr^{(n)}$, leaving $n$ implicit.
\end{definition}

\begin{example}
    $\Sq^{(1)}$ is the (Rezk) nerve $N$ of a category and $\Sq^{(2)}$ is the squares functor $\Sq$ of \Cref{def:squaresfunc}.
\end{example}

It turns out that diagonal approximation still holds as a corollary of the $n = 2$ case. For an $n$-simplicial space $X$, we denote by $\diag X$ the simplicial space
\begin{equation*}
    \begin{tikzcd}
        \Delta^\mathrm{op}\arrow{r}{\diag} & (\Delta^{\times n})^\mathrm{op}\arrow{r}{X} & \Ani.
    \end{tikzcd}
\end{equation*}
We then again have a composition along the diagonal map
\begin{equation*}
    \begin{tikzcd}[row sep=0pt]
        \diag X\arrow{r}{\diag\circ\mathrm{unit}} &[10pt] \diag \Sq^{(n)}\Gr X\arrow{r} & N(\Gr X). \\
        & F : [m]^{\times n}\to \Gr X\arrow[mapsto]{r} & ([m]\xrightarrow{\diag} [m]^{\times n}\xrightarrow{F} \Gr X)
    \end{tikzcd}
\end{equation*}

\begin{proposition}
    Composition along the diagonal $\diag X\to N(\Gr X)$ transposes to an equivalence $\ascat\diag X\simeq \Gr X$.
\end{proposition}
\begin{proof}
    We induct on $n$. The case $n = 1$ is trivial and the case $n = 2$ is given by \Cref{prop:diaglocalequalsgr}. For $n > 2$, notice that $\Gr^{(n)}$ may be computed as the composite
    \begin{equation*}
        \begin{tikzcd}
            \PSh(\Delta^{\times n})\simeq \Fun(\Delta^\mathrm{op}, \PSh(\Delta^{\times (n - 1)}))\arrow{r}{\Gr^{(n - 1)}_\ast} & \Fun(\Delta^\mathrm{op},\Cat)\arrow[hook]{r}{N_\ast} & \PSh(\Delta^{\times 2})\arrow{r}{\Gr} & \Cat.
        \end{tikzcd}
    \end{equation*}
    Moreover, by the inductive hypotheses and the $n = 2$ case, these fit into a commutative diagram
    \begin{equation*}
        \begin{tikzcd}
        	{\PSh(\Delta^{\times n})} &[20pt] {\PSh(\Delta^{\times 2})} & {\PSh(\Delta^{\times 2})} \\
        	{\PSh(\Delta)} & {\PSh(\Delta)} \\
        	& \Cat
        	\arrow["{N_\ast\,\circ\, \Gr^{(n - 1)}_\ast}", from=1-1, to=1-2]
        	\arrow["\diag", from=1-1, to=2-1]
        	\arrow[equals, from=1-2, to=1-3]
        	\arrow["\diag", from=1-2, to=2-2]
        	\arrow["\Gr", bend left=20, from=1-3, to=3-2]
        	\arrow["\ascat"', from=2-1, to=3-2]
        	\arrow["\ascat", from=2-2, to=3-2]
        \end{tikzcd}
    \end{equation*}
    yielding the result.
\end{proof}

\subsubsection{Right fibrations and reduction of dimension}

In order to collapse dimensions in our trisimplicial spaces, we will need the notion of right fibrations of simplicial spaces.

\begin{definition}
    A morphism $p : X\to Y$ of simplicial spaces is said to be a \emph{right fibration} if
    \begin{equation*}
        \begin{tikzcd}
            X_n\arrow{r}{\langle n\rangle^\ast}\arrow{d}{p_n} & X_{0}\arrow{d}{p_{0}} \\
            Y_n\arrow{r}{\langle n\rangle^\ast} & Y_{0}
        \end{tikzcd}
    \end{equation*}
    is Cartesian for every $n\ge 1$.\par
    Dually, $p$ is said to be a \emph{left fibration} if the analogous square with $\langle 0\rangle$ rather than $\langle n\rangle$ is Cartesian. Alternatively, $p$ is a left fibration if $p^\mathrm{op} : X^\mathrm{op}\to Y^\mathrm{op}$ is a right fibration.
\end{definition}

The following result explains the choice of notation. This theorem is originally due to Rasekh \cite[Theorem 4.18
and 5.1]{rasekh2023yoneda} via model-dependent methods, and then later re-derived via model independent techniques by Barkan and Steinebrunner \cite{barkan2025segalification}. Let $\mb L_{\mathrm{CS}}\simeq N\circ\ascat$ be localization onto complete Segal spaces, i.e.\ onto the essential image of $N : \Cat\to \PSh(\Delta)$.

\begin{theorem}[Rasekh]\label{thm:rightfibpresbylocal}
    For every simplicial space $X$, there is a pair of adjoint equivalences
    \begin{equation*}
        \begin{tikzcd}
        	{\mb L_\mathrm{CS}:\PSh(\Delta)_{/X}^\mathrm{r-fib}} && {\PSh(\Delta)_{/\mb L_\mathrm{CS}X}^\mathrm{r-fib}: X\times_{\mb L_\mathrm{CS}X}(-)}.
        	\arrow[shift left=2, from=1-1, to=1-3]
        	\arrow["\simeq"{description}, draw=none, from=1-1, to=1-3]
        	\arrow[shift left=2, from=1-3, to=1-1]
        \end{tikzcd}
    \end{equation*}
    Moreover, a map $\mc C\to\mc D$ of categories is a right fibration if and only if $N(\mc C)\to N(\mc D)$ is.
\end{theorem}

With this notion in mind, we define and study a class of bisimplicial spaces which will appear repeatedly.

\begin{definition}
    \begin{enumerate}
        \item Say a bisimplicial space $X$ is \emph{factorial} if for all $m\ge 0$, $(\id, \langle 0\rangle)^\ast : X(-, m)\to X(-,0)$ is a right fibration.
        \item Let $X$, $Y$ be factorial. A map $X\to Y$ is said to be \emph{base compatible} if for all $n\ge 0$, the square
        \begin{equation*}
            \begin{tikzcd}
            	{X(n,-)} & {Y(n,-)} \\
            	{X(0,-)} & {Y(0,-)}
            	\arrow[from=1-1, to=1-2]
            	\arrow["{(\langle n\rangle,\; \id)^\ast}"', from=1-1, to=2-1]
            	\arrow["{(\langle n\rangle,\; \id)^\ast}", from=1-2, to=2-2]
            	\arrow[from=2-1, to=2-2]
            \end{tikzcd}
        \end{equation*}
        is Cartesian.
    \end{enumerate}
\end{definition}

\begin{remark}
    In \cite{juran2026orthogonal}, Juran shows that $\Fact :\mathrm{OFS}\to \PSh(\Delta^{\times 2})$ is fully faithful with essential image factorial bisimplicial spaces $X$ such that for all $n$, $m\ge 0$, $X(n,-)$ and $X(-,m)$ are complete Segal spaces. Thus factorial bisimplicial spaces should be thought of as a mild weakening of belonging to the essential image of $\Fact$.
\end{remark}

\begin{proposition}\label{prop:equivofbaseimpbasecompat}
    Let $X\to Y$ be a map of factorial bisimplicial spaces. The following are equivalent:
    \begin{enumerate}
        \item $X\to Y$ is base compatible and $X(0,0)\to Y(0,0)$ is an equivalence
        \item $X(-,0)\to Y(-,0)$ is an equivalence.
    \end{enumerate}
\end{proposition}
\begin{proof}
    (i) implies (ii) by definition of base compatibility and preservation of equivalences under pullback. For (ii) $\Rightarrow$ (i), consider the commutative cube
    \begin{equation*}
        \begin{tikzcd}
        	& {Y(n,0)} && {X(n,0)} \\
        	{Y(n,m)} && {X(n,m)} \\
        	& {Y(0,0)} && {X(0,0)}. \\
        	{Y(0,m)} && {X(0,m)}
        	\arrow["\simeq"{description, pos=0.3}, from=1-2, to=1-4]
        	\arrow[from=1-2, to=3-2]
        	\arrow[from=1-4, to=3-4]
        	\arrow[from=2-1, to=1-2]
        	\arrow[from=2-1, to=2-3]
        	\arrow["\lrcorner"{anchor=center, pos=0.125}, draw=none, from=2-1, to=3-2]
        	\arrow[from=2-1, to=4-1]
        	\arrow[from=2-3, to=1-4]
        	\arrow["\lrcorner"{anchor=center, pos=0.125}, draw=none, from=2-3, to=3-4]
        	\arrow[from=2-3, to=4-3]
        	\arrow["\simeq"{description, pos=0.3}, from=3-2, to=3-4]
        	\arrow[from=4-1, to=3-2]
        	\arrow[from=4-1, to=4-3]
        	\arrow[from=4-3, to=3-4]
        \end{tikzcd}
    \end{equation*}
    The left and right faces are Cartesian by factoriality of $X$ and $Y$, and the back face is Cartesian since both horizontal arrows are equivalences. It follows by pasting that the front face is Cartesian, as required.
\end{proof}

Given a bisimplicial space $X$, an intermediate approximation to $\Gr X$ is given by localizing onto $\Fun(\Delta^\mathrm{op},\Cat)\subseteq\PSh(\Delta^{\times 2})$, i.e.\ by passing to associated categories along one of the dimensions. The next lemma shows that this procedure preserves factorial bisimplicial spaces and base compatible maps.

\begin{proposition}\label{prop:CSpreservesnicestuff}
    Let $(-)_\mathrm{CS} : \PSh(\Delta^{\times 2})\to \PSh(\Delta^{\times 2})$ be the functor which sends a bisimplicial space $X$ to
    \begin{equation*}
        \begin{tikzcd}[row sep=0pt]
            \Delta^\mathrm{op}\arrow{r} & \PSh(\Delta) \\
            {[n]}\arrow[mapsto]{r} & \mb L_\mathrm{CS}X(n,-)
        \end{tikzcd}
    \end{equation*}
    where $\mb L_\mathrm{CS}$ is localization onto complete Segal spaces. Then $(-)_\mathrm{CS}$ preserves factorial bisimplicial spaces and base compatible maps.
\end{proposition}
\begin{proof}
    First, notice that by definition a bisimplicial space $X$ is factorial if and only if for every $n\ge 0$, $(\langle n\rangle, \id)^\ast : X(n,-)\to X(0,-)$ is a left fibration. Using this alternative characterization and (the dual of) \Cref{thm:rightfibpresbylocal} which says that localization onto complete Segal spaces preserves left fibrations, we see that $(-)_\mathrm{CS}$ preserves factorial bisimplicial spaces.\par
    For the second claim, let $X\to Y$ be a base compatible map of factorial bisimplicial spaces. We then have for all $n\ge 0$ a commutative cube
    \begin{equation}\label{eq:fundamentalcube}
        \begin{tikzcd}
        	& {X_\mathrm{CS}(n,-)} && {Y_\mathrm{CS}(n,-)} \\
        	{X(n,-)} && {Y(n,-)} \\
        	& {X_\mathrm{CS}(0,-)} && {Y_\mathrm{CS}(0,-)} \\
        	{X(0,-)} && {Y(0,-)}
        	\arrow[from=1-2, to=1-4]
        	\arrow[from=1-2, to=3-2]
        	\arrow[from=1-4, to=3-4]
        	\arrow[from=2-1, to=1-2]
        	\arrow[from=2-1, to=2-3]
        	\arrow[from=2-1, to=4-1]
        	\arrow[from=2-3, to=1-4]
        	\arrow[from=2-3, to=4-3]
        	\arrow[from=3-2, to=3-4]
        	\arrow[from=4-1, to=3-2]
        	\arrow[from=4-1, to=4-3]
        	\arrow[from=4-3, to=3-4]
        \end{tikzcd}
    \end{equation}
    where the front face is Cartesian by base compatibility and the left and right faces are Cartesian by (the dual of) \Cref{thm:rightfibpresbylocal}. To complete the proof, we need to show that the back face is Cartesian.\par
    To do this, notice that both vertical morphisms in the back face are (the nerve of) left fibrations of categories. Thus being Cartesian may be checked fiberwise. Since $X(0,0)\to \ascat X_\mathrm{CS}(0,-)$ is essentially surjective, it suffices to check on fibers belonging to points of $X(0,0)$, i.e.\ it will suffice to verify that the outer square of
    \begin{equation}\label{eq:outersquaretocheck}
        \begin{tikzcd}
        	F & {X_\mathrm{CS}(n,-)} & {Y_\mathrm{CS}(n,-)} \\
        	{\mathrm{const}\, X(0,0)} & {X_\mathrm{CS}(0,-)} & {Y_\mathrm{CS}(0,-)}
        	\arrow[from=1-1, to=1-2]
        	\arrow[from=1-1, to=2-1]
        	\arrow["\lrcorner"{anchor=center, pos=0.125}, draw=none, from=1-1, to=2-2]
        	\arrow[from=1-2, to=1-3]
        	\arrow[from=1-2, to=2-2]
        	\arrow[from=1-3, to=2-3]
        	\arrow[from=2-1, to=2-2]
        	\arrow[from=2-2, to=2-3]
        \end{tikzcd}
    \end{equation}
    is Cartesian, where $\mathrm{const}$ denotes the constant simplicial space functor. However, the outer square of \eqref{eq:outersquaretocheck} agrees with the outer square of
    \begin{equation*}
        \begin{tikzcd}
        	F & {X(n,-)} & {X_\mathrm{CS}(n,-)} & {Y_\mathrm{CS}(n,-)} \\
        	{\mathrm{const}\, X(0,0)} & {X(0,-)} & {X_\mathrm{CS}(0,-)} & {Y_\mathrm{CS}(0,-)}
        	\arrow[from=1-1, to=1-2]
        	\arrow[from=1-1, to=2-1]
        	\arrow["\lrcorner"{anchor=center, pos=0.125}, draw=none, from=1-1, to=2-2]
        	\arrow[from=1-2, to=1-3]
        	\arrow[from=1-2, to=2-2]
        	\arrow["\lrcorner"{anchor=center, pos=0.125}, draw=none, from=1-2, to=2-3]
        	\arrow[from=1-3, to=1-4]
        	\arrow[from=1-3, to=2-3]
        	\arrow[from=1-4, to=2-4]
        	\arrow[from=2-1, to=2-2]
        	\arrow[from=2-2, to=2-3]
        	\arrow[from=2-3, to=2-4]
        \end{tikzcd}
    \end{equation*}
    which, making use of \eqref{eq:fundamentalcube}, agrees with the outer square of
    \begin{equation*}
        \begin{tikzcd}
        	F & {X(n,-)} & {Y(n,-)} & {Y_\mathrm{CS}(n,-)} \\
        	{\mathrm{const}\, X(0,0)} & {X(0,-)} & {Y(0,-)} & {Y_\mathrm{CS}(0,-)}.
        	\arrow[from=1-1, to=1-2]
        	\arrow[from=1-1, to=2-1]
        	\arrow["\lrcorner"{anchor=center, pos=0.125}, draw=none, from=1-1, to=2-2]
        	\arrow[from=1-2, to=1-3]
        	\arrow[from=1-2, to=2-2]
        	\arrow["\lrcorner"{anchor=center, pos=0.125}, draw=none, from=1-2, to=2-3]
        	\arrow[from=1-3, to=1-4]
        	\arrow[from=1-3, to=2-3]
        	\arrow["\lrcorner"{anchor=center, pos=0.125}, draw=none, from=1-3, to=2-4]
        	\arrow[from=1-4, to=2-4]
        	\arrow[from=2-1, to=2-2]
        	\arrow[from=2-2, to=2-3]
        	\arrow[from=2-3, to=2-4]
        \end{tikzcd}
    \end{equation*}
    However, the above outer square is evidently Cartesian by pasting, as claimed.
\end{proof}

From this we may deduce our primary tool for collapsing dimensions.

\begin{lemma}\label{lemma:dimcollapserightfib}
    Let $X$, $Y$ be factorial bisimplicial spaces and suppose we are given a map $Y\to X$ such that
    \begin{enumerate}
        \item $Y(-,0)\to X(-,0)$ is an equivalence
        \item $\ascat Y(0,-)\to \ascat X(0,-)$ is an equivalence.
    \end{enumerate}
    Then $Y_\mathrm{CS}\to X_\mathrm{CS}$ is an equivalence of bisimplicial spaces, and in particular the induced map $\Gr Y\to \Gr X$ is an equivalence.
\end{lemma}
\begin{proof}
    By \Cref{prop:equivofbaseimpbasecompat} we have that $X\to Y$ is base compatible. Using this, by \Cref{prop:CSpreservesnicestuff} we deduce that $Y_\mathrm{CS}\to X_\mathrm{CS}$ is a base compatible map between factorial bisimplicial spaces.\par
    Since both bisimplicial spaces are factorial, this map is an equivalence if and only if the maps $Y_\mathrm{CS}(0,-)\to X_\mathrm{CS}(0,-)$ and $Y_\mathrm{CS}(-,0)\to X_\mathrm{CS}(-,0)$ on horizontal and vertical fragments are equivalences. (ii) is precisely the statement that this map is an equivalence on horizontal fragments. For the map on vertical fragments, by base compatibility, for all $n\ge 0$ we have a pullback square
    \begin{equation*}
        \begin{tikzcd}
        	{Y_\mathrm{CS}(n,0)} & {X_{\mathrm{CS}}(n,0)} \\
        	{Y_\mathrm{CS}(0,0)} & {X_\mathrm{CS}(0,0)}
        	\arrow[from=1-1, to=1-2]
        	\arrow[from=1-1, to=2-1]
        	\arrow["\lrcorner"{anchor=center, pos=0.125}, draw=none, from=1-1, to=2-2]
        	\arrow[from=1-2, to=2-2]
        	\arrow["\simeq", from=2-1, to=2-2]
        \end{tikzcd}
    \end{equation*}
    where the bottom map is an equivalence by (ii). The claim then follows by preservation of equivalences under pullback.\par
    The final claim follows as $\Sq$ lands inside $\Fun(\Delta^\mathrm{op},\Cat)\subseteq \PSh(\Delta^{\times 2})$ so $\Gr$ factors through $(-)_\mathrm{CS}$.
\end{proof}

\begin{remark}\label{rem:genofrightfiblemma}
    Since $\Sq$ is symmetric in both factors, $\Gr$ is invariant under swapping the two factors of $\Delta^\mathrm{op}$. Additionally, there is a natural equivalence $\Gr X^\mathrm{op}\simeq (\Gr X)^\mathrm{op}$. It follows that the analogous versions of \Cref{lemma:dimcollapserightfib} hold with coordinates swapped and/or right fibration replaced with left fibration in the definition of factorial bisimplicial spaces.
\end{remark}

\subsubsection{Trisimplicial representations of labeled spans}

\begin{definition}
    Denote by $\Pull(\mc C, E, S)$ the sub-trisimplicial space of $\Sq^{(3)}(\mc C)$ consisting at level $(n,m,\ell)$ of connected components of diagrams $D : [n]\times [m]\times [\ell]\to\mc C$ such that
    \begin{enumerate}
        \item $D$ sends each $(i,j,k)\to (i',j,k)$ to a morphism in $E$
        \item $D$ sends each $(i,j,k)\to (i,j',k)$ to a morphism in $S$
        \item $D$ sends each $(i,j,k)\to (i,j,k')$ to a morphism in $\mc C$
        \item all squares in $[n]\times [m]\times [\ell]$ consisting of morphisms which alter at most one coordinate are sent under $D$ to Cartesian squares in $\mc C$.
    \end{enumerate}
\end{definition}

\begin{lemma}\label{lemma:CntoC0rightfib}
    Let $\mb C = \Pull(\mc C, E, S)^{2,3\mathrm{-op}}$. Then for every $m,\ell$, the map
    \begin{equation*}
        \begin{tikzcd}
            (\id,\langle 0\rangle,\langle 0\rangle)^\ast : \mb C(-, m, \ell)\arrow{r} & \mb C(-, 0, 0)
        \end{tikzcd}
    \end{equation*}
    is a right fibration of simplicial spaces.
\end{lemma}
\begin{proof}
    By arguments similar to \Cref{prop:goodisdetectonspine} and \Cref{lemma:AnmdescrbyKan}, one may show that a diagram $D : [n]\times [m]\times [\ell]\to\mc C$ belongs to $\mb C(n,m,\ell)$ if and only if
    \begin{enumerate}
        \item $\{n\}\times [m]\times [\ell]\hookrightarrow [n]\times [m]\times [\ell]\xrightarrow{D} \mc C$ belongs to $\mb C(0, m, \ell)$
        \item $D$ is right Kan extended from its restriction to
        \begin{equation*}
            \begin{tikzcd}
                {\{n\}\times [m]\times [\ell]\cup [n]\times \{0\}\times \{0\}}\arrow[hook]{r} & {[n]\times [m]\times [\ell]}\arrow{r}{D} & \mc C,
            \end{tikzcd}
        \end{equation*}
    \end{enumerate}
    and moreover that these right Kan extensions exist and are computed pointwise. In particular, one deduces that the square
    \begin{equation*}
        \begin{tikzcd}[column sep=45pt]
            \mb C(n,m,\ell)\arrow{r}{(\langle n\rangle,\,\id,\,\id)^\ast}\arrow{d}[left]{(\id,\,\langle 0\rangle,\,\langle 0\rangle)^\ast} & \mb C(0, m, \ell)\arrow{d}[right]{(\id,\,\langle 0\rangle,\,\langle 0\rangle)^\ast} \\
            \mb C(n,0,0)\arrow{r}[below]{(\langle n\rangle,\, \id,\,\id)^\ast} & \mb C(0,0,0)
        \end{tikzcd}
    \end{equation*}
    is Cartesian, i.e.\ that $\mb C(-, m, \ell)\longrightarrow \mb C(-,0,0)$ is a right fibration.
\end{proof}

\begin{lemma}\label{lemma:DntoD0rightfib}
    Let $\mb D = \Pull^G(\mc C, E)^{2\mathrm{-op}}$ for some $(\mc C, E, G)\in\mathrm{LabPair}$. Then for every $m$, the map
    \begin{equation*}
        (\id,\langle 0\rangle)^\ast : \mb D(-, m)\longrightarrow \mb D(-, 0)
    \end{equation*}
    is a right fibration.
\end{lemma}
\begin{proof}
    Since each $\mb D(-, m)$ is a Segal space, by \cite[Proposition 1.7]{de2016segal} this is equivalent to the a priori weaker condition (i) of \Cref{prop:condcnrsegal}. The result then follows from \Cref{lemma:equivcondforsegalforD}.
\end{proof}

We also have one other trisimplicial space of interest.

\begin{definition}
    Let $\mathrm{PullSq}(\mc C, E)$ denote the sub-trisimplicial space of $\Sq^{(3)}(\mc C)$ consisting of components of diagrams $D : [n]\times [m]\times [\ell]\to\mc C$ such that
    \begin{enumerate}
        \item for all $0\le i\le \ell$, $D|_{[n]\times [m]\times \{i\}}$ belongs to $\Pull(\mc C, E)_{n,m}$
        \item for all $0\le i\le m$, $D|_{[n]\times \{i\}\times [\ell]}$ belongs to $\Pull(\mc C, E)_{n,\ell}$.
    \end{enumerate}
\end{definition}

That is to say, the $(n,m,\ell)$-simplex space of $\mathrm{PullSq}(\mc C, E)$ is the groupoid of $n\times m\times\ell$ commutative grids built from cubes of the form
\begin{equation*}
    \begin{tikzcd}
    	& \bullet && \bullet \\
    	\bullet && \bullet \\
    	& \bullet && \bullet \\
    	\bullet && \bullet
    	\arrow["{\in \mc C}"', from=1-2, to=1-4]
    	\arrow["{\in E}"'{pos=0.7}, from=1-2, to=3-2]
    	\arrow["{\in E}", from=1-4, to=3-4]
    	\arrow["{\in \mc C}"', from=2-1, to=1-2]
    	\arrow["{\in \mc C}"'{pos=0.8}, from=2-1, to=2-3]
    	\arrow["{\in E}", from=2-1, to=4-1]
    	\arrow["{\in \mc C}"', from=2-3, to=1-4]
    	\arrow["{\in E}"{pos=0.3}, from=2-3, to=4-3]
    	\arrow["{\in \mc C}"'{pos=0.3}, from=3-2, to=3-4]
    	\arrow["{\in \mc C}"', from=4-1, to=3-2]
    	\arrow["{\in \mc C}"', from=4-1, to=4-3]
    	\arrow["{\in \mc C}"', from=4-3, to=3-4]
    \end{tikzcd}
\end{equation*}
where the front, back, left and right faces are assumed to be Cartesian, but the top and bottom faces are only assumed to be commutative.

\begin{lemma}[Expansion]\label{lemma:expansion}
    There exists an equivalence $\Gr \mathrm{PullSq}(\mc C, E)\simeq \Gr \Pull(\mc C, E)$ under which
    \begin{enumerate}
        \item the two inclusions
        \begin{equation*}
            \Pull(\mc C, E)\simeq \mathrm{PullSq}(\mc C, E)_{\bullet, 0, \bullet}\longrightarrow \mathrm{PullSq}(\mc C, E)
        \end{equation*}
        and
        \begin{equation*}
            \Pull(\mc C, E)\simeq \mathrm{PullSq}(\mc C, E)_{\bullet, \bullet, 0}\longrightarrow \mathrm{PullSq}(\mc C, E)
        \end{equation*}
        give inverses upon taking $\Gr$
        \item there is a commutative diagram
        \begin{equation*}
            \begin{tikzcd}
            	{\Gr\mathrm{PullSq}(\mc C, E)_{0,\bullet,\bullet}\simeq \Gr\Sq\mc C} & {\mc C \simeq \Gr\Pull(\mc C, E)_{0,\bullet}} \\
            	{\Gr\mathrm{PullSq}(\mc C, E)} & {\Gr \Pull(\mc C, E)}
            	\arrow[from=1-1, to=1-2]
            	\arrow[from=1-1, to=2-1]
            	\arrow[from=1-2, to=2-2]
            	\arrow["\simeq", from=2-1, to=2-2]
            \end{tikzcd}
        \end{equation*}
        where the top morphism is the counit.
    \end{enumerate}
\end{lemma}
\begin{proof}
    Let $\mb C = \mathrm{PullSq}(\mc C, E)$. By a routine Kan extension argument similar to \Cref{lemma:CntoC0rightfib}, one sees that for every $m$, $\ell$, the map $(\id,\langle m\rangle, \langle \ell\rangle)^\ast : \mb C(-,m,\ell)\to \mb C(-,0,0)$ is a right fibration.\par
    Now, consider the bisimplicial space $\mb C'$ given by
    \begin{equation*}
        \begin{tikzcd}
            \Delta^\mathrm{op}\times\Delta^\mathrm{op}\arrow{r}{\id\times \diag} &[30pt] \Delta^\mathrm{op}\times\Delta^\mathrm{op}\times\Delta^\mathrm{op}\arrow{r}{\mb C} & \Ani.
        \end{tikzcd}
    \end{equation*}
    Seeing as there is an equality $\diag\mb C = \diag \mb C'$, diagonal approximation gives an equivalence $\Gr\mb C'\simeq \Gr\mb C$. We then have a composition along the diagonal map
    \begin{equation*}
        \begin{tikzcd}[row sep=0pt]
            F : \mb C'\arrow{r} & \Pull(\mc C, E) \\
            D : [n]\times [m]\times [m]\to \mc C\arrow[mapsto]{r} & {[n]\times [m]\xrightarrow{\id\times \diag}[n]\times [m]\times [m]\xrightarrow{D}\mc C}.
        \end{tikzcd}
    \end{equation*}
    The conditions of \Cref{lemma:dimcollapserightfib} (see also \Cref{rem:genofrightfiblemma}) then hold for $F$. Indeed, $F(-,0)$ evidently induces an equivalence, and $F(0,-)$ induces an equivalence after passing to associated categories since it is the map $\diag\Sq(\mc C)\to N(\Gr \Sq(\mc C))\simeq N(\mc C)$ occurring in diagonal approximation for $\Sq(\mc C)$. It follows that $\Gr\mb C\simeq \Gr \mb C'\simeq \Gr \Pull(\mc C, E)$.\par
    The claims about this equivalence are readily checked.
\end{proof}

The content of this lemma is as follows: Given a map $\Pull(\mc C, E)\to \Sq(\mc D)$ corresponding on squares, i.e.\ $(1,1)$-simplices, to
\begin{equation*}
    \begin{tikzcd}
    	x & {y'} && {D(x)} & {D(y')} \\
    	y & z && {D(y)} & {D(z)}
    	\arrow["{\bar g}", from=1-1, to=1-2]
    	\arrow["{\bar f}"', from=1-1, to=2-1]
    	\arrow["\lrcorner"{anchor=center, pos=0.125}, draw=none, from=1-1, to=2-2]
    	\arrow[""{name=0, anchor=center, inner sep=0}, "f"', from=1-2, to=2-2]
    	\arrow["{G(\bar g)}", from=1-4, to=1-5]
    	\arrow[""{name=1, anchor=center, inner sep=0}, "{F(\bar f)}"', from=1-4, to=2-4]
    	\arrow["{F(f)}", from=1-5, to=2-5]
    	\arrow["g", from=2-1, to=2-2]
    	\arrow["{G(g)}"', from=2-4, to=2-5]
    	\arrow[shorten <= 15pt, shorten >= 25pt, squiggly={
               pre length=15pt, post length=25pt
             }, from=0, to=1]
    \end{tikzcd}
\end{equation*}
it extends essentially uniquely to a map $\mathrm{PullSq}(\mc C, E)\to \Sq^{(3)}(\mc D)$ and this extension on cubes, i.e.\ $(1,1,1)$-simplices, is given by
\begin{equation*}
    \begin{tikzcd}
    	& {x_2} && {y'_2} &&&& {D(x_2)} && {D(y'_2)} \\
    	{x_1} && {y'_1} &&&& {D(x_1)} && {D(y'_1)} \\
    	& {y_2} && {z_2} &&&& {D(y_2)} && {D(z_2)} \\
    	{y_1} && {z_1} &&&& {D(y_1)} && {D(z_1)}
    	\arrow["{\bar{g_2}}"'{pos=0.3}, from=1-2, to=1-4]
    	\arrow["{\bar{f_2}}"'{pos=0.3}, from=1-2, to=3-2]
    	\arrow["{f_2}"'{pos=0.3}, from=1-4, to=3-4]
    	\arrow["{G(\bar{g_2})}"{pos=0.3}, from=1-8, to=1-10]
    	\arrow["{F(\bar{f_2})}"{pos=0.3}, from=1-8, to=3-8]
    	\arrow["{F(f_2)}"{pos=0.3}, from=1-10, to=3-10]
    	\arrow["{h_x}", from=2-1, to=1-2]
    	\arrow["{\bar{g_1}}"'{pos=0.3}, from=2-1, to=2-3]
    	\arrow["{\bar{f_1}}"'{pos=0.3}, from=2-1, to=4-1]
    	\arrow["{h_{y'}}", from=2-3, to=1-4]
    	\arrow["{f_1}"'{pos=0.3}, from=2-3, to=4-3]
    	\arrow["{G(h_x)}"{pos=0.2}, from=2-7, to=1-8]
    	\arrow["{G(\bar{g_1})}"'{pos=0.3}, from=2-7, to=2-9]
    	\arrow[""{name=0, anchor=center, inner sep=0}, "{F(\bar{f_1})}"'{pos=0.3}, from=2-7, to=4-7]
    	\arrow["{G(h_{y'})}"{pos=0.2}, from=2-9, to=1-10]
    	\arrow["{F(f_1)}"{pos=0.3}, from=2-9, to=4-9]
    	\arrow["{g_2}"'{pos=0.3}, from=3-2, to=3-4]
    	\arrow["{G(g_2)}"'{pos=0.3}, from=3-8, to=3-10]
    	\arrow["{h_y}"', from=4-1, to=3-2]
    	\arrow["{g_1}"'{pos=0.3}, from=4-1, to=4-3]
    	\arrow["{h_z}"', from=4-3, to=3-4]
    	\arrow["{G(h_y)}"'{pos=0.7}, from=4-7, to=3-8]
    	\arrow["{G(g_1)}"'{pos=0.3}, from=4-7, to=4-9]
    	\arrow["{G(h_z)}"', from=4-9, to=3-10]
    	\arrow[shorten <= 15pt, shorten >= 25pt, squiggly={
               pre length=15pt, post length=25pt
             }, from=3-4, to=0]
    \end{tikzcd}
\end{equation*}
where the left, right, front and back Cartesian faces are mapped under the original map, and the top and bottom faces, which are only assumed to be commutative, are sent to their pushforward under $G$.\par
Using this expansion lemma, we may construct our desired equivalence $\Gr \Pull(\mc C, E, S)\simeq \Span^{\mc K}(\mc C, E)$.

\begin{theorem}\label{thm:trisimpgluing}
    There exists an equivalence
    \begin{equation*}
        \Gr\Pull(\mc C, E, S)^{2,3\mathrm{-op}}\simeq \Span^{\mc K}(\mc C, E)
    \end{equation*}
    under which
    \begin{enumerate}
        \item the composite
        \begin{equation*}
            \Span(\mc C, E)\simeq \Gr \Pull(\mc C, E, S)^{2,3\mathrm{-op}}_{\bullet,0,\bullet}\to \Gr\Pull(\mc C, E, S)^{2,3\mathrm{-op}}\simeq \Span^{\mc K}(\mc C, E)
        \end{equation*}
        corresponds to the map $\iota$ of \Cref{prop:descr2ofSpanG}
        \item the composite
        \begin{equation*}
            \begin{tikzcd}
                (\TwC)^\mathrm{op}\arrow{r}[below]{\simeq}[above]{{\text{\normalfont\Cref{thm:mainthmintext}}}} &[25pt] \Gr\Pull(\mc C, E, S)^{2,3\mathrm{-op}}_{0,\bullet,\bullet}\arrow{r} & \Gr\Pull(\mc C, E, S)^{2,3\mathrm{-op}}\simeq \Span^{\mc K}(\mc C, E)
            \end{tikzcd}
        \end{equation*}
        corresponds to the map $(\TwC)^\mathrm{op}\simeq \Span^{\mc K}(\mc C, \mc C^\simeq)\to \Span^{\mc K}(\mc C, E)$.
    \end{enumerate}
\end{theorem}
\begin{proof}
    Properties (i) and (ii) will be readily checked from the construction, so we construct the desired map and show that it is an equivalence.\par
    We start with the equivalence
    \begin{equation*}
        \Gr \Pull(\mc C, S)\simeq \TwC
    \end{equation*}
    coming from \Cref{thm:mainthmintext} which corresponds to a map
    \begin{equation}\label{eq:mainthmmapfortri}
        \Pull(\mc C, S)\longrightarrow \Sq(\TwC)
    \end{equation}
    which, from its construction, factors as a composite
    \begin{equation}\label{eq:horizfactsthroughs}
        \Pull(\mc C, S)\to \Sq(\mc C\xrightarrow{s} \TwC \;{=\joinrel=}\; \TwC)\to  \Sq(\TwC)
    \end{equation}
    where the latter map is the forgetful map.\par
    Now, we may apply \Cref{lemma:expansion} to \eqref{eq:mainthmmapfortri} to produce a map
    \begin{equation*}
        \mathrm{PullSq}(\mc C, S)\longrightarrow \Sq^{(3)}(\TwC).
    \end{equation*}
    Up to reordering coordinates, we have an inclusion $\Pull(\mc C, E, S)\hookrightarrow \mathrm{PullSq}(\mc C, S)$ and we may restrict the above map to obtain a map
    \begin{equation}\label{eq:inducedtrisimplicialmap}
        \Pull(\mc C, E, S)\longrightarrow \Sq^{(3)}(\TwC).
    \end{equation}
    Now we set $\mb C = \Pull(\mc C, E, S)$ and consider the bisimplicial space $\mb C'$ given by
    \begin{equation*}
        \begin{tikzcd}
            \Delta^\mathrm{op}\times\Delta^\mathrm{op}\arrow{r}{\id\times \diag} &[30pt] \Delta^\mathrm{op}\times\Delta^\mathrm{op}\times\Delta^\mathrm{op}\arrow{r}{\mb C} & \Ani.
        \end{tikzcd}
    \end{equation*}
    Taking the diagonal along the first two dimensions of $\Sq^{(3)}(\TwC)$ and restricting along the diagonal yields a map
    \begin{equation}\label{eq:restrtodiag}
        \begin{tikzcd}[row sep=0pt]
            \Sq^{(3)}(\TwC)_{\bullet,\bullet,\ast}\arrow{r} & \Sq(\TwC) \\
            D : [n]\times [m]\times [m]\to \TwC\arrow[mapsto]{r} & {[n]\times [m]\xrightarrow{\id\times \diag} [n]\times [m]\times [m]\xrightarrow{D}\TwC}.
        \end{tikzcd}
    \end{equation}
    Restricting \eqref{eq:inducedtrisimplicialmap} to the diagonal on the first two dimensions then post-composing with \eqref{eq:restrtodiag} gives a map
    \begin{equation*}
        \mb C'\longrightarrow \Sq(\TwC).
    \end{equation*}\par
    By \eqref{eq:horizfactsthroughs}, this map factors as
    \begin{equation*}
        \mb C'\longrightarrow \Sq(\TwC \;{=\joinrel=}\; \TwC\xleftarrow{s} E)\longrightarrow  \Sq(\TwC)
    \end{equation*}
    and, by the pasting lemma for Cartesian squares, the projection $\mb C'\to\Sq(\mc C)$ lands inside $\Pull(\mc C, E)$. Thus we have a commutative square
    \begin{equation*}
        \begin{tikzcd}
            \mb C'\arrow{r}\arrow{d} & \Sq(\TwC \;{=\joinrel=}\; \TwC\xleftarrow{s} E)\arrow{d} \\
            \Pull(\mc C, E) \arrow{r} & \Sq(\mc C)
        \end{tikzcd}
    \end{equation*}
    which induces a map $\mb C'\to \Pull^{\mc K}(\mc C, E)$.\par
    Taking opposites in the second coordinate, we obtain a map
    \begin{equation*}
        F : (\mb C')^{2\mathrm{-op}}\longrightarrow \Pull^{\mc K}(\mc C, E)^{2\mathrm{-op}}.
    \end{equation*}
    We claim $F$ is an equivalence after gluing. The result will then follow since we have an equivalence $\Gr \Pull(\mc C, E, S)^{2,3\mathrm{-op}}\simeq \Gr (\mb C')^{2\mathrm{-op}}$ via diagonal approximation as $\diag \Pull(\mc C, E, S)^{2,3\mathrm{-op}} = \diag (\mb C')^{2\mathrm{-op}}$.\par
    To show this, we apply \Cref{lemma:dimcollapserightfib}. By \Cref{lemma:CntoC0rightfib,lemma:DntoD0rightfib}, the bisimplicial spaces occurring are factorial. Condition (i) follows as there is a natural chain of equivalences $(\mb C')^{2\mathrm{-op}}(-, 0)\simeq N(E)\simeq \Pull^{\mc K}(\mc C, E)^{2\mathrm{-op}}(-, 0)$ the composite of which corresponds to $F(-,0)$. Finally, condition (ii) follows from \Cref{lemma:ascatofD0} since the composite
    \begin{equation*}
        \begin{tikzcd}
            \diag \Pull(\mc C, S)^\mathrm{op}\simeq (\mb C')^{2\mathrm{-op}}(0,-)\arrow{r}{F(0,-)} &[30pt] \Pull^{\mc K}(\mc C, E)^{2\mathrm{-op}}(0, -)\arrow{r} & N((\TwC)^\mathrm{op})
        \end{tikzcd}
    \end{equation*}
    is the diagonal approximation map occurring under the identification $\Gr \Pull(\mc C, S)^\mathrm{op}\simeq (\TwC)^\mathrm{op}$ given by \Cref{thm:mainthmintext}.
\end{proof}

\subsection{Main result}

With the trisimplicial space description of the labeled span category, we are now equipped to construct the canonical enhancement of a given $3$-functor formalism to encode Poincar\'e duality and Thom twists.

\begin{theorem}\label{thm:extforweakcohom}
    Let $\mc D : \Span(\mc C, E)^\otimes\to \Cat$ be a $3$-functor formalism and $S$ the class of weakly cohomologically smooth morphisms. Then there exists a canonical extension of $\mc D$ along $\iota : \Span(\mc C, E)\to \Span^{\mc K}(\mc C, E)$ to a functor
    \begin{equation*}
        \mc D^\mathrm{Th} : \Span^{\mc K}(\mc C, E)\longrightarrow \Cat
    \end{equation*}
    where $\mc K$ is the functor
    \begin{equation*}
        \begin{tikzcd}[row sep=0pt]
            \mc C^\mathrm{op}\arrow{r} & \CAlg(\Ani) \\
            x\arrow[mapsto]{r} & \mc K(\Vect_{(\mc C, S)}(x)).
        \end{tikzcd}
    \end{equation*}
    Moreover, there is a canonical identification of the exceptional pullback
    \begin{equation*}
        (-)^! : S^\mathrm{op}\longrightarrow \Cat
    \end{equation*}
    with
    \begin{equation*}
        \begin{tikzcd}
            S^\mathrm{op} \arrow{r}{(-)_\natural^\mathrm{op}} & (\TwC)^\mathrm{op}\arrow{r} & \Span^{\mc K}(\mc C, E)\arrow{r}{\mc D^\mathrm{Th}} & \Cat.
        \end{tikzcd}
    \end{equation*}
\end{theorem}
\begin{proof}
    Under the equivalence $\Span(\mc C, E)\simeq \Gr \Pull(\mc C, E)^{2\mathrm{-op}}$, the $3$-functor formalism $\mc D$ corresponds to a map
    \begin{equation*}
        \Pull(\mc C, E)^{2\mathrm{-op}}\to \Sq(\Cat).
    \end{equation*}
    By a similarly proven variant of \Cref{lemma:expansion}, this corresponds uniquely to a map
    \begin{equation*}
        \mb C^{3\mathrm{-op}}\to \Sq^{(3)}(\Cat)
    \end{equation*}
    where $\mb C$ is the sub-trisimplicial space of $\Sq^{(3)}(\mc C)$ consisting of diagrams $D : [n]\times [m]\times [\ell]\to\mc C$ such that
    \begin{enumerate}
        \item for all $0\le i\le n$, $D|_{\{i\}\times [m]\times [\ell]}$ belongs to $\Pull(\mc C, E)_{m,\ell}$
        \item for all $0\le i\le m$, $D|_{[n]\times \{i\}\times [\ell]}$ belongs to $\Pull(\mc C, E)_{n,\ell}$.
    \end{enumerate}
    We then restrict this map to obtain a map
    \begin{equation}\label{eq:preadjointpd}
        \Pull(\mc C, E, S)^{3\mathrm{-op}}\longrightarrow \Sq^{(3)}(\Cat).
    \end{equation}
    From here, we curry out the first and third dimension. Doing this to $\Sq^{(3)}(\Cat)$ results in the functor
    \begin{equation*}
        \begin{tikzcd}[row sep=0pt]
            \Delta^\mathrm{op}\times\Delta^\mathrm{op}\arrow{r} & \PSh(\Delta) \\
            {([n],[\ell])}\arrow[mapsto]{r} & N(\Fun([n]\times [\ell], \Cat)).
        \end{tikzcd}
    \end{equation*}
    Now, for any category $C$ and $2$-category $\mc X$, the left adjoints in $\Fun(C, \mc X)$ consist precisely of right adjointable natural transformations. By \cite[Lemma 4.5.13]{heyer20246}, if $f$ is weakly cohomologically smooth, then for every Cartesian square
    \begin{equation*}
        \begin{tikzcd}
        	x & {y'} \\
        	y & z
        	\arrow["{\bar g}", from=1-1, to=1-2]
        	\arrow["{\bar f}"', from=1-1, to=2-1]
        	\arrow["\lrcorner"{anchor=center, pos=0.125}, draw=none, from=1-1, to=2-2]
        	\arrow["f"', from=1-2, to=2-2]
        	\arrow["g", from=2-1, to=2-2]
        \end{tikzcd}
    \end{equation*}
    in $\mc C$, the natural map $\bar g^\ast f^!\to \bar f^!g^\ast$ is an equivalence and if $g,\bar g\in E$, then the natural map $\bar g_!\bar f^!\to f^!g_!$ is an equivalence. It follows that, after currying the first and third dimensions, the map \eqref{eq:preadjointpd} lands inside the non-full subcategory of $\Fun([n]\times [\ell], \Cat)$ consisting of left adjoints. Lifting to homotopy coherent adjunctions and passing to right adjoints, we obtain a map
    \begin{equation*}
        \Pull(\mc C, E, S)^{2,3\mathrm{-op}}\longrightarrow \Sq^{(3)}(\Cat).
    \end{equation*}
    By \Cref{thm:trisimpgluing}, this corresponds to a map $\Span^{\mc K}(\mc C, E)\to \Cat$ which gives the desired extension.
\end{proof}

To finish, we also wish to enhance this extension to include all virtual vector bundles in the case where all the Thom twists happen to be invertible. For this we will prove the following proposition.

\begin{proposition}\label{prop:localizationoflabeledspan}
    Let $(\mc C, E, G)\in\mathrm{LabPair}$. Then, letting $G^\mathrm{gp}$ be the pointwise group completion of $G$, the induced map $\Span^G(\mc C, E)\to \Span^{G^\mathrm{gp}}(\mc C, E)$ is a localization at the class of morphisms of the form $X = (X, \xi) = X$.
\end{proposition}

For this, we recall a lemma.

\begin{lemma}\label{lemma:fibrationsandlocalization}
    Let
    \begin{equation*}
        \begin{tikzcd}
        	{\mc D} & {\mc D'} \\
        	{\mc C} & {\mc C[W^{-1}]}
        	\arrow[from=1-1, to=1-2]
        	\arrow["p"', from=1-1, to=2-1]
        	\arrow["\lrcorner"{anchor=center, pos=0.125}, draw=none, from=1-1, to=2-2]
        	\arrow[from=1-2, to=2-2]
        	\arrow[from=2-1, to=2-2]
        \end{tikzcd}
    \end{equation*}
    be a Cartesian square of categories where both vertical morphisms are left fibrations and $\mc C\to\mc C[W^{-1}]$ is a localization at a class $W$ of morphisms. Then $\mc D\to\mc D'$ is a localization of $\mc D$ at the collection of morphisms $p^{-1}(W)$.
\end{lemma}
\begin{proof}
    This follows from \cite[\href{https://kerodon.net/tag/038U}{Tag 038U}]{kerodon}.
\end{proof}

\begin{proof}[Proof of \Cref{prop:localizationoflabeledspan}]
    Let $Y = \Pull^{G}(\mc C, E)^{2\mathrm{-op}}$ and $X = \Pull^{G^\mathrm{gp}}(\mc C, E)^{2\mathrm{-op}}$. Then we have a natural map $Y\to X$ of bisimplicial spaces which is an equivalence on vertical fragments, both yielding $N(E)$. By \Cref{prop:equivofbaseimpbasecompat}, it follows that $Y\to X$ is base compatible. Hence, by \Cref{prop:CSpreservesnicestuff}, we have that $Y_\mathrm{CS}\to X_\mathrm{CS}$ is a base compatible map of factorial bisimplicial spaces.\par
    This yields a pullback square
    \begin{equation*}
        \begin{tikzcd}[column sep=40pt]
        	{\ascat Y_{\mathrm{CS}}(n,-)} & {\ascat X_{\mathrm{CS}}(n,-)} \\
        	{(\mc C^G)^\mathrm{op} \simeq \ascat Y_\mathrm{CS}(0,-)} & {\ascat X_\mathrm{CS}(0,-)\simeq (\mc C^{G^\mathrm{gp}})^\mathrm{op}}
        	\arrow[from=1-1, to=1-2]
        	\arrow[from=1-1, to=2-1]
        	\arrow["\lrcorner"{anchor=center, pos=0.125, rotate=0}, draw=none, from=1-1, to=2-2]
        	\arrow[from=1-2, to=2-2]
        	\arrow[from=2-1, to=2-2]
        \end{tikzcd}
    \end{equation*}
    of categories for all $n\ge 0$ where the vertical morphisms are left fibrations. Note that we have also applied \Cref{lemma:ascatofD0} to identify $\ascat Y_\mathrm{CS}(0,-)$ and $\ascat X_\mathrm{CS}(0,-)$. By \Cref{prop:locofCG} $\mc C^G\to \mc C^{G^\mathrm{gp}}$ is a localization at the collection of morphisms of the form $(\id_x, \xi)$ for $x\in\mc C$, $\xi\in G(x)$. We deduce from \Cref{lemma:fibrationsandlocalization} that
    \begin{equation*}
        \ascat Y_\mathrm{CS}(n, -)\longrightarrow \ascat X_\mathrm{CS}(n, -)
    \end{equation*}
    is a localization for every $n$.\par
    At this point, we recall some basic facts about epimorphisms. For a category $\mc D$, we say that a morphism $f : x\to y$ is an epimorphism if for every $z\in \mc D$, the induced map $\map_{\mc D}(y, z)\to \map_{\mc D}(x, z)$ is a monomorphism of spaces, i.e.\ an inclusion of connected components. Since a map of spaces is a monomorphism if and only if its diagonal is an equivalence, if $\mc D$ has coproducts then $f$ is an epimorphism precisely when the codiagonal of $f$ is an equivalence. By this codiagonal criterion, if $\mc D$ has coproducts we observe that epimorphisms in $\Fun(K, \mc D)$ are precisely the pointwise epimorphisms for any category $K$. From this, we learn that $Y_\mathrm{CS}\to X_\mathrm{CS}$ is an epimorphism in $\Fun(\Delta^\mathrm{op}, \Cat)$.\par
    Consider now the induced map
    \begin{equation*}
        \Span^G(\mc C, E)\simeq \Gr Y\simeq \Gr Y_\mathrm{CS}\longrightarrow \Gr X_\mathrm{CS}\simeq \Gr X\simeq \Span^{G^\mathrm{gp}}(\mc C, E).
    \end{equation*}
    For any category $\mc D$, we have a commutative diagram
    \begin{equation}\label{eq:commsquareloc}
        \begin{tikzcd}
        	{\map_{\Cat}(\Span^{G^\mathrm{gp}}(\mc C, E),\mc D)} & {\map_{\Cat}(\Span^G(\mc C, E), \mc D)} \\
        	{\map_{\Fun(\Delta^\mathrm{op}, \Cat)}(X_\mathrm{CS}, \Sq(\mc D))} & {\map_{\Fun(\Delta^\mathrm{op}, \Cat)}(Y_\mathrm{CS}, \Sq(\mc D))}.
        	\arrow[from=1-1, to=1-2]
        	\arrow["\simeq", from=1-1, to=2-1]
        	\arrow["\simeq"', from=1-2, to=2-2]
        	\arrow[from=2-1, to=2-2]
        \end{tikzcd}
    \end{equation}
    By the above, the bottom morphism is a monomorphism of spaces and hence so is the top morphism. Moreover, by taking degenerate diagrams we have a coCartesian section
    \begin{equation*}
        \ascat Y_\mathrm{CS}(0,-)\longrightarrow\ascat Y_\mathrm{CS}(n, -)
    \end{equation*}
    which surjects the morphisms inverted in the localization $\ascat Y_\mathrm{CS}(n, -)\longrightarrow \ascat X_\mathrm{CS}(n, -)$. It follows that a morphism $F : Y_\mathrm{CS}\to \Sq(\mc D)$ belongs to the essential image of
    \begin{equation*}
        \map_{\Fun(\Delta^\mathrm{op}, \Cat)}(X_\mathrm{CS}, \Sq(\mc D)) \longrightarrow \map_{\Fun(\Delta^\mathrm{op}, \Cat)}(Y_\mathrm{CS}, \Sq(\mc D))
    \end{equation*}
    if and only if $F_1 : \ascat Y_\mathrm{CS}(0,-)\to \mc D$ belongs to the essential image of
    \begin{equation*}
        \map_{\Cat}((\mc C^{G^\mathrm{gp}})^\mathrm{op}, \mc D)\simeq \map_{\Cat}(\ascat X_\mathrm{CS}(0,-), \mc D)\longrightarrow \map_{\Cat}(\ascat Y_\mathrm{CS}(0, -), \mc D)\simeq \map_{\Cat}((\mc C^{G})^\mathrm{op}, \mc D).
    \end{equation*}
    Translating this condition under the equivalences of \eqref{eq:commsquareloc}, we see that the top morphism surjects precisely those components consisting of $F : \Span^G(\mc C, E)\to\mc D$ which send each span of the form $X = (X,\xi) = X$ to an equivalence.\par
    Since the map $\Span^G(\mc C, E)\to \Span^{G^\mathrm{gp}}(\mc C, E)$ sends this collection $W$ of arrows to equivalences, we get an induced map $\Span^G(\mc C, E)[W^{-1}]\to \Span^{G^\mathrm{gp}}(\mc C, E)$ which gives an equivalence of corepresenting functors, as required.
\end{proof}

As a corollary, we may extend Thom twists to all virtual vector bundles when restricting attention to cohomologically smooth morphisms.

\begin{corollary}\label{cor:invThomtwistext}
    Let $\mc D : \Span(\mc C, E)^\otimes\to \Cat$ be a $3$-functor formalism and $S$ the class of cohomologically smooth morphisms. Then there exists a canonical extension of $\mc D$ along $\iota : \Span(\mc C, E)\to \Span^{K}(\mc C, E)$ to a functor
    \begin{equation*}
        \mc D^\mathrm{Th} : \Span^{K}(\mc C, E)\longrightarrow \Cat
    \end{equation*}
    where $K$ is the functor
    \begin{equation*}
        \begin{tikzcd}[row sep=0pt]
            \mc C^\mathrm{op}\arrow{r} & \CAlg(\Ani) \\
            x\arrow[mapsto]{r} & K(\Vect_{(\mc C, S)}(x)).
        \end{tikzcd}
    \end{equation*}
    Moreover, there is a canonical identification of the exceptional pullback
    \begin{equation*}
        (-)^! : S^\mathrm{op}\longrightarrow \Cat
    \end{equation*}
    with
    \begin{equation*}
        \begin{tikzcd}
            S^\mathrm{op} \arrow{r}{(-)_\natural^\mathrm{op}} & (\mc C^K)^\mathrm{op}\arrow{r} & \Span^{K}(\mc C, E)\arrow{r}{\mc D^\mathrm{Th}} & \Cat.
        \end{tikzcd}
    \end{equation*}
\end{corollary}
\begin{proof}
    Take the extension $\mc D^{\mathrm{Th}'}$ of $\mc D$ coming from \Cref{thm:extforweakcohom} and restrict it to $\Span^{\mc K}(\mc C, E)$ where $\mc K$ is the functor
    \begin{equation*}
        \begin{tikzcd}[row sep=0pt]
            \mc C^\mathrm{op}\arrow{r} & \CAlg(\Ani) \\
            x\arrow[mapsto]{r} & \mc K(\Vect_{(\mc C, S)}(x)).
        \end{tikzcd}
    \end{equation*}
    in the obvious manner. By \Cref{prop:localizationoflabeledspan} and \Cref{prop:partialKthygroupcmpl}, our claimed extension follows by showing that each Thom twist is invertible.\par
    By the same argument as in \Cref{cor:exttovirtvbexch}, it will suffice to show that for each vector bundle
    \begin{equation*}
        \begin{tikzcd}
            \xi = x\arrow{r}{s} & e\arrow[two heads]{r}{p} & x,
        \end{tikzcd}
    \end{equation*}
    the associated Thom twist is invertible. But we have that
    \begin{equation*}
        \begin{aligned}
            \Sigma^{\xi} &\simeq s^\ast p^! \\
            &\simeq s^\ast(p^!(1_x)\otimes p^\ast) \\
            &\simeq s^\ast p^!(1_x)\otimes (-).
        \end{aligned}
    \end{equation*}
    Since $p$ is cohomologically smooth, $p^!(1_x)$ is $\otimes$-invertible and hence so is $s^\ast p^!(1_x)$, as required.
\end{proof}

\bibliographystyle{alpha}
\bibliography{references}

\end{document}